\pdfoutput=1
\RequirePackage{ifpdf}
\ifpdf 
\documentclass[pdftex]{sigma}
\else
\documentclass{sigma}
\fi

\def\s{{\mathfrak{S}}}

\def\Z{\mathbb{Z}}
\def\Q{\mathbb{Q}}
\def\R{\mathbb{R}}
\def\C{\mathbb{C}}
\def\N{\mathbb{N}}

\numberwithin{equation}{section}

\newtheorem{Theorem}{Theorem}[section]
\newtheorem{Corollary}[Theorem]{Corollary}
\newtheorem{Lemma}[Theorem]{Lemma}
\newtheorem{Proposition}[Theorem]{Proposition}
\newtheorem{Problem}[Theorem]{Problem}
\newtheorem{Problems}[Theorem]{Problems}
\newtheorem{Project}[Theorem]{Project}
\newtheorem{Conjecture}[Theorem]{Conjecture}
\newtheorem{Question}[Theorem]{Question}
\newtheorem{Claim}[Theorem]{Claim}
 { \theoremstyle{definition}
\newtheorem{Definition}[Theorem]{Definition}
\newtheorem{Example}[Theorem]{Example}
\newtheorem{Examples}[Theorem]{Examples}
\newtheorem{Remark}[Theorem]{Remark}
\newtheorem{Comments}[Theorem]{Comments}
\newtheorem{Exercises}[Theorem]{Exercises}
\newtheorem{Exercise}[Theorem]{Exercise}
\newtheorem{Challenge}[Theorem]{Challenge}
 }

\begin{document}

\allowdisplaybreaks

\newcommand{\arXivNumber}{1502.00426}

\renewcommand{\PaperNumber}{002}

\FirstPageHeading

\ShortArticleName{On Some Quadratic Algebras}

\ArticleName{On Some Quadratic Algebras I $\boldsymbol{\frac{1}{2}}$: \\
Combinatorics of
 Dunkl and Gaudin Elements,\\ Schubert,
Grothendieck, Fuss--Catalan,\\ Universal Tutte and Reduced Polynomials}

\Author{Anatol N.~{KIRILLOV}~$^{\dag\ddag\S}$}

\AuthorNameForHeading{A.N.~Kirillov}

\Address{$^\dag$~Research Institute of Mathematical Sciences (RIMS), Kyoto, Sakyo-ku 606-8502, Japan}
\EmailD{\href{mailto:kirillov@kurims.kyoto-u.ac.jp}{kirillov@kurims.kyoto-u.ac.jp}}
\URLaddressD{\url{http://www.kurims.kyoto-u.ac.jp/~kirillov/}}

\Address{$^\ddag$~The Kavli Institute for the Physics and Mathematics of the Universe (IPMU),\\
\hphantom{$^\ddag$}~5-1-5 Kashiwanoha, Kashiwa, 277-8583, Japan}

\Address{$^\S$~Department of Mathematics, National Research University Higher School of Economics, \\
\hphantom{$^\S$}~7~Vavilova Str., 117312, Moscow, Russia}

\ArticleDates{Received March 23, 2015, in f\/inal form December 27, 2015; Published online January 05, 2016}

\Abstract{We study some combinatorial and algebraic properties of certain
quadratic algebras related with dynamical classical and classical Yang--Baxter equations.}


\Classification{14N15; 53D45; 16W30}

\begin{flushright}
\begin{minipage}{90mm}
\it To the memory of Alain Lascoux 1944--2013, the great Mathematician, from
whom I have learned a~lot about the Schubert and Grothendieck polynomials.
\end{minipage}
\end{flushright}

{\small \tableofcontents}

\bigskip

\subsection*{Extended abstract}

We introduce and study a certain class of quadratic algebras,
which are nonhomogeneous in general,
together with the distinguish set of mutually commuting elements inside of
each, the so-called {\it Dunkl elements}. We describe relations among the
Dunkl elements in the case of a family of quadratic algebras corresponding to
 a certain splitting of the {\it universal classical Yang--Baxter relations}
into two {\it three term relations}. This result is a further extension
and generalization of analogous results obtained in~\cite{FK,P} and~\cite{KM}. As an application we describe explicitly the set of relations among
 the Gaudin elements in the group ring of the symmetric group, cf.~\cite{MTV}.
We also study relations among the Dunkl elements in the case of
(nonhomogeneous) quadratic algebras related with the {\it universal dynamical
 classical Yang--Baxter relations}. Some relations of results obtained in
papers~\cite{FK,K,KM3} with those obtained in~\cite{GRTV}
are pointed out. We also identify a subalgebra generated by the generators
corresponding to the {\it simple roots} in the extended Fomin--Kirillov
algebra with the $\rm DAHA$, see Section~\ref{section4.3}.

The set of generators of algebras in question, naturally corresponds to the
set of edges of the {\it complete graph} $K_n$ (to the set of edges and loops
of the complete graph with (simple) loops~${\widetilde{K}}_n$ in
{\it dynamical} and {equivariant} cases). More generally, starting from any
subgraph $\Gamma$ of the complete graph with simple loops~${\widetilde{K}}_{n}$ we def\/ine a (graded) subalgebra $3T_n^{(0)}(\Gamma)$ of
the (graded) algebra $3T_n^{(0)}({\widetilde{K}}_n)$~\cite{K2}. In the case
of loop-less graphs $\Gamma \subset K_n$ we state {\it conjecture},
 Conjecture~\ref{conjecture4.2} in the main text, which relates the {\it Hilbert polynomial}
of the abelian quotient~$3T_n^{(0)}(\Gamma)^{ab}$ of the algebra $3T_n^{(0)}(\Gamma)$ and the {\it chromatic polynomial} of the graph~$\Gamma$ we are
started with\footnote{We {\it expect} that a similar conjecture is true for any f\/inite
(oriented) matroid~$\cal{M}$. Namely, one (A.K.) can def\/ine an analogue of the
three term relations algebra $3T^{(0)}({\cal{M}})$ for any (oriented) matroid~$\cal{M}$. We {\it expect} that the abelian quotient
$3T^{(0)}({\cal{M}})^{ab}$ of the algebra $3T^{(0)}({\cal{M}})$ is
isomorphic to the {\it Orlik--Terao}
 algebra~\cite{OTe}, denoted by ${\rm OT}({\cal{M}})$ (known also as {\it even}
 version of the Orlik--Solomon algebra, denoted by ${\rm OS}^{+}({\cal{M}})$ )
associated with matroid~$\cal{M}$~\cite{Cor}.
 Moreover, the anticommutative quotient of the {\it odd} version of the
algebra $3T^{(0)}({\cal{M}})$, as we {\it expect}, is isomorphic to the Orlik--Solomon algebra ${\rm OS}({\cal{M}})$ associated with matroid~${\cal{M}}$, see, e.g.,
\cite{B, GR}. In particular,
\begin{gather*}
{\rm Hilb}(3T^{(0)}\big({\cal{M}})^{ab},t\big)= t^{r({\cal{M}})} {\rm Tutte}\big({\cal{M}}; 1+t^{-1},0\big).
\end{gather*}
We {\it expect} that the Tutte polynomial of a matroid, ${\rm Tutte}({\cal{M}},x,y)$,
 is related with the Betti polynomial of a~matroid $\cal{M}$. Replacing
relations $u_{ij}^2=0$, $\forall\, i, j$, in the def\/inition of the algebra
$3T^{(0)}(\Gamma)$ by relations $u_{ij}^2=q_{ij}$, $\forall\, i, j$, $(i,j) \in E(\Gamma)$, where $\{q_{ij}\}_{(i,j) \in E(\Gamma)}$, $q_{ij}=q_{ji}$, is a~collection of {\it central} elements, give rise to a~{\it quantization} of the Orlik--Terao algebra ${\rm OT}(\Gamma)$. It seems an interesting {\it task} to clarify
combinatorial/geometric signif\/icance of {\it noncommutative} versions of
Orlik--Terao algebras (as well as Orlik--Solomon ones) def\/ined as follows:
 ${\cal{OT}}(\Gamma) := 3T^{(0)}(\Gamma)$, its ``quantization''~$3T^{({\boldsymbol{q}})}(\Gamma)^{ab}$ and $K$-theoretic analogue $3T^{({\boldsymbol{q}})}(\Gamma, \beta)^{ab}$, cf.\ Def\/inition~\ref{definition3.1}, in the theory of
hyperplane arrangements.
{\it Note} that a small modif\/ication of arguments in~\cite{Li} as were used
for the proof of our Conjecture~\ref{conjecture4.2}, gives rise to a theorem that the
algebra $3T_n(\Gamma)^{ab}$ is isomorphic to the Orlik--Terao algebra~${\rm OT}(\Gamma)$ studied in~\cite{ScT}.\label{footnote1}}\footnote{In the case of simple graphs our Conjecture~\ref{conjecture4.2} has been proved in~\cite{Li}.}.
We check our conjecture for the complete graphs~$K_n$ and the complete bipartite graphs~$K_{n,m}$. Besides, in the case of
{\it complete multipartite graph $K_{n_1,\ldots,n_r}$}, we identify the
commutative subalgebra in the algebra $3T_{N}^{(0)}(K_{n_1,\ldots,n_r})$,
 $N=n_1+\cdots+n_r$, generated by the elements
\begin{gather*}
\theta_{j,k_j}^{(N)}:=e_{k_j}\big(\theta_{N_{j-1}+1}^{(N)},\ldots,
\theta_{N_{j}}^{(N)}\big), \\ 1 \le j \le r, \quad 1 \le k_j \le n_j, \quad N_j:=n_1+\cdots+
n_j, \quad N_0=0,
\end{gather*}
 with the cohomology ring $H^{*}({\cal{F}}l_{n_1,\ldots,n_r},\Z)$ of the
partial f\/lag variety ${\cal{F}}l_{n_1,\dots,n_r}$. In other words, the set of
(additive) Dunkl elements $\big\{ \theta_{N_{j-1}+1}^{(N)},\ldots,
\theta_{N_j}^{(N)}\big\}$ plays a role of the {\it Chern roots} of the
tautological vector bundles~$\xi_j$, $j=1,\ldots,r$, over the partial f\/lag
variety ${{\cal{F}}l}_{n_1,\ldots,n_r}$, see Section~\ref{section4.1.2} for details.
In a similar fashion, the set of {\it multiplicative} Dunkl elements
$\big\{\Theta_{N_{j-1}+1}^{(N)},\ldots,\Theta_{N_{j}}^{(N)} \big\}$ plays a role of
the {\it $K$-theoretic version of Chern roots} of the tautological vector
bundle~$\xi_j$ over the~partial f\/lag variety ${\cal{F}}l_{n_1,\ldots,n_r}$.
As a byproduct for a given set of weights ${\boldsymbol{\ell}} = \{ \ell_{ij} \}_{1 \le i < j \le r}$ we compute the {\it Tutte polynomial}
$T(K_{n_1,\ldots,n_k}^{({\boldsymbol{\ell}})},x,y)$ of the ${\boldsymbol{\ell}}$-weighted
complete multipartite graph $K_{n_1,\ldots,n_k}^{({\boldsymbol{\ell}})}$, see
Section~\ref{section4}, Def\/inition~\ref{section4.4} and Theorem~\ref{section4.3}.
More generally, we introduce {\it universal Tutte polynomial}
\begin{gather*}
T_n(\{q_{ij}\},x,y) \in \Z[\{q_{ij}\}] [x,y]
\end{gather*}
in such a way that for any collection of non-negative integers ${\boldsymbol{m}}=
\{ m_{ij} \}_{1 \le i < j \le n}$
and a subgraph $\Gamma \subset K_n^{({\boldsymbol{m}})}$ of the weighted complete
graph on $n$ labeled vertices with each edge $(i,j) \in K_n^{({\boldsymbol{m}})}$
appears with multiplicity~$m_{ij}$, the specialization
\begin{gather*}
 q_{ij} \longrightarrow 0 \quad \text{if edge} \ \ (i,j) \notin \Gamma,\qquad
q_{ij} \longrightarrow [m_{ij}]_y := \frac{y^{m_{ij}}-1}{y-1} \quad \text{if edge} \ \ (i,j)
\in \Gamma
\end{gather*}
of the universal Tutte polynomial is equal to the Tutte polynomial of graph
$\Gamma$ multiplied by $(x-1)^{\kappa(\Gamma)}$, see Section~\ref{section4.1.2},
Theorem~\ref{theorem4.3}, and {\it comments and examples}, for details.

 We also introduce and study a family of {\it $($super$)$ $6$-term relations}
 algebras, and suggest a~def\/inition of ``multiparameter quantum
deformation'' of the algebra of the curvature of $2$-forms of the Hermitian
linear bundles over the complete f\/lag variety ${\cal{F}}l_n$. This algebra
can be treated as a natural generalization of the (multiparameter) quantum
cohomology ring $QH^{*}({\cal{F}}l_n)$, see Section~\ref{section4.2}. In a similar fashion as in the case of three term relations algebras, for any subgraph~$\Gamma \subset K_n$, one~(A.K.) can also def\/ine an algebra
$6T^{(0)}(\Gamma)$ and projection\footnote{We treat this map as an algebraic version of the homomorphism which
sends the curvature of a Hermitian vector bundle over a smooth algebraic variety to its cohomology class, as well as a splitting of classical Yang--Baxter
relations (that is six term relations) in a couple of three term relations.}
\begin{gather*}
\text{Ch}\colon \ 6T^{(0)}(\Gamma) \longrightarrow 3T^{(0)}(\Gamma).
\end{gather*}
Note that subalgebra ${\cal{A}}(\Gamma):= {\mathbb{Q}}[\theta_1,\ldots,\theta_n] \subset 6T^{(0)}(\Gamma)^{ab}$
 generated by additive Dunkl elements
\begin{gather*}
 \theta_i = \sum_{j \atop (ij) \in E(\Gamma)} u_{ij}
 \end{gather*}
is closely related with problems have been studied in \cite{PSS, SS}, \dots, and \cite{T} in the case $\Gamma= K_n$, see Section~\ref{section4.2.2}. We want to draw
attention of the reader to the following {\it problems} related with {\it arithmetic Schubert\footnote{See for example~\cite{T} and the literature quoted therein.}
 and Grothendieck calculi}:
 \begin{enumerate}\itemsep=0pt
\item[(i)] Describe (natural) quotient $6T^{\dagger}(\Gamma)$ of the algebra
$6T^{(0)}(\Gamma)$ such that the natural epimorphism
${\rm pr}\colon {\mathbb{A}}(\Gamma) \longrightarrow {\cal{A}}(\Gamma)$ turns out to be {\it isomorphism}, where we denote by
${\mathbb{A}}(\Gamma)$ a~subalgebra of~$6T^{\dagger}(\Gamma)$ generated over~${\mathbb{Q}}$ by additive Dunkl elements.

\item[(ii)] It is not dif\/f\/icult to see \cite{K} that {\it multiplicative} Dunkl
elements $\{ \Theta_i \}_{1 \le i \le n}$ also mutually commute in the algebra~$6T^{(0)}$, cf.\ Section~\ref{section3.2}. {\it
Problem} we are interested in is to describe commutative subalgebras generated by {\it multiplicative} Dunkl elements in the algebras~$6T^{\dagger}(\Gamma)$ and~$6T^{(0)}(\Gamma)^{ab}$. In the latter case one will come to the $K$-theoretic version of algebras studied in~\cite{PSS}, \dots.
\end{enumerate}

Yet another objective of our paper\footnote{This part of our paper had its origin in the study/computation of
relations among the additive and multiplicative Dunkl elements in the quadratic
 algebras we are interested in, as well as the author's attempts to construct
a monomial basis in the algebra~$3T_n^{(0)}$ and f\/ind its Hilbert series for~$n \ge 6$. As far as I'm aware these {\it problems} are still widely open.}
 is to describe several combinatorial
properties of some special elements in the associative quasi-classical
Yang--Baxter algebras~\cite{K}, including among others, the so-called
{\it Coxeter element} and the {\it longest element}. In the
{\it case} of {\it Coxeter element} we relate the corresponding reduced
polynomials introduced in~\cite[Exercise~6.C5(c)]{ST3}, and independently in
\cite{K}, cf.~\cite{K2}, with the $\beta$-Grothendieck polynomials~\cite{FK1} for some special permutations~$\pi_{k}^{(n)}$. More generally, we
identify the $\beta$-Grothendieck polynomial~$\mathfrak{G}_{\pi_{k}^{(n)}}^{(\beta)}(X_n)$ with a~certain weighted sum running over the set of $k$-dissections of a convex $(n+k+1)$-gon. In particular we show that the specialization
$\mathfrak{G}_{\pi_{k}^{(n)}}^{(\beta)}(1)$ of
the $\beta$-Grothendieck polynomial $\mathfrak{G}_{\pi_{k}^{(n)}}^{(\beta)}(X_n)$ counts the number of {\it $k$-dissections} of a convex $(n+k+1)$-gon
according to the number of diago\-nals involved. When the number of diagonals in
 a $k$-dissection is the maximal possible (equals to $n(2 k-1)-1$), we
recover the well-known fact that the number of $k$-triangulations of a convex
$(n+k+1)$-gon is equal to the value of a certain Catalan--Hankel determinant,
see, e.g.,~\cite{SS}. In Section~\ref{section5.4.2} we study multiparameter generalizations
of reduced polynomials associated with Coxeter elements.

We also show that for a certain $5$-parameters family of vexillary permutations, the specia\-li\-za\-tion $x_i=1$, $\forall\, i \ge 1$, of the
corresponding {\it $\beta$-Schubert} polynomials
${\s}_{w}^{(\beta)}(X_n)$ turns out to be coincide either with the
Fuss--Narayana polynomials and their generalizations, or with a
$(q,\beta)$-deformation of $\rm VSASM$ or that of $\rm CSTCPP$ numbers, see
Corollary~\ref{corollary5.2}{\rm B}.
As examples we show that
\begin{enumerate}\itemsep=0pt
\item[(a)] the reduced polynomial corresponding to
a monomial $x_{12}^{n} x_{23}^{m}$ counts the number of {\it $(n,m)$-Delannoy
paths} according to the number of $NE$-steps, see Lemma~\ref{lemma5.3};

\item[(b)]
if $\beta=0$, the reduced polynomial corresponding to monomial
$(x_{12} x_{23})^n x_{34}^{k}$, $n \ge k$, counts the number of~$n$
up, $n$ down permutations in the symmetric group ${\mathbb{S}}_{2n+k+1}$, see
Proposition~\ref{proposition5.10}; see also Conjecture~\ref{conjecture5.10}.
\end{enumerate}

 We also point out on a conjectural connection between the sets of
{\it maximal} compatible sequences for the permutation $\sigma_{n,2n,2,0}$
and that $\sigma_{n,2n+1,2,0}$ from one side, and the set of ${\rm VSASM}(n)$ and
that of ${\rm CSTCPP}(n)$ correspondingly, from the other, see Comments~\ref{comments5.7} for
details. Finally, in Sections~\ref{section5.1.1} and~\ref{section5.4.1} we introduce and study a~multiparameter generalization of reduced polynomials considered
in~\cite[Exercise~6.C5(c)]{ST3}, as well as that of the Catalan, Narayana and
(small) Schr\"oder numbers.

 In the {\it case} of the {\it longest element} we relate the
corresponding reduced polynomial with the Ehrhart polynomial of the
Chan--Robbins--Yuen polytope, see Section~\ref{section5.3}. More generally, we relate the
$(t,\beta)$-reduced polynomial corresponding to monomial
\begin{gather*}
 \prod_{j=1}^{n-1} x_{j,j+1}^{a_{j}} \prod_{j=2}^{n-2}
\left(\prod_{k=j+2}^{n} x_{jk} \right), \qquad a_j \in \Z_{\ge 0}, \qquad \forall\, j,
\end{gather*}
with positive $t$-deformations of the Kostant partition function and that of
the Ehrhart polynomial of some f\/low polytopes, see Section~\ref{section5.3}.

In Section~\ref{section5.4} we investigate reduced polynomials associated with certain
monomials in the algebra $({\widehat{{\rm ACYB}}})^{ab}_n(\beta)$, known also as
Gelfand--Varchenko algebra~\cite{K3,K}, and study its combinatorial properties. Our main objective in Section~\ref{section5.4.2} is to study reduced polynomials for Coxeter element treated in a certain multiparameter deformation of the
 (noncommutative) quadratic algebra ${\widehat{{\rm ACYB}}}_n(\alpha,\beta)$.
Namely, to each dissection of a convex $(n+2)$-gon we associate a certain
weight and consider the generating function of all dissections of~$(n+2)$-gon selected taken with that weight. One can show that the reduced polynomial
corresponding to the Coxeter element in the deformed algebra is equal to that
generating function. We show that certain specializations of that reduced
polynomial coincide, among others, with the Grothendieck polynomials corresponding to the permutation $1 \times w_{0}^{(n-1)} \in \mathbb{S}_n$, the Lagrange inversion formula, as well as give rise to combinatorial (i.e., positive
expressions) multiparameters deformations of Catalan and Fuss--Catalan, Motzkin, Riordan and Fine numbers, Schr\"{o}der numbers and Schr\"{o}der trees. We
{\it expect} (work in progress) a~similar connections between Schubert and Grothendieck polynomials associated with the Richardson permutations $1^{k} \times
w_{0}^{(n-k)}$, $k$-dissections of a convex $(n+k+1)$-gon investigated in the
present paper, and $k$-dimensional Lagrange--Good inversion formula studied
from combinatorial point of view, e.g., in~\cite{CL, Ge}.

\section{Introduction}\label{section1}

The Dunkl operators have been introduced in the later part of 80's of the
last century by Charles Dunkl \cite{Du,Du1} as a powerful mean to
study of harmonic and orthogonal polynomials related with f\/inite Coxeter
groups. In the present paper we don't need the def\/inition of Dunkl
 operators for arbitrary (f\/inite) Coxeter groups, see, e.g.,~\cite{Du}, but only for the special case of the symmetric group ${\mathbb S}_n$.

\begin{Definition}\label{definition1.1} Let $P_n= \C[x_1,\ldots,x_n]$ be the ring of polynomials in
variables $x_1,\ldots,x_n$. The type~$A_{n-1}$ (additive) rational Dunkl
operators $D_1, \ldots, D_n$ are the dif\/ferential-dif\/ference operators of
the following form
\begin{gather}\label{equation1.1}
 D_i= \lambda {\partial \over \partial x_i} + \sum_{j \not= i}
{1-s_{ij} \over x_i-x_j},
\end{gather}
Here $s_{ij}$, $1 \le i < j \le n$, denotes the exchange (or permutation)
operator, namely,
\begin{gather*}
s_{ij}(f)(x_1,\ldots,x_i,\ldots,x_j,\ldots, x_n)=f(x_1,\ldots,x_j,\ldots,
x_i,\ldots, x_n),
\end{gather*}
${\partial \over \partial x_i}$ stands for the derivative w.r.t.\ the variable
$x_i$, $\lambda \in \C$ is a parameter.
\end{Definition}

The key property of the Dunkl operators is the following result.

\begin{Theorem}[C.~Dunkl~\cite{Du}]\label{theorem1.1} For any finite Coxeter group $(W,S)$,
where $S=\{s_1,\ldots,s_l\}$ denotes the set of simple reflections, the
Dunkl operators $D_i:=D_{s_{i}}$ and $D_j:=D_{s_{j}}$ pairwise
commute: $D_i D_j=
D_jD_i$, $1 \le i ,j \le l$.
\end{Theorem}

Another fundamental property of the Dunkl operators which f\/inds a wide variety
of applications in the theory of integrable systems, see, e.g.,~\cite{HW}, is
the following statement:
the operator
\begin{gather*}
\sum_{i=1}^{l} (D_i)^2
\end{gather*}
``essentially'' coincides with the Hamiltonian of the rational Calogero--Moser
model related to the f\/inite Coxeter group~$(W,S)$.

\begin{Definition}\label{definition1.2}
 Truncated (additive) Dunkl operator (or the Dunkl operator at
critical level), denoted by ${\cal D}_i$, $i=1,\ldots,l$,~is an operator of
the form~\eqref{equation1.1} with parameter $\lambda=0$.
\end{Definition}

For example, the type $A_{n-1}$ rational truncated Dunkl operator has the
following form
\begin{gather*}
 {\cal D}_i = \sum_{j \not= i} {1-s_{ij} \over x_i-x_j }.
\end{gather*}

Clearly the truncated Dunkl operators generate a commutative algebra.
The important property of the truncated Dunkl operators is the following
result discovered and proved by C.~Dunkl~\cite{Du1}; see also~\cite{Ba} for a~more recent proof.

\begin{Theorem}[C.~Dunkl~\cite{Du1}, Yu.~Bazlov~\cite{Ba}] \label{theorem1.2} For any finite Coxeter
group $(W,S)$ the algebra over $\Q$ generated by the truncated Dunkl
operators ${\cal D}_1,\ldots,{\cal D}_l$ is canonically isomorphic to the
coinvariant algebra ${\cal{A}}_W$ of the Coxeter group~$(W,S)$.
\end{Theorem}

Recall that for a f\/inite {\it crystallographic} Coxeter group~$(W,S)$ the
coinvariant algebra~${\cal{A}}_W$ is isomorphic to the cohomology ring
$H^{*}(G/B, \Q)$ of the f\/lag variety~$G/B$, where $G$ stands for the Lie group
corresponding to the crystallographic Coxeter group~$(W,S)$ we started with.

\begin{Example} \label{example1.1}
In the case when $W= {\mathbb S}_n$ is the symmetric group,
Theorem~\ref{theorem1.2} states that the algebra over~$\Q$ generated by the truncated
Dunkl operators ${\cal D}_i=\sum\limits_{j \not= i} {1-s_{ij} \over x_i-x_j}$, $i=1,
 \ldots,n$, is canonically isomorphic to the cohomology ring of the full f\/lag
 variety ${\cal F}l_n$ of type~$A_{n-1}$
\begin{gather}\label{equation1.2}
\Q[{\cal D}_1,\ldots,{\cal D}_n] \cong \Q[x_1,\ldots,x_n] / J_n,
\end{gather}
where $J_n$ denotes the ideal generated by the elementary symmetric
polynomials $ \{e_k(X_n), \, 1 \le k \le n\}$.

Recall that the elementary symmetric polynomials~$e_i(X_n)$,
$i=1,\ldots,n$, are def\/ined through the generating function
\begin{gather*}
 1+\sum_{i=1}^{n} e_i(X_n) t^{i}=\prod_{i=1}^{n} (1+t x_i),
 \end{gather*}
where we set $X_n:=(x_1,\ldots,x_n)$. It is well-known that in the case
 $W=\mathbb{S}_n$, the isomorphism~\eqref{equation1.2} can be def\/ined over the ring of
integers $\Z$.
\end{Example}

Theorem~\ref{theorem1.2} by C.~Dunkl has raised a number of natural questions:
\begin{enumerate}\itemsep=0pt
\item[(A)] What is the algebra generated by the {\it truncated}
\begin{itemize}\itemsep=0pt
\item trigonometric,
\item elliptic,
\item super, matrix, \dots,
\begin{enumerate}\itemsep=0pt
\item[(a)] additive Dunkl operators?

\item[(b)] Ruijsenaars--Schneider--Macdonald operators?

\item[(c)] Gaudin operators?
\end{enumerate}
\end{itemize}

\item[(B)] Describe commutative subalgebra generated by the Jucys--Murphy
elements in
\begin{itemize}\itemsep=0pt
\item the group ring of the symmetric group;
\item the Hecke algebra;

\item the Brauer algebra, ${\rm BMW}$ algebra, \dots.
\end{itemize}

\item[(C)] Does there exist an analogue of Theorem~\ref{theorem1.2} for
\begin{itemize} \itemsep=0pt

\item classical and quantum equivariant cohomology and equivariant
$K$-theory rings of the partial f\/lag varieties?

\item chomology and $K$-theory rings of af\/f\/ine f\/lag varieties?

\item diagonal coinvariant algebras of f\/inite Coxeter groups?

\item complex ref\/lection groups?
\end{itemize}
\end{enumerate}

The present paper is an extended introduction to a few items from Section~5 of~\cite{K}.

The main purpose of my paper ``On some quadratic algebras,~II'' is to give
some partial answers on the above questions basically in the case of the
symmetric group~${\mathbb S}_n$.

The purpose of the {\it present paper} is to draw attention to an
interesting class of nonhomogeneous quadratic algebras closely connected
(still mysteriously!) with dif\/ferent branches of Mathematics such as
classical and quantum Schubert and Grothendieck calculi,
low-dimensional topology,
classical, basic and elliptic hypergeometric functions,
algebraic combinatorics and graph theory,
integrable systems, etc.

What we try to explain in~\cite{K} is that upon passing to {\it a suitable
representation} of the quadratic algebra in question, the subjects mentioned
above, are a manifestation of certain general properties of that quadratic
algebra.

From this point of view, we treat the commutative subalgebra generated (over a~universal {\it Lazard} ring~${\mathbb{L}}_n$~\cite{Lazard1955})
by
the additive (resp.\ multiplicative) truncated Dunkl elements in the algebra
$3T_n(\beta)$, see Def\/inition~\ref{definition3.1}, as {\it universal cohomology}
(resp.\ {\it universal $K$-theory})~ring of the complete f\/lag variety
${\cal F}l_n$. The classical or quantum
cohomology (resp.\ the classical or quantum $K$-theory) rings of the f\/lag
variety~${\cal F}l_n$ are certain quotients of that {\it universal ring.}

For example, in~\cite{KM2} we have computed relations among the (truncated)
Dunkl elements $\{ \theta_i, \, i=1,\ldots,n \}$ in the {\it elliptic
representation} of the algebra $3T_n(\beta=0)$. We {\it expect} that the
commutative subalgebra obtained is isomorphic to {\it elliptic cohomology
ring} (not def\/ined yet, but see~\cite{GRa, GKV}) of the f\/lag variety~${\cal F}l_n$.

Another example from \cite{K}. Consider the algebra $3T_n(\beta=0)$.
One can prove~\cite{K} the following {\it identities} in the algebra
$3T_n(\beta=0)$:
\begin{enumerate}\itemsep=0pt
\item[(A)] {\it summation formula}
\begin{gather*}
 \sum_{j=1}^{n-1}
 \left( \prod_{b=j+1}^{n-1} u_{b,b+1} \right) u_{1,n}
\left(\prod_{b=1}^{j-1} u_{b,b+1} \right)=\prod_{a=1}^{n-1} u_{a,a+1};
\end{gather*}

\item[(B)] {\it duality transformation formula}, let $m \le n$, then
\begin{gather*}
 \sum_{j=m}^{n-1} \left( \prod_{b=j+1}^{n-1} u_{b,b+1} \right)
\left[\prod_{a=1}^{m-1} u_{a,a+n-1} u_{a,a+n} \right] u_{m,m+n-1}
\left( \prod_{b=m}^{j-1} u_{b,b+1} \right) \\
\qquad\quad{} +
\sum_{j=2}^{m}\left[\prod_{a=j}^{m-1}u_{a,a+n-1} u_{a,a+n} \right] u_{m,n+m-1} \left(\prod_{b=m}^{n-1} u_{b,b+1} \right) u_{1,n}\\
\qquad {} = \sum_{j=1}^{m} \left[ \prod_{a=1}^{m-j} u_{a,a+n} u_{a+1,a+n} \right] \left(
\prod_{b=m}^{n-1}u_{b,b+1} \right) \left[ \prod_{a=1}^{j-1}
u_{a,a+n-1} u_{a,a+n} \right].
\end{gather*}
\end{enumerate}

One can check that upon passing to the {\it elliptic representation} of the
algebra $3T_n(\beta=0)$, see Section~\ref{section3.1} or~\cite{KM2}, for the def\/inition of {\it elliptic representation}, the above
identities~(A) and~(B) f\/inally end up correspondingly, to be
the {\it summation formula} and the $N=1$ case of the {\it duality
transformation formula} for multiple elliptic hypergeometric series (of type
$A_{n-1})$, see, e.g.,~\cite{NK} or Appendix~\ref{appendixA.6} for the explicit forms of the latter. After passing to the so-called {\it Fay representation}~\cite{K}, the
identities~(A) and~(B) become correspondingly to be the
summation formula and duality transformation formula for the Riemann theta
functions of genus $g > 0$~\cite{K}. These formulas in the case~$g \ge 2$
seems to be new.

Worthy to mention that the relation (A) above can be treated as
a~``non-commutative analogue'' of the well-known recurrence relation among
the {\it Catalan numbers}. The study of ``descendent relations'' in the
quadratic algebras in question was originally motivated by the author
attempts to construct a~{\it monomial basis} in the algebra $3T_n^{(0)}$,
and compute ${\rm Hilb}(3T_n^{(0)},t)$ for $n \ge 6$.
These problems are still widely open, but gives rise the author to discovery
 of several interesting connections with
\begin{itemize}\itemsep=0pt
\item classical and quantum Schubert and Grothendieck calculi,
\item combinatorics of reduced decomposition of some special elements in
the symmetric group,
\item combinatorics of generalized {\it Chan--Robbins--Yuen} polytopes,
\item relations among the Dunkl and Gaudin elements,
\item computation of Tutte and chromatic polynomials of the weighted complete
multipartite graphs, etc.
\end{itemize}

A few words about the content of the {\it present paper}.
Example~\ref{example1.1} can be viewed as an illustration of the main problems we are
treated in Sections~\ref{section2} and~\ref{section3} of the present paper, namely the following ones.
\begin{itemize}\itemsep=0pt
\item Let $\{u_{ij},\, 1 \le i,j \le n\}$ be a set of generators of a certain algebra over a commutative ring~$K$. The f\/irst {\it problem} we are
interested in is to describe
``a natural set of relations'' among the generators
$\{u_{ij}\}_{1 \le i,j \le n}$ which
implies the pair-wise commutativity of {\it dynamical Dunkl elements}
\begin{gather*}
\theta_{i}= \theta_{i}^{(n)}=:\sum\limits_{j=1}^{n} u_{ij}, \qquad 1 \le i \le n.
\end{gather*}

\item Should this be the case then we are interested in to describe the
algebra generated by ``the integrals of motions'', i.e., to describe the
quotient of the algebra of polynomials $K[y_1,\ldots,y_n]$ by the two-sided ideal~${\cal{J}}_n$ generated by non-zero polynomials
$F(y_1,\ldots,y_n)$ such that $F(\theta_1,\ldots,\theta_n)=0$ in the algebra
over ring~$K$ generated by the elements \linebreak $\{u_{ij} \}_{1 \le i,j \le n}$.

\item We are looking for a set of additional relations which imply that
the elementary symmetric polynomials $e_k(Y_n)$, $1 \le k \le n$, belong to the
set of integrals of motions. In other words, the value of elementary symmetric
 polynomials $e_k(y_1,\ldots, y_n)$, $1 \le k \le n$, on the Dunkl elements
$\theta_1^{(n)},\ldots,\theta_n^{(n)}$ {\it do not depend} on the
variables $\{u_{ij}, \, 1 \le i \not= j \le n\}$. If so,
one can def\/ined {\it deformation} of elementary symmetric polynomials, and
make use of it and the Jacobi--Trudi formula, to def\/ine deformed Schur
functions, for example. We try to realize this program in Sections~\ref{section2} and~\ref{section3}.
\end{itemize}

In Section~\ref{section2}, see Def\/inition~\ref{definition2.2}, we introduce the
so-called {\it dynamical classical Yang--Baxter algebra} as ``a natural
quadratic algebra'' in which the Dunkl elements form a pair-wise commuting
family. It is the study of the algebra generated by the (truncated) Dunkl
elements that is the main objective of our investigation in \cite{K} and the
present paper. In Section~\ref{section2.1} we describe few representations of
the dynamical classical Yang--Baxter algebra~${\rm DCYB}_n$ related with
\begin{itemize}\itemsep=0pt
\item quantum cohomology $QH^{*}({\cal{F}}l_n)$ of the
complete f\/lag variety ${\cal{F}}l_n$, cf.~\cite{FGP};
\item quantum equivariant
cohomology $QH^{*}_{T^n \times C^{*}}(T^{*}{\cal{F}}l_n)$ of the cotangent
 bundle $T^{*} {\cal{F}}l_n$ to the complete f\/lag variety, cf.~\cite{GRTV};
\item Dunkl--Gaudin and Dunkl--Uglov representations, cf.~\cite{MTV,U}.
\end{itemize}

In Section~\ref{section3}, see Def\/inition~\ref{section3.1}, we introduce the algebra
$3HT_n (\beta)$, which seems to be the most
general (noncommutative) deformation of the (even) Orlik--Solomon algebra of
type~$A_{n-1}$,
such that it's still possible to describe relations among the Dunkl elements,
see Theorem~\ref{theorem3.1}. As an application we describe explicitly a set of relations
among the (additive) Gaudin/Dunkl elements, cf.~\cite{MTV}.
It should be stressed at this
place that we treat the Gaudin elements/operators (either additive or
multiplicative) as {\it images} of the
{\it universal} Dunkl elements/operators (additive or multiplicative)
in the {\it Gaudin representation} of the algebra~$3HT_n(0)$. There are
se\-ve\-ral other important representations of that algebra, for example, the
Calogero--Moser, Bruhat, Buchstaber--Felder--Veselov (elliptic), Fay trisecant
($\tau$-functions), adjoint, and so on, considered (among others) in~\cite{K}.
 Specif\/ic properties of a representation chosen\footnote{For example, in the cases of either {\it Calogero--Moser} or
{\it Bruhat} representations one has an additional constraint, namely,
$u_{ij}^2=0$ for all $i \not= j$. In the case of {\it Gaudin} representation
one has an additional constraint $u_{ij}^2 = p_{ij}^2$, where the (quantum)
parameters $ \{p_{ij}={1 \over x_i-x_j}, \, i \not=j \}$, satisfy
{\it simultaneously} the {\it Arnold} and {\it Pl\"{u}cker} relations, see
Section~\ref{section2},~{\bf II}. Therefore, the (small) quantum cohomology ring of the
 type $A_{n-1}$ full f\/lag variety ${{\cal F}l}_{n}$ and the Bethe subalgebra(s)
 (i.e., the subalgebra generated by Gaudin elements in the algebra
$3HT_n(0)$) correspond to {\it different specializations} of ``{\it quantum
parameters}'' $\{q_{ij} :=u_{ij}^2 \}$ of the {\it universal cohomology
ring} (i.e., the subalgebra/ring in $3HT_n(0)$ generated by (universal) Dunkl
elements). For more details and examples, see Section~\ref{section2.1} and~\cite{K}.}
(e.g., {\it Gaudin representation}) imply some
additional relations among the images of the universal Dunkl elements (e.g., {\it Gaudin elements}) should to be unveiled.

We start Section \ref{section3} with def\/inition of algebra~$3T_n(\beta)$ and
its ``Hecke'' $3HT_n(\beta)$ and ``elliptic'' $3MT_n(\beta)$ quotients.
In particular we def\/ine an elliptic representation of the algebra~$3T_n(0)$~\cite{KM2}, and show how the well-known elliptic solutions of the quantum
Yang--Baxter equation due to A.~Belavin and V.~Drinfeld, see, e.g.,~\cite{BD}, S.~Shibukawa and K.~Ueno~\cite{SU+}, and G.~Felder and V.~Pasquier~\cite{Fe},
can be plug in to our construction, see Section~\ref{section3.1}. At the end of
Section~\ref{section3.1.1} we point out on a {\it mysterious} (at least for the author)
appearance of the Euler numbers and ``traces'' of the Brauer algebra in the
 equivariant Pieri rules hold for the algebra $3TM_n(\beta, {\boldsymbol{q}}, \psi)$ stated in
 Theorem~\ref{theorem3.1}.

In Section \ref{section3.2} we introduce a {\it multiplicative} analogue of
the Dunkl elements $\{ \Theta_j \in 3T_n(\beta)$, $1 \le j \le n\}$ and
describe the commutative subalgebra in the algebra $3T_n(\beta)$ generated
by multiplicative Dunkl elements~\cite{KM}. The latter commutative subalgebra
 turns out to be isomorphic to the quantum equivariant $K$-theory of the
complete f\/lag variety ${\cal{F}}l_n$~\cite{KM}.

In Section~\ref{section3.3} we describe relations among the truncated
Dunkl--Gaudin elements. In this case the quantum parameters $q_{ij}=
p_{ij}^2$, where parameters $\{ p_{ij}= (z_i-z_j)^{-1} ,\, 1 \le i < j \le n \}$ satisfy the both Arnold and Pl\"{u}cker relations. This observation
has made it possible to describe a set of additional {\it rational
relations} among the Dunkl--Gaudin elements, cf.~\cite{MTV}.

In Section~\ref{section3.4} we introduce an equivariant version of
multiplicative Dunkl elements, called {\it shifted Dunkl elements} in our
paper, and describe (some) relations among the latter. This result is a
generalization of that obtained in Section~\ref{section3.1} and~\cite{KM}. However we
don't know any geometric interpretation of the commutative subalgebra
generated by shifted Dunkl elements.

In Section \ref{section4.1} for any subgraph $\Gamma \subset K_n$ of the
complete graph~$K_n$ we introduce\footnote{Independently the algebra $3T_n^{(0)}(\Gamma)$ has been studied in~\cite{BLM}, where the reader can f\/ind some examples and conjectures.}
\cite{K2,K}, algebras $3T_n(\Gamma)$ and $3T_n^{(0)}(\Gamma)$ which
 can be seen as analogues of algebras~$3T_n$ and~$3T_n^{(0)}$
correspondingly\footnote{To avoid confusions, it must be emphasized that the def\/ining
relations for algebras $3T_n(\Gamma)$~and $3T_n(\Gamma)^{(0)}$ may have more
then three terms.}.

We want to point out in the Introduction, cf.\ footnote~\ref{footnote1}, that an analog of the algebras~$3T_n$ and~$3T_n^{(\beta)}$, $3HT_n$, etc.\ treated in the present paper, can be def\/ined for any (oriented or not)
{\it matroid}~$\cal{M}$. We denote these algebras as $3T({\cal{M}})$ and
$3T^{(\beta)}({\cal{M}})$. One can show~(A.K.) that the {\it abelianization}
of the algebra $3T^{(\beta)}({\cal{M}})$, denoted by
${3T^{(\beta)}({\cal{M}})}^{ab}$, is isomorphic to the {\it Gelfand--Varchenko}
algebra corresponding to a matroid~$\cal{M}$, whereas the algebra ${3T^{(\beta=
0)}({\cal{M}})}^{ab}$ is isomorphic to the (even) {\it Orlik--Solomon} algebra
${\rm OS}^{+}({\cal{M}})$ of a matroid $\cal{M}$.\footnote{For a def\/inition and basic properties of the Orlik--Solomon algebra corresponding to a matroid, see, e.g., \cite{GR,Kawahara2004}.} We consider and treat the
algebras $3T({\cal{M}})$, $3HT({\cal{M}})$, \dots, as
{\it equivariant noncommutative}~(or {\it quantum}) versions of the (even)
Orlik--Solomon algebras associated with matroid (including hyperplane,
graphic, \dots\ arrangements). However a meaning of a quantum deformation of the
(even or odd) Orlik--Solomon algebra suggested in the present paper, is missing, even for the braid arrangement of type~$A_{n}$. Generalizations of the
 Gelfand--Varchenko algebra has been suggested and studied in~\cite{K3,K} and in the present paper under the name quasi-associative Yang--Baxter
algebra, see Section~\ref{section5}.

In the present paper we basically study the {\it abelian quotient} of the
algebra $3T_n^{(0)}(\Gamma)$, where graph $\Gamma$ has no loops and multiple
edges, since we expect some applications of our approach to the
theory of {\it chromatic polynomials} of planar graphs, in particular to
the complete multipartite graphs $K_{n_1,\ldots,n_{r}}$ and the grid graphs~$G_{m,n}$.\footnote{See \url{http://reference.wolfram.com/language/ref/GridGraph.html} for a def\/inition of {\it grid
graph} $G_{m,n}$.} Our main results hold for
the complete multipartite, cyclic and line graphs. In particular we compute
their {\it chromatic} and {\it Tutte} polynomials, see Proposition~\ref{proposition4.2} and
Theorem~\ref{theorem4.3}. As a~byproduct we compute the Tutte polynomial of the
${\boldsymbol{\ell}}$-weighted complete multipartite graph
$K_{n_1,\ldots,n_r}^{({\boldsymbol{\ell}})}$ where ${\boldsymbol{\ell}}=\{\ell_{ij}\}_{1 \le i < j \le r}$, is a collection of weights, i.e., a set of non-negative integers.

More generally, for a set of variables $ \{ \{q_{ij} \}_{1 \le i < j \le n},
x,y \}$ we def\/ine {\it universal Tutte polynomial}
$T_n(\{q_{ij}\},x,y) \in \Z[q_{ij}] [x,y]$ such that for any collection of non-negative integers $\{m_{ij} \}_{1 \le i < j \le n}$ and a subgraph $\Gamma
\subset K_n^{({\boldsymbol{m}})}$ of the complete graph~$K_n$ with each edge~$(i,j)$
 comes with multiplicity~$m_{ij}$, the specialization
\begin{gather*}
q_{ij} \longrightarrow 0 \quad \text{if~edge} \ \ (i,j) \notin \Gamma, \qquad q_{ij}
\longrightarrow [m_{ij}]_{y}:= \frac{y^{m_{ij}}-1}{y-1} \quad \text{if~edge} \ \ (i,j) \in
\Gamma
\end{gather*}
of the universal Tutte polynomial $T_n(\{q_{ij} \},x,y)$ is equal to the Tutte
polynomial of graph~$\Gamma$ multiplied by the factor $(t-1)^{\kappa(\Gamma)}$:
\begin{gather*}
 (x-1)^{\kappa(\Gamma)} {\rm Tutte} (\Gamma,x,y):= T_n(\{q_{ij}\}, x,y) \bigg |_{q_{ij}=0 \, \text{if} \, (i,j) \notin \Gamma \atop
 q_{ij}={[m_{ij}]}_{y} \, \text{if} \, (i,j) \in \Gamma}.
 \end{gather*}
Here and after $\kappa(\Gamma)$ demotes the number of connected components
of a~graph~$\Gamma$. In other words, one can treat the universal Tutte
polynomial $T_n(\{q_{ij} \},x,y)$ as a ``reproducing kernel'' for the Tutte
polynomials of all (loop-less) graphs with the number of vertices not exceeded~$n$.

We also state Conjecture~\ref{conjecture4.2} that for any loopless graph~$\Gamma$ (possibly
with multiple edges) the algebra ${3T_{|\Gamma|}^{(0)}(\Gamma)}^{ab}$ is
isomorphic to the even Orlik--Solomon algebra ${\rm OS}^{+}({\cal{A}}_{\Gamma})$
of the {\it graphic arrangement} associated with graph~$\Gamma$ in question\footnote{For simple graphs, i.e., without loops and multiple edges, this conjecture has been proved in~\cite{Li}.}.

At the end we emphasize that the case of the complete graph $\Gamma=K_n$
reproduces the results of the present paper and those of~\cite{K}, i.e.,
the case of the full f\/lag variety ${\cal F}l_n$. The case of the
{\it complete multipartite graph} $\Gamma=K_{n_{1},\ldots, n_{r}}$ reproduces
the analogue of results stated in the present paper for the case of full
f\/lag variety ${\cal F}l_n$, to the case of the {\it partial flag}
variety~${\cal F}_{n_{1},\ldots,n_{r}}$, see~\cite{K} for details.

In Section~\ref{section4.1.3} we sketch how to generalize our constructions
and some of our results to the case of the Lie algebras of {\it classical
types}\footnote{One can def\/ine an analogue of the algebra~$3T_n^{(0)}$ for the root
 system of~$BC_n$ and $C_n^{\vee}C_n$-types as well, but we are omitted these
 cases in the present paper.}.

In Section~\ref{section4.2} we brief\/ly overview our results concerning yet
another interesting family of quadratic algebras, namely the {\it six-term
relations algebras}~$6T_n$, $6T_n^{(0)}$ and related ones. These algebras
also contain a distinguished set of mutually commuting elements called
{\it Dunkl elements} $\{\theta_i,\,i=1,\ldots, n \}$ given by
$\theta_i= \sum\limits_{j \not= i} r_{ij}$, see Def\/inition~\ref{definition4.10}.

In Section~\ref{section4.2.2} we introduce and study the algebra
$6T_n^{\bigstar}$ in greater detail. In particular we introduce a ``quantum
deformation'' of the algebra generated by the curvature of $2$-forms of of
the Hermitian linear bundles over the f\/lag variety~${\cal{F}}l_n$,
cf.~\cite{PSS}.

In Section~\ref{section4.2.3} we state our results concerning the {\it
classical Yang--Baxter algebra} ${\rm CYB}_n$ and the $6$-term relation algebra~$6T_n$. In particular we give formulas for the Hilbert series of these
 algebras. These formulas have been obtained independently in~\cite{BEE} The
paper just mentioned, contains a description of a basis in the algebra~$6T_n$, and much more.

In Section~\ref{section4.2.4} we introduce a {\it super analog} of the
algebra~$6T_n$, denoted by~$6T_{n,m}$, and compute its Hilbert series.

Finally, in Section~\ref{section4.3} we introduce {\it extended nil-three
term relations} algebra~${3 \mathfrak{T}}_n$ and describe a subalgebra inside
of it which is isomorphic to the double af\/f\/ine Hecke algebra of type~$A_{n-1}$,
cf.~\cite{Ch}.

In Section~\ref{section5} we describe several combinatorial properties of some
special elements in the associative quasi-classical Yang--Baxter algebra\footnote{The algebra $\widehat{{\rm ACYB}}_n$ can be treated as ``one-half''
of the algebra~$3T_n(\beta)$. It appears that the basic relations
among the Dunkl elements, which do {\it not} mutually commute anymore, are
still {\it valid}, see Lemma~\ref{lemma5.1}.}, denoted by $\widehat{{\rm ACYB}}_n$. The main results in that direction were
motivated and obtained as a by-product, in the process of the study of the
{\it the structure} of the algebra~$3HT_n(\beta)$. More specif\/ically, the
main results of Section~\ref{section5} were obtained in the course of
 ``hunting for descendant relations'' in the algebra mentioned, which is an
 important problem to be solved to construct {\it a basis} in the
nil-quotient algebra $3T_n^{(0)}$. This {\it problem} is still widely-open.

The results of Section~\ref{section5.1}, see Proposition~\ref{proposition5.1}, items~(1)--(5), are more or less well-known among the specialists
in the subject, while those of the item~(6) seem to be new. Namely,
we show that the polynomial $Q_n(x_{ij}=t_i)$
from~\cite[Exercise~6.C8(c)]{ST3}, essentially coincides with the
$\beta$-deformation \cite{FK1} of the Lascoux--Sch\"utzenberger
Grothendieck polynomial \cite{LS} for some particular permutation.
The results of Proposition~\ref{proposition5.1}(6), point out on a deep connection
between reduced forms of monomials in the algebra~$\widehat{{\rm ACYB}}_n$ and the
Schubert and Grothendieck calculi. This observation was the starting point
for the study of some combinatorial properties of certain specializations of
the Schubert, the $\beta$-Grothendieck~\cite{FK2} and the double $\beta$-Grothendieck polynomials in Section~\ref{section5.2}. One of the main
results of Section~\ref{section5.2} can be stated as follows.

\begin{Theorem}\label{theorem1.3} \quad
\begin{enumerate}\itemsep=0pt

\item[$(1)$] Let $w \in \mathbb{S}_n$ be a permutation, consider the
specialization $x_1:=q$, $x_i=1$, $\forall\, i \ge 2$, of the $\beta$-Grothendieck
polynomial $\mathfrak{G}_{w}^{(\beta)}(X_n)$. Then
\begin{gather*} {\cal{R}}_{w}(q,\beta +1) := \mathfrak{G}_{w}^{(\beta)}(x_1=q,\, x_i=1,
\, \forall\, i \ge 2) \in \N [q,1+\beta].
\end{gather*}
In other words, the polynomial
${\cal{R}}_{w}(q,\beta)$ has non-negative integer coefficients\footnote{For a more general result see Appendix~\ref{appendixA.1}, Corollary~\ref{corollary6.2}.}.

For late use we define polynomials
\begin{gather*}
\mathfrak{R}_{w}(q,\beta) := q^{1-w(1)} {\cal {R}}_{w}(q,\beta).
\end{gather*}

\item[$(2)$] Let $w \in \mathbb{S}_n$ be a permutation, consider the
specialization $x_i:=q$, $y_i=t$, $\forall\, i \ge 1$, of the double
$\beta$-Grothendieck
polynomial $\mathfrak{G}_{w}^{(\beta)}(X_n,Y_n)$. Then
\begin{gather*}
\mathfrak{G}_{w}^{(\beta -1)}(x_i:=q,\, y_i :=t, \, \forall \, i \ge 1) \in \N [q,t,
\beta].
\end{gather*}

\item[$(3)$] Let $w$ be a permutation, then
\begin{gather*}
 \mathfrak{R}_{w}(1,\beta)=\mathfrak{R}_{1 \times w}(0,\beta).
 \end{gather*}
Note that ${\cal{R}}_{w}(1,\beta) ={\cal {R}}_{w^{-1}}(1,\beta)$,
but ${\cal{R}}_{w}(t,\beta) \not={\cal {R}}_{w^{-1}}(t,\beta)$, in general.
\end{enumerate}
\end{Theorem}

 For the reader convenience we collect some basic def\/initions and results
concerning the $\beta$-Grothendieck polynomials in Appendix~\ref{appendixA.1}.

Let us observe that $\mathfrak{R}_{w}(1,1)= \mathfrak{S}_{w}(1)$, where
$\mathfrak{S}_{w}(1)$ denotes the specialization $x_i:=1$, $\forall \, i \ge 1$,
of the Schubert polynomial $\mathfrak{S}_{w}(X_n)$ corresponding to
permutation~$w$. Therefore,
$\mathfrak{R}_{w}(1,1)$ is equal to the number of {\it compatible sequences}~\cite{BJS} (or {\it pipe dreams}, see, e.g.,~\cite{SS}) corresponding to
permutation~$w$.
\begin{Problem} \label{problem1.1}

Let $w \in \mathbb{S}_n$ be a permutation and $l:= \ell(w)$ be its length.
Denote by ${\rm CS}(w) = \{ {\boldsymbol{a}}=(a_1 \le a_2 \le \cdots \le a_l) \in \N^{l}\}$
the set of
 compatible sequences~{\rm \cite{BJS}} corresponding to permutation~$w$.
\begin{itemize} \itemsep=0pt
\item Define statistics $r({\boldsymbol{a}})$ on the set of all
compatible sequences ${\rm CS}_n := \coprod\limits_{{w \in \mathbb{S}_n}} {\rm CS}(w)$
in a such way that
\begin{gather*}
 \sum _{{\boldsymbol{a}} \in {\rm CS}(w)} q^{a_1} \beta ^{r({\boldsymbol{a}})} =
{\cal {R}}_{w}(q,\beta).
\end{gather*}

\item Find a geometric interpretation, and
investigate combinatorial and algebra-geometric pro\-per\-ties of
polynomials $\mathfrak{S}_{w}^{(\beta)}(X_n)$,
where for a permutation $w \in \mathbb{S}_n$ we denoted by
$\mathfrak{S}_{w}^{(\beta)}(X_n)$ the $\beta$-Schubert polynomial
 defined as follows
\begin{gather*}
 \mathfrak{S}_{w}^{(\beta)}(X_n) = \sum_{{\boldsymbol{a}} \in {\rm CS}(w)} \beta^{r({\boldsymbol{a}})} \prod_{i=1}^{l:=\ell(w)} x_{a_{i}}.
 \end{gather*}
\end{itemize}
\end{Problem}

We {\it expect} that polynomial $\mathfrak{S}_{w}^{(\beta)}(1)$ coincides
with the Hilbert polynomial of a certain graded commutative ring naturally
associated to permutation~$w$.

\begin{Remark}
It should be mentioned that, in general, the principal specialization
\begin{gather*}
\mathfrak{G}_{w}^{(\beta-1)}\big(x_i:=q^{i-1}, \,\forall \, i \ge 1\big)
\end{gather*}
 of the $(\beta-1)$-Grothendieck polynomial may have {\it negative}
coef\/f\/icients.
\end{Remark}

Our main objective in Section~\ref{section5.2} is to study the polynomials
$\mathfrak{R}_{w}(q,\beta)$ for a special class of permutations in the
symmetric group $\mathbb{S}_{\infty}$.
Namely, in Section~\ref{section5.2} we study some combinatorial properties of
polynomials $\mathfrak{R}_{ \varpi_{\lambda,\phi}}(q,\beta)$ for the f\/ive
parameters family of {\it vexillary}
permutations $ \{\varpi_{\lambda,\phi} \}$ which have the {\it shape}
$\lambda:= \lambda_{n,p,b}=(p(n-i+1)+b$, $i=1,\ldots,n+1)$
and {\it flag}
$\phi:= \phi_{k,r} = (k+r (i-1),~i=1, \ldots,n+1)$.

This class of permutations is notable for many reasons, including that the
specialized value of the Schubert polynomial $\mathfrak{S}_{\varpi_{\lambda,
\phi}}(1)$ admits a nice product formula\footnote{One can prove a product formula for the principal
specialization $\mathfrak{S}_{\varpi_{\lambda,\phi}}(x_i := q^{i-1},\,
\forall\, i \ge 1)$ of the correspon\-ding Schubert polynomial. We don't need a~such formula in the present paper.}, see Theorem~\ref{theorem5.6}. Moreover, we describe also some interesting connections of
polynomials $\mathfrak{R}_{\varpi_{\lambda,\phi}}(q,\beta)$ with plane
partitions, the Fuss--Catalan numbers\footnote{We def\/ine the (generalized)
Fuss--Catalan numbers to be ${\rm FC}_{n}^{(p)}(b):={1+b \over 1+b+(n-1)p} {n p + b \choose n}$. Connection of
the Fuss--Catalan numbers with the {\it $p$-ballot numbers} ${\rm Bal}_{p}(m,n):=
 {n-mp+1 \over n+m+1}~{n+m+1 \choose m}$ and the {\it Rothe numbers}
$R_{n}(a,b):= {a \over a+b n} {a+b n \choose n}$ can be described as follows
\begin{gather*}
{\rm FC}_{n}^{(p)}(b)=R_{n}(b+1,p)={\rm Bal}_{p-1}(n,(n-1) p+b).
\end{gather*}}
and Fuss--Narayana polynomials,
$k$-triangulations and $k$-dissections of a~convex polygon, as well as a connection with two families of ${\rm ASM}$. For
example, let $\lambda=(b^{n})$ and $\phi=(k^n)$ be rectangular shape
partitions, then the polynomial
$\mathfrak{R}_{ \varpi_{\lambda,\phi}}(q,\beta)$ def\/ines a~$(q,\beta)$-deformation of the number of (ordinary) plane partitions\footnote{Let $\lambda$ be a partition. An ordinary plane partition (plane
partition for short)bounded by~$d$ and shape $\lambda$ is a~f\/illing of the
shape $\lambda$ by the numbers from the set $\{0,1,\ldots,d \}$ in such a
way that the numbers along columns and rows are weakly {\it decreasing}.
A {\it reverse} plane partition bounded by
$d$ and shape $\lambda$ is a~f\/illing of the shape $\lambda$ by the numbers from the set $\{0,1,\ldots,d \}$ in such a way that the numbers along columns
and rows are weakly {\it increasing}.}
sitting in the box~$b \times k \times n$. It seems an interesting
{\it problem} to f\/ind an algebra-geometric interpretation of polynomials
$\mathfrak{R}_{w}(q,\beta)$ in the general case.

\begin{Question}\label{question1.1}
 Let $a$ and $b$ be mutually prime positive integers. Does
 there exist a family of permutations $w_{a,b} \in {\mathbb{S}}_{ab(a+b)}$ such
that the specialization $x_i =1$, $\forall\, i$ of the Schubert polyno\-mial~$\s_{w_{a,b}}$ is equal to the rational Catalan number $C_{a/b}$? That is
\begin{gather*}
\s_{w_{a,b}}(1)={1 \over a+b} {a+b \choose a}.
\end{gather*}
\end{Question}

 Many of the computations in Section~\ref{section5.2} are based on the following
determinantal formula for $\beta$-Grothendieck polynomials corresponding to
grassmannian permutations, cf.~\cite{L}.

\begin{Theorem}[see Comments~\ref{comments5.5}(b)]\label{theorem1.4}
 If $w=\sigma_{\lambda}$ is the grassmannian permutation with
shape $\lambda =(\lambda_,\ldots, \lambda_n)$ and a unique
descent at position~$n$, then\footnote{The equality
\begin{gather*}
\mathfrak{G}_{\sigma_{\lambda}}^{(\beta)}(X_n)=
{\operatorname{DET}\big| x_i^{\lambda_j +n-j}
(1+\beta x_i)^{j-1} \big|_{1 \le i,j \le n} \over \prod\limits_{1 \le i < j \le n}
(x_i-x_j) },
\end{gather*}
has been proved independently in \cite{MS}.}
\begin{gather*}
({\rm A}) \quad \mathfrak{G}_{\sigma_{\lambda}}^{(\beta)}(X_n)=
\operatorname{DET}\big|h_{\lambda_j +i,j}^{(\beta)}(X_n)\big|_{1 \le i,j \le n} =
{\operatorname{DET} \big| x_i^{\lambda_j +n-j}
(1+\beta x_i)^{j-1} \big|_{1 \le i,j \le n} \over \prod\limits_{1 \le i < j \le n}
(x_i-x_j) },
\end{gather*}
where $X_{n}=(x_i,x_{1},\ldots,x_n)$, and for any set of variables~$X$,
\begin{gather*}
h_{n,k}^{(\beta)}(X) = \sum_{a=0}^{k-1}~{k-1 \choose a} h_{n-k+a}(X)
\beta^{a},
\end{gather*}
and $h_k(X)$ denotes the complete symmetric polynomial of degree~$k$ in the variables from the set~$X$.
\begin{gather*}
({\rm B}) \quad {\mathfrak G}_{\sigma_{\lambda}}(X,Y) =
{ \operatorname{DET}\Big| \prod\limits_{a=1}^{\lambda_{j}+n-j} (x_i +y_a + \beta x_i y_a) (1+\beta x_i)^{j-1} \Big|_{1 \le i,j \le n} \over \prod\limits_{1 \le i < j \le n} (x_i-x_j) }.
\end{gather*}
\end{Theorem}

In Sections~\ref{section5.2.2} and~\ref{section5.4.2} we study connections of Grothendieck polynomial
associated with the Richardson permutation~$ w_{k}^{(n)}= 1^{k} \times w_{0}^{(n-k)}$, $k$-dissections of a convex $(n+k+1)$-gon, generalized reduced
polynomial corresponding to a certain monomial in the algebra~${\widehat{{\rm ACYB}}}_{n}$ and the Lagrange inversion formula. In the case of generalized Richardson
permutation~$w_{n,p}^{(k)}$ corresponding to the
$k$-shifted dominant permutations~$w^{(p,n)}$ associated with the Young diagram
$ \lambda_{p,n}:= p(n-1,n-2,\ldots,1)$, namely, $w_{n,p}^{(k)} = 1^{k} \times
w^{(p,n)}$, we treat only the case $k=1$, see also~\cite{EM}. In the case
$k \ge 2$ one comes to a {\it task} to count and f\/ind a lattice path type
interpretation for the number of {\it $k$-{\bf p}gulations} of a convex
$n$-gon that is the number of partitioning of a convex $n$-gon on parts which
are all equal to a convex $(p+2)$-gon, by a (maximal) family of diagonals such
 that each diagonal has at most $k$ internal intersections with the members of a family of diagonals selected.

In Section~\ref{section5.3} we give a partial answer on Question~{6.C8}(d) by R.~Stanley~\cite{ST3}. In particular, we relate the reduced
polynomial corresponding to monomial
\begin{gather*}
\bigl( x_{12}^{a_{2}} \cdots {x_{n-1,n}}^{a_{n}} \bigr) \prod_{j=2}^{n-2}
\prod_{k=j+2}^{n} x_{jk},\qquad a_j \in \Z_{\ge 0}, \qquad \forall\, j,
\end{gather*}
with the Ehrhart polynomial of the generalized Chan--Robbins--Yuen polytope,
 if $a_2=\cdots=a_n=m+1$, cf.~\cite{Me2}, with a $t$-deformation of the
Kostant partition function of type $A_{n-1}$ and the Ehrhart polynomials of
some f\/low polytopes, cf.~\cite{MM}.

In Section~\ref{section5.4} we investigate certain specializations of the
reduced polynomials corresponding to monomials of the form
\begin{gather*}
 x_{12}^{m_1} \cdots x_{n-1,n}^{m_n}, \qquad m_j \in \Z_{\ge 0}, \qquad \forall\, j.
 \end{gather*}
First of all we observe that the corresponding specialized reduced polynomial
appears to be a {\it piece-wise polynomial function} of parameters
${\boldsymbol{m}}=(m_1,\ldots, m_n) \in (\R_{\ge 0})^{n}$, denoted by $P_{{\boldsymbol{m}}}$.
It is an interesting {\it problem} to compute the {\it Laplace transform} of
that piece-wise polynomial function. In the present paper we compute the value
of the function $P_{{\boldsymbol{m}}}$ in the dominant chamber
${\cal{C}}_n = (m_1 \ge m_2 \ge \cdots \ge m_n \ge 0)$, and give a~combinatorial interpretation of the values of that function in points
$(n,m)$ and $(n,m,k)$, $n \ge m \ge k$.

For the reader convenience, in Appendices~\ref{appendixA.1}--\ref{appendixA.6} we collect some
useful auxiliary information about the subjects we are treated in the present
paper.

Almost all results in Section~\ref{section5} state that some two specif\/ic sets have the
same number of elements. Our proofs of these results are pure algebraic. It is
an interesting {\it problem} to f\/ind {\it bijective proofs} of results from
Section~\ref{section5} which generalize and extend remarkable bijective proofs presented
in \cite{MM,SS,St, W} to the {\it cases} of
\begin{itemize}\itemsep=0pt
\item the $\beta$-Grothendieck polynomials,
\item the (small) Schr\"oder numbers,
\item $k$-dissections of a convex $(n+k+1)$-gon,
\item special values of reduced polynomials.
\end{itemize}

We are planning to treat and present these bijections in separate
publication(s).

We {\it expect} that the reduced polynomials corresponding to the higher-order
powers of the Coxeter elements also admit an interesting combinatorial
interpretation(s). Some preliminary results in this direction are discussed in
Comments~\ref{comments5.8}.

At the end of introduction I want to add a few remarks.

(a)~After a suitable modif\/ication of the algebra $3HT_n$, see~\cite{KM3}, and the case $\beta \not= 0$ in~\cite{K}, one can compute the set
of relations among the (additive) Dunkl elements (def\/ined in Section~\ref{section2},
equation~\eqref{equation2.1}). In the case $\beta=0$ and $q_{ij}=q_i \delta_{j-i,1}$,
$1 \le i < j \le n$, where $\delta_{a,b}$ is the Kronecker delta symbol, the
commutative algebra generated by additive Dunkl elements~\eqref{equation2.3} appears to be
``almost'' isomorphic to the equivariant quantum
cohomology ring of the f\/lag variety ${\cal F}l_n$, see~\cite{KM3} for
details. Using the {\it multiplicative} version of Dunkl elements, see
Section~\ref{section3.2}, one can extend the results from~\cite{KM3} to the case of equivariant quantum
$K$-theory of the f\/lag variety~${\cal F}l_n$, see~\cite{K}.

(b)~As it was pointed out previously, one can def\/ine an analogue of the algebra
$3T_n^{(0)}$ for any (oriented) matroid ${\cal{M}}_n$, and state a conjecture
which connects the Hilbert polynomial of the algebra $3T_n^{(0)}(({\cal{M}}_n)^{ab},t)$ and the chromatic polynomial of matroid ${\cal{M}}_n$. We {\it
expect} that algebra $3T_n^{(\beta=1)}({\cal{M}}_n)^{ab}$ is isomorphic to
the {\it Gelfand--Varchenko} algebra associated with matroid~$\cal{M}$. It
is an interesting {\it problem} to f\/ind a combinatorial meaning of the algebra
$3T_n^{(\beta)}({\cal{M}}_n)$ for $\beta =0$ and $\beta \not= 0$.

(c)~Let $R$ be a (graded) ring (to be specif\/ied later) and
${\mathfrak{F}}_{n^{2}}$ be the free associative algebra over~$R$ with the set
 of generators $\{ u_{ij},\, 1 \le i,j \le n \}$. In the subsequent text we
will distinguish the set of generators $\{ u_{ii} \}_{1 \le i \le n}$ from
that $\{u_{ij}\}_{1 \le i \not= j \le n}$, and set
\begin{gather*}
x_i:=u_{ii}, \qquad i=1,\ldots, n.
\end{gather*}

A guiding idea to choose def\/initions and perform constructions in the present
paper is to impose a set of relations~${\cal{R}}_n$ among the generators
$\{x_{i} \}_{1 \le i \le n}$ and that $\{u_{ij} \}_{1 \le i \not= j \le n}$
which ensure the mutual commutativity of the following elements
\begin{gather*}
 \theta_{i}^{(n)}:=\theta_i = x_i + \sum_{j \not= i}^{n} u_{ij}, \qquad i=1,\ldots,n,
 \end{gather*}
in the algebra ${\cal{F}}_{n^{2}} / {\cal{R}}_{n}$, as well as to have a good chance to describe/compute

$\bullet$~``Integral of motions'', that is f\/inding a big enough set of
algebraically independent polynomials (quite possibly that polynomials are
trigonometric or elliptic ones) $I_{\alpha}^{(n)}(y_1,\ldots,y_n) \in R [Y_n]$ such that
\begin{gather*}
I_{\alpha}^{(n)}\big(\theta_1^{(n)}, \ldots, \theta_n^{(n)}\big) \in R [X_n],
\qquad \forall\, \alpha,
\end{gather*}
in other words, the latter specialization of any integral of motion has to be
 independent of the all generators $\{ u_{ij} \}_{1 \le i \not=j \le n}$.

$\bullet$~Give a presentation of the algebra ${\cal{I}}_{n}$ generated by the
integral of motions that is to f\/ind a set of def\/ining relations among the
elements $\theta_1,\ldots, \theta_n$, and describe a~$R$-basis
$ \big\{m_{\alpha}^{(n)}\big\}$ in the algebra~${\cal{I}}_{n}$.

$\bullet$~Generalized Littlewood--Richardson and Murnaghan--Nakayama problems.
 Given an integral of motion $I_{\beta}^{(m)}(Y_m)$ and an integer $n \ge m$,
f\/ind an explicit positive (if possible) expression in the quotient algebra
${\cal{F}}_{n^{2}} / {\cal{R}}_{n}$ of the element
\begin{gather*}
 I_{\beta}^{(m)}\big(\theta_{1}^{(n)}, \ldots , \theta_{m}^{(n)}\big).
 \end{gather*}
For example in the case of the 3-term relations algebra $3T_n^{(0)}$ (as
 well as its equivariant, quantum, etc.\ versions) the generalized
 Littlewood--Richardson problem is to f\/ind a positive expression in the algebra~$3T_n^{(0)}$ for the element $\s_{w}\big(\theta_1^{(n)},\ldots,\theta_{m}^{(n)}\big)$, where $\s_{w}(Y_n)$ stands for the Schubert polynomial corresponding to a~permutation $w \in \mathbb{S}_n$.

Generalized Murnaghan--Nakayama problem consists in f\/inding a combinatorial
expression in the algebra $3T_n^{(0)}$ for the element $ \sum\limits_{i=1}^{m}
(\theta_{i}^{(n)})^{k}$.

Partial results concerning these problems have been obtained as far as we
aware in \cite{FK,K2, K, KMa,MPP, P}.

$\bullet$~``Partition functions''. Assume that the (graded) algebra ${\cal{I}}_{n}$ generated over $R$ by the
elements $\theta_1,\ldots, \theta_n$ has f\/inite dimension/rank, and the (non
zero) maximal degree component ${\cal{I}}_{\max}^{(n)}$ of that algebra has
dimension/rank one and generated by an element $\omega$.~For any element
$ g \in {\cal{F}}_{n^{2}}$ let us denote by $\operatorname{Res}_{\omega}(g)$ an element in~$R$ such that
\begin{gather*}
\overline{g} = \operatorname{Res}_{\omega}(g) \omega,
\end{gather*}
where we denote by $\overline{g}$ the image of element $g$ in the component
${\cal{I}}_{\max}^{(n)}$.

We def\/ine {\it partition function} associated with the algebra ${\cal{I}}_{n}$ as follows
\begin{gather*}
{\cal{Z}}({\cal{I}}_{n}) = \operatorname{Res}_{w} \bigg( \exp \bigg(\sum_{\alpha} q_{\alpha} m_{\alpha}^{(n)} \bigg) \bigg),
\end{gather*}
where $ \{q_{\alpha}\}$ is a set of parameters which is consistent in
one-to-one correspondence with a~basis $\big\{m_{\alpha}^{(n)}\big\}$ chosen.

We are interesting in to f\/ind a closed formula for the partition function
${\cal{Z}}({\cal{I}}_{n})$ as well as that for a {\it small} partition function
\begin{gather*}
{\cal{Z}}^{(0)}({\cal{I}}_{n}):= \operatorname{Res}_{\omega} \bigg(\exp \bigg( \sum_{1 \le i, j \le n} \lambda_{ij} u_{ij} \bigg) \bigg),
\end{gather*}
where $\{\lambda_{ij} \}_{1 \le i,j \le n}$ stands for a set of parameters.
One can show~\cite{K4} that the partition function ${\cal{Z}}({\cal{I}}_{n})$
associated with algebra $3T_n^{{\boldsymbol{q}}}$ satisf\/ies the famous Witten--Dijkraaf--Verlinde--Verlinde equations.

As a preliminary steps to perform our guiding idea we
\begin{enumerate}\itemsep=0pt
\item[(i)] investigate properties of the abelianization of the algebra
${\cal{F}}_{n^{2}} / {\cal{R}}_n$. Some unexpected connections with
the theory of hyperplane arrangements and graph theory are discovered;

\item[(ii)] investigate a variety of descendent relations coming from the
def\/ining relations. Some polynomials with interesting combinatorial properties
are naturally appear.
\end{enumerate}

To keep the size of the present paper reasonable, several new results are
presented as exercises.

We conclude Introduction by a short historical remark. As far as we aware, the
commutative version of $3$-term relations which provided the framework for a
def\/inition of the FK algebra~${\cal{E}}_n$~\cite{FK} and a plethora of its
generalizations, have been frequently used implicitly in the theory of elliptic
functions and related topics, starting  at least from the middle of the 19th century, see, e.g.,~\cite{WW} for references, and up to now, and for sure will be used for ever.
 The key point is that the Kronecker sigma function
\begin{gather*}
\sigma_{z}(w):= {\frac{ \sigma(z-w) \theta' (0)}{\sigma(z) \sigma(-w)}},
\end{gather*}
where $\sigma(z)$ denotes the Weierstrass sigma function, satisf\/ies the
quadratic three terms {\it addition formula} or {\it functional equation}
 discovered, as far as we aware, by
 K.~Weierstrass. In fact this functional equation~is really equivalent\footnote{We refer the reader to a nice written paper by Tom H.~Koornwinder \cite{Koornwinder2014} for more historical information.}
 to the famous Jacobi--Riemann three term relation of degree four between
the Riemann theta functions~$\theta(x)$. In the rational degeneration of theta
functions, the three term relation between Kronecker sigma functions turns to the
famous three term Jacobi identity which can be treated as an associative
analogue of the Jacobi identity in the theory of Lie algebras.

To our best knowledge, in an abstract form that is as a set of def\/ining
relations in a certain algebra, an anticommutative version of three term
relations had been appeared in a~remarkable paper by V.I.~Arnold~\cite{Arnold1969}. Nowadays these relations are known as {\it Arnold relations}.
These relations and its various generalizations play fundamental role in the
theory of arrangements, see, e.g.,~\cite{OT}, in topology, combinatorics and
many other branches of Mathematics.

In commutative set up abstract form of $3$-term relations has been invented by
O.~Mathieu~\cite{Mathieu1995}. In the context of the braided Hopf algebras (of type~A)
$3$-term relations like algebras (as some examples of the Nichols algebras)
 have appeared in papers by A.~Milinski and H.-J.~Schneider (2000), N.~Andruskiewitsch (2002), S.~Madjid (2004), I.~Heckenberger (2005) and many others\footnote{We refer the reader to the site
\url{https://en.wikipedia.org/wiki/Nichols_algebra}
for basic def\/initions and results concerning Nichols' algebras and references on vast literature treated dif\/ferent aspects of the theory of Nichols' algebras and braided Hopf algebras.}.

It is well-known that the Nichols algebra associated with the symmetric group~${\mathbb{S}}_n$ and trivial conjugacy class is a quotient of the algebra~$FK_n$. It is still an open {\it problem} to prove (or disprove) that these two algebras are isomorphic.

\section{Dunkl elements}\label{section2}

Having in mind to fulf\/ill conditions suggested by our guiding line mentioned
in Introduction as far as it could be done till now, we are led to introduce the following algebras\footnote{Surprisingly enough, in many cases to f\/ind relations among the
elements $\theta_1, \ldots, \theta_n$ there is no need to require that the
elements $\{\theta_i \}_{1 \le i \le n}$ are pairwise commute.}.

\begin{Definition}[additive Dunkl elements] \label{definition2.1}
The (additive) Dunkl elements $\theta_i$, $i=1,\dots,n$, in the
algebra ${\cal F}_n$ are def\/ined to be
\begin{gather}\label{equation2.1}
\theta_i=x_i+\sum_{j=1 \atop j \not=i}^{n} u_{ij}.
\end{gather}
\end{Definition}

We are interested in to f\/ind ``natural relations'' among the generators
$\{u_{ij} \}_{1 \le i,j \le n}$ such that the Dunkl elements~\eqref{equation2.1} are
pair-wise {\it commute}. One of the natural conditions which is the
commonly accepted in the theory of integrable systems, is

\begin{itemize}\itemsep=0pt
\item locality conditions:
\begin{gather}
 (a) \quad [x_i,x_j]=0 \qquad \text{if} \ \ i \not= j,\nonumber\\
(b) \quad u_{ij} u_{kl}=u_{kl} u_{ij} \qquad \text{if} \ \ i \not=j, \ \ k \not= l \ \ \text{and} \ \ \{i,j \} \cap \{k,l \}=\varnothing.\label{equation2.2}
\end{gather}
\end{itemize}

\begin{Lemma}\label{lemma2.1} Assume that elements $\{u_{ij} \}$ satisfy the locality condition~\eqref{equation2.1}. If $i \not= j$,
then
\begin{gather*}
 [\theta_i,\theta_j]= \biggl[ x_i+ \sum_{k \not= i,j} u_{ik}, u_{ij}+u_{ji}
\biggr] +\biggl[u_{ij},\sum_{k=1}^{n} x_k \biggr]+ \sum_{k \not= i,j} w_{ijk},
\end{gather*}
where
\begin{gather}
w_{ijk}=[u_{ij},u_{ik}+u_{jk}]+[u_{ik},u_{jk}]+[x_i,u_{jk}]+[u_{ik},x_j]+
[x_k,u_{ij}].\label{equation2.3}
\end{gather}
\end{Lemma}

Therefore in order to ensure that the Dunkl elements form a pair-wise
{\it commuting} family, it's natural to assume that the following
conditions hold
\begin{itemize}
\item unitarity:
\begin{gather}
[u_{ij}+u_{ji},u_{kl}]=0=[u_{ij}+u_{ji},x_k] \qquad \text{for all distinct} \ \ i,\,j,\,k,\, l,\label{equation2.4}
\end{gather}
i.e., the elements $u_{ij}+u_{ji}$ are {\it central}.

\item ``conservation laws'':
\begin{gather}
 \left[\sum_{k=1}^{n} x_k,u_{ij}\right] =0 \qquad \text{for all} \ \ i,\, j,
\label{equation2.5}
 \end{gather}
i.e., the element $E:=\sum\limits_{k=1}^{n} x_k$ is {\it central},

\item unitary dynamical classical Yang--Baxter relations:
\begin{gather}\label{equation2.6}
 [u_{ij},u_{ik}+u_{jk}]+[u_{ik},u_{jk}]+[x_i,u_{jk}]+[u_{ik},x_j]+
[x_k,u_{ij}]=0,
\end{gather}
if $i$, $j$, $k$ are pair-wise distinct.
\end{itemize}

\begin{Definition}[dynamical six term relations algebra $6DT_n$] \label{definition2.2}
We denote by $6DT_n$ (and frequently will use also notation ${\rm DCYB}_n$) the
quotient of the algebra ${\cal F}_n$ by the two-sided ideal generated by relations~\eqref{equation2.2}--\eqref{equation2.6}.
\end{Definition}

Clearly, the Dunkl elements~\eqref{equation2.1} generate a commutative subalgebra inside of
 the algebra~$6DT_n$, and the sum $\sum\limits_{i=1}^{n} \theta_i =\sum\limits_{i=1}^{n} x_i$ belongs to the center of the algebra~$6DT_n$.

\begin{Remark}\label{remark2.0}
Occasionally we will call the Dunkl elements of the form \eqref{equation2.1} by {\it
dynamical Dunkl elements} to distinguish the latter from {\it truncated
Dunkl elements}, corresponding to the case $x_i=0$, $\forall\, i$.
\end{Remark}

\subsection[Some representations of the algebra $6DT_n$]{Some representations of the algebra $\boldsymbol{6DT_n}$} \label{section2.1}

\subsubsection{Dynamical Dunkl elements and equivariant quantum cohomology}\label{section2.1.1}

{\bf (I)} ( cf.~\cite{FGP}). Given a set $q_1,\ldots,q_{n-1} $ of mutually
 commuting parameters, def\/ine
\begin{gather*}
q_{ij}=\prod_{a=i}^{j-1} q_a \qquad \text{if} \quad i < j,
\end{gather*}
 and set $q_{ij}=q_{ji}$ in the
case $i > j$. Clearly, that if $i < j < k$, then $q_{ij}q_{jk}=q_{ik}$.

Let $z_1, \ldots,z_n$ be a set of (mutually commuting) variables. Denote by
$P_n:=\Z[z_1,\ldots,z_n]$ the corresponding ring of polynomials. We consider
the variable~$z_i$, $i=1,\ldots,n$, also as the operator acting on the ring
of polynomials $P_n$ by multiplication on the variable~$z_i $.

Let $s_{ij} \in \mathbb{S}_n$ be the transposition that swaps the letters~$i$
and $j$ and f\/ixes the all other letters $k \not= i,j$. We consider the
transposition~$s_{ij}$ also as the operator which acts on the ring~$P_n$ by
interchanging $z_i$ and $z_j$, and f\/ixes all other variables. We denote by
\begin{gather*}
 \partial_{ij}={1-s_{ij} \over z_i-z_j },\qquad \partial_{i} :=
\partial_{i,i+1},
\end{gather*}
the divided dif\/ference operators corresponding to the transposition
$s_{ij}$ and the simple transposition $s_{i}:= s_{i,i+1}$ correspondingly.
 Finally we
def\/ine operator (cf.~\cite{FGP})
\begin{gather*}
\partial_{(ij)}:= \partial_{i}\cdots \partial_{j-1}\partial_{j}\partial_{j-1} \cdots \partial_{i} \qquad \text{if} \ \ i < j.
\end{gather*}
The operators $\partial_{(ij)}$, $1 \le i < j \le n$, satisfy (among other
things) the following set of relations (cf.~\cite{FGP})
\begin{itemize}\itemsep=0pt
\item $[z_j,\partial_{(ik)}]=0$ if $j \notin [i,k]$, $\Big[\partial_{(ij)},\sum\limits_{a=i}^{j} z_a\Big]=0$,
\item $[\partial_{(ij)},\partial_{(kl)}] =
\delta_{jk} [z_j,\partial_{(il)}]+\delta_{il} [\partial_{(kj)},z_i]$ if $i < j$, $k < l$.
\end{itemize}

Therefore, if we set $u_{ij}=q_{ij} \partial_{(ij)}$ if $i < j$, and
$u_{ij}=-u_{ji}$ if $i > j$, then for a triple $i < j < k$ we will have
\begin{gather*}
[u_{ij},u_{ik}+u_{jk}]+[u_{ik},u_{jk}]+[z_i,u_{jk}]+[u_{ik},z_j]+[z_k,u_{jk}]\\
\qquad{} =q_{ij}q_{jk}[\partial_{(ij)},\partial_{(jk)}]+q_{ik}[\partial_{(ik)},z_j]=0.
\end{gather*}
Thus the elements $\{ z_i,\, i=1,\ldots,n \}$ and
$\{u_{ij},\, 1 \le i < j \le n \}$ def\/ine a representation of the algebra
${\rm DCYB}_n$, and therefore the Dunkl elements
\begin{gather*}
\theta_i:= z_i+ \sum_{j \not= i} u_{ij}=z_i-\sum_{j < i}q_{ji}\partial_{(ji)}+\sum_{j > i} q_{ij}\partial_{(ij)}
\end{gather*}
form a pairwise commuting family of operators acting on the ring of
polynomials
\begin{gather*}
\Z[q_1,\ldots,q_{n-1}][z_1,\ldots,z_n],
\end{gather*}
 cf.~\cite{FGP}. This representation
has been used in~\cite{FGP} to construct the small quantum cohomology ring
of the complete f\/lag variety of type~$A_{n-1}$.

{\bf (II)} Consider degenerate af\/f\/ine Hecke algebra ${\mathfrak {H}}_n$
generated by the central element~$h$, the elements of the symmetric group
${\mathbb{S}}_n$, and the mutually commuting elements $y_1,\ldots,y_n$,
subject to relations
\begin{gather*}
 s_i y_i - y_{i+1} s_i = h, \quad 1 \le i < n, \qquad s_i y_j =y_j s_i, \quad j \not= i,i+1,
\end{gather*}
where $s_i$ stand for the simple transposition that swaps only indices~$i$
and~$i+1$. For $i < j$, let $s_{ij}=s_i\cdots s_{j-1}s_{j}s_{j-1}\cdots
s_{i}$ denotes the permutation that swaps only indices~$i$ and~$j$. It is
an easy exercise to show that
\begin{itemize}\itemsep=0pt
\item $[y_j,s_{ik}] = h [s_{ij},s_{jk}]$ if $i < j < k$,
\item $y_{i} s_{ik} -s_{ik} y_{k} = h + h s_{ik} \sum\limits_{ i < j < k} s_{jk}$ if $ i < k$.
\end{itemize}
Finally, consider a set of mutually commuting parameters $\{p_{ij}, \, 1 \le i
\not= j \le n, \, p_{ij}+p_{ji}=0 \}$, subject to the constraints
\begin{gather*}
 p_{ij} p_{jk} = p_{ik} p_{ij} + p_{jk} p_{ik} + h p_{ik}, \qquad i < j < k.
 \end{gather*}

\begin{Comments} \label{comments2.1} If parameters $\{p_{ij} \}$ are {\it invertible}, and
satisfy relations
\begin{gather*}
 p_{ij} p_{jk} = p_{ik} p_{ij} + p_{jk} p_{ik} + \beta p_{ik}, \qquad i < j < k,
 \end{gather*}
 then one can rewrite the above displayed relations in the following form
\begin{gather*}
 1+ {\beta \over p_{ik}} =\left(1+{\beta \over p_{ij}} \right)
\left(1 +{\beta \over p_{jk}}\right), \qquad 1 \le i < j < k \le n .
\end{gather*}
Therefore there exist parameters $\{q_1,\ldots,q_n \}$ such that
$1+\beta /p_{ij}=q_i/q_j$, $1 \le i < j \le n$. In other words, $p_{ij} =
{ \beta q_j \over q_j -q_j}$, $1 \le i < j \le n$. However in general,
 there are many other types of solutions, for example, solutions related to
the Heaviside function\footnote{See \url{https://en.wikipedia.org/wiki/Heaviside_step_function}.}
$H(x)$, namely, $p_{ij} = H(x_i-x_j)$, $x_i \in \R$, $\forall\, i$, and its
discrete analogue, see Example~{\bf (III)} below. In the both cases $\beta= -1$; see also Comments~\ref{comments2.3} for other examples.
\end{Comments}

To continue presentation of Example~{\bf (II)}, def\/ine elements
$u_{ij}= p_{ij} s_{ij}$, $1 \le i \not= j \le n$.

\begin{Lemma}[dynamical classical Yang--Baxter relations]\label{lemma2.2}
\begin{gather}\label{equation2.7}
 [u_{ij}, u_{ik}+u_{jk}]+[u_{ik},u_{jk}]+[u_{ik},y_j] = 0, \qquad 1 < i < j < k \le n.
\end{gather}
 \end{Lemma}

Indeed,
\begin{gather*}
u_{ij}u_{jk}=u_{ik}u_{ij}+u_{jk}u_{ik}+h p_{ik}s_{ij}s_{jk}, \qquad
u_{jk}u_{ij}=u_{ij}u_{ik}+u_{ik}u_{jk}+h p_{ik}s_{jk}s_{ij},
\end{gather*}
and moreover, $[y_j,u_{ik}] = h p_{ik} [s_{ij},s_{jk}]$.

Therefore, the elements
\begin{gather*}
\theta_i = y_i - h \sum_{j < i} u_{ij} + h \sum_{i < j} u_{ij}, \qquad i=1,\ldots,n,
\end{gather*}
form a mutually commuting set of elements in the algebra
$\Z[\{p_{ij} \}] \otimes_{\Z} {\mathfrak{H}}_n$.

\begin{Theorem} \label{theorem2.1} Define matrix $M_n =(m_{i,j})_{1 \le i,j \le n}$ as follows
\begin{gather*}
m_{i,j}(u;z_1,\ldots,z_n) = \begin{cases}
u-z_i& \text{if $i=j$},\\
-h-p_{ij} & \text{if $i < j$}, \\
p_{ij} & \text{if $i > j$}.
\end{cases}
\end{gather*}
Then
\begin{gather*}
\operatorname{DET} \big | M_n(u;\theta_1,\ldots,\theta_n) \big |
= \prod_{j=1}^{n} (u-y_j).
\end{gather*}
Moreover, let us set $q_{ij}:=h^{2}( p_{ij}+p_{ij}^2) =
h^2 q_{i}q_{j}(q_{i}-q_{j})^{-2}$, $i < j$, then
\begin{gather*}
e_k(\theta_1,\ldots,\theta_n) = e_k^{({\boldsymbol{q}})}(y_1,\ldots,y_n), \qquad
1 \le k \le n,
\end{gather*}
where $e_k(x_1,\ldots,x_n)$ and $e_k^{({\boldsymbol{q}})}(x_1,\ldots,x_n)$ denote
correspondingly the classical and multiparameter quantum~{\rm \cite{FK}}
elementary polynomials\footnote{For the reader convenience we remind~\cite{FK} a~def\/inition of
the quantum elementary polynomial $e_{k}^{\boldsymbol{q}}(x_1,\ldots, x_n)$. Let
${\boldsymbol{q}}:= \{q_{ij} \}_{1 \le i <j \le n}$ be a collection of ``quantum
parameters'', then
\begin{gather*}
e_{k}^{\boldsymbol{q}}(x_1,\ldots,x_n)=\sum_{\ell} \sum_{1 \le i_{1} < \cdots < i_{\ell} \le n \atop j_{1} > i_{1}, \ldots ,j_{\ell} > i_{\ell}} e_{k-2\ell}(X_{
\overline{I \cup J}}) \prod_{a=1}^{\ell} q_{i_{a},j_{a}},
\end{gather*}
where $I=(i_1, \ldots, i_{\ell})$, $J=(j_1,\ldots,j_{\ell})$ should be
distinct elements of the set $\{1, \ldots, n \}$, and
$X_{\overline{I \cup J}}$ denotes set of variables~$x_a$ for which the
subscript~$a$ is neither one of~$i_m$ nor one of the~$j_m$.}.
\end{Theorem}

Let's stress that the elements $y_i$ and $\theta_j$ {\it do not} commute
in the algebra ${\mathfrak{H}}_n$, but the symmetric functions of
$y_1,\ldots,y_n$, i.e., the center of the algebra ${\mathfrak{H}}_n$, do.

A few remarks in order. First of all, $u_{ij}^2=p_{ij}^{2}$ are central
elements. Secondly, in the case $ h=0 $ and $y_i=0$, $\forall\, i$, the equality
\begin{gather*}
\operatorname{DET} \big | M_n(u;x_1,\ldots,x_n) \big | = u^{n}
\end{gather*}
describes the set of {\it polynomial} relations among the Dunkl--Gaudin
elements (with the following choice of parameters $p_{ij}=(q_i-q_j)^{-1}$ are
taken). And our f\/inal remark is that according to~\cite[Section~8]{GRTV}, the
quotient ring
\begin{gather*}
{\cal {H}}_n^{{\boldsymbol{q}}} : = \Q[y_1,\ldots,y_n]^{{\mathbb {S}}_n} \otimes
\Q[\theta_1, \dots, \theta_n] \otimes \Q[h]\Big/ \bigg\langle M_n(u;\theta_1,
\ldots,\theta_n)=\prod_{j=1}^{n}(u-y_j) \bigg\rangle
\end{gather*}
is isomorphic to the quantum equivariant cohomology ring of the cotangent
bundle $T^{*}{\cal {F}}l_n$ of the complete f\/lag variety of type $A_{n-1}$,
 namely,
\begin{gather*}
{\cal {H}}_n^{{\boldsymbol{q}}} \cong QH^{*}_{T^{n} \times \C^{*}} (T^{*} {\cal{F}}l_n)
\end{gather*}
with the following choice of quantum parameters: $Q_i:=h q_{i+1}/ q_{i}$,
$i=1,\ldots,n-1$.

On the other hand, in~\cite{KM3} we computed the so-called {\it multiparameter
deformation} of the equivariant cohomology ring of the complete f\/lag variety
of type~$A_{n-1}$.
\begin{center}
\framebox{\parbox[t]{6in}{A deformation def\/ined in~\cite{KM3} depends on parameters
$\{q_{ij},\,1 \le i < j \le n \}$ without any constraints are imposed. For the
special choice of parameters
\begin{gather*}
q_{ij}:= h^2 {q_i~q_j \over (q_i-q_j)^2}
\end{gather*}
the multiparameter deformation of the equivariant cohomology ring of the type
$A_{n-1}$ complete f\/lag variety~${\cal{F}}l_n$ constructed in~\cite{KM3}, is
isomorphic to the ring~${\cal {H}}_n^{{\boldsymbol{q}}}$. }}
\end{center}

\begin{Comments} \label{comments2.2} Let us f\/ix a set of independent parameters $\{q_1,\ldots,
q_n \}$ and def\/ine new parameters
\begin{gather*}
\left\{ q_{ij}:= h p_{ij}(p_{ij}+h)= h^2
{q_i q_j \over (q_i-q_j)^2 } \right\}, \quad 1 \le i < j \le n,\! \qquad \text{where} \quad p_{ij}= {q_{j}
\over q_{i}-q_{j}}, \quad i < j.
\end{gather*}
We set $\deg(q_{ij})= 2$, $\deg(p_{ij})= 1$, $\deg(h)=1$.

The new parameters $\{q_{ij} \}_{1 \le i < j \le n}$,
do not free anymore, but satisfy rather complicated algebraic relations. We
display some of these relations soon, having in mind a question:
is there some
intrinsic meaning of the algebraic variety def\/ined by the set of def\/ining
relations among the ``quantum parameters''~$\{q_{ij} \}$?

Let us denote by
${{\cal {A}}}_{n,h}$ the quotient ring of the ring of polynomials
$\Q[h][x_{ij},\, 1 \le i < j \le n]$ modulo the ideal generating by
polynomials $f(x_{ij})$ such that the specialization $x_{ij}=q_{ij}$ of a~polynomial $f(x_{ij})$, namely $f(q_{ij})$, is equal to zero. The algebra
${\cal{A}}_{n,h}$ has a natural f\/iltration, and we denote by
${\cal{A}}_n =\operatorname{gr}{\cal{A}}_{n,h}$ the corresponding associated graded
algebra.

To describe (a part of) relations among the parameters $\{q_{ij}\}$ let us
observe that parame\-ters~$\{p_{ij}\}$ and~$\{q_{ij} \}$ are related by the
following identity
\begin{gather*}
q_{ij} q_{jk} -q_{ik}( q_{ij} +q_{jk})+h^2 q_{ik} =
2 p_{ij} p_{ik} p_{jk}(p_{ik}+h) \qquad \text{if} \quad i < j < k.
\end{gather*}
Using this identity we can f\/ind the following relations among parameters in
question
\begin{gather}
q_{ij}^2 q_{jk}^2 +q_{ij}^2 q_{ik}^2 + h^4q_{ik}^2 q_{jk}^2 -
2q_{ij} q_{ik} q_{jk}(q_{ij}+q_{jk} +q_{ik}) \nonumber\\
\qquad {} -2 h^2 q_{ik}(q_{ij} q_{jk}+q_{ij} q_{ik}+q_{jk} q_{ik})
 = 8 h q_{ij}q_{ik} q_{jk} p_{ik},
\label{eq:xdef}
\end{gather}
 if $ 1 \le i < j < k \le n$.

Finally, we come to a relation of degree $8$
among the ``quantum parameters'' $\{ q_{ij} \}$
\begin{gather*}
\big(\text{l.h.s.\ of \eqref{eq:xdef}}\big)^2 = 64 h^2 q_{ij}^2 q_{ik}^3 q_{jk}^2, \qquad
1 \le i <j < k \le n.
\end{gather*}
There are also higher degree relations among the parameters $\{q_{ij} \}$
some of whose in degree $16$ follow from the deformed Pl\"{u}cker
relation between parameters $\{p_{ij} \}$:
\begin{gather*}
 {1 \over p_{ik} p_{jl}}={1 \over p_{ij} p_{kl}} + {1 \over p_{il} p_{jk}} +
{h \over p_{ij} p_{jk} p_{kl}}, \qquad i < j < k < l.
\end{gather*}
However, we don't know how to describe the algebra ${{\cal{A}}}_{n,h}$
generated by quantum parameters $\{ q_{ij} \}_{1 \le i < j \le n}$ even for
$n=4$.

The algebra ${\cal{A}}_n = \operatorname{gr}({\cal{A}}_{n,h})$ is isomorphic to the quotient
algebra of $\Q[x_{ij}, \, 1 \le i < j \le n]$ modulo the ideal generated by the
set of relations between ``quantum parameters''
\begin{gather*}
 \left\{\overline{q}_{ij}:=
\left({1 \over z_i-z_j} \right)^2\right\}_{1 \le i < j \le n},
\end{gather*}
which correspond to the Dunkl--Gaudin elements $\{\theta_i \}_{1 \le i \le n}$, see Section~\ref{section3.2} below for details. In this case the parameters
$\{ \overline{q}_{ij} \}$ satisfy the following relations
\begin{gather*}
\overline{q}_{ij}^2 \overline{q}_{jk}^2+\overline{q}_{ij}^2
\overline{q}_{ik}^2+
\overline{q}_{jk}^2 \overline{q}_{ik}^2 = 2 \overline{q}_{ij} \overline{q}_{ik} \overline{q}_{jk}(\overline{q}_{ij} + \overline{q}_{jk}+\overline{q}_{jk})
\end{gather*}
which correspond to the relations~\eqref{eq:xdef} in the special case $h=0$. One can
f\/ind a set of relations in degrees~$6$,~$7$ and~$8$, namely for a given pair-wise distinct integers $1 \le i,j,k,l \le n$, one has
\begin{itemize}\itemsep=0pt
\item one relation in degree $6$
\begin{gather*}
\overline{q}_{ij}^2\overline{q}_{ik}^2\overline{q}_{il}^2+
\overline{q}_{ij}^2\overline{q}_{jk}^2\overline{q}_{jl}^2+
\overline{q}_{ik}^2\overline{q}_{jk}^2\overline{q}_{kl}^2+
\overline{q}_{il}^2\overline{q}_{jl}^2\overline{q}_{kl}^2\\
\qquad {}-2 \overline{q}_{ij} \overline{q}_{ik} \overline{q}_{il} \overline{q}_{jk}
\overline{q}_{jl} \overline{q}_{kl} \left(
{{{\overline{q}_{ij}} \over {\overline{q}_{kl}}}}+
{{{\overline{q}_{kl}} \over {\overline{q}_{ij}}}}+
{{{\overline{q}_{ik}} \over {\overline{q}_{jl}}}}+
{{{\overline{q}_{jl}} \over {\overline{q}_{ik}}}}+
{{{\overline{q}_{il}} \over {\overline{q}_{jk}}}}+
{{{\overline{q}_{jk}} \over {\overline{q}_{il}}}} \right)\\
\qquad{} +
8 \overline{q}_{ij} \overline{q}_{ik} \overline{q}_{il}\overline{q}_{jk}
\overline{q}_{jl}\overline{q}_{kl} = 0;
\end{gather*}

\item three relations in degree $7$
\begin{gather*}
 \overline{q}_{ik}
\bigl(\overline{q}_{ij} \overline{q}_{il} \overline{q}_{kl} -
\overline{q}_{ij} \overline{q}_{il} \overline{q}_{jk} +
\overline{q}_{ij} \overline{q}_{jk} \overline{q}_{kl} -
\overline{q}_{il} \overline{q}_{jk} \overline{q}_{kl} \bigr)^2 \\
\qquad {} =
 8 \overline{q}_{ij}^2 \overline{q}_{ik}^2 \overline{q}_{jk} \overline{q}_{kl}
\bigl(\overline{q}_{jk}+\overline{q}_{jl}+ \overline{q}_{kl} \bigr) -
4 \overline{q}_{ij}^2 \overline{q}_{il}^2\overline{q}_{jl}
\bigl( \overline{q}_{jk}^2+ \overline{q}_{kl}^2 \bigr) ,
\end{gather*}

\item one relation in degree $8$
\begin{gather*}
\overline{q}_{ij}^2 \overline{q}_{il}^2\overline{q}_{jk}^2\overline{q}_{kl}^2+
\overline{q}_{ij}^2 \overline{q}_{ik}^2 \overline{q}_{jl}^2 \overline{q}_{kl}^2+
\overline{q}_{ik}^2 \overline{q}_{il}^2 \overline{q}_{jk}^2 \overline{q}_{jl}^2 =
2 \overline{q}_{ij} \overline{q}_{ik} \overline{q}_{il}\overline{q}_{jk}
\overline{q}_{jl}\overline{q}_{kl} \bigl(\overline{q}_{ij}\overline{q}_{kl}+
\overline{q}_{ik}\overline{q}_{jl}+\overline{q}_{il}\overline{q}_{jk} \bigr),
\end{gather*}
\end{itemize}
 However we don't know does the list of relations displayed above, contains
the all independent relations among the elements
$\{\overline{q}_{ij} \}_{1 \le i < j \le n}$ in degrees~$6$, $7$ and~$8$,
even for $n=4$. In degrees $\ge 9$ and $n \ge 5$ some independent relations
should appear.

Notice that the parameters $\big\{p_{ij}= { hq_j \over q_i-q_j}, \, i < j \big\}$
satisfy the so-called {\it Gelfand--Varchenko} relations, see, e.g.,~\cite{K3}
\begin{gather*}
p_{ij} p_{jk}=p_{ik} p_{ij} + p_{jk} p_{ik}+ h p_{ik}, \qquad i < j < k,
\end{gather*}
whereas parameters $\big\{ {\overline{p}}_{ij}= {1 \over q_i - q_j}, \,i <j \big\}$
satisfy the so-called {\it Arnold} relations
\begin{gather*}
 {\overline{p}}_{ij}{\overline{p}}_{jk}=
{\overline{p}}_{ik}{\overline{p}}_{ij}+{\overline{p}}_{jk}{\overline{p}}_{ik},
\qquad i < j < k.
\end{gather*}

\begin{Project}\label{project2.1}
Find Hilbert series ${\rm Hilb}({\cal{A}}_n,t)$ for $n \ge 4$.\footnote{This is a particular case of more general problem we are
interested in. Namely, let $\{f_{\alpha} \in \R[x_1,\ldots,x_n] \}_{1 \le
\alpha \le N}$ be a collection of linear forms, and $k \ge 2$ be an integer.
Denote by $I(\{ f_{\alpha} \})$ the ideal in the ring of polynomials $\R[z_1,
\ldots,z_N]$ generated by polynomials $\Phi(z_1,\ldots,z_N)$ such that
\begin{gather*}
\Phi\big(f_1^{-k},\ldots,f_{N}^{-k}\big) =0.
\end{gather*}
{\it Compute} the Hilbert series (polynomial?) of the quotient algebra
$\R[z_1,\ldots,z_N] / I(\{f_{\alpha} \})$.}
\end{Project}

For example, ${\rm Hilb}({\cal{A}}_3,t)={(1+t)(1+t^2) \over (1-t)^2}$.

Finally, if we set $q_i:= \exp(h z_i)$ and take the limit
$ \lim\limits_{h \to 0} \frac{h^2 q_i q_j}{(q_i-q_j)^2}$,
as a result we obtain the Dunkl--Gaudin parameter ${\overline{q}}_{ij}=
\frac{1}{(z_i-z_j)^2}$.
\end{Comments}

{\bf (III)} Consider the following representation of the degenerate
af\/f\/ine Hecke algebra ${\mathfrak{H}}_n$ on the ring of polynomials $P_n =
\Q[x_1,\ldots,x_n]$:
\begin{itemize}\itemsep=0pt
\item the symmetric group ${\mathbb{S}}_n$ acts on $P_n$ by means of
operators
\begin{gather*}
 \overline{s}_i =1+ (x_{i+1}-x_{i}-h) \partial_i, \qquad i=1,\ldots,n-1,
 \end{gather*}
 \item $y_i$ acts on the ring $P_n$ by multiplication on the variable
$x_i$: $y_i(f(x)) = x_i f(x)$, $f \in P_n$. Clearly,
\begin{gather*}
y_i \overline{s_i}-y_{i+1} \overline{s_i} = h \qquad \text{and} \qquad
y_i({\overline{s}}_i-1)=({\overline{s}}_i-1) y_{i+1} +x_{i+1} - x_i- h.
\end{gather*}
\end{itemize}
In the subsequent discussion we will identify the operator of multiplication
 by the variable~$x_i$, namely the operator~$y_i$, with~$x_i$.

This time def\/ine $u_{ij}= p_{ij} (\overline{s}_i-1)$, if $i < j$ and set
 $u_{ij} = -u_{ji}$ if $i > j$, where parameters
$\{p_{ij} \}$ satisfy the same conditions as in the previous example.
\begin{Lemma}\label{lemma2.3} The elements $\{ u_{ij},\, 1 \le i < j \le n \}$, satisfy the
dynamical classical Yang--Baxter relations displayed in Lemma~{\rm \ref{lemma2.2}}, equation~\eqref{equation2.7}.
\end{Lemma}

Therefore, the Dunkl elements
\begin{gather*}
 \overline{\theta}_i := x_i+\sum_{j \atop j \not= i} u_{ij}, \qquad i=1,\ldots,n,
 \end{gather*}
form a commutative set of elements.

\begin{Theorem}[\cite{GRTV}] \label{theorem2.2} Define matrix $\overline{M}_n =(\overline{m}_{ij})_{1 \le i,j \le n}$ as follows
\begin{gather*}
\overline{m}_{i,j}(u;z_1,\ldots,z_n) = \begin{cases}
u-z_i + \sum\limits_{j \not= i} h p_{ij}& \text{if $i=j$},\\
-h-p_{ij} & \text{if $i < j$}, \\
p_{ij} & \text{if $i > j$}.
\end{cases}
\end{gather*}
Then
\begin{gather*}
\operatorname{DET} \big | \overline{M}_n(u;\overline{\theta}_1,\ldots,\overline{\theta}_n)
\big | = \prod_{j=1}^{n}(u-x_j).
\end{gather*}
\end{Theorem}

\begin{Comments}\label{comments2.3}
Let us list a few more representations of the {\it dynamical}
 classical Yang--Baxter relations.
\begin{itemize}\itemsep=0pt
\item Trigonometric Calogero--Moser representation. Let $i < j$, def\/ine
\begin{gather*}
u_{ij}={x_j \over x_i-x_j } (s_{ij}-\epsilon), \qquad \epsilon=0 \ \text{or} \ 1, \\
 s_{ij}(x_i)
=x_j, \qquad s_{ij}(x_{j})=x_i, \qquad s_{ij}(x_k)=x_k, \qquad \forall \, k \not= i,j.
\end{gather*}

\item Mixed representation:
\begin{gather*}
 u_{ij} = \left({\lambda_j \over \lambda_i-\lambda_j} - {x_j \over x_i-x_j}\right)
(s_{ij}-\epsilon), \qquad \epsilon = 0 \ \text{or} \ 1, \qquad s_{ij}(\lambda_k) =\lambda_k, \qquad \forall\,
 k.
 \end{gather*}
\end{itemize}
We set $u_{ij}=-u_{ji}$, if $i > j$. In all cases we def\/ine Dunkl elements to
be $\theta_i=\sum\limits_{j \not= i} u_{ij}$.

Note that operators
\begin{gather*}
r_{ij} =\left({\lambda_i+\lambda_j \over \lambda_i-\lambda_j} - {x_i+x_j \over
x_i-x_j}\right) s_{ij}
\end{gather*}
satisfy the three term relations: $r_{ij}r_{jk}=r_{ik}r_{ij}+r_{jk}r_{ik}$,
and $r_{jk}r_{ij}=r_{ij}r_{jk}+r_{ik}r_{jk}$, and thus satisfy the
{\it classical} Yang--Baxter relations.
\end{Comments}

\subsubsection[Step functions and the Dunkl--Uglov representations of the degenerate
af\/f\/ine Hecke algebras~\cite{U}]{Step functions and the Dunkl--Uglov representations\\ of the degenerate
af\/f\/ine Hecke algebras~\cite{U}}\label{section2.1.2}

Con\-sider step functions $\eta^{\pm}\colon \R \longrightarrow \{0,1 \}$
\begin{gather*}
\text{(Heaviside function)} \qquad
\eta^{+}(x) =\begin{cases}
 1 & \text{if $x \ge 0$},\\
 0 & \text{if $x < 0$},
\end{cases} \qquad
 \eta^{-}(x) =\begin{cases}
 1 & \text{if $x > 0$},\\
 0 & \text{if $x \le 0$}.
\end{cases}
\end{gather*}
For any two real numbers $x_i$ and $x_j$ set $\eta_{ij}^{\pm}=\eta^{\pm}(x_i-x_j)$.

\begin{Lemma} \label{lemma2.4} The functions $\eta_{ij}$ satisfy the following relations
\begin{gather*}
\eta_{ij}^{\pm}+\eta_{ji}^{\pm}=1 + \delta_{x_i,x_j}, \qquad
(\eta_{ij}^{\pm})^{2}= \eta_{ij}^{\pm},\\
\eta_{ij}^{\pm} \eta_{jk}^{\pm} = \eta_{ik}^{\pm} \eta_{ij}^{\pm}+
\eta_{jk}^{\pm} \eta_{ik}^{\pm} - \eta_{ik}^{\pm},
\end{gather*}
where $\delta_{x,y}$ denotes the Kronecker delta function.
\end{Lemma}

To introduce the Dunkl--Uglov operators~\cite{U} we need a few more
def\/initions and notation. To start with, denote by $\Delta_i^{\pm}$ the
f\/inite dif\/ference operators: $\Delta_{i}^{\pm}(f)(x_1,\ldots,x_n)=f(\ldots,
x_i \pm 1,\ldots)$. Let as before, $ \{s_{ij},\, 1 \le i \not= j \le n,\,
s_{ij}=s_{ji} \}$, denotes the set of transpositions in the symmetric group
${\mathbb {S}}_n$. Recall that $s_{ij}(x_i)=x_j$, $s_{ij}(x_k)=x_k$, $\forall\, k \not= i,j$.
Finally def\/ine Dunkl--Uglov operators $d_i^{\pm} \colon \R^n \longrightarrow \R^n$
to be
\begin{gather*}
d_i^{\pm} =\Delta_i^{\pm} + \sum_{j < i} \delta_{x_i,x_j}- \sum_{j < i}
\eta_{ji}^{\pm} s_{ij} + \sum_{j > i} \eta_{ij}^{\pm} s_{ij}.
\end{gather*}
To simplify notation, set $u_{ij}^{\pm}:=\eta_{ij}^{\pm} s_{ij}$ if $i < j$, and
${\widetilde{\Delta}}_i^{\pm}= \Delta_{i}^{\pm} +\sum\limits_{j < i} \delta_{x_i,x_j}$.
\begin{Lemma} \label{lemma2.5} The operators $\{ u_{ij}^{\pm},\, 1 \le i < j \le n \}$ satisfy the
following relations
\begin{gather*}
\big[u_{ij}^{\pm},u_{ik}^{\pm} + u_{jk}^{\pm}\big]+\big[u_{ik}^{\pm},u_{jk}^{\pm}\big] +
\bigg[u_{ik}^{\pm}, \sum_{j < i} \delta_{x_i,x_j}\bigg] =0 \qquad \text{if} \ \ i < j < k.
\end{gather*}
\end{Lemma}

 From now on we {\it assume that $ x_i \in \Z$, $\forall\, i$},
that is, we will work with the restriction of the all operators def\/ined at
beginning of Example~\ref{example2.1}(c), to the subset $\Z^n
\subset \R^n$. It is easy to see that under the assumptions $x_i \in \Z$, $\forall\, i$, we will have
\begin{gather}\label{equation2.10}
 \Delta_j^{\pm} \eta_{ij}^{\pm} = (\eta_{ij}^{\pm} \mp \delta_{x_i,x_j})
\Delta_{i}^{\pm}.
\end{gather}
 Moreover, using relations~\eqref{equation2.13}, \eqref{equation2.14} one can prove that

\begin{Lemma} \label{lemma2.6} \quad
\begin{itemize}\itemsep=0pt
\item $[u_{ij}^{\pm}, {\widetilde{\Delta}}_i^{\pm} +
{\widetilde{\Delta}}_j^{\pm} ] =0$,

\item $[ u_{ik}^{\pm},{\widetilde{\Delta}_j}^{\pm}]=
\big[u_{ik}^{\pm}, \sum\limits_{j < i} \delta_{x_i,x_j}\big]$, $ i < j < k$.
\end{itemize}
\end{Lemma}

\begin{Corollary} \label{corollary2.1}\quad
\begin{itemize}\itemsep=0pt
\item The operators $\{ u_{ij}^{\pm}, \, 1 \le i < j < k \le n \}$,
and ${\widetilde{\Delta}}_i^{\pm}$, $i=1,\ldots,n $ satisfy the dynamical
classical Yang--Baxter relations
\begin{gather*}
\big[u_{ij}^{\pm},u_{ik}^{\pm} + u_{jk}^{\pm}\big]+\big[u_{ik}^{\pm},u_{jk}^{\pm}\big] +
\big[u_{ik}^{\pm}, {\widetilde{\Delta}}_j\big] =0 \qquad \text{if} \ \ i < j < k.
\end{gather*}

\item The operators $\{s_i:=s_{i,i+1}, \, 1 \le i < n, \, and\,
{\widetilde{\Delta}}_j^{\pm}, \, 1 \le j \le n \}$ give rise to two
representations of the degenerate affine Hecke algebra~${\mathfrak{H}}_n$.
In particular, the Dunkl--Uglov operators are mutually commute:
$[d_i^{\pm},d_j^{\pm}]=0$~{\rm \cite{U}}.
\end{itemize}
\end{Corollary}

\subsubsection{Extended Kohno--Drinfeld algebra and Yangian Dunkl--Gaudin
elements}\label{section2.1.3}

\begin{Definition}\label{definition2.3}
 Extended Kohno--Drinfeld algebra is an associative algebra over
$\Q[\beta]$ ge\-ne\-ra\-ted by the elements $\{z_1,\ldots,z_n \}$ and
$\{ y_{ij} \}_{1 \le i \not= j \le n}$ subject to the set of relations
\begin{enumerate}\itemsep=0pt
\item[(i)] The elements $\{ y_{ij} \{_{1 \le i \not= j \le n}$ satisfy the Kohno--Drinfeld relations
\begin{itemize}\itemsep=0pt
\item $y_{ij}=y_{ji}$, $[y_{ij},y_{kl}]=0$ if $i$, $j$, $k$, $l$ are distinct,

\item $[y_{ij},y_{ik}+y_{jk}]=0=[y_{ij}+y_{ik},y_{jk}]$ if $i < j < k$.
\end{itemize}

\item[(ii)] The elements $z_1,\ldots,z_n$ generate the free associative algebra
${\cal{F}}_n$.

\item[(iii)] Crossing relations:
\begin{itemize}\itemsep=0pt
\item $[z_i, y_{jk}] =0$ if $i \not= j,k$, $[z_i,z_{j}]= \beta
[y_{ij},z_i]$ if $i \not= j$.
\end{itemize}
\end{enumerate}
\end{Definition}

To def\/ine the (Yangian) Dunkl--Gaudin elements, cf.~\cite{GRTV}, let us consider
a set of elements $\{p_{ij} \}_{1 \le i \not= j \le n}$ subject to relations
\begin{itemize}\itemsep=0pt
\item $p_{ij}+p_{ji}= \beta$, $[p_{ij},y_{kl}]=0 =[p_{ij},z_k]$ for all
$i$, $j$, $k$,

\item $p_{ij} p_{jk}= p_{ik} ( p_{jk}-p_{ji} )$ if $i < j < k$.
\end{itemize}

Let us set $u_{ij}= p_{ij} y_{ij}$, $i \not= j$, and def\/ine the (Yangian)
Dunkl--Gaudin elements as follows
\begin{gather*}
\theta_i = z_i+ \sum_{j \not= i} u_{ij}, \qquad i=1,\ldots, n.
\end{gather*}

\begin{Proposition}[cf.~\protect{\cite[Lemma~3.5]{GRTV}}] \label{proposition2.1}
The elements $\theta_1,\ldots,\theta_n$ form a mutually commuting family.
\end{Proposition}

Indeed, let $i < j$, then
\begin{gather*}
[\theta_i,\theta_j] =
 [z_i,z_j] + \beta [z_i, y_{ij}] + p_{ij} [y_{ij},z_i+z_j]\\
 \hphantom{[\theta_i,\theta_j] = }{}
 + \sum_{k \not= i,j} \big(p_{ik} p_{jk} \big[y_{ij}+y_{ik}, y_{jk} \big]
+p_{ik} p_{ji} \big[ y_{ij},y_{ik}+y_{jk} \big] \big) =0.
\end{gather*}
A representation of the extended Kohno--Drinfeld algebra has been constructed
in~\cite{GRTV}, namely one can take
\begin{gather*}
y_{ij}:= T_{ij}^{(1)} T_{ji}^{(1)} -T_{jj}^{(1)}=y_{ji},\qquad z_{i}:=
\beta T_{ii}^{(2)} -\frac{\beta}{2} T_{ii}^{(1)} \big (T_{ii}^{(1)}-1\big),\\
 p_{ij}:= \frac{\beta q_j}{q_i -q_j}, \qquad i \not= j,
\end{gather*}
where $q_1,\ldots,q_n$ stands for a set of mutually commuting {\it quantum}
parameters, and $ \big\{T_{ij}^{(s)} \big\}_{1 \le i,j \le n \atop s \in \Z_{\ge 0}}$
denotes the set of generators of the Yangian $Y({\mathfrak{gl}}_n)$, see, e.g.,~\cite{Mo}.

A proof that the elements $\{z_i \}_{1 \le i \le n}$ and $\{y_{ij} \}_{1 \le i \not= j \le n}$ satisfy the extended Kohno--Drinfeld algebra relations
is based on the following relations, see, e.g., \cite[Section~3]{GRTV},
\begin{gather*}
\big[T_{ij}^{(1)}, T_{kl}^{(s)}\big]= \delta_{il} T_{kj}^{(s)} - \delta_{jk} T_{il}^{(s)} , \qquad i,j,k,l = 1,\ldots,n, \qquad s \in \Z_{\ge 0}.
\end{gather*}

\subsection[``Compatible'' Dunkl elements, Manin matrices and
algebras related with weighted\\ complete graphs $r K_{n}$]{``Compatible'' Dunkl elements, Manin matrices and
algebras\\ related with weighted complete graphs $\boldsymbol{r K_{n}}$}\label{section2.2}

Let us consider a collection of generators $\{ u_{ij}^{(\alpha)},\,1 \le i,j
\le n,\, \alpha =1,\ldots,r \}$, subject to the following relations
\begin{itemize}\itemsep=0pt
\item either the unitarity
(the case of sign~``${+}$'') or the symmetry relations (the case of sign~``${-}$'')\footnote{More generally one can impose the $q$-symmetry conditions
\begin{gather*}
u_{ij}+q u_{ji} = 0, \qquad 1 \le i < j \le n
\end{gather*}
and ask about relations among the local Dunkl elements to ensure the
commutativity of the global ones. As one might expect, the matrix
$Q := \big(\theta_j^{(a)}\big)_{1 \le a \le r \atop 1 \le j \le n}$ composed from the
local Dunkl elements should be a~$q$-Manin matrix. See, e.g.,~\cite{CF}, or
\url{https://en.wikipedia.org/wiki/Manin.matrix} for a def\/inition and basic properties of the latter.}
\begin{gather}\label{equation2.11}
 u_{ij}^{(\alpha)} \pm u_{ji}^{(\alpha)}=0, \qquad \forall \, \alpha, i,j,
\end{gather}

\item {\it local $3$-term relations}:
\begin{gather}\label{equation2.12}
u_{ij}^{(\alpha)} u_{jk}^{(\alpha)}+ u_{jk}^{(\alpha)} u_{ki}^{\alpha)}+
u_{ki}^{(\alpha)} u_{ij}^{(\alpha)}=0, \qquad i,j,k \ \ \text{are distinct}, \quad 1 \le
\alpha \le r.
\end{gather}
\end{itemize}
We def\/ine {\it global} 3-term relations algebra $3T_{n,r}^{(\pm)}$ as ``compatible product'' of the local 3-term relations algebras. Namely, we
require that the elements
\begin{gather*}
U_{ij}^{({\boldsymbol{\lambda}})}:= \sum_{\alpha=1}^{r} \lambda_{\alpha}
u_{ij}^{(\alpha)}, \qquad 1 \le i,j \le n,
\end{gather*}
satisfy the global 3-term relations
\begin{gather*}
U_{ij}^{(\boldsymbol{\lambda})} U_{jk}^{(\boldsymbol{\lambda})} +
U_{jk}^{(\boldsymbol{\lambda})} U_{ki}^{(\boldsymbol{\lambda})} +
U_{ki}^{(\boldsymbol{\lambda})} U_{ij}^{(\boldsymbol{\lambda})} = 0
\end{gather*}
 for all values of parameters
$\{\lambda_i \in \R , \, 1 \le \alpha \le r \}$.

It is easy to check that our request is equivalent to a validity of the
following sets of relations among the generators $\big\{u_{ij}^{(\alpha)} \big\}$
\begin{enumerate}\itemsep=0pt
\item[(a)] {\it local $3$-term relations}: $ u_{ij}^{(\alpha)} u_{jk}^{\alpha)}+u_{jk}^{(\alpha)} u_{ki}^{(\alpha)} +u_{ki}^{\alpha)} u_{ij}^{(\alpha)} =0$,
\item[(b)] {\it $6$-term crossing relations}:
\begin{gather*}
u_{ij}^{(\alpha)} u_{jk}^{(\beta)}+u_{ij}^{(\beta)} u_{jk}^{(\alpha)}+
u_{k,i}^{(\alpha)} u_{ij}^{(\beta)} u_{ki}^{(\alpha)} + u_{jk}^{(\alpha)}
u_{ki}^{(\beta)} + u_{jk}^{(\beta)} u_{ki}^{(\alpha)}=0,
\end{gather*}
$i$, $j$, $k$ are distinct, $\alpha \not= \beta$.
\end{enumerate}

Now let us consider {\it local} Dunkl elements
\begin{gather*}
\theta_{i}^{(\alpha)}:= \sum_{j \neq i} u_{ij}^{(\alpha)}, \qquad j=1,\ldots,n, \quad
\alpha=1,\ldots,r.
\end{gather*}
It follows from the local 3-term relations \eqref{equation2.12} that for a f\/ixed
$\alpha \in [1,r]$ the local Dunkl elements
$\big\{ \theta_i^{(\alpha)} \big\}_{1 \le i \le n \atop 1 \le \alpha \le r}$
 either mutually commute
(the sign~``$+$''), or pairwise anticommute (the sign~``$-$''). Similarly, the
global 3-term relations imply that the global Dunkl
elements
\begin{gather*}
 \theta_i^{(\lambda)}:= \lambda_1 \theta_i^{(1)}+ \cdots + \lambda_r \theta_i^{(r)} = \sum_{j \not=i} U_{ij}^{(\lambda)}, \qquad
 i=1,\ldots,n,
 \end{gather*}
also either mutually commute (the case~``$+$'') or pairwise anticommute
(the case~``$-$'').

Now we are looking for a set of relations among the local Dunkl elements
which is a consequence of the commutativity (anticommutativity) of the
global Dunkl elements.
 It is quite clear that if $i < j$, then
\begin{gather*}
\big[\theta_i^{(a)},\theta_j^{(b)}\big]_{\pm} =\sum_{a=1}^{r} \lambda_{a}^2 \big[\theta_i^{(a)},\theta_j^{(a)}\big]_{\pm} + \sum_{1 \le a < b \le r}
\lambda_{a} \lambda_b
\big(\big[\theta_i^{(a)},\theta_j^{(b)}\big]_{\pm}+\big[\theta_i^{(b)},
\theta_j^{(a)}\big]_{\pm} \big),
\end{gather*}
and the commutativity (or anticommutativity) of the global Dunkl elements for
all $(\lambda_1,\ldots, \lambda_r) \in \R^{r}$ is equivalent to the following set of relations
\begin{itemize}\itemsep=0pt
\item $[\theta_i^{(a)},\theta_j^{(a)}]_{\pm} =0$,

\item $[\theta_i^{(a)},\theta_j^{(b)}]_{\pm} + [\theta_i^{(b)},\theta_j^{(a)}]_{\pm} = 0$, $a < b$ and $i < j$,
where by def\/inition we set $[a,b]_{\pm}:=a b \mp b a$.
\end{itemize}

In other words, the matrix $\varTheta_n: = \big(\theta_i^{(a)}\big)_{1 \le a \le r
\atop 1 \le i \le n}$
should be either a {\it Manin matrix} (the case~``$+$''), or its super analogue
(the case~``$-$''). Clearly enough that a similar
construction can be applied to the algebras studied in Section~\ref{section2}, {\bf I}--{\bf III},
and thus it produces some interesting examples of the Manin matrices.
It is an interesting {\it problem} to describe
the algebra generated by the local Dunkl elements $\big\{ \theta_i^{(a)}\big\}_{1 \le
a \le r \atop 1 \le i \le n}$ and a commutative subalgebra generated by the
global Dunkl elements inside the former. It is also an interesting {\it
question} whether or not the coef\/f\/icients
$C_1,\ldots ,C_n$ of the column characteristic polynomial $\operatorname{Det}^{\rm col} |\varTheta_n - t I_n| = \sum\limits_{k=0}^{n} C_k t^{n-k}$ of the Manin matrix
$\varTheta_n$ generate a commutative subalgebra? For a def\/inition of the
column determinant of a matrix, see, e.g.,~\cite{CF}.

However a close look at this problem and the question posed needs an
additional treatment and has been omitted from the content of the present
paper.

Here we are looking for a ``natural conditions'' to be imposed on the set
of generators $\{ u_{ij}^{\alpha} \}_{1 \le \alpha
\le r \atop 1 \le i,j \le n}$ in order to ensure that the local
Dunkl elements satisfy the commutativity (or anticommutativity) relations:
\begin{gather*}
\big[\theta_i^{(\alpha)},\theta_j^{(\beta)}\big]_{\pm}=0, \qquad \text{for all} \ \ 1 \le i < j
\le n, \qquad 1 \le \alpha, \beta \le r.
\end{gather*}

The ``natural conditions'' we have in mind
are
\begin{itemize}\itemsep=0pt
\item {\it locality relations}:
\begin{gather}\label{equation2.13}
\big[u_{ij}^{(\alpha)}, u_{kl}^{(\beta)}\big]_{\pm} = 0 ,
\end{gather}

\item {\it twisted classical Yang--Baxter relations}:
\begin{gather}\label{equation2.14}
 \big[u_{ij}^{(\alpha)}, u_{jk}^{(\beta)}\big]_{\pm}+ \big[u_{ik}^{(\alpha)},u_{ji}^{(\beta)}\big]_{\pm}+ \big[u_{ik}^{(\alpha)},u_{jk}^{(\beta)}\big]_{\pm}=0,
\end{gather}
if $i$, $j$, $k$, $l$ are distinct and $1 \le \alpha, \beta \le r$.
\end{itemize}

Finally we def\/ine a multiple analogue of the three term relations algebra,
denoted by \linebreak $3T^{\pm}(r K_n)$, to be the quotient of the global $3$-term
relations algebra $3T_{n,r}^{\pm}$ modulo the two-sided ideal generated by
the left hand sides of relations~\eqref{equation2.13},~\eqref{equation2.14} and that of the
following relations
\begin{itemize}\itemsep=0pt
\item $ \big(u_{ij}^{(\alpha)} \big)^2=0$,
$\big[u_{ij}^{(\alpha)},u_{ij}^{(\beta)}\big]_{\pm}=0$, for all $i \not= j$, $\alpha
\not= \beta $.
\end{itemize}

The outputs of this construction are
\begin{itemize}\itemsep=0pt
\item commutative (or anticommutative) quadratic algebra $3T^{(\pm)}(r K_n)$ generated by
the elements $\big\{u_{ij}^{(\alpha)} \big\}_{ 1 \le i < j \le n \atop \alpha =1,
\ldots,r} $,

\item a family of $nr$ either mutually commuting (the case~``$+$''), or
 pair-wise anticommuting (the case~``$-$'') local Dunkl elements
$\big\{ \theta_{i}^{(\alpha)} \big\} _{i=1,\ldots,n \atop \alpha=1,\ldots,r} $.
\end{itemize}

We {\it expect} that the subalgebra generated by local Dunkl elements in the
algebra $3T^{+}( r K_n)$ is closely related (isomorphic for $r=2$) with
the coinvariant algebra of
the diagonal action of the symmetric group ${\mathbb{S}}_n$ on the ring of
polynomials $\Q\big[ X_{n}^{(1)},\ldots,X_{n}^{(r)}\big]$, where $X_{n}^{(j)}$ stands
for the set of variables $\big\{x_1^{(j)},\ldots,x_n^{(j)} \big\}$. The algebra $3T^{-}(2 K_n)^{\rm anti}$ has been studied in~\cite{K} and~\cite{BDK}.
In the present paper we state only our old conjecture.
\begin{Conjecture}[A.N.~Kirillov, 2000]\label{conjecture2.1}
\begin{gather*}
{\rm Hilb}\big(3T^{-}( 3 K_n)^{\rm anti},t\big) = (1+t)^n (1+n t)^{n-2} ,
\end{gather*}
where for any algebra~$A$ we denote by $A^{\rm anti}$ the quotient of algebra~$A$
by the two-sided ideal generated by the set of anticommutators $\{ a b + b a \,|\, (a,b) \in A \times A \}$.
\end{Conjecture}

According to observation of M.~Haiman~\cite{H}, the number $2^n (n+1)^{n-2}$
is thought of as being equal to the dimension of the space of triple
coinvariants of the symmetric group~$\mathbb{S}_n$.

\subsection{Miscellany}\label{section2.3}

\subsubsection[Non-unitary dynamical classical Yang--Baxter algebra
${\rm DCYB}_n$]{Non-unitary dynamical classical Yang--Baxter algebra
$\boldsymbol{{\rm DCYB}_n}$}\label{section2.3.1}

Let $\widetilde{{\cal A}_n}$ be the quotient of the algebra ${\mathfrak F}_n$
by the two-sided ideal generated by the rela\-tions~\eqref{equation2.2}, \eqref{equation2.5} and~\eqref{equation2.6}. Consider elements
\begin{gather*}
\theta_i=x_i+\sum_{a \not= i} u_{ia} \qquad \text{and} \qquad {\bar {\theta_j}}= -x_j+
\sum_{b \not= j} u_{bj}, \qquad 1\le i < j \le n.
\end{gather*}
Clearly, if $i < j$, then
\begin{gather*}
[\theta_i,{\bar \theta_j}] +[x_i,x_j] = \left[\sum_{k=1}^{n} x_k , u_{ij}\right] +
\sum_{k \not= i,j} w_{ikj},
\end{gather*}
where the elements $w_{ijk}$, $i < j$, have been def\/ined in Lemma~\ref{lemma2.1}, equation~\eqref{equation2.3}.

Therefore the elements $\theta_i$ and ${\bar \theta_j}$ {\it commute}
 in the algebra ${\widetilde{A}_n}$.

In the case when $x_i=0$ for all $i=1, \ldots,n$, the relations
\begin{gather*}
w_{ijk} := [u_{ij},u_{ik}+u_{jk}]+[u_{ik},u_{jk}]=0 \qquad \text{if $i$, $j$, $k$ are all distinct},
\end{gather*}
 are well-known as the {\it non-unitary classical Yang--Baxter relations}.
 Note that for a given triple of pair-wise distinct~$(i,j,k)$ one has in fact
6~relations. These six relations imply that $[\theta_i,{\bar {\theta_j}}]=0$.
 However, in general,
\begin{gather*}
 [\theta_i,\theta_j]=\biggl[\sum_{k \not= i,j} u_{ik},
u_{ij}+u_{ji} \biggr] \not= 0.
\end{gather*}

{\bf Dynamical classical Yang--Baxter algebra ${\rm DCYB}_n$.}
In order to ensure the commutativity relations among the Dunkl
elements~\eqref{equation2.1}, i.e., $[\theta_i,\theta_j]=0$ for all $i$, $j$, let us remark
that if $i \not=j$, then
\begin{gather*}
[\theta_,\theta_j] = [x_i+u_{ij},x_j+ u_{ji}]+
 [x_i+x_j,u_{ij}] + \left [u_{ij},\sum_{k=1}^{n} x_k\right] \\
 \hphantom{[\theta_,\theta_j] =}{}
 +
\sum_{k=1 \atop k \not= i,j}^{n} [u_{ij}+u_{ik},u_{jk}]+[u_{ik}, u_{ji}] +[x_i,u_{jk}]+[u_{ik},x_j]+ [x_k,u_{ij}].
\end{gather*}

\begin{Definition} \label{definition2.4}
Def\/ine {\it dynamical non-unitary classical Yang--Baxter algebra ${\rm DNUCYB}_n$}
to be the quotient of the free associative algebra
$\Q \langle \{ x_{i}, \,1 \le i \le n\}, \, \{ u_{ij} \}_{1 \le i \not= j \le n}
 \rangle$
by the two-sided ideal generated by the following set of relations
\begin{itemize}\itemsep=0pt
\item zero curvature conditions:
\begin{gather}\label{equation2.15}
[x_i+u_{ij},x_j+ u_{ji}]=0, \qquad 1 \le i \not= j \le n,
\end{gather}

\item conservation laws conditions:
\begin{gather*}
\left[u_{ij},\sum_{k=1}^{n} x_k\right] =0 \qquad \text{for all} \ \ i \not= j, k.
\end{gather*}

\item crossing relations:
\begin{gather*}
[x_i+x_j,u_{ij}]=0, \qquad i \not= j.
\end{gather*}

\item twisted dynamical classical Yang--Baxter relations:
\begin{gather*}
[u_{ij}+u_{ik},u_{jk}]+[u_{ik}, u_{ji}] +[x_i,u_{jk}]+[u_{ik},x_j]+
[x_k,u_{ij}]=0,
\end{gather*}
$i$, $j$, $k$ are distinct.
\end{itemize}
\end{Definition}

It is easy to see that the twisted classical Yang--Baxter relations
\begin{gather}\label{equation2.17}
[u_{ij}+u_{ik},u_{jk}]+[u_{ik},u_{ji}]= 0, \qquad i,j,k \ \ \text{are distinct},
\end{gather}
for a f\/ixed triple of distinct indices $i$, $j$, $k$ contain in fact~$3$ dif\/ferent relations whereas the non-unitary classical
Yang--Baxter relations
\begin{gather*}
[u_{ij}+u_{ik},u_{jk}]+[u_{ij}, u_{ik}], \qquad i,j,k \ \ \text{are distinct},
\end{gather*}
contain $6$ dif\/ferent relations for a f\/ixed triple of distinct indices
$i$, $j$, $k$.

\begin{Definition}\label{definition2.5} \quad
\begin{itemize}\itemsep=0pt

\item Def\/ine {\it dynamical classical Yang--Baxter algebra ${\rm DCYB}_n$} to
be the quotient of the algebra ${\rm DNUCYB}_n$ by the two-sided ideal generated by
the elements
\begin{gather*}
\sum_{k \not=i,j} [u_{ik},u_{ij}+u_{ji}] \qquad \text{for all} \ \ i \not= j.
\end{gather*}

\item Def\/ine {\it classical Yang--Baxter algebra ${\rm CYB}_n$} to be the
quotient of the dynamical classical Yang--Baxter algebra ${\rm DCYB}_n$ by the set
of relations
\begin{gather*}
x_i=0\qquad \text{for} \ \ i=1,\dots,n.
\end{gather*}
\end{itemize}
\end{Definition}

\begin{Example}\label{examples2.1}
Def\/ine
\begin{gather*}
p_{ij}(z_1,\ldots,z_n) =\begin{cases}
\dfrac{z_i}{z_i-z_j} & \text{if $1 \le i < j \le n$},\vspace{1mm}\\
 - \dfrac{z_j}{z_j-z_i} & \text{if $ n \ge i > j \ge 1$}.
\end{cases}
\end{gather*}
Clearly, $p_{ij}+p_{ji}=1$. Now def\/ine operators $u_{ij} = p_{ij} s_{ij}$,
and the truncated Dunkl operators to be $\theta_i = \sum\limits_{j \not= i}
 u_{ij}$, $i=1, \ldots,n$. All these operators act on the f\/ield of rational
functions $\Q(z_1,\ldots,z_n)$; the operator $s_{ij}=s_{ji}$ acts as the
exchange operator, namely, $s_{ij}(z_i)=z_j$, $s_{ij}(z_k)=z_k$, $\forall\, k \not=
i,j$, $s_{ij}(z_j)=z_i$.

Note that this time one has
\begin{gather*}
p_{12} p_{23} = p_{13} p_{12}+p_{23} p_{13} - p_{13}.
\end{gather*}

It is easy to see that the operators $\{ u_{ij},\, 1
\le i \not= j \le n \}$ satisfy relations~\eqref{equation3.1}, and
therefore, satisfy the twisted classical Yang--Baxter relations~\eqref{equation2.14}. As a~corollary we obtain that the truncated Dunkl operators $\{\theta_i,\,i=1,
\ldots,n \}$ are pair-wise commute. Now consider the Dunkl operator $D_i=
\partial_{{z_{i}}} + h \theta_i$, $i=1,\ldots,n$, where~$h$ is a~parameter.
Clearly that $ [ \partial_{{z_i}} +\partial_{{z_j}}, u_{ij}]=0$, and
therefore $[D_i,D_j]=0$, $\forall\, i,j$. It easy to see that
\begin{gather*}
 s_{i,i+1} D_{i} - D_{i+1} s_{i,i+1}=h, \qquad [D_i, s_{j,j+1}]=0 \qquad \text{if} \ \ j \not=i, i+1.
\end{gather*}
In such a manner we come to the well-known representation of the degenerate
af\/f\/ine Hecke algebra ${\mathfrak H}_n$.
\end{Example}

\subsubsection{Dunkl and Knizhnik--Zamolodchikov elements}\label{section2.3.2}

 Assume that $\forall \, i$, $x_i=0$, and generators $\{u_{ij},\,
1\le i < j \le n \}$ satisfy the locality conditions~\eqref{equation2.2} and the
classical Yang--Baxter relations
\begin{gather*}
[u_{ij},u_{ik}+u_{jk}]+[u_{ik},u_{jk}]=0 \qquad \text{if} \ \ 1 \le i < j < k \le n.
\end{gather*}
Let $y,z,t_1,\ldots,t_n $ be parameters, consider the rational function
\begin{gather*}
 F_{\rm CYB}(z;{\boldsymbol{t}}):= F_{\rm CYB}(z;t_1,\ldots,t_n)= \sum_{1 \le i < j \le n} {(t_i
-t_j) u_{ij} \over (z-t_i)(z-t_j)}.
\end{gather*}
Then
\begin{gather*}
 [F_{\rm CYB}(z;{\boldsymbol{t}}),F_{\rm CYB}(y;{\boldsymbol{t}})]=0 \qquad \text{and} \qquad \operatorname{Res}_{z=t_{i}}
F_{\rm CYB}(z;{\boldsymbol{t}}) =\theta_i .
\end{gather*}

Now assume that a set of generators $\{ c_{ij},\, 1 \le i \not= j
\le n\}$ satisfy the locality and symmetry (i.e., $c_{ij}=c_{ji}$) conditions,
and the Kohno--Drinfeld relations:
\begin{gather*}
 [c_{ij}, c_{kl}]=0 \qquad \text{if} \ \ \{i,j\} \cap \{k,l\} = {\varnothing}, \\
[c_{ij},c_{jk}+c_{ik}]=0=[c_{ij}+c_{ik},c_{jk}], \qquad i < j < k.
\end{gather*}

Let $y,z,t_1,\ldots,t_n $ be parameters, consider the rational function
\begin{gather*}
 F_{\rm KD}(z;{\boldsymbol{t}}):= F_{\rm KD}(z;t_1,\ldots,t_n)=\sum_{1 \le i \not= j \le n}
{ c_{ij} \over (z-t_i)(t_i-t_j)} = \sum_{1 \le i < j \le n} {c_{ij} \over
(z-t_i)(z-t_j)}.
\end{gather*}
 Then
\begin{gather*}
[F_{\rm KD}(z;{\boldsymbol{t}}),F_{\rm KD}(y;{\boldsymbol{t}})] =0 \qquad \text{and} \qquad \operatorname{Res}_{z=t_i}
F_{\rm KD}(z;{\boldsymbol{t}})= {\rm KZ}_{i},
\end{gather*}
where
\begin{gather*}
{\rm KZ}_i=\sum_{j=1 \atop j \not=i}^{n} {c_{ij} \over t_i -t_j}
\end{gather*}
 denotes the truncated {\it Knizhnik--Zamolodchikov} element.

\subsubsection{Dunkl and Gaudin operators}\label{section2.3.3}

{\bf (a)~Rational Dunkl operators.} Consider the quotient of the algebra
${\rm DCYB}_n$, see Def\/i\-ni\-tion~\ref{definition2.2}, by the
two-sided ideal generated by elements
\begin{gather*}
 \{[x_i+x_j, u_{ij}] \} \qquad \text{and} \qquad \{ [x_k,u_{ij}],\, k \not= i,j \}.
\end{gather*}
Clearly the Dunkl elements~\eqref{equation2.1} mutually commute. Now let us consider
the so-called {\it Calogero--Moser} representation of the algebra ${\rm DCYB}_n$ on
the ring of polynomials $R_n:=\R[z_1,\ldots,z_n]$ given by
\begin{gather*}
x_i(p(z))= \lambda {\partial p(z) \over \partial z_i}, \qquad u_{ij}(p(z))=
{1 \over
 z_i-z_j} (1-s_{ij}) p(z),\qquad p(z) \in R_n.
 \end{gather*}
 The symmetric group ${\mathbb{S}}_n$ acts on the ring $R_n$ by means of
transpositions $s_{ij} \in {\mathbb{S}}_n$:
$s_{ij}(z_i)=z_j$, $s_{ij}(z_j)=z_i$, $s_{ij}(z_k)=z_k$ if $k \not=i,j$.

In the Calogero--Moser representation the Dunkl elements~$\theta_i$ becomes the
 rational Dunkl operators~\cite{Du}, see Def\/inition~\ref{definition1.1}. Moreover, one has
$[x_k,u_{ij}]=0$ if$k \not= i,j$, and
\begin{gather*}
x_i u_{ij}=u_{ij} x_j + {1 \over z_i - z_j} (x_i-x_j-u_{ij}), \qquad
x_j u_{ij}=u_{ij} x_i - {1 \over z_i - z_j} (x_i-x_j-u_{ij}).
\end{gather*}

{\bf (b)~Gaudin operators.}
The Dunkl--Gaudin representation of the algebra ${\rm DCYB}_n$ is def\/ined on the
f\/ield of rational functions $K_n:= \R(q_1,\ldots,q_n)$~and given by
\begin{gather*}
 x_i(f(q)):= \lambda {\partial f(q) \over \partial q_i},\qquad u_{ij}= {s_{ij}
\over q_i-q_j}, \qquad f(q) \in K_n,
\end{gather*}
but this time we {\it assume} that $w(q_i)=q_i$, $\forall\, i \in [1,n]$ and
for all $w \in {\mathbb{S}}_n$. In the Dunkl--Gaudin representation the
Dunkl elements becomes the rational Gaudin operators, see, e.g.,~\cite{MTV}.
Moreover, one has $[x_k,u_{ij}]= 0$, if~$k \not=i,j$, and
\begin{gather*}
x_i u_{ij}= u_{ij} x_j - {u_{ij} \over q_i-q_j},\qquad x_j u_{ij}= u_{ij} x_i+
{u_{ij} \over q_i -q_j}.
\end{gather*}

\begin{Comments}\label{comments2.4}
 It is easy to check that if $f \in \R[z_1,\ldots,z_n]$, and $x_i:= {\frac{\partial}{\partial z_{i}}}$, then the
following commutation relations are true
\begin{gather*}
x_i f= f x_i + \frac{\partial}{\partial_{z_{i}}}(f), \qquad u_{ij} f = s_{ij}(f) u_{ij} + \partial_{z_{i},z_{j}}(f).
\end{gather*}
Using these relations it easy to check that in the both cases~$({\boldsymbol{a}})$ and~$({\bf b})$ the elementary symmetric polynomials $e_k(x_1,\ldots,x_n)$
commute with the all generators $ \{u_{ij} \}_{1 \le i,j \le n}$, and
therefore commute with the all Dunkl elements $\{\theta_i \}_{1 \le i \le n}$.
 Let us {\it stress} that $[\theta_i, x_k] \not= 0$ for all $1 \le i,k \le n$.
\end{Comments}

\begin{Project}\label{project2.2}
Describe a commutative algebra generated by the Dunkl elements
$\{\theta_i \}_{1 \le i \le n}$ and the elementary symmetric polynomials
$ \{ e_k(x_1,\ldots,x_n) \}_{1 \le k \le n}$.
\end{Project}

\subsubsection[Representation of the algebra $3T_n$ on the free algebra
$\Z \langle t_1,\ldots,t_n \rangle$]{Representation of the algebra $\boldsymbol{3T_n}$ on the free algebra
$\boldsymbol{\Z \langle t_1,\ldots,t_n \rangle}$}\label{section2.3.4}

Let ${\mathcal{F}}_n =\Z \langle t_1,\ldots,t_n \rangle $ be free associative
algebra over the ring of integers $\Z$, equipped with the action of the
 symmetric group $\mathbb{S}_n$: $s_{ij}(t_i)=t_j$, $s_{ij}(t_k)=t_k$, $\forall\,
k \not=i,j$.

Def\/ine the action of $u_{ij} \in 3T_n$ on the set of generators of the algebra
$\mathcal{F}_n$ as follows
\begin{gather*}
u_{ij}(t_k)= \delta_{i,k} t_i t_j - \delta_{j,k} t_j t_i.
\end{gather*}

The action of generator $u_{ij}$ on the whole algebra $\mathcal{F}_n$ is
def\/ined by linearity and the twisted Leibniz rule:
\begin{gather*}
u_{ij}(1)=0, \qquad u_{ij}(a+b)=u_{ij}(a)+u_{ij}(b),\qquad u_{ij}(a b)=u_{ij}(a) b+s_{ij}(a) u_{ij}(b).
\end{gather*}
It is easy to see from~\eqref{equation2.15} that
\begin{gather*}
s_{ij} u_{jk}= u_{ik} s_{ij},\qquad s_{ij} u_{kl}=u_{kl} s_{ij} \qquad \text{if} \ \ \{i,j\}
\cap \{ k,l\}=\varnothing, \qquad u_{ij}+u_{ji}=0.
\end{gather*}
Now let us consider operator
\begin{gather*}
u_{ijk}:= u_{ij} u_{jk}-u_{jk} u_{ik} -u_{ik} u_{ij},\qquad 1 \le i < j < k \le n.
\end{gather*}
\begin{Lemma}\label{lemma2.7}
\begin{gather*}
u_{ijk}(a b)=u_{ijk}(a) b+s_{ij} s_{jk}(a) u_{ijk}(b),\qquad a,b \in \mathcal{F}_n.
\end{gather*}
\end{Lemma}

\begin{Lemma}\label{lemma2.8}
\begin{gather*}
 u_{ijk}(a)=0\qquad \forall\, a \in \mathcal{F}_n .
 \end{gather*}
\end{Lemma}

Indeed,
\begin{gather*}
u_{ijk}(t_i)=-u_{jk} (u_{ij}(t_i))-u_{ik}(u_{ij}(t_i)) = -t_i u_{jk}(t_k) -u_{ik}(t_i) t_j =t_{i}(t_k t_j) -(t_i t_k) t_j =0,\\
u_{ijk}(t_k)=u_{ij}(u_{jk}(t_k)) -u_{jk}(u_{ik}(t_k)) =-u_{ij}(t_kt_j)+u_{jk}(t_kt_i)=t_k(u_{ij}(t_j) + u_{jk}(t_k)t_i = 0, \\
u_{ijk}(t_j)= u_{ij}(u_{jk}(t_j)) - u_{ik}(u_{ij}(t_j))= - u_{ij}(t_j)t_k-
t_ju_{ik}(t_i) =(t_jt_i)t_k-t_j(t_it_k)=0.
\end{gather*}
Therefore Lemma~\ref{lemma2.8} follows from Lemma~\ref{lemma2.7}.

Let $\mathcal{F}_n^{\bullet}$ be the quotient of the free algebra
$\mathcal{F}_n$ by the two-sided ideal generated by elements
$t_i^2t_j-t_jt_i^2$, $1 \le i \not= j \le n$. Since
$u_{i,j}^2(t_i)=t_it_j^2-t_j^2t_i$, one can
def\/ine a representation of the algebra~$3T_n^{(0)}$ on that
$\mathcal{F}_n^{\bullet}$. One can also def\/ine a representation of the
algebra $3T_n^{(0)}$ on that~$\mathcal{F}_n^{(0)}$, where
$\mathcal{F}_n^{(0)}$ denotes the quotient of the algebra~$\mathcal{F}_n$ by
the two-sided ideal generated by elements
$\{ t_{i}^2,\, 1 \le i \le n \}$. Note that
$(u_{i,k} u_{j,k} u_{i,j})(t_k)=[t_i t_j t_i,t_k] \not= 0$ in the algebra
$\mathcal{F}_n^{(0)}$, but the elements $u_{i,j} u_{i,k} u_{j,k} u_{i,j}$,
$ 1 \le i < j < k \le n$, which belong to the kernel of the Calogero--Moser
representation \cite{K}, act trivially both on the algebras~$\mathcal{F}_n^{(0)}$ and that~$\mathcal{F}_n^{\bullet}$.

Note f\/inally that the algebra $\mathcal{F}_n^{(0)}$ is
{\it Koszul} and has Hilbert series
${\rm Hilb}\big(\mathcal{F}_n^{(0)},t\big)={1+t \over 1-(n-1) t}$, whereas
the algebra $\mathcal{F}_n^{\bullet}$ is {\it not} Koszul for
$n \ge 3$, and
\begin{gather*}
{\rm Hilb}
(\mathcal{F}_n^{\bullet},t)= {1 \over (1-t)(1-(n-1)t)(1-t^2)^{n-1}}.
\end{gather*}

In Appendix~\ref{appendixA.5} we apply the representation introduced in this section to the
study of relations in the subalgebra $Z_{n}^{(0)}$ of the algebra
$3T_{n}^{(0)}$ generated by the elements $u_{1,n}, \ldots,u_{n-1,n}$. To
distinguish the generators $\{ u_{ij} \}$ of the algebra $3T_n^{(0)}$ from
the introduced in this section opera\-tors~$u_{ij}$ acting on it, in Appendix~\ref{appendixA.5} we will use for the latter notation $\nabla_{ij}:= u_{ij}$.

\subsubsection{Kernel of Bruhat representation}\label{section2.3.5}

Bruhat representations, classical and quantum, of algebras~$3T_{n}^{(0)}$ and
 $3QT_{n}$ can be seen as a~connecting link between commutative subalgebras
generating by either additive or multiplicative Dunkl elements in these
algebras, and classical and quantum Schubert and Grothendieck calculi.

$(\bf Ia)$ {\bf Bruhat representation of algebra} $3T_n^{(0)}$, cf.~\cite{FK}.
Def\/ine action of $u_{i,j} \in 3T_n^{(0)}$ on the group ring of the symmetric
group $\Z[{\mathbb S}_n]$ as follows: let $w \in {\mathbb S}_n$, then
\begin{gather*} u_{i,j} w=\begin{cases}
w s_{ij} & \text{if \ $l(w s_{ij})=l(w)+1 $}, \\
0 & \text{otherwise}.
\end{cases}
\end{gather*}
Let us remind that $s_{ij} \in {\mathbb S}_n$ denotes the transposition
that interchanges~$i$ and~$j$ and f\/ixes each $k \not= i,j$; for each
permutation $u \in {\mathbb S}_n$, $l(u)$ denotes its length.

$(\bf Ib)$ {\bf Quantum Bruhat representation of algebra} $3QT_n$, cf.~\cite{FK}.
Let us remind that algebra $3QT_n$ is the quotient of the 3-term relations
algebra $3T_n$ by the two-sided ideal generated by the elements
\begin{gather*}
\{u_{ij}^2, |j-i| \ge 2 \} \bigcup \{u_{i,i+1}^2=q_i,~i=1,\ldots,n-1 \}.
\end{gather*}
Def\/ine the $\Z[q]-$linear action of $u_{i,j} \in 3QT_n$, $i < j$, on the
extended group ring of the symmetric group $\Z[q] [{\mathbb S}_n]$ as follows: let $w \in {\mathbb S}_n$, and $q_{ij}=q_i q_{i+1} \cdots q_{j-1}$, $i < j$, then
\begin{gather*}
u_{i,j} w=\begin{cases}
w s_{ij} & \text{if \ $l(w s_{ij})=l(w)+1 $}, \\
q_{ij} w s_{ij} & \text{if \ $l(ws_{ij})=l(w)-l(s_{ij})$}, \\
0 & \text{otherwise}.
\end{cases}
\end{gather*}
Let us remind, see, e.g.,~\cite{M}, that in general one has
\begin{gather*}
 l(w s_{ij})=\begin{cases}
l(w)-2 e_{ij}-1 & \text{if} \ \ w(i) > w(j), \\
l(w)+2~e_{ij}+1 & \text{if} \ \ w(i) < w(j).
\end{cases}
\end{gather*}
Here $e_{ij}(w)$ denotes the number of $k$ such that $i < k < j$ and $w(k)$
lies between~$w(i)$ and~$w(j)$. In particular, $l(ws_{ij})=l(w)+1$ if\/f
$e_{ij}(w)=0$ and $w(i) < w(j)$;
$l(ws_{ij})=l(w)-l(s_{ij})=l(w)-2(j-i)+1$ if\/f $w(i) > w(j)$ and $e_{ij}=j-i-1$ is the maximal possible.

$({\bf II})$ {\bf Kernel of the Bruhat representation.}
It is not dif\/f\/icult to see that the following elements of degree
three and four belong to the kernel of the Bruhat representation:
\begin{gather*}
({\bf IIa}) \quad u_{i,j}u_{i,k}u_{i,j} \qquad \text{and} \qquad u_{i,k} u_{j,k} u_{i,k} \qquad
\text{if} \ \ 1 \le i < j < k \le n;\\
({\bf IIb}) \quad u_{i,k}u_{i,l}u_{j,l} \qquad \text{and} \qquad u_{j,l}u_{i,l}u_{i,k};\\
({\bf IIc}) \quad u_{il} u_{ik} u_{jl} u_{il}, \qquad u_{il}u_{ij}u_{kl}u_{il}, \qquad
u_{ik}u_{il}u_{jk}u_{ik}, \\
\hphantom{({\bf IIc}) \quad} u_{ij}u_{ik}u_{il}u_{ij}, \qquad u_{ik}u_{il}u_{ij}u_{ik}
\qquad \text{if} \ \ 1 \le i < j < k <l \le n .
\end{gather*}

This observation motivates the following def\/inition.
\begin{Definition}\label{definition2.6} {\it The reduced 3-term relation algebra $3T_n^{\rm red}$} is
def\/ined to be the quotient of the algebra $3T_n^{(0)}$ by the two-sided ideal
generated by the elements displayed in {\bf IIa}--{\bf IIc} above.
\end{Definition}

\begin{Example}\label{example2.1}
\begin{gather*}
{\rm Hilb}\big(3T_3^{\rm red},t\big)=(1,3,4,1), \qquad \dim \big(3T_3^{\rm red}\big) =9,\\
{\rm Hilb}\big(3T_4^{\rm red},t\big)=(1,6,19,32,19,6,1),\qquad \dim \big(3T_4^{\rm red}\big)=84, \\
{\rm Hilb}\big(3T_5^{\rm red},t\big)=(1,10,55,190,383,370,227,102,34,8,1), \qquad \dim\big(3T_5^{\rm red}\big)=1374.
\end{gather*}
\end{Example}

We {\it expect} that $\dim (3T_n^{red})_{{n \choose 2}-1} = 2(n-1)$ if $n \ge 3$.
\begin{Theorem} \label{theorem2.3} \quad
\begin{enumerate}\itemsep=0pt
\item[$1.$] The algebra $3T_n^{\rm red}$ is finite-dimensional, and its Hilbert
polynomial has degree ${n \choose 2}$.

\item[$2.$] The maximal degree ${n \choose 2}$ component of the algebra
$3T_n^{\rm red}$ has dimension one and generated by any element which is
equal to the product $($in any order$)$ of all generators of the algebra
$3T_n^{\rm red}$.

\item[$3.$] The subalgebra in $3T_n^{\rm red}$ generated by the elements
$\{ u_{i,i+1},\, i=1,\ldots,n-1 \}$ is canonically isomorphic to the
nil-Coxeter algebra~${\rm NC}_n$. In particular, its Hilbert polynomial is equal
to $[n]_{t} !:= \prod\limits_{j=1}^{n} {(1-t^{j}) \over 1-t}$, and the element
$\prod\limits_{j=1}^{n-1} \prod\limits_{a=j}^{1} u_{a,a+1}$ of degree ${n \choose 2}$
generates the maximal degree component of the algebra $3T_n^{\rm red}$.

\item[$4.$] The subalgebra over $\Z$ generated by the Dunkl elements
$\{\theta_1, \ldots, \theta_n \}$ in the algebra $3T_n^{\rm red}$ is canonically
isomorphic to the cohomology ring $H^{*}({\cal F}l_{n}, \Z)$ of the type~$A$ flag variety ${\cal F}l_n$.
\end{enumerate}
\end{Theorem}

A def\/inition of the nil-Coxeter algebra ${\rm NC}_n$ one can f\/ind in Section~\ref{section4.1.1}.
It is known, see~\cite{Ba} or Section~\ref{section4.1.1}, that the subalgebra generated
by the elements $\{ u_{i,i+1},\,i=1,\ldots,n-1 \}$ in the whole algebra
$3T_n^{(0)}$ is canonically isomorphic to the nil-Coxeter algebra~${\rm NC}_n$ as well.

We {\bf expect} that the kernel of the Bruhat representation of the algebra
$3T_n^{(0)}$ is generated by all {\it monomials} of the form
$u_{i_{1},j_{1}} \cdots u_{i_{k},j_{k}}$ such that the sequence of
transpositions $t_{i_{1},j_{1}},\ldots, t_{i_{k},j_{k}}$
does not correspond to {\it a~path} in the Bruhat graph of the symmetric group~${\mathbb S}_n $. For example if $1 \le i < j < k < l \le n$, the elements $u_{i,k}u_{i,l}u_{j,l}$ and $u_{j,l}u_{i,l}u_{i,k}$ do belong to the
kernel of the Bruhat representation.

 \begin{Problem} \label{problem2.1}\quad
\begin{enumerate}\itemsep=0pt
\item[$1.$] The image of the Bruhat representation of the algebra
$3T_n^{(0)}$
defines a subalgebra
\begin{gather*}
\operatorname{Im}\big(3T_n^{(0)}\big) \subset \operatorname{End}_{\Q} (\Q[{\mathbb{S}}_{n}] ).
\end{gather*}
Does this image isomorphic to the algebra $3T_n^{\rm red}$? Compute Hilbert polynomials of algebras $\operatorname{Im}\big(3T_n^{(0)}\big)$ and
$3T_n^{\rm red}$.

\item[$2.$] Describe the image$($s$)$ of the affine nil-Coxeter algebra
${\widetilde{{\rm NC}}}_n$, see Section~{\rm \ref{section4.1.1}}, in the algebras $3T_n^{\rm red}$ and
$ \operatorname{End}_{\Q}(\Q[{\mathbb{S}}_{n}] )$.
\end{enumerate}
\end{Problem}

\subsubsection[The Fulton universal ring \cite{Fu}, multiparameter quantum
cohomology of f\/lag\\ varieties \cite{FK} and the full Kostant--Toda
lattice \cite{FTL+,FTL}]{The Fulton universal ring \cite{Fu}, multiparameter
quantum cohomology\\ of f\/lag varieties \cite{FK} and the full Kostant--Toda
lattice \cite{FTL+,FTL}}\label{section2.3.6}

Let $X_{n}= (x_1,\ldots, x_{n})$ be be a set of variables, and
\begin{gather*}
{\boldsymbol{g}}:= {\boldsymbol{g}}^{(n)}= \{g_a[b] \,|\, a \ge 1, \, b \ge 1, \, a+b \le n
 \}
\end{gather*}
 be a set of parameters; we put $\deg(x_i)=1$ and $\deg(g_a[b])=b+1$, and set
$g_k[0]:=x_k$, $k=1,\ldots,n$. For a subset $S \subset [1,n]$ we denote by $X_{S}$ the set of variables $\{ x_i\,|\, i \in S \}$.

Let $t$ be an auxiliary variable, denote by $M=(m_{ij})_{1 \le i,j \le n}$
the matrix of size~$ n$ by~$n$ with the following elements:
\begin{gather*}
m_{i,j}=\begin{cases}
x_i+t & \text{if \ $ i=j$}, \\
g_i[j-i]& \text{if \ $j > i$}, \\
-1& \text{if \ $i-j=1$}, \\
0 & \text{if \ $i-j > 1$}.
\end{cases}
\end{gather*}
Let $P_n(X_n,t) =\det |M |$.

\begin{Definition} \label{definition2.7}
The Fulton universal ring ${\cal R}_{n-1}$ is def\/ined to be the
quotient\footnote{If $P(t,X_n)=\sum\limits_{k \ge 1} f_k(X_n) t^k$, $f_k(X_n) \in \Q[Xn]$ is a polynomial, we denote by
$\langle P(t,X_n) \rangle$
 the ideal in the polynomial ring $\Q[X_n]$ generated by the coef\/f\/icients $\{f_1,f_2,\ldots \}$.\label{footnote26}}
\begin{gather*}
{\cal R}_{n-1} = \Z\big[{\boldsymbol{g}}^{(n)}\big][x_1,\ldots,x_{n}]/
\langle P_n(X_{n},t)-t^{n} \rangle.
\end{gather*}
\end{Definition}

\begin{Lemma} \label{lemma2.9}
Let $P_n(X_n,t)=\sum\limits_{k=0}^{n}c_k(n)t^{n-k}$, $c_0(n)=1$. Then
\begin{gather}\label{equation2.20}
 c_k(n):=c_k\big(n;X_n,{\boldsymbol{g}}^{(n)}\big)=
\sum_{{1 \le i_1 < i_2 < \cdots < i_s < n
 \atop j_1 \ge 1,\ldots, j_s \ge 1}
\atop m:=\sum(j_a+1) \le n} \prod_{a=1}^{s}g_{i_a}[j_a]
e_{k-m}\big(X_{[1,n] {\setminus} \bigcup\limits_{a=1}^{s} [i_a,i_a+j_a ]} \big),
\end{gather}
where in the summation we assume additionally that the sets
$[i_a,i_a+j_a]:= \{i_a,i_{a}+1,\ldots,i_a+j_a \}$, $a=1,\ldots, s$, are
pair-wise disjoint.
\end{Lemma}

It is clear that ${\cal R}_{n-1} = \Z[{\boldsymbol{g}}^{(n)}][x_1,\ldots,x_{n}]/
\langle c_{n}(1),\ldots,c_{n}(n) \rangle $.
One can easily see that the coef\/f\/icients $c_k(n)$ and $g_m[k]$ satisfy the
following recurrence relations~\cite{Fu}:
\begin{gather*}\label{equation2.21}
c_k(n)=c_k(n-1)+\sum_{a=0}^{k-1}g_{n-a}[a]c_{k-a-1}(n-a-1),\qquad c_0(n)=1,
\\
g_m[k]=c_{k+1}(m+k)-c_{k+1}(m+k-1)-
\sum_{a=0}^{k-1}g_{m+k-a}[a]c_{k-a}(m+k-a), \\
 g_{m}[0]:=x_m.
\end{gather*}

On the other hand, let $\{ q_{ij} \}_{1 \le i < j \le n}$ be a set of
(quantum) parameters, and $e_k^{({\boldsymbol{q}})}(X_n)$ be the multiparameter
quantum elementary polynomial introduced in~\cite{FK}. We are interested in
to describe a set of relations between the parameters $ \{ g_i[j] \}_{i \ge 1, j \ge 1 \atop i+j \le n}$ and the quantum parameters $\{ q_{ij} \}_{1 \le i < j \le n}$ which implies that
\begin{gather*}
c_k(n)= e_{k}^{({\boldsymbol{q}})}(X_n) \qquad \text{for} \quad k= 1,\ldots,n.
\end{gather*}
To start with, let us recall the recurrence relations among the quantum
elementary polynomials, cf.~\cite{P}. To do so, consider the generating
function
\begin{gather*}
E_n\big( X_n; \{q_{ij} \}_{1 \le i < j \le n}\big) = \sum_{k=0}^{n}
e_k^{({\boldsymbol{q}})}(X_n) t^{n-k}.
\end{gather*}

\begin{Lemma}[\cite{FGP,P}]\label{lemma2.10}
One has
\begin{gather*}
 E_n\big( X_n; \{q_{ij} \}_{1 \le i < j \le n}\big) = (t+x_n)
E_{n-1}\big( X_{n-1}; \{q_{ij} \}_{1 \le i < j \le n-1}\big) \\
\hphantom{E_n\big( X_n; \{q_{ij} \}_{1 \le i < j \le n}\big) =}{}+
\sum_{j=1}^{n-1} q_{jn} E_{n-2}\big( X_{[1,n-1] {\setminus} \{j \}};
\{q_{a,b } \}_{1 \le a < b \le n-1 \atop a\not=j, b \not= j} \big).
\end{gather*}
\end{Lemma}

\begin{Proposition}\label{proposition2.2}
Parameters $\{g_a[b] \}$ can be expressed polynomially in
terms of quantum parameters $\{q_{ij} \}$ and variables $x_1,\ldots, x_n$,
in a such way that
\begin{gather*}
c_k(n) = e_k^{({\boldsymbol{q}})}(X_n), \qquad \forall \, k,n.
\end{gather*}
Moreover,
\begin{itemize}\itemsep=0pt
\item $g_a[b]= \sum\limits_{k=1}^{a} q_{k,a+b} \prod\limits_{j=a+1}^{a+b-1} (x_j-x_k)
+ \text{lower degree polynomials in $x_1,\ldots,x_n$}$,
\item the quantum parameters $\{ q_{ij} \}$ can be presented as rational functions in terms of variables $x_1,\ldots,x_n$ and polynomially in terms
of parameters $ \{g_a[b] \}$ such that the equality $c_k(n) = e_k^{({\boldsymbol{q}})}(X_n)$ holds for all~$k$,~$n$.
\end{itemize}
\end{Proposition}

In other words, the transformation
\begin{gather*}
\{q_{ij}\}_{1 \le i < j \le n}
\longleftrightarrow \{g_{a}[b] \}_{a+b \le n \atop a \ge 1, \, b \ge 1}
\end{gather*}
def\/ines a ``birational transformation'' between the algebra
$\Z[{\boldsymbol{g}}^{(n)}] [X_n] / \langle P_n(X_n,t)-t^n \rangle $ and
multiparameter quantum deformation of the algebra $H^{*}({\cal{F}}l_n,\Z)$.

\begin{Example}\label{example2.2}
Clearly,
\begin{gather*}
 g_{n-1}[1]=\sum_{j=1}^{n-1} q_{j,n}, \quad n \ge 2 \qquad \text{and} \qquad g_{n-2}[2]= \sum_{j=1}^{n-2} q_{jn} (x_{n-1}-x_{j}),
 \quad n \ge 3.
 \end{gather*}
Moreover
\begin{gather*}
g_1[3]= q_{14} \big((x_2-x_1)(x_3-x_1)+q_{23}-q_{12} \big)+ q_{24}
\big(q_{13}-q_{12} \big), \\
g_2[3]= q_{15} \big((x_3-x_1)(x_4-x_1)+q_{24}+q_{34}-q_{12}-q_{13} \big)\\
\hphantom{g_1[3]=}{} +
q_{25} \big((x_3-x_2)(x_4-x_2) +q_{14}+q_{34}-q_{12}-q_{23} \big) +
q_{35} \big(q_{14}+q_{24}-q_{13}-q_{23} \big).
\end{gather*}
\end{Example}

\begin{Comments}\label{comments2.5}
The full Kostant--Toda lattice (FKTL for short) has been introduced in the
end of $70's$ of the last century by B.~Kostant and since that
time has been extensively studied both in Mathematical and Physical literature.
We refer the reader to the original paper by B.~Kostant~\cite{FTL+,FTL} for the def\/inition of the ${\rm FKTL}$ and its basic
properties. In the present paper we just want to point out on a connection of
the Fulton universal ring and hence the multiparameter deformation of the
cohomology ring of complete f\/lag varieties, and polynomial integral of motion
of the FKTL. Namely,
\begin{center}
\framebox{\parbox[t]{4in}{Polynomials $c_k(n;X_n,{\boldsymbol{g}}^{(n)})$ def\/ined by
\eqref{equation2.20} coincide with the polynomial integrals of motion of the FKTL.}}
\end{center}

It seems an interesting task to clarify a meaning of the ${\rm FKTL}$ rational
integrals of motion in the context of the universal Schubert calculus~\cite{Fu} and the algebra $3HT_n(0)$, as well as any meaning of universal
Schubert or Grothendieck polynomials in the context of the Toda or full
Kostant--Toda lattices.
\end{Comments}

\section[Algebra $3HT_n$]{Algebra $\boldsymbol{3HT_n}$}\label{section3}

Consider the twisted classical Yang--Baxter relation
\begin{gather*}
[u_{ij}+u_{ik},u_{jk}]+[u_{ik},u_{ji}]=0,
\end{gather*}
where $i$, $j$, $k$ are distinct.
Having in mind applications of the Dunkl elements to combinatorics and
algebraic geometry, we split the above relation into two relations
\begin{gather}\label{equation3.1}
u_{ij}u_{jk}=u_{jk}u_{ik}-u_{ik}u_{ji} \qquad \text{and} \qquad
u_{jk}u_{ij}=u_{ik}u_{jk}-u_{ji}u_{ik}
\end{gather}
and impose the following {\it unitarity} constraints
\begin{gather*}
u_{ij}+u_{ji}=\beta,
\end{gather*}
where $\beta$ is a central element.
Summarizing, we come to the following def\/inition.

\begin{Definition} \label{definition3.1}
Def\/ine algebra $3T_n(\beta)$ to be the quotient of the free
associative algebra
\begin{gather*}
\Z[\beta] \langle u_{ij},\, 1 \le i < j \le n
\rangle
\end{gather*}
 by the set of relations
\begin{itemize}\itemsep=0pt
\item {\it locality}: $u_{ij} u_{kl}=u_{kl}u_{ij}$ if $\{i,j\}
\cap \{k,l\}=\varnothing$,

\item {\it $3$-term relations}:
$u_{ij} u_{jk}=u_{ik} u_{ij}+u_{jk} u_{ik}- \beta u_{ik}$, and
$u_{jk} u_{ij}=u_{ij} u_{ik}+u_{ik} u_{jk}-\beta u_{ik}$
if~$1 \le i < j < k \le n$.
\end{itemize}
\end{Definition}

It is clear that the elements $\{ u_{ij},\, u_{jk},\, u_{ik},\, 1 \le i < j <k
\le n \}$ satisfy the classical Yang--Baxter relations, and therefore, the
elements $\big\{\theta_{i}:=\sum\limits_{j \not= i} u_{ij}, \,1=1,\ldots,n \big\}$ form a~mutually commuting set of elements in the algebra $3T_n(\beta)$.
\begin{Definition}\label{definition3.2}
 We will call $ \theta_1,\ldots,\theta_n$ by the
(universal) additive Dunkl elements.
\end{Definition}

For each pair of indices $i < j$, we def\/ine element $q_{ij}:=u_{ij}^{2} - \beta u_{ij} \in 3T_n(\beta)$.
\begin{Lemma} \label{lemma3.1} \quad
\begin{enumerate}\itemsep=0pt
\item[$1.$] The elements $\{ q_{ij}, \, 1 \le i < j \le n \}$ satisfy the Kohno--Drinfeld relations
$($known also as the horizontal four term relations$)$
\begin{gather*}
 q_{ij} q_{kl}=q_{kl} q_{ij} \qquad \text{if} \ \ \{i,j\} \cap \{k,l\} = \varnothing, \\
 [q_{ij},q_{ik}+q_{jk}]=0,\qquad [q_{ij}+q_{ik},q_{jk}]=0 \qquad \text{if} \ \ i < j < k.
 \end{gather*}

\item[$2.$] For a triple $(i < j < k)$ define $u_{ijk}:= u_{ij}-u_{ik}+u_{jk}$.
Then
\begin{gather*}
u_{ijk}^2= \beta u_{ijk}+q_{ij}+q_{ik}+q_{jk}.
\end{gather*}

\item[$3.$] Deviation from the Yang--Baxter and Coxeter relations:
\begin{gather*}
 u_{ij} u_{ik} u_{jk}-u_{jk} u_{ik} u_{ij}=[u_{ik},q_{ij}]=
[q_{jk},u_{ik}],\\
u_{ij}u_{jk}u_{ij}-u_{jk}u_{ij}u_{jk}= q_{ij}u_{ik}-u_{ik}q_{jk}.
\end{gather*}
\end{enumerate}
\end{Lemma}

\begin{Comments} \label{comments3.1}
It is easy to see that the horizontal 4-term relations
listed in Lemma~\ref{lemma3.1}(1), are consequences of the locality conditions among
 the generators $\{q_{ij}\}$, together with the commutativity conditions
among the Jucys--Murphy elements
\begin{gather*}
d_i:= \sum_{j=i+1}^{n} q_{ij},\qquad i=2,\ldots,n,
\end{gather*}
namely, $[d_i,d_j]=0$. In \cite{K} we describe some properties of a
commutative subalgebra generated by the Jucys--Murphy elements in the (nil\footnote{That is the quotient of the Kohno--Drinfeld algebra generated by the elements $\{q_{ij}\}$ by the two-sided ideal generated by the elements $\{q_{ij}^2\}_{1 \le i,j \le n}$.})
Kohno--Drinfeld algebra. It is well-known that the Jucys--Murphy elements
generate a maximal commutative subalgebra in the group ring of the symmetric
group ${\mathbb S}_n$. It is an open {\it problem}
\begin{center}
\framebox{\parbox[t]{4 in}{describe def\/ining
relations among the Jucys--Murphy elements in the group ring
$\Z[{\mathbb S}_n]$. }}
\end{center}
\end{Comments}

Finally we introduce the ``Hecke quotient'' of the algebra $3T_n(\beta)$,
denoted by $3HT_n(\beta)$.
\begin{Definition} \label{definition3.3}
Def\/ine algebra $3HT_n(\beta)$ to be the quotient of the algebra
$3T_n(\beta)$ by the set of relations
\begin{gather*}
 q_{ij} q_{kl}=q_{kl} q_{ij} \qquad \text{for all} \ \ i,\,j,\,k,\,l.
 \end{gather*}
\end{Definition}
In other words we assume that the all elements $\{q_{ij},\, 1 \le i < j \le n \}$ are {\it central} in the algebra $3T_n(\beta)$.
From Lemma~\ref{lemma3.1} follows immediately that in the algebra $3HT_n(\beta)$ the
elements $\{u_{ij} \}$ satisfy the multiplicative (or quantum) Yang--Baxter
relations
\begin{gather}\label{equation3.2}
 u_{ij} u_{ik} u_{jk}=u_{jk} u_{ik} u_{ij} \qquad \text{if} \ \ i < j < k.
\end{gather}
To underline the dependence of the algebra $3HT_n(\beta)$ on the central
elements ${\boldsymbol{q}}:=\{q_{ij}\}$, we will use for the former the notation
$3T_n^{({\boldsymbol{q}})}(\beta)$ as well.
\begin{Exercises}[some relations in the algebra $3T_n^{({\boldsymbol{q}})}(\beta)$] \label{Exercises3.1} \quad

\begin{enumerate}\itemsep=0pt
\item[1.] Noncommutative analogue of recurrence relation among the Catalan
numbers \cite{K2, K}, cf.\ Section~\ref{section5.1}.
Let $k$, $n$ be positive integers, $ k < n$ and $i_1,\ldots,i_k$,
$1 \le i_{k} < n $, be a collection of pairwise distinct integers. {\it Prove} the following identity in the algebra $3T_n^{({\boldsymbol{q}})}(\beta)$\footnote{{\it Hint}: denote the r.h.s.\ of of the identity stated in item~(1)
 by~$R_{I}$. One possible proof is based on induction and examination of the
element $R_{ I \cup \{i_{k+2} \}}:= u_{i_{a_{1},i_{a+2}}} R_{I} -R_{I} u_{i_{a+1},i_{a_{i+2}}}$.}
\begin{gather*}
 \prod_{a=1}^{k} u_{i_{a},i_{a+1}} +\sum_{r=2}^{k+1} \left( \prod_{a=r}^{n}
 \beta (u_{i_{a},i_{a+1}}) u_{i_{a_{k+1}},i_{a_{1}}} \left( \prod_{a=1}^{r-2}
u_{i_{a},i_{a+1}}\right)\right)\\
\qquad{} = \beta \prod_{a=1}^{k} u_{i_{a},i_{a+1}} -
 \beta \big( u_{i_{1},i_{k+1}} R_{ I{\setminus} \{i_{k+1\}}} - R_{I{\setminus}\{i_{k+1}\}} u_{i_{k},i_{k+1}} \big),
 \end{gather*}
where $R_{I}$ denotes the r.h.s.\ of the above identity.
For example,
\begin{gather*}
12\,23 +23\,31 + 31\,12 = \beta (12-13+23),\\
 12\,23\,34 + 23\,34\,41+ 34\,41\,12 + 41\,12\,23\\
 \qquad {} = \beta ( 12\,23 - 14(12 -13+23) + (12-13 +23) 34 ),
\end{gather*}
where we use short notation $ij:=u_{ij}$.
See Introduction, {\it summation formula},~{\bf A}, for an interpretation of
the above formula in the case $\beta=0$, $q_{ij}=0$, $\forall\, i,j$.
{\it Note} that the above formula does not depend on deformation
(or quantum) parameters $\{q_{ij}\}$, in particular it also true for the
algebras $3T(\Gamma)$ associated with a simple graph~$\Gamma$, and gives rise to quantum as well as $K$-theoretic deformations of the Orlik--Terao algebra
of a~simple graph, cf.~\cite{Li}.

\item[2.] Cyclic relations, cf.~\cite{FK}.
Let $i_1,i_2,\ldots,i_k$, $1 \le i_a \le n$ be a collection of pairwise
distinct integers.
{\it Show} that
\begin{gather*}
\sum_{r=1}^{k-1} \left(\prod_{a=r+1}^{k} u_{i_{1},i_{a}} \right) \left(\prod_{a=2}^{r} u_{i_{1},i_{a}} \right) u_{i_{r+1},i_{1}}
= - \left( \sum_{a=2}^{k}
q_{i_{1},i_{a}} \left(\prod_{b =a+1}^{k } u_{i_{a},i_{b}} \right) \left( \prod_{b=2}^{a-1} u_{i_{a},i_{b}} \right) \right).
\end{gather*}
For example, $12\,13\,14\,21 + 13\,14\,12\,31+14\,12\,13\,41= -q_{12}\,23\,24 -q_{13}\,34\,32 - q_{14}\,42\,43$.
\end{enumerate}
Note that the r.h.s.\ does not depend on parameter $\beta$.
\end{Exercises}

\subsection[Modif\/ied three term relations algebra $3MT_n(\beta,\psi)$]{Modif\/ied three term relations algebra $\boldsymbol{3MT_n(\beta,\psi)}$}\label{section3.1}

Let $\beta$, $\{q_{ij}=q_{ji},\,\psi_{ij}=\psi_{ji},\, 1 \le i, j \le n \}$,
be a set of mutually commuting elements.

\begin{Definition}\label{definition3.7}
Modif\/ied $3$-term relation algebra $3MT_n(\beta, {\boldsymbol{q}}, \psi)$
 is an associative algebra over the ring of polynomials
$\Z[ \beta, q_{ij},\psi_{ij}]$ with the set of generators
$\{ u_{ij},\, 1 \le i , j \le n \}$ subject to the set of relations
\begin{itemize}\itemsep=0pt
\item $u_{ij}+u_{ji}= \beta$, $u_{ij}u_{kl}=u_{kl}u_{ij}$
if~$\{i,j \}
\cap \{k,l \} =\varnothing$,

\item {\it three term relations}:
\begin{gather*}
 u_{ij} u_{jk}+u_{ki} u_{ij}+u_{jk} u_{ki}= \beta (u_{ij}+u_{ik}+u_{jk}) \qquad \text{if} \ \ i,\,j,\,k \ \ \text{are distinct},
\end{gather*}

\item $u_{ij}^{2} =\beta u_{uj}+ q_{ij}+\psi_{ij} $ if $i \not= j$,

\item $u_{ij} \psi_{kl}= \psi_{kl} u_{ij}$ if $\{i,j\} \cap \{k,l \} =
\varnothing$,

\item {\it exchange relations}: $u_{ij} \psi_{jk}= \psi_{ik} u_{ij}$
if $i$, $j$, $k$ are distinct,

\item elements $\beta$, $\{ q_{ij}, \, 1 \le i,j \le n \}$ are
{\it central}.
\end{itemize}
\end{Definition}

It is easy to see that in the algebra $3MT_n(\beta, {\boldsymbol{q}},{\boldsymbol{\psi}})$ the
generators $\{u_{ij} \}$
satisfy the modif\/ied {\it Coxeter} and modif\/ied {\it quantum Yang--Baxter
relations}, namely
\begin{itemize}\itemsep=0pt
\item {\it modified Coxeter relations}: $u_{ij}u_{jk}u_{ij}-u_{jk}u_{ij}u_{jk}= (q_{ij}-q_{jk})u_{ik}$,

\item {\it modified quantum Yang--Baxter relations}:
\begin{gather*}
u_{ij} u_{ik} u_{jk}-u_{jk} u_{ik} u_{ij} = (\psi_{jk}- \psi_{ij}) u_{ik}
\end{gather*}
if $i$, $j$, $k$ are distinct.
\end{itemize}

Clearly the additive Dunkl elements $\big\{\theta_i:=\sum\limits_{j \not=i} u_{ij},\, i=1,
\ldots,n \big\}$ generate a commutative subalgebra in $3MT_n(\beta,\psi)$.

It is still possible to describe relations among the additive Dunkl elements~\cite{K}, cf.~\cite{KM2}. However we don't know any geometric interpretation
of the commutative algebra obtained. It is not unlikely that this commutative
subalgebra is a common generalization of the small quantum cohomology and
elliptic cohomology (remains to be def\/ined!) of complete f\/lag varieties.

The algebra $3MT_n(\beta=0,{\boldsymbol{q}}={\boldsymbol{0}},\psi)$ has an elliptic representation \cite{K,KM2}. Namely,
\begin{gather*}
 u_{ij}:= \sigma_{\lambda_i -\lambda_j}(z_i-z_j) s_{ij}, \qquad q_{ij} =
\wp(\lambda_i -\lambda_j), \qquad \psi_{ij}= - \wp(z_i-z_j),
\end{gather*}
where $\{ \lambda_i, \, i=1,\ldots,n \}$ is a set of parameters (e.g., complex
numbers), and $\{z_1,\ldots,z_n \}$ is a set of variables;
$s_{ij}$, $i < j$, denotes the transposition that swaps~$i$ on~$j$ and f\/ixes
all other variables;
\begin{gather*}
 \sigma_{\lambda}(z):= \frac{\theta(z-\lambda) \theta'(0)}{\theta(z)\theta(\lambda)}
 \end{gather*}
denotes the {\it Kronecker sigma function}; $\wp(z)$ denotes the {\it
Weierstrass $P$-function}.

{\bf ``Multiplicative'' version of the elliptic representation.}
Let $q$ be parameter. In this place we will use the same symbol $\theta(x)$ to denote the ``multiplicative'' version of the Riemann theta function
\begin{gather*}
\theta(x):=\theta(x;q)= (x;q)_{\infty} (q/x;q)_{\infty},
\end{gather*}
where by def\/inition $ (x;q)_{\infty}= (x)_{\infty}=\prod\limits_{k \ge 0}(1-x~q^{k})$.
Let us state some well-known properties of the Riemann theta function:
\begin{itemize}\itemsep=0pt
\item $\theta(qx;q)=\theta(1/x;q)=-x^{-1} \theta(x;q)$,

\item functional equation:
\begin{gather*}
x/y \theta\big(u x^{\pm 1}\big) \theta\big(y v^{\pm 1}\big) + \theta\big(u v^{\pm 1}\big)
\theta\big(x y^{\pm 1}\big)=\theta\big(u y^{\pm 1}\big) \theta\big(x v^{\pm 1}\big),
\end{gather*}
where by def\/inition $\theta(x y^{\pm 1}):= \theta(x y) \theta(x y^{-1})$.

\item Jacobi triple product identity:
\begin{gather*}
 (q;q)_{\infty} \theta(x;q) = \sum_{n \in \Z} (-x)^{n} q^{{n \choose 2}}.
 \end{gather*}
\end{itemize}

One can easily check that after the change of variables
\begin{gather*}
x:= \left(\frac{z^2}{\lambda w}\right)^{1/2}, \qquad y:=\left(\frac{w}{\lambda}\right)^{1/2}, \qquad
u:=\left(\frac{w}{\lambda \mu^2}\right)^{1/2}, \qquad v:=(w \lambda)^{1/2},
\end{gather*}
the functional equation for the Riemann theta function $\theta(x)$ takes the following form
\begin{gather*}
\sigma_{\lambda}(z)~\sigma{\mu}(w) = \sigma_{\lambda \mu}(z)\sigma_{\mu}(w/z)
+ \sigma_{\lambda \mu}(w) \sigma_{\lambda}(z/w),
\end{gather*}
where
\begin{gather*}
\sigma_{\lambda}(z):=\frac{\theta(z/ \lambda)}{\theta(z) \theta\big(\lambda^{-1}\big)}
\end{gather*}
denotes the (multiplicative) {\it Kronecker sigma function}.
 Therefore, the operators
\begin{gather*}
u_{ij}(f):= \sigma_{\lambda_i/\lambda_j}(z_i/z_j) s_{ij}(f),
\end{gather*}
where $s_{ij}$ denotes the exchange operator which swaps the variables $z_i$
and $z_{j} $, namely $s_{ij}(z_i)=z_j$, $s_{ij}(z_j)=z_i$, $s_{ij}(z_k)=z_k$,
$\forall\, k \not= i,j$, and $s_{ij}$ acts trivially on dynamical parameters
$\lambda_i$, namely, $s_{ij}(\lambda_k)=\lambda_k$, $\forall\, k$, {\it give rise} to a representation of the algebra $3MT_n(\beta=0,{\boldsymbol{q}}={\boldsymbol{0}},\psi)$.

The $3$-term relations among the elements $\{u_{ij} \}$ are consequence (in
fact equivalent) to the famous {\it Jacobi--Riemann} $3$-term relation of
degree~$4$ among the theta function $\theta(z)$, see, e.g., \cite[p.~451,
Example~5]{WW}. In several cases, see Introduction, relations~({\bf A}) and~{(\bf B}), identities among the Riemann theta functions can be rewritten in
terms of the elliptic Kronecker sigma functions and turn out to be a~consequence of certain relations in the algebra $3MT_n(\beta=0,{\boldsymbol{q}}={\boldsymbol{0}},\psi)$ for some
integer~$n$, and vice versa\footnote{It is commonly believed that any identity between the Riemann
theta functions is a consequence of the Jacobi--Riemann three term relations
among the former. However we do not expect that the all hypergeometric type
identities among the Riemann theta functions can be obtained from certain
relations in the algebra $3MT_n(\beta=0$, ${\boldsymbol{q}}={\boldsymbol{0}},\psi)$ after applying the {\it elliptic
representation} of the latter.}.

The algebra $3HT_n(\beta)$ is the quotient of algebra $3MT_n(\beta,{\boldsymbol{q}},\psi)$ by the two-sided ideal ge\-ne\-ra\-ted by the elements $\{ \psi_{ij} \}$.
Therefore the
elements $\{u_{ij} \}$ of the algebra $3HT_n(\beta)$ satisfy the quantum Yang--Baxter relations $u_{ij}u_{ik}u_{jk}=u_{jk}u_{ik}u_{ij}$, $i <j < k$,
and as a consequence, the multiplicative Dunkl elements
\begin{gather*}
\Theta_i =\prod_{a=i-1}^{1} (1+h u_{a,i})^{-1} \prod_{a=i+1}^{n} (1+h
u_{i,a}),\qquad i=1,\ldots,n, \qquad u_{0,i}=u_{i,n+1}=0
\end{gather*}
generate a commutative subalgebra in the algebra $3HT_n(\beta)$, see
Section~\ref{section3.1}. We emphasize that the Dunkl elements $\Theta_j$, $j=1,\ldots, n$,
do not pairwise commute in the algebra $3MT_n(\beta,{\boldsymbol{q}}, {\boldsymbol{\psi}})$, if
$\psi_{ij} \not= 0$ for some $i \not= j$. One way to construct a~multiplicative analog of additive Dunkl elements $\theta_i:=\sum\limits_{j \not=i}
u_{ij}$ is to add a new set of mutually commuting generators denoted by
$\{\rho_{ij}, \,\rho_{ij}+\rho_{ji}=0$, $1 \le i \not= j \le n\}$ subject
to the crossing relations
\begin{itemize}\itemsep=0pt
\item $\rho_{ij}$ commutes with $\beta$, $q_{kl}$ and $\psi_{k,l}$ for
all $i$, $j$, $k$, $l$,

\item $\rho_{ij} u_{kl}=u_{kl} \rho_{ij}$ if $\{i,j\} \cap \{k,l \} =
\varnothing$, $\rho_{ij}u_{jk}=u_{jk}\rho_{ik}$ if $i$, $j$, $k$ are distinct,

\item $\rho_{ij}^2 - \beta \rho_{ij} +\psi_{ij}= \rho_{jk}^2-\beta \rho_{jk}+ \psi_{jk}$ for all triples $1 \le i < j < k \le n$.
\end{itemize}

 Under these assumptions one can check that elements
\begin{gather*}
 R_{ij}:= \rho_{ij}+ u_{ij},\qquad 1 \le i < j \le n
\end{gather*}
satisfy the quantum Yang--Baxter relations
\begin{gather*}
R_{ij} R_{ik} R_{jk} =R_{jk} R_{ik} R_{ij},\qquad i < j < k.
\end{gather*}
In the case of {\it elliptic representation} def\/ined above, one can take
\begin{gather*}
 \rho_{ij}:= \sigma_{\mu}(z_i-z_j),
 \end{gather*}
where $\mu \in \C^{*}$ is a parameter. This solution to the quantum Yang--Baxter equation has been discovered in~\cite{SU+}. It can be seen as an
operator form of the famous (f\/inite-dimensional) solution to~${\rm QYBE}$ due to A.~Belavin and V.~Drinfeld~\cite{BD}. One can go to one step more and add
to the algebra in question a new set of generators corresponding to the shift
operators
$T_{i,q}\colon z_i \longrightarrow q z_i$, cf.~\cite{Fe}. In this case one
can def\/ine {\it multiplicative Dunkl elements} which are closely related
with the elliptic Ruijsenaars--Schneider--Macdonald operators.

\subsubsection{Equivariant modif\/ied three term relations algebra}\label{section3.1.1}

Let ${\boldsymbol{h}}= (h_2,\ldots,h_n)$ be a set of parameters. We def\/ine {\it equivariant modified $3$-term relations
algebra $3EMT_n(\beta,h,{\boldsymbol{q}},\psi)$} to be the extension of the algebra
$3TM_n(\beta,{\boldsymbol{q}},\psi)$ by the set of mutually commuting generators
 $\{y_1,\ldots,y_n \}$ subject to the crossing relations
\begin{itemize}\itemsep=0pt
\item $y_i u_{jk}=u_{jk}y_i$ if $ i \not= j,k$, $y_i u_{ij} =u_{ij} y_j +h_{j}$, $y_j u_{ij}= u_{ij} y_i -h_{j}$, $ i < j$,

\item $ [y_k, q_{ij}] = 0 = [y_k, \psi_{ij}]$ for all $i$, $j$, $k$.
\end{itemize}

It is clear that the additive Dunkl elements $ \theta_i = y_i +\sum\limits_{j \not= i} u_{ij}$, $i=1,\ldots,n$, are pair-wise commute.
For simplicity's sake, we
shall restrict our consideration to the case $\beta =0$.
\begin{Theorem}[generalized Pieri's rule, cf.~\cite{K,KM2, P}] \label{theorem3.1}
Let $ 1 \le m \le n$, then
\begin{gather*}
e_k^{({\boldsymbol{h}}, {\boldsymbol{q}})}\big(\theta_1^{(n)},\ldots, \theta_{m}^{(n)}\big):=
\sum_{A \subset [1,m], \, |A|=2r \atop B \subset [1,m] {\setminus} A, |B|=2s}
 H_{r}(A) M_B(\{q_{ij} \}) e_{k-2r-2s}(\Theta_{[1,m] {\setminus} (A \cup B)}) \\
\qquad{}=
\sum_{A \subset [1,m]} Y_{A} \sum_{B \subset [1,m] {\setminus} A \atop |B|=
2s} (-1)^{s} M_{B}(\{\psi_{ij}\}) \sum_{I \subset [1,n] {\setminus} A,\, I \cap B = \varnothing \atop |A|+|B|+|I|=k}
\prod_{{(i_{\alpha},j_{\alpha}) \in I \times I \atop 1 \le i_{\alpha} \le m < j_{\alpha} \le n, \, \forall \alpha} \atop
i_1,\ldots,i_{I} \, \text{are distinct}} u_{i_{\alpha},j_{\alpha}},
\end{gather*}
where for any subset $C \subset [1,n]$ we put $Y_{C}:=\prod\limits_{c \in C} y_{c}$, and $e_{\ell}(\Theta_C)=e_m(\{ \theta_c\}_{c \in C})$ stands for the degree~$\ell$ elementary symmetric polynomial of the elements
$\{\theta_c\}_{c \in C}$,
$e_{k}(\{\theta_c\}_{c \in C}) = \delta_{0,k}$ if $k \le 0$;
 if $B \subset [1,n]$, $|B|=2s$, we set
\begin{gather*}
M_B(\{\psi_{ij}) =\prod_{{ L \subset B, \, |L|=s \atop (i_1,\ldots,i_s) \subset L} \atop (j_1,\ldots,j_s) \subset B \setminus L, \, i_{\alpha} < j_{\alpha}, j_{\alpha} \in B {\setminus} L, \, i_{\alpha} < n\, \forall \, \alpha} \psi_{i_{\alpha},j_{\alpha}};
\end{gather*}
in a similar manner one can define $M_B(\{q_{ij} \})$; finally we set
\begin{gather*}
 H_{r}(A)= h_{a_{2r}} \Biggl( \sum_{(a_1,\ldots,a_{r-1}) \subset A {\setminus}
 \{a_{2r} \}} \prod_{j=1}^{r-1} \max(a_j-2j+1,0)~h_{a_{j}} \Biggr).
\end{gather*}
\end{Theorem}

It is not dif\/f\/icult to {\it show} that
\begin{gather*}
 H_{r}(A) \big |_{h_{a}=1, \, a \in A} = (2 r -1) !! ,
 \end{gather*}
as well as the number of dif\/ferent monomials which appear in $H_r([1,2r])$ is
equal to the Catalan number ${\rm Cat}_{r}$. For example,
\begin{gather*}
H_{3}([1,6])=h_{6}
(h_{24} +2 h_{25}+ 2 h_{34}+4 h_{35}+6 h_{45}),\\
H_4([1,8])=h_{8}\big(h_{246} +2 h_{247}+ 2 h_{256}+4 h_{257}+6 h_{267} +
2 h_{346}+ 4 h_{347}+ 4 h_{356} \\
\hphantom{H_4([1,8])=}{} +8 h_{357} +12 h_{367} + 6 h_{456}+ 12 h_{457} + 18 h_{467} +24 h_{567}\big).
\end{gather*}

\begin{Exercise}\label{exercise3.1}
Write
\begin{gather*}
H_{r}([1,2r])= h_{a_{2r}} \Bigg( \sum_{A:=(a_1,\ldots,a_{r-1}) \subset [1,2r-1] \atop a_j \ge 2 j} c_{A}^{(r)} h_{A} \Bigg),
\end{gather*}
where $c_{A}^{(r)}:= \prod\limits_{j=1}^{r-1} \max(a_{j}-2j+1,0)$ and $h_{A}:= \prod\limits_{a \in A} h_{a}$. {\it Show} that
\begin{gather}
\sum_{A:=(a_1,\ldots,a_{r-1}) \subset [1,2r-1] \atop a_{j} \ge 2 j}
\big( c_{A}^{(r)} \big)^2 = E_{r},\label{****}
\end{gather}
where $E_r$ denotes the $r$-th Euler number, see, e.g., \cite[$A000364$]{SL}.

{\it Find} representation theoretic interpretation of numbers
$\{ c_{A}^{(r)} \}$ and the identity \eqref{****}.
\end{Exercise}

Clearly,
\begin{gather*}
\sum_{A:=(a_1,\ldots,a_{r-1}) \subset [1,2r-1] \atop a_{j} \ge 2 j}
 c_{A}^{(r)} = (2r-1) !! .
\end{gather*}

\begin{Question} \label{question3.1} Does there exist a semisimple algebra ${\mathfrak{A}}(r)$,
$\dim ({\mathfrak{A}}(r))=E_r$ such that the all irreducible representations
 $\pi_{A}^{(r)}$ of the algebra ${\mathfrak{A}}(r)$ are in one-to-one
correspondence with the set $ {\cal{P}}(r):= \{ A=(a_1,\ldots,a_{r-1}) \subset [1,2r-1], \, a_{j} \ge 2 j, \, \forall\, j\}$ and $\dim (\pi_{A}) =
c_{A}^{(r)}$, $\forall \, A \in {\cal{P}}(r)$?
\end{Question}

 It is worth noting that the Dunkl element $\theta_i$, $1 \le i \le n$, doesn't
 commute either with~$y_j$, $j \not=i$ or any~$\psi_{kl}$. On the other hand
one can check easily that $[e_k(y_1,\ldots,y_n), \theta_i]= 0$, $\forall\, k,i$.

\subsection{Multiplicative Dunkl elements}\label{section3.2}

Since the elements $u_{ij}$, $u_{ik}$ and $u_{jk}$, $i < j <k$, satisfy the
classical and {\it quantum} Yang--Baxter relations~\eqref{equation3.1} and~\eqref{equation3.2},
one can def\/ine a multiplicative analogue denoted by~$\Theta_i$, $1 \le i
\le n$, of the Dunkl elements~$\theta_i$. Namely, to start with, we def\/ine
 elements
\begin{gather*}
 h_{ij}:= h_{ij}(t)= 1+t u_{ij}, \qquad i \not= j.
 \end{gather*}
We consider $h_{ij}(t)$ as an element of the algebra $\widetilde {3HT_n} :=
3HT_n(\beta) \otimes \Z[[ q_{ij}^{\pm 1},t,x,y, \ldots]]$, where we assume
that the all parameters $\{ q_{ij},t,x,y, \ldots \}$ are
{\it central} in the algebra $\widetilde {3HT_n}$.

\begin{Lemma} \label{lemma3.2}\quad
\begin{enumerate}\itemsep=0pt
\item[$(1a)$] $h_{ij}(x) h_{ij}(y)=h_{ij}(x+y+\beta xy)+q_{ij} xy$,
\item[$(1b)$] $h_{ij}(x) h_{ji}(y)=h_{ij}(x-y)+\beta y-q_{ij} x y$ if $i < j$.
\end{enumerate}
It follows from $(1b)$ that $h_{ij}(t) h_{ji}(t) =1+\beta t-t^2 q_{ij}$ if~$i <j$,
and therefore the elements $ \{h_{ij} \}$ are invertible in the algebra
$\widetilde{3HT_n}$.
\begin{enumerate}\itemsep=0pt
\item[$(2)$] $h_{ij}(x) h_{jk}(y)=h_{jk}(y) h_{ik}(x)+h_{ik}(y) h_{ij}(x)-
h_{ik}(x+y+\beta xy)$,
\item[$(3)$] multiplicative Yang--Baxter relations:
\begin{gather*}
h_{ij} h_{ik} h_{jk}=h_{jk} h_{ik} h_{ij} \qquad \text{if} \ \ i < j < k,
\end{gather*}
\item[$(4)$] define multiplicative Dunkl elements $($in the algebra $\widetilde{3HT_n})$ as follows
\begin{gather*}
\Theta_j:=\Theta_j(t)= \left( \prod_{a=j-1}^{1} h_{aj}^{-1} \right) \left( \prod_{a=n}^{j+1} h_{ja} \right), \qquad 1 \le j \le n.
\end{gather*}
\end{enumerate}
Then the multiplicative Dunkl elements pair-wise commute.
\end{Lemma}

 Clearly
\begin{gather*}
 \prod_{j=1}^{n} \Theta_j =1, \qquad \Theta_j=1+t \theta_j+t^2( \cdots) \qquad \text{and} \qquad
\Theta_{I}\prod_{i \notin I, \, j \in I \atop i < j}\big(1+t \beta -t^2 q_{ij}\big)
 \in 3HT_n(\beta).
\end{gather*}

Here for a subset $I \subset [1,n]$ we use notation $\Theta_{I}= \prod\limits_{a \in I} \Theta_{a}$. Note, that the element $\Theta_{I}$ is a~product of
(exactly!) $k(n-k)$ terms of a form $h_{i_{\alpha}j_{\alpha}}$, where $k:= |I|$.

Our main result of this section is a~description of relations among the
multiplicative Dunkl elements.

\begin{Theorem}[A.N.~Kirillov and T.~Maeno~\cite{KM}]\label{theorem3.2}
In the algebra $3HT_n(\beta)$ the following relations hold true
\begin{gather*}
 \sum_{I \subset [1,n] \atop |I|=k} \Theta_{I} \prod_{i \notin I, \, j \in J
\atop i < j} \big(1+t \beta-t^2 q_{ij}\big)= {n \brack k}_{1+t \beta}.
\end{gather*}
Here ${n \brack k}_{q}$ denotes the $q$-Gaussian
polynomial.
\end{Theorem}

\begin{Corollary}\label{corollary3.1}
Assume that $q_{ij} \not= 0$ for all $1 \le i < j \le n$. Then the all
elements $\{u_{ij} \}$ are invertible and $u_{ij}^{-1} =q_{ij}^{-1} (u_{ij}-
\beta)$. Now define elements $\Phi_i \in \widetilde{3HT_n}$ as follows
\begin{gather*}
 \Phi_i= \left\{\prod_{a=i-1}^{1} u_{ai}^{-1} \right\} \left\{\prod_{a=n}^{i+1} u_{ia} \right\}, \qquad i=1,\ldots,n.
 \end{gather*}
Then we have
\begin{enumerate}\itemsep=0pt
\item[$(1)$] relationship among $\Theta_j$~and~$\Phi_j$:
\begin{gather*}
 t^{n-2j+1} \Theta_j\big(t^{-1}\big) \arrowvert_{t=0}=(-1)^{j} \Phi_j,
 \end{gather*}

\item[$(2)$] the elements $\{ \Phi_i,\, 1 \le i \le n,\}$ generate a~commutative
subalgebra in the algebra ${\widetilde{3HT_n}}$,

\item[$(3)$] for each $k=1,\ldots,n$, the following relation in the algebra
$3HT_n$ among the elements $\{ \Phi_i \}$ holds
\begin{gather*}
 \sum_{I \subset [1,n] \atop |I|=k} \prod_{i \notin I, \, j \in I \atop i < j}
(-q_{ij}) \Phi_{I} =\beta^{k(n-k)},
\end{gather*}
where $\Phi_{I}:= \prod\limits_{a \in I} \Phi_a$.
\end{enumerate}
\end{Corollary}

In fact the element $\Phi_i$ admits the following ``reduced expression'' (i.e., one with the minimal number of terms involved) which
is useful for proofs and applications
\begin{gather}\label{equation3.4}
 \Phi_i= \Biggl\{ \overrightarrow{\prod_{j \in I}} \Biggl\{ \overrightarrow{\prod_{i \in I_{+}^{c} \atop i < j}}~
u_{ij}^{-1} \Biggr\} \Biggr\} \Biggl\{\overrightarrow{\prod_{j \in I_{+}^{c}}}
\Biggl\{ \overrightarrow{\prod_{i \in I \atop i < j }} u_{ij} \Biggr\} \Biggr\}.
\end{gather}
Let us explain notations. For any (totally) ordered set
$I=(i_1 < i_2 < \cdots < i_k)$ we denote by~$I_{+}$ the set~$I$ with the
opposite order, i.e., $I_{+}=(i_k > i_{k-1} > \cdots > i_1)$;
if $I \subset [1,n]$, then we set $I^{c} := [1,n]{\setminus}I$. For any
(totally) ordered set~$I$ we denote by $\overrightarrow{\prod\limits_{i \in I}}$ the
ordered product according to the order of the set~$I$.

Note that the total number of terms in the r.h.s.\ of \eqref{equation3.4} is equal to
$i(n-i)$.

Finally, from the ``reduced expression''~\eqref{equation3.4} for the element~$\Phi_i$ one
can see that
\begin{gather*}
\prod_{i \notin I, j \in I \atop i < j}(-q_{ij}) \Phi_{I}=
\Biggl\{\overrightarrow{\prod_{j \in I}} \Biggl\{ \overrightarrow{\prod_{i \in I_{+}^{c} \atop i < j}} (\beta-u_{ij}) \Biggr\} \Biggr\}
\Biggl\{\overrightarrow{
\prod_{j \in I_{+}^{c}}} \Biggl\{\overrightarrow{\prod_{i \in I \atop i < j}} u_{ij} \Biggr\} \Biggr\}:= \widetilde{\Phi_I} \in 3HT_n.
\end{gather*}
Therefore the identity
\begin{gather*}
\sum_{I \subset [1,n] \atop |I|=k} \widetilde{\Phi_{I}} = \beta^{k(n-k)}
\end{gather*}
is true in the algebra $3HT_n$ for any set of parameters $\{q_{ij} \}$.

\begin{Comments} \label{comments3.2}
In fact from our proof of Theorem~\ref{theorem3.1} we can deduce more
general statement, namely, consider integers~$m$ and~$k$~such that $1 \le k
\le m \le n$. Then
\begin{gather}\label{equation3.5}
\sum_{I \subset [1,m] \atop |I|=k} \Theta_{I}
\prod_{i \in [1,m] {\setminus} I, \, j \in J \atop i < j} \big(1+t \beta-t^2 q_{ij}\big)=
 {m \brack k} _{1+t \beta} +\sum_{A \subset [1,n], \, B \subset
[1,n] \atop |A|=|B|=r} u_{A,B},
\end{gather}
where, by def\/inition, for two sets $A=(i_1,\ldots,i_r)$ and
$B=(j_1,\ldots,j_r)$ the symbol $u_{A,B}$ is equal to the (ordered) product
$\prod\limits_{a=1}^{r} u_{i_{a},j_{a}}$. Moreover, the elements of the sets~$A$ and~$B$ have to satisfy the following conditions:
\begin{itemize}\itemsep=0pt
\item for each $a=1,\ldots,r$ one has
$1 \le i_a \le m < j_a \le n$, and $ k \le r \le k(n-k)$.
\end{itemize}
Even more, if $r=k$, then sets $A$ and $B$ have to satisfy the following
additional conditions:
\begin{itemize}\itemsep=0pt
\item $B=(j_1 \le j_2 \le \cdots \le j_k)$, and the elements of the set~$A$ are pair-wise distinct.
\end{itemize}

In the case $\beta=0$ and $r=k$, i.e., in the case of additive (truncated) Dunkl
elements, the above statement, also known as the quantum Pieri formula, has
been stated as conjecture in~\cite{FK}, and has been proved later in~\cite{P}.
\begin{Corollary}[\cite{KM}]\label{corollary3.2}
In the case when $\beta=0$ and $q_{ij}=q_i \delta_{j-i,1}$, the algebra over
$\Z[q_1,\ldots,q_{n-1}]$ generated by the multiplicative Dunkl elements
$\{ \Theta_i \, \text{and}\, \Theta_i^{-1}, \, 1 \le i \le n \}$ is canonically isomorphic
to the quantum $K$-theory of the complete flag variety ${\cal F}l_n$ of type~$A_{n-1}$.
\end{Corollary}

It is still an open {\it problem} to describe explicitly the set of monomials
$\{ u_{A,B} \}$ which appear in the r.h.s.\ of~\eqref{equation3.5} when $r > k$.
\end{Comments}

\subsection{Truncated Gaudin operators}\label{section3.3}

Let $\{ p_{ij},\, 1 \le i \not= j \le n
\}$ be a set of mutually commuting parameters. We assume that parameters
$\{p_{ij}\}_{1\le i < j \le n}$ are invertible and satisfy the Arnold
relations
\begin{gather*}
{1 \over p_{ik}}=
{1 \over p_{ij}}+{1 \over p_{jk}}, \qquad i < j , k.
\end{gather*}
For example one can take $p_{ij}=(z_i-z_j)^{-1}$, where $z=(z_1,\ldots,z_n) \in
 (\C {\setminus} 0)^{n}$.

\begin{Definition} \label{definition3.5} Truncated (rational) Gaudin operator corresponding to the set
of parame\-ters~$\{p_{ij}\}$ is def\/ined to be
\begin{gather*} G_i= \sum_{j \not= i} p_{ij}^{-1} s_{ij}, \qquad 1 \le i \le n,
\end{gather*}
where~$s_{ij}$ denotes the exchange operator which switches variables~$x_i$
and~$x_j$, and f\/ixes parame\-ters~$\{ p_{ij} \}$.

We consider the Gaudin operator $G_i$ as an element of the group ring
$\Z[\{p_{ij}^{\pm 1} \}][{\mathbb S}_n]$, call this element
$G_i \in \Z[\{p_{ij}^{\pm 1} \}][{\mathbb S}_n]$, $i=1,\ldots,n$, by {\it
Gaudin element} and denoted it by~$\theta_i^{(n)}$.
\end{Definition}

It is easy to see that the elements $u_{ij}:=p_{ij}^{-1} s_{ij}$, $1 \le i \not=
 j \le n$, def\/ine a representation of the algebra $3HT_n(\beta)$ with
parameters $\beta=0$ and $q_{ij}=u_{ij}^{2}=p_{ij}^2$.

Therefore one can consider the (truncated) Gaudin elements as a special case
of the (truncated) Dunkl elements. Now one can rewrite the relations among
the Dunkl elements, as well as the quantum Pieri formula \cite{FK,P}, in terms of the Gaudin elements.

The key observation which allows to rewrite
the quantum Pieri formula as a certain relation among the Gaudin elements, is
the following one:
parameters $\big\{p_{ij}^{-1} \big\}$ satisfy the
Pl\"ucker relations
\begin{gather*}
 {1 \over p_{ik} p_{jl}}={1 \over p_{ij} p_{kl}}+{1 \over p_{il} p_{jk}} \qquad
\text{if} \ \ i < j < k < l.
\end{gather*}

To describe relations among the Gaudin elements $\theta_i^{(n)}$, $i=1,\ldots,n$,
we need a bit of notation. Let $\{p_{ij}\}$ be a set of invertible
parameters as before,
$i_a < j_a$, $a=1,\ldots,r$.
Def\/ine polynomials in the variables ${\boldsymbol{h}}=(h_1,\ldots,h_n)$
\begin{gather}\label{equation3.6}
G_{m,k,r}^{(n)}({\boldsymbol{h}},\{ p_{ij} \})=\sum_{I \subset [1,n-1] \atop |I|=r}
{1 \over \prod\limits_{i \in I} p_{in} } \sum_{J \subset [1,n] \atop |I|+m=|J|+k} {n- | I \cup J | \choose n-m -|I|}
{\tilde h}_{J},
\end{gather}
where
\begin{gather*}
{\tilde h}_{J}=\sum_{K \subset J,\, L \subset J \atop |K|= |L|,
\, K \bigcap L = \varnothing} \prod_{j \in J {\setminus} (K \cup L)} h_{j}
 \prod_{k_a \in K, \, l_a \in L} p_{k_a,l_a}^2,
 \end{gather*}
and summation runs over subsets $K=\{k_1 < k_2 < \cdots < k_r \}$
 and $L=\{l_1 < l_2 < \cdots < l_r \} \subset J \}$, such that
$k_a < l_a$, $a=1, \ldots,r$.

\begin{Theorem}[relations among the {\it Gaudin} elements~\cite{K},
cf.~\cite{MTV}] \label{theorem3.3} \quad
\begin{enumerate}\itemsep=0pt
\item[$(1)$] Under the assumption that elements $\{p_{ij},\, 1 \le i < j \le n \}$
are invertible, mutually commute and satisfy the Arnold relations, one has
\begin{gather}
 G_{m,k,r}^{(n)}\big(\theta_1^{(n)},\ldots,\theta_n^{(n)}, \{ p_{ij} \}\big)=0 \qquad \text{if} \ \ m >k,\nonumber
\\
 G_{0,0,r}^{(n)}\big(\theta_1^{(n)},\ldots,\theta_n^{(n)},\{p_{ij}\}\big)=
e_{r}(d_2, \ldots,d_n),\label{equation3.7}
\end{gather}
 where $d_2,\ldots,d_n$ denote the Jucys--Murphy elements in the group ring
$\Z[\mathbb{S}_n]$~of the symmetric group $\mathbb{S}_n$, see Comments~{\rm \ref{comments3.1}} for
a definition of the Jucys--Murphy elements.

\item[$(2)$] Let $J= \{j_1 < j_2 < \cdots < j_r \} \subset [1,n]$, define matrix
 $M_{J}:= (m_{a,b})_{1 \le a,b \le r}$, where
\begin{gather*}
m_{a,b}:= m_{a,b}({\boldsymbol{h}};\{p_{ij}\} ) =\begin{cases}
 h_{j_{a}} & \text{if \ $ a=b$},\\
p_{{j_a},{j_b}} & \text{if \ $a < b$}, \\
- p_{{j_b},{j_a}} & \text{if \ $a > b$}.
\end{cases}
\end{gather*}
Then
\begin{gather*}
 {\tilde{h}}_{J} =\operatorname{DET} \vert M_{J} \vert.
 \end{gather*}
\end{enumerate}
\end{Theorem}

\begin{Examples} \label{examples3.1} \quad

(1) Let us display the polynomials $G_{m,k,r}^{(n)}({\boldsymbol{h}},\{ p_{ij} \})$
a few cases
\begin{gather*}
 G_{m,0,r}^{(n)}({\boldsymbol{h}}, \{p_{ij} \})=\sum_{I \subset [1,n-1]
\atop |I| =r} \prod_{i \in I}p_{in}^{-1} \Biggl( \sum_{J \subset [1,n]
\atop {|J|=m+r,\, I \subset J}} {\tilde{h}}_{J} \Biggr),\\
 G_{m,k,0}^{(n)}({\boldsymbol{h}}, \{p_{ij} \})= {n-m+k \choose k}
e_{m-k}^{{\boldsymbol{q}}}(h_1,\ldots,h_n),\\
 G_{m,1,r}^{(n)}({\boldsymbol{h}}, \{p_{ij} \})=\sum_{I \subset [1,n-1] \atop
 |I|=r} \prod_{ i \in I} p_{in}^{-1} \Biggl( \sum_{ J \subset [1,n] \atop {I
 \subset J,\, |J|=m+r}}\!\! (n-m-r+1) {\tilde{h}}_{J}+
 \!\!\sum_{{J \subset [1,n] \atop {|J|=m+r-1,\, |I \cup J| =m+r}}}\!\! {\tilde{h}}_{J}
\Biggr).
\end{gather*}

(2)~Let us list the relations~\eqref{equation3.7} among the Gaudin elements in
the case $n=3$. First of all, the Gaudin elements satisfy the ``standard''
relations among the Dunkl elements
$\theta_1 + \theta_2+ \theta_3 =0$, $\theta_1\theta_2+\theta_1\theta_3+
\theta_2\theta_3+q_{12}+q_{13}+q_{23}=0$,
$\theta_1\theta_2\theta_3+
q_{12} \theta_3+q_{13} \theta_2+q_{23} \theta_1=0$. Moreover, we have
additional relations which are specif\/ic for the Gaudin elements
\begin{gather*}
 G_{2,0,1}^{(3)}= {1 \over p_{13}} (\theta_1 \theta_2 +\theta_1 \theta_3+q_{12}+
q_{13} )+{1 \over p_{23}} (\theta_1 \theta_2+ \theta_2 \theta_3 +q_{12}+
q_{23} ) =0,
\end{gather*}
the elements $p_{23} \theta_1+p_{13} \theta_2$ and $\theta_1 \theta_2$
are central.
\end{Examples}

It is well-known that the elementary symmetric polynomials
$e_r(d_2,\ldots,d_n):=C_r$, $r=1,\ldots,n-1$, generate the center of the
group ring $\Z[p_{ij}^{\pm 1}][\mathbb{S}_n]$, whereas the Gaudin elements
$\{ \theta_i^{(n)},\, i=1,\ldots,n \}$, generate a maximal commutative
subalgebra ${\cal B}(p_{ij})$, the so-called {\it Bethe subalgebra},
in $\Z[p_{ij}^{\pm 1}][{\mathbb S}_n]$. It is well-known, see, e.g.,~\cite{MTV}, that ${\cal B}(p_{ij})=
\bigoplus_{\lambda \vdash n} {\cal B}_{\lambda}(p_{ij})$, where
${\cal B}_{\lambda}(p_{ij})$ is
the $\lambda$-isotypic component of ${\cal B}(p_{ij})$. On each $\lambda$-isotypic component the value of the central element~$C_k$ is the explicitly
 known constant~$c_k(\lambda)$.

It follows from~\cite{MTV} that the relations~\eqref{equation3.7} together with relations
\begin{gather*}
G_{0,0,r}\big(\theta_1^{(n)},\ldots,\theta_n^{(n)},\{p_{ij}\}\big)= c_r(\lambda)
\end{gather*}
 are the def\/ining relations for the algebra ${\cal B}_{\lambda}(p_{ij})$.

Let us remark that in the def\/inition of the Gaudin elements we can use
{\it any} set of mutually commuting, invertible elements
$\{p_{ij} \}$ which
satisf\/ies the Arnold conditions. For example, we can take
\begin{gather*}
p_{ij}:= {q^{j-2}(1-q) \over 1-q^{j-i}}, \qquad 1 \le i < j \le n.
\end{gather*}
It is not dif\/f\/icult to see that in this case
\begin{gather*}
\lim_{q \rightarrow 0}{\theta_J^{(n)} \over p_{1j}}=-d_j=- \sum_{a=1}^{j-1} s_{aj},
\end{gather*}
where as before, $d_j$ denotes the Jucys--Murphy element in the group ring
$\Z[{\mathbb S}_n]$ of the symmetric group ${\mathbb S}_n$. Basically from
relations~\eqref{equation3.7} one can deduce the relations among the Jucys--Murphy
elements
$d_2,\ldots,d_n$ after plugging in~\eqref{equation3.6} the values
$p_{ij}:= {q^{j-2}(1-q) \over 1-q^{j-i}}$ and passing to the limit
$q \rightarrow 0$. However the real computations are rather involved.

Finally we note that the {\it multiplicative} Dunkl/Gaudin elements
$\{ \Theta_i,\, 1,\ldots,n \}$ also generate a maximal commutative subalgebra in
the group ring $\Z[p_{ij}^{\pm 1}][{\mathbb S}_n]$. Some relations among
the elements $\{\Theta_l\}$ follow from Theorem~\ref{theorem3.2}, but we don't know
an analogue of relations~\eqref{equation3.7} for the multiplicative Gaudin
elements, but see~\cite{MTV}.

 \begin{Exercises}\label{exercises3.2}
Let $A = (a_{i,j})$ be a $2m \times 2m$ {\it skew-symmetric} matrix.
The {\it Pfaffian} and {\it Hafnian} of $A$ are def\/ined correspondingly by
the equations
\begin{gather*} 
 \mathrm{Pf}(A) = \frac{1}{2^m m!}\sum_{\sigma\in S_{2m}}\mathrm{sgn}(\sigma)\prod_{i=1}^{m}a_{\sigma(2i-1),\sigma(2i)},\\
 \mathrm{Hf}(A)= \frac{1}{2^m m!}
\sum_{\sigma\in S_{2m}} \prod_{i=1}^{m}a_{\sigma(2i-1),\sigma(2i)},
\end{gather*}
where ${\mathbb S}_{2m}$ is the symmetric group and $\operatorname{sgn}(\sigma)$ is the
signature of a permutation $\sigma \in {\mathbb S}_{2m}$.\footnote{See, e.g.,
\url{https://en.wikipedia.org/wiki/Pfaffian}.}

Now let $n$ be a positive integer, and~$\{p_{ij}, \, 1 \le i \not= j \le n, \, p_{ij}+p_{ji}=0 \}$ be a set of
skew-symmetric, invertible and mutually commuting elements. We set $p_{ii}=0$
for all~$i$, and ${\boldsymbol{q}}:= \big\{ p_{ij}^2 \big\}_{1 \le i < j \le n}$.

Now let us {\it assume} that the elements $\{p_{ij} \}_{1 \le i < j \le n}$ satisfy the Pl\"{u}ker relations for the elements
$\big\{ p_{ij}^{-1} \big\}_{1 \le i < j \le n}$, namely,
\begin{gather*}
{1 \over p_{ik} p_{jl}}={1 \over p_{ij} p_{kl}}+{1 \over p_{il} p_{jk}} \qquad \text{for all} \quad 1 \le i <j < k < l \le n.
\end{gather*}
\begin{enumerate}\itemsep=0pt
\item[(a)] Let $n$ be an {\it even} positive integer. Let us def\/ine
$A_{n}(p_{ij}):=(p_{ij})_{1 \le i,j \le n}$ to be the $n \times n$
skew-symmetric matrix corresponding to the family $\{p_{ij} \}_{1 \le i < j
\le n}$.
Show that
\begin{gather*}
\operatorname{DET} | A_n(p_{ij})| = \operatorname{Hf}\big(A_n\big(p_{ij}^2\big)\big).
\end{gather*}

\item[(b)] Let $n$ be a positive integer, and $z_1,\ldots,z_n$ be a set of
mutually commuting variables, def\/ine polynomials $H_{i}(z_1,\ldots,z_n \,|\, \{p_{ij} \})$,
$i=1,\ldots,n$ from the equation
\begin{gather*}
\operatorname{DET} |\operatorname{diag}(t+z_1, \ldots,t+z_n)+A_n(p_{ij})| =t^n+\sum_{i=1}^{n}
 H_{i}(z_1,\ldots,z_n \,|\, \{p_{ij} \}) t^{n-i},
 \end{gather*}
where $ \operatorname{diag}(t+z_1, \ldots,t+z_n)$ means the diagonal matrix.
\end{enumerate}

Show that
for $k=1,\ldots,n$ the polynomial $H_k(z_1,\ldots,z_n \,|\, \{p_{ij} \})$ is
equal to the
multiparameter quantum elementary polynomial
$e_{k}^{({\boldsymbol{q}})}(z_1, \ldots,z_n)$, see, e.g.,~\cite{FK}, or Theorem~\ref{theorem2.1}.

 For example, take $n=4$, then
 \begin{gather*}
 \operatorname{DET}| A(p_{ij})| = (p_{12} p_{34}-
p_{13} p_{24} +p_{14} p_{23})^2 = p_{12}^2 p_{34}^2 +
p_{13}^2 p_{24}^2 +p_{14}^2 p_{23}^2 \\
\hphantom{\operatorname{DET}| A(p_{ij})| =}{}
-2 p_{12} p_{13} p_{23} p_{14} p_{24} p_{34} \left( \frac{1}{p_{12} p_{34}} - \frac{1}{p_{13} p_{24}} +
\frac{1}{p_{14} p_{23}} \right) \\
\hphantom{\operatorname{DET}| A(p_{ij})|}{}
 =
p_{12}^2 p_{34}^2 +p_{13}^2 p_{24}^2 +p_{14}^2 p_{23}^2 = \operatorname{Hf}\big(A_4\big(\big\{p_{ij}^2 \big\}\big)\big).
\end{gather*}
The last equality follows from the Pl\"{u}cker relations for parameters~$\{ p_{ij}^{-1}\}$.

 On the other hand, if one assumes that a set of skew symmetric parameters
 $\{r_{ij} \}_{1 \le i < j \le n}$,
$r_{ij}+r_{ji}=0$, satisf\/ies the
``standard'' Pl\"{u}cker relations, namely
\begin{gather*}
r_{ik} r_{jl}=r_{ij} r_{kl}+r_{il} r_{jk}, \qquad i < j < k < l,
\end{gather*}
 then $\operatorname{DET}|A_n(r_{ij})| = 0 $.
\end{Exercises}

\subsection[Shifted Dunkl elements $\mathfrak{d}_i$ and $\mathfrak{D}_i$]{Shifted Dunkl elements $\boldsymbol{\mathfrak{d}_i}$ and $\boldsymbol{\mathfrak{D}_i}$}\label{section3.4}

As it was stated in Corollary~\ref{corollary3.2}, the {\it truncated} additive and
multiplicative Dunkl elements in the algebra $3HT_n(0)$ generate over the ring
of polynomials $\Z[q_1,\ldots,q_{n-1}]$ correspondingly the
{\it quantum cohomology} and {\it quantum $K$-theory}
rings of the full f\/lag variety ${\cal{F}}l_n$. In order to describe the
corresponding {\it equivariant} theories, we will introduce the {\it shifted} additive and multiplicative Dunkl elements. To start with we need at f\/irst to
introduce an extension of the algebra~$3HT_n(\beta)$.

Let $\{z_1,\ldots,z_n \}$ be a set of mutually commuting elements and
$\{\beta,{\boldsymbol{h}}=(h_2,\ldots,h_{n}), t, q_{ij}=q_{ji}$, $1 \le i,j \le n \}$
be a set of parameters. We set $h_{n}:=0$.

\begin{Definition}[cf.\ Def\/inition~\ref{definition3.1}]\label{definition3.20}
Def\/ine algebra
$\overline{3TH_n(\beta,{\boldsymbol{h}})}$ to be the semi-direct
product of the algebra $3TH_n(\beta)$ and the ring of polynomials
$\Z[{\boldsymbol{h}},t][z_1,\ldots,z_n]$ with respect to the crossing relations
\begin{enumerate}\itemsep=0pt
\item[(1)] $z_i u_{kl}=u_{kl} z_i$ if $i \notin \{k,l\}$,
\item[(2)] $z_i u_{ij}=u_{ij} z_j+\beta z_i+h_{j}$, $z_j u_{ij}=u_{ij} z_i-\beta z_i-h_{j}$ if $1 \le i < j < k \le n$.
\end{enumerate}
\end{Definition}

Now we set as before $h_{ij}:=h_{ij}(t)=1+t u_{ij}$.
\begin{Definition}\label{definition3.21}\quad
\begin{itemize}\itemsep=0pt
\item Def\/ine shifted additive Dunkl elements to be
\begin{gather*}
\mathfrak{d}_i= z_i + \sum_{ i < j} u_{ij} - \sum_{i > j} u_{ji}.
\end{gather*}

\item Def\/ine shifted multiplicative Dunkl elements to be
\begin{gather*}
\mathfrak{D}_i=\left( \prod_{a=i-1}^{1} h_{ai}^{-1} \right) (1+z_i)\left(
\prod_{a=n}^{i+1} h_{ia} \right).
\end{gather*}
\end{itemize}
\end{Definition}
\begin{Lemma}\label{lemma3.3}
\begin{gather*}
[\mathfrak{d}_i, \mathfrak{d}_j]=0, \qquad [\mathfrak{D}_i, \mathfrak{D}_j]=0 \qquad \text{for all} \quad i,\, j.
\end{gather*}
\end{Lemma}

Now we stated an analogue of Theorem~\ref{theorem3.1} for shifted multiplicative Dunkl
elements.

As a preliminary step, for any subset $I \subset [1,n]$ let us set
$\mathfrak{D}_{I} = \prod\limits_{a \in I} \mathfrak{D}_a$. It is clear that
\begin{gather*}
\mathfrak{D}_{I} \prod_{i \notin I,\, j \in I \atop i < j}
\big(1+t \beta-t^2 q_{ij}\big) \in \overline{3HT_n(\beta,{\boldsymbol{h}})}.
\end{gather*}

\begin{Theorem} \label{theorem3.4}
In the algebra $\overline{3HT_n(\beta,{\boldsymbol{h}})}$ the following relations hold
true
\begin{gather*}
 \sum_{I \subset [1,n] \atop |I|=k} \mathfrak{D}_{I} \prod_{i \notin I, j \in J \atop i < j} \big(1+t \beta-t^2 q_{ij}\big)\\
 \qquad {}=
 \sum_{I \subset [1,n] \atop I=\{1 \le i_1 < \ldots < i_k \le n \}} \prod_{a= 1}^{k} (1+t \beta)^{n-k-i_{a}+a} \left(z_{i_{a}} (1+ t \beta)^{i_{a}-a} +1 + h_{i_{a}} \frac{(1+t \beta)^{i_{a}- a}-1}{\beta} \right).
\end{gather*}
\end{Theorem}

In particular, if $\beta=0$, we will have
\begin{Corollary}\label{corollary3.3} In the algebra $\overline{3HT_n(0,{\boldsymbol{h}})}$ the following
relations hold
\begin{gather}\label{equation3.9}
\sum_{I \subset [1,n] \atop |I|=k} \mathfrak{D}_{I} \prod_{i \notin I, j \in J \atop i < j} \big(1-t^2 q_{ij}\big)= \sum_{I \subset [1,n] \atop I=\{ 1 \le i_1,\ldots,i_k \le n \}} \prod_{a=1}^{k} \bigl(z_{i_{a}}+1 +t h_{i_{a}}
(i_{a}-a) \bigr).
\end{gather}
\end{Corollary}

\begin{Conjecture} \label{conjecture3.1} If $h_1= \cdots =h_{n-1}=1$, $t=1$ and $q_{ij}=\delta_{i,j+1}$,
then the subalgebra generated by multiplicative Dunkl elements
$\mathfrak{D}_i$, $i=1,\ldots,n$, in the algebra
$\overline{3HT_n(0,{\boldsymbol{h}} ={\boldsymbol{1}})}$ $($and $t=1)$, is isomorphic to the equivariant
quantum $K$-theory of the complete flag variety ${\cal{F}}l_n$.
\end{Conjecture}

Our proof is based on induction on $k$ and the following relations in
the algebra $\overline{3HT_n(\beta,{\boldsymbol{h}})}$
\begin{gather*}
h_{ji}\cdot ( 1+x_j) =h_{j}+(1 +\beta) x_j-x_i + (1+x_i) h_{ji}, \qquad
h_{ji} h_{jk} =h_{jk} h_{ki}+h_{ik} h_{ji} -1- \beta,
\end{gather*}
if $i < j < k$, and we set $h_{ij}:= h_{ij}(1)$. These relations allow to
 reduce the left hand side of the relations listed in Theorem~\ref{theorem3.4} to the case
when $z_i=0$, $h_i =0$, $\forall\, i$. Under these assumptions one needs to proof
the following relations in the algebra $3HT_n(\beta)$, see Theorem~\ref{theorem3.2},
\begin{gather*}
\sum_{I \subset [1,n] \atop |I|=k} \mathfrak{D}_{I} \prod_{i \notin I, \, j \in J \atop i < j} \big(1+t \beta-t^2 q_{ij}\big)= {n \brack k}_{1+t \beta}.
\end{gather*}
In the case $\beta =0$ the identity~\eqref{equation3.9} has been proved in~\cite{KM}.

One of the main steps in our proof of Theorem~\ref{theorem3.1} is the following explicit
formula for the elements~$\mathfrak{D}_{I}$.
\begin{Lemma}\label{lemma3.4}
 One has
\begin{gather*}
 \widetilde{\mathfrak{D}_I}:=\mathfrak{D}_{I} \prod_{i \notin I,\, j \in I
\atop i < j}\big(1+t \beta-t^2 q_{ij}\big) = \prod_{b \in I}^{\nearrow}
\Bigg(\prod_{a \notin I \atop a < b}^{\searrow} h_{ba} \Bigg)
 \prod_{a \in I}^{\nearrow} \Bigg((1+z_a) \prod_{b \notin I \atop a < b}^{\searrow} h_{ab} \Bigg).
\end{gather*}
\end{Lemma}

Note that if $a < b$, then $h_{ba}= 1+\beta t - u_{ab}$. Here we have used
the symbol
\begin{gather*}
\prod_{b \in I}^{\nearrow}
\Bigg(\prod_{a \notin I \atop a < b}^{\searrow} h_{ba} \Bigg)
\end{gather*}
to denote the following product. At f\/irst, for a given element $b \in I$ let
us def\/ine the set $I(b):=\{ a \in [1,n] \backslash I, \, a < b\}:=
(a_1^{(b)} < \cdots < a_{p}^{(b)})$ for some $p$ (depending on~$b$). If $I=
(b_1 < b_2 < \cdots < b_k)$, i.e., $b_i=a_i^{(b)}$, then we set
\begin{gather*}
\prod_{b \in I}^{\nearrow}
\Bigg(\prod_{a \notin I \atop a < b}^{\searrow} h_{ba} \Bigg)= \prod_{j=1}^{k}
 \Bigg( u_{b_{j},a_{s}} u_{b_{j},a_{s-1}} \cdots u_{b_{j},a_{1}} \Bigg).
 \end{gather*}
For example, let us take $n=6$ and $I=(1,3,5)$, then
\begin{gather*}
\widetilde{\mathfrak{D}_I}=h_{32}h_{54}h_{52}(1+z_1)h_{16}h_{14}h_{12} (1+
z_3)h_{36}h_{34}(1+z_5)h_{56}.
\end{gather*}

Let us {\it stress} that the element $\widetilde{\mathfrak{D}_I} \in
 \overline{3HT_n(\beta)}$ is a linear combination of {\it square free}
monomials and therefore, a computation of the left hand side of the equality
stated in Theorem~\ref{theorem3.3} can be performed in the ``classical case'' that is in
the case $q_{ij}=0$, $\forall\, i < j$. This case corresponds to the computation
 of the classical equivariant cohomology of the type~$A_{n-1}$ complete f\/lag
variety~${\cal{F}}l_n$ if $h=1$.

A proof of the $\beta=0$ case given in \cite[Theorem~1]{KM}, can be
immediately extended to the case $\beta \not= 0$.
\begin{Exercises}\label{exercises3.3}\quad
\begin{enumerate}\itemsep=0pt
\item[(1)] Show that
\begin{gather*}
\sum_{1 \le i_1 < \cdots < i_k \le n } \prod_{a= 1}^{k} (1+\beta)^{n-k-i_{a}+a} = {n \brack k}_{1+t \beta}.
\end{gather*}
\item[(2)] $(\beta,h)$-Stirling polynomials of the second type.
Def\/ine polynomials $S_{n,k}(\beta,h)$ as follows
\begin{gather*}
S_{n,k}(\beta,h) = \sum_{I \subset [1,n] \atop I=\{ 1 \le i_1,\ldots,i_k \le n \}} \prod_{a=1}^{k} \left(\beta^{n-k-i_{a}+a} + h\frac{\beta^{n-k-i_{a}+a}-1}{\beta-1} \right).
\end{gather*}
Show that
\begin{gather*}
 S_{n,k}(1,1) = {n+1 \brace k+1},\qquad S_{n,k}(\beta,0)= {n \brack k}_{\beta}.
 \end{gather*}
 \end{enumerate}
\end{Exercises}

\section[Algebra $3T_n^{(0)}(\Gamma)$ and Tutte polynomial of graphs]{Algebra $\boldsymbol{3T_n^{(0)}(\Gamma)}$ and Tutte polynomial of graphs}\label{section4}

\subsection{Graph and nil-graph subalgebras, and partial f\/lag varieties}\label{section4.1}

Let's consider the set $R_n:= \{ (i,j) \in \Z \times \Z \,|\, 1 \le i < j
\le n \}$ as the set of edges of the complete graph~$K_n$ on~$n$ labeled
vertices $v_1,\ldots,v_n$. Any subset $S \subset R_n$ is the set of edges
of a unique subgraph $\Gamma:=\Gamma_{S}$ of the complete graph~$K_n$.

\begin{Definition}[graph and nil-graph subalgebras]\label{definition4.1}
The graph subalgebra $3T_n(\Gamma)$
(resp.\ nil-graph subalgebra $3T_n^{(0)}(\Gamma)$) corresponding to a subgraph
$\Gamma \subset K_n$ of the complete graph $K_n$,
is def\/ined to be the subalgebra in the algebra $3T_n$ (resp.~$3T_n^{(0)}$)
generated by the elements $\{u_{ij}\,|\, (i,j) \in \Gamma \}$.
\end{Definition}

In subsequent Sections~\ref{section4.1.1} and~\ref{section4.1.2} we will study some examples of
graph subalgebras corresponding to the complete multipartite graphs, cycle
graphs and linear graphs.

\subsubsection[Nil-Coxeter and af\/f\/ine nil-Coxeter subalgebras in $3T_n^{(0)}$]{Nil-Coxeter and af\/f\/ine nil-Coxeter subalgebras in $\boldsymbol{3T_n^{(0)}}$}\label{section4.1.1}

Our f\/irst example is concerned with the case when the graph $\Gamma$
corresponds to either the set $S:= \{(i,i+1) \,|\, i=1,\ldots,n-1 \}$ of
{\it simple roots} of type $A_{n-1}$, or the set $S^{\rm af\/f}:= S \cup
\{(1,n) \}$ of {\it affine simple roots} of type $A_{n-1}^{(1)}$.

{\samepage \begin{Definition} \label{definition4.2}\quad
\begin{enumerate}\itemsep=0pt

\item[(a)] Denote by ${\widetilde {\rm NC}_{n}}$ subalgebra in the
algebra $3T_{n}^{(0)}$ generated by the elements~$u_{i,i+1}$, $1 \le i \le n-1$.

\item[(b)] Denote by ${\widetilde {\rm ANC}_{n}}$ subalgebra in the algebra
$3T_{n}^{(0)}$ generated by the elements $u_{i,i+1}$, $1 \le i \le n-1$ and~$- u_{1,n}$.
\end{enumerate}
\end{Definition}}

\begin{Theorem} \label{theorem4.1}\quad
\begin{enumerate}\itemsep=0pt
\item[{\rm (A)}] The subalgebra ${\widetilde {\rm NC}_{n}}$ is
canonically isomorphic to the nil-Coxeter algebra ${\rm NC}_{n}$. In particular,
${\rm Hilb}({\widetilde {\rm NC}_{n}},t)=[n]_{t}{!}$ $($cf.~{\rm \cite{Ba})}.

\item[{\rm (B)}] The subalgebra ${\widetilde {\rm ANC}_{n}}$ has finite dimension and
its Hilbert polynomial is equal to
\begin{gather*}
{\rm Hilb}({\widetilde {\rm ANC}_{n}},t)= [n]_{t} \prod_{ 1 \le j \le n-1}[j(n-j)]_{t}=
[n]_{t}{!}\prod_{1 \le j \le n-1} [j]_{t^{n-j}}.
\end{gather*}
In particular, $\dim {\widetilde {\rm ANC}_{n}}=(n-1)! n !$,
$\deg_{t} {\rm Hilb}({\widetilde {\rm ANC}_{n}},t)= {{n+1}\choose 3}$.

\item[{\rm (C)}] The kernel of the map $ \pi\colon {\widetilde {\rm ANC}_{n}}
 \longrightarrow {\widetilde {\rm NC}_{n}}$, $ \pi(u_{1,n}) = 0$,
 $\pi(u_{i,i+1}) = u_{i,i+1}$, $1 \le i \le n-1$, is generated by the following elements:
\begin{gather*}
 f_{n}^{(k)}= \prod_{j=k}^{1} \prod_{a=j}^{n-k+j-1} u_{a,a+1}, \qquad 1 \le k \le n-1.
 \end{gather*}
\end{enumerate}
\end{Theorem}

Note that $\deg~f_{n}^{(k)}=k(n-k)$.

The statement (C) of Theorem~\ref{theorem4.1} means that the element $f_{n}^{(k)}$ which
does not contain the generator $u_{1,n}$, can be written as a linear
combination of degree $k(n-k)$ monomials in the algebra ${\widetilde
{\rm ANC}_{n}}$, each contains the generator $u_{1,n}$ at least once. By this
means we obtain a~set of all extra relations (i.e., additional to those
in the algebra $\widetilde {\rm NC}_{n}$) in the algebra
${\widetilde {\rm ANC}_{n}}$. Moreover, each monomial~$M$ in all linear
combinations mentioned above,
appears with coef\/f\/icient $(-1)^{\#|u_{1,n} \in M|+1}$. For example,
\begin{gather*}
 f_4^{(1)}:=u_{1,2}u_{2,3}u_{3,4}=u_{2,3}u_{3,4}u_{1,4}+u_{3,4}u_{1,4}u_{1,2}+
u_{1,4}u_{1,2}u_{2,3},\\
f_4^{(2)}:=u_{2,3}u_{3,4}u_{1,2}u_{2,3}=
u_{1,2}u_{3,4}u_{2,3}u_{1,4}+u_{1,2}u_{2,3}u_{1,4}u_{1,2}+
u_{2,3}u_{1,4}u_{1,2}u_{3,4}\\
\hphantom{f_4^{(2)}:=}{} +u_{3,4}u_{2,3}u_{1,4}u_{3,4}-
u_{1,4}u_{1,2}u_{3,4}u_{1,4}.
\end{gather*}
Worthy of mention is that $\dim(\widetilde{\rm ANC}_{n})= (n-1)! n!$ is equal
to the number of (directed) Hamiltonian cycles in the complete bipartite graph
$K_{n,n}$, see, e.g., \cite[$A010790$]{SL} for additional information.
\begin{Remark}\label{remark4.1} More generally, let $(W,S)$ be a f\/inite crystallographic
Coxeter group of rank~$l$ with the set of exponents
$1=m_1 \le m_2 \le \cdots \le m_l$.

Let ${\cal B}_{W}$ be the corresponding Nichols--Woronowicz algebra, see, e.g.,~\cite{Ba}. Follow~\cite{Ba}, denote by $\widetilde {\rm NC}_{W}$ the subalgebra in
${\cal B}_{W}$ generated by the elements $[\alpha_s] \in {\cal B}_{W}$
 corresponding to simple roots $s \in S$. Denote by $\widetilde {\rm ANWC}_{W}$
 the subalgebra in ${\cal B}_{W}$ generated by $\widetilde {\rm NC}_{W}$ and the
 element~$[a_{\theta}]$, where $[a_{\theta}]$ stands for the element in
${\cal B}_{W}$ corresponding to the highest root $\theta$ for~$W$. In other
words, $\widetilde {\rm ANWC}_{W}$ is the image of the algebra
$\widetilde {\rm ANC}_{W}$ under the natural map
$B{\cal E}(W) \longrightarrow {\cal B}_{W}$, see, e.g., \cite{Ba,KMa}.
 It follows from \cite[Section~6]{Ba},
 that ${\rm Hilb}(\widetilde {\rm NC}_{W},t)=\prod\limits_{i=1}^{l}[m_{i}+1]_{t}$.
\end{Remark}

\begin{Conjecture}[Yu.~Bazlov and A.N.~Kirillov, 2002]\label{conjecture4.1}
\begin{gather*}
{\rm Hilb}\big(\widetilde {\rm ANWC}_{W},t\big)=
\prod_{i=1}^{l}{{1-t^{m_{i}+1} \over 1-t^{m_i}}} {\prod_{i=1}^{l} {1-t^{a_i} \over
1-t }}=
P_{\rm af\/f}(W,t) \prod_{i=1}^{l} (1-t^{a_i}),
\end{gather*}
where
\begin{gather*}
P_{\rm af\/f}(W,t):=\sum_{w \in W_{\rm af\/f}} t^{l(w)}=
\prod_{i=1}^{l}{(1+t+\cdots+t^{m_i}) \over 1-t^{m_i}}
\end{gather*}
denotes the Poincar\'e polynomial corresponding to the affine Weyl group
$W_{\rm af\/f}$, see {\rm \cite[p.~245]{Bo}};
$a_{i}:=( 2 \rho,\alpha_{i}^{\vee})$,
$1 \le i \le l$, denote the coefficients of the decomposition of the sum of
positive roots~$2 \rho$ in terms of the simple roots~$\alpha_i$.
\end{Conjecture}

In particular,
\begin{gather*}
\dim \widetilde {\rm ANWC}_{W}=|W|{\prod\limits_{i=1}^{l}a_i \over
\prod\limits_{i=1}^{l} m_{i}} \qquad \text{and} \qquad \deg {\rm Hilb}\big(\widetilde {ANWC}_{W},t\big)=\sum\limits_{1=1}^{l}
a_{i}.
\end{gather*}
It is well-known that the product~$\prod\limits_{i=1}^{l}{1-t^{a_{i}}
\over 1-t^{m_i}}$ is a
symmetric (and {\it unimodal}?) polynomial with non-negative integer
coef\/f\/icients.

\begin{Example}\label{example4.1}\quad
\begin{alignat*}{3}
& (a) \quad && {\rm Hilb}\big(\widetilde{\rm ANC}_{3},t\big)=[2]_{t}^2[3]_{t}, \qquad
 {\rm Hilb}\big(\widetilde{ANC}_{4},t\big)=[3]_{t}^2[4]_{t}^2, &\\
 &&&
 {\rm Hilb}\big(\widetilde{ANC}_{5},t\big)=[4]_{t}^2[5]_{t}[6]_{t}^2,& \\
& (b)\quad && {\rm Hilb}(B{\cal E}_2,t)=(1+t)^4\big(1+t^2\big)^2, &\\
&&& {\rm Hilb}(\widetilde{\rm ANC}_{B_2},t)= (1+t)^3\big(1+t^2\big)^2=P_{\rm af\/f}(B_2,t)\big(1-t^3\big)\big(1-t^4\big).&\\
& (c) \quad && {\rm Hilb}\big(\widetilde{\rm ANC}_{B_3},t\big)=
 (1+t)^3\big(1+t^2\big)^2\big(1+t^3\big)\big(1+t^4\big)\big(1+t+t^2\big)\big(1+t^3+t^6\big)& \\
 &&& \hphantom{{\rm Hilb}\big(\widetilde{\rm ANC}_{B_3},t\big)}{} =
P_{\rm af\/f}(B_3,t)\big(1-t^5\big)\big(1-t^8\big)\big(1-t^9\big).
\end{alignat*}
Indeed, $m_{B_3}=(1,3,5)$, $a_{B_3}=(5,8,9)$.
\end{Example}

\begin{Definition}\label{definition4.3}
 Let $\langle \widetilde{\rm ANC}_{n} \rangle$ denote the two-sided
ideal in $3T_{n}^{(0)}$ generated by the ele\-ments $ \{ u_{i,i+1} \}$,
$1 \le i \le n-1$, and $u_{1,n}$. Denote by $U_{n}$ the quotient
$U_{n}=3T_{n}^{0} / \langle \widetilde{\rm ANC}_{n} \rangle$.
\end{Definition}

\begin{Proposition}\label{proposition4.1}
\begin{gather*}
U_{4} \cong \langle u_{1,3},u_{2,4} \rangle \cong \Z_{2} \times \Z_{2},\qquad
 U_{5} \cong \langle u_{1,4},u_{2,4},u_{2,5},u_{3,5},u_{1,3} \rangle
\cong \widetilde{\rm ANC}_{5}.
\end{gather*}
\end{Proposition}

In particular, ${\rm Hilb}\big(3T_5^{(0)},t\big)=\bigl[ {\rm Hilb}({\widetilde{\rm ANC}}_5,t) \bigr]^2$.

\subsubsection{Parabolic 3-term relations algebras and partial f\/lag varieties}\label{section4.1.2}

In fact one can construct an analogue of the algebra $3HT_n$
and a commutative subalgebra inside it, for {\it any} graph
$\Gamma=(V,E)$ on $n$ vertices, possibly with loops and multiple edges~\cite{K}. We denote this algebra by $3T_n(\Gamma)$, and denote by
$3T_n^{(0)}(\Gamma)$ its {\it nil-quotient}, which may be considered as
a~``classical limit of the algebra $3T_n(\Gamma)$''.

The case of the complete graph $\Gamma=K_n$ reproduces the
results of the present paper and those of~\cite{K}, i.e., the case of the full
f\/lag variety ${\cal F}l_n$. The case of the {\it complete multipartite graph}
$\Gamma=K_{n_{1},\ldots, n_{r}}$ reproduces the analogue of results stated
 in the present paper for the full f\/lag variety ${\cal F}l_n$, to the case of
the {\it partial flag} variety~${\cal F}_{n_{1},\ldots,n_{r}}$, see~\cite{K} for details.

We {\it expect} that in the case of the complete graph
with all edges having the same multiplicity~$m$, denoted by either
$\Gamma = K_{n}^{(m)}$, or~$m K_{n}$ in the present paper, the
commutative subalgebra generated by the Dunkl elements in the algebra
$3T_n^{(0)}(\Gamma)$ is related to the algebra of coinvariants of the diagonal
 action of the symmetric group~$\mathbb{S}_n$ on the ring of polynomials
$\Q\big[X_n^{(1)},\ldots,X_n^{(m)}\big]$, where we set $X_n^{(i)} =
\big\{ x_1^{(i)},\ldots,x_n^{(i)} \big\}$.

\begin{Example} \label{example4.2} Take $\Gamma=K_{2,2}$. The algebra $3T^{(0)}(\Gamma)$ is
generated by four elements $\{a=u_{13}$, $b=u_{14}, c=u_{23}, d=u_{24} \}$
subject to the following set of (def\/ining) relations
\begin{itemize}\itemsep=0pt
\item $ a^2=b^2=c^2=d^2=0$, $ c b=b c$, $a d=d a$,

\item $a b a+b a b=0=a c a+c a c$, $b d b+d b d=0=c d c+d c d$,\\
$a b d-b d c-c a b+d c a=0=a c d-b a c-c d b+d b a$,

\item $a b c a+a d b c+b a d b+b c a d+c a d c+d b c d=0$.
\end{itemize}
It is not dif\/f\/icult to see that\footnote{Hereinafter we shell use notation
$ (a_{0},a_{1},\ldots,a_{k})_{t} := a_{0}+a_{1} t +
\cdots + a_{k} t^{k}$.}
\begin{gather*}
{\rm Hilb}\big(3T^{(0)}(K_{2,2}),t\big)=[3]_{t}^{2} [4]_{t}^{2},\qquad
{\rm Hilb}\big(3T^{(0)}(K_{2,2})^{ab},t\big)=(1,4,6,3).
\end{gather*}
Here for any algebra $A$ we denote by $A^{ab}$ its {\it abelianization}\footnote{See \url{http://groupprops.subwiki.org/wiki/Abelianization}.}.
\end{Example}

The commutative subalgebra in $3T^{(0)}(K_{2,2})$, which corresponds to
the intersection
$3T^{(0)}(K_{2,2})$ $\cap \Z[\theta_1,\theta_2,\theta_3,\theta_4]$,
is generated by the elements $c_1:=\theta_1+\theta_2=(a+b+c+d)$ and
$c_2:= \theta_1 \theta_2=(ac+ca+bd+db+ad+bc)$. The elements $c_1$ and~$c_2$
commute and satisfy the following relations
\begin{gather*}
c_{1}^{3}-2 c_1 c_2=0,\qquad c_{2}^{2}-c_{1}^{2} c_2=0.
\end{gather*}
The ring of polynomials $\Z[c_1,c_2]$ is isomorphic to the cohomology ring
$H^{*}(\operatorname{Gr}(2,4),\Z)$ of the Grassmannian variety $\operatorname{Gr}(2,4)$.

To continue exposition, let us take $m \le n$, and consider the complete
multipartite graph $K_{n,m}$ which corresponds to the Grassmannian variety
$\operatorname{Gr}(n,m+n)$. One can show
\begin{gather*}
{\rm Hilb}\big(3T_{n+m}^{(0)}(K_{n,m})^{ab},t\big)= \sum_{k=0}^{n-1} (-1)^{k}
(1+(n-k) t)^{m-1} \prod_{j=1}^{n-k} (1+j t) {n \brace n-k} \\
\hphantom{{\rm Hilb}\big(3T_{n+m}^{(0)}(K_{n,m})^{ab},t\big)}{}
 =t^{n+m-1} {\rm Tutte}\big(K_{n,m},1+t^{-1},0\big),
\end{gather*}
where ${n \brace k } := S(n,k)$ denotes the
Stirling numbers of the second kind, that is the number of ways to partition
a~set of~$n$ labeled objects into~$k$ nonempty unlabeled subsets, and for
any graph $\Gamma$,
${\rm Tutte}(\Gamma,x,y)$ denotes the {\it Tutte polynomial\/}\footnote{See, e.g., \url{https://en.wikipedia.org/wiki/Tutte_polynomial}. It is
well-known that
\begin{gather*}
{\rm Tutte}(\Gamma,1+t,0)= {(-1)^{|\Gamma|}} {t^{-\kappa(\Gamma)}}
{\rm Chrom}(\Gamma,-t),
\end{gather*}
where for any graph $\Gamma$, $|\Gamma|$ is equal to the number of vertices
and $\kappa(\Gamma)$ is equal to the number of connected components of~$\Gamma$. Finally ${\rm Chrom}(\Gamma,t)$ denotes the {\it chromatic polynomial}
corresponding to graph~$\Gamma$, see, e.g.,~\cite{We}, or \url{https://en.wikipedia.org/wiki/Chromatic_polynomial}.\label{footnote32}}
corresponding to graph~$\Gamma$.

It is well-known that the Stirling numbers $S(n,k)$ satisfy the following
identities
\begin{gather*}
\sum_{k=0}^{n-1} (-1)^k~S(n,n-k) \prod_{j=1}^{n-k} (1+j t)=(1+t)^n,\qquad
\sum_{n \ge k} {n \brace k } {x^n \over n !}={(e^x-1)^k
\over k !}.
\end{gather*}
Let us observe that 
\begin{gather*}
\dim\big(3T^{(0)}(K_{n,n})^{ab}\big)=
 \sum_{k=0}^{n-1} (-1)^{k} (n+1-k)^{n-1} (n+1-k) ! {n \brace n-k }\\
 \hphantom{\dim\big(3T^{(0)}(K_{n,n})^{ab}\big)}{} =
\sum_{k=1}^{n+1} ((k-1) !)^2 { n+1 \brace k}^2 ,
\end{gather*}
see, e.g., \cite[$A048163$]{SL}.

Moreover, if $m \ge 0$, then
\begin{gather*}
 \sum_{n \ge 1} \dim\big(3T^{(0)}(K_{n,n+m})^{ab}\big) { t^{n}} = \sum_{k \ge 1}
{k^{k+m-1} (k-1) ! t^{k} \over \prod\limits_{j=1}^{k-1} (1+k j t) },\\
 \sum_{n \ge 1} {\rm Hilb}\big(3T^{(0)}(K_{n,m})^{ab},t\big) z^{n-1}=
\sum_{k \ge 0} (1+k t)^{m-1} \prod_{j=1}^{k} {{z (1+j t) \over 1+j z}}.
\end{gather*}

\begin{Comments}[poly-Bernoulli numbers]\label{commenst4.1}
Based on listed above identities involving the Stirling numbers $S(n,k)$, one
can prove the following {\it combinatorial} formula
\begin{gather}\label{equation4.1}
 \dim \big(3T^{(0)}(K_{n,m})^{ab}\big) = \sum_{j=0}^{\min(n,m)}(j !)^2
{n+1 \brace j+1} {m+1 \brace j+1} = {B_{n}^{(-m)}} = {B_{m}^{(-n )}},
\end{gather}
where $B_n^{(k)}$ denotes the {\it poly-Bernoulli number} introduced by M.~Kaneko~\cite{Kan}.

On the other hand, it is well-known, see, e.g., footnote~\ref{footnote32}, that for any graph
 $\Gamma$ the specialization ${\rm Tutte}(\Gamma; 2,0)$ of the Tutte polynomial associated with graph~$\Gamma$, counts the number of acyclic orientations of~$\Gamma$. Therefore, the poly-Bernulli number $B_n^{(-m)}$ counts the number of acyclic orientatations of the complete bipartite graph~$K_{n,m}$.

For example, $\dim \big(3T^{(0)}(K_{3,3})^{ab}\big)= 230 =1 + 7^2 +(2 !)^2 6^2 + (3 !)^2$, cf.\ Example~\ref{examples4.16}(3).

For the reader's convenient, we recall below a def\/inition of
{\it poly-Bernoulli numbers}. To start with, let $k$ be an integer, consider
the formal power series
\begin{gather*}
\operatorname{Li}_k(z):= \sum_{n=1}^{\infty} {z^n \over n^k}.
\end{gather*}
If $k \ge 1$, $\operatorname{Li}_k(z)$ is the $k$-th {\it polylogarithm}, and if
 $k \le 0$, then $\operatorname{Li}_k(z)$ is a~rational function. Clearly
$\operatorname{Li}_1(z)=-\ln(1-z)$. Now def\/ine {\it poly-Bernoulli numbers} through the
generating function
\begin{gather*}
{\operatorname{Li}_k(1-e^{-z}) \over {1-e^{-z}}} = \sum_{n=0}^{\infty} B_n^{(k)} {z^n
\over n !}.
\end{gather*}
Note that a combinatorial formula for the numbers $B_n^{(-k)}$ stated in
\eqref{equation4.1} is a consequence of the following identity~\cite{Kan}
\begin{gather*}
 \sum_{n=0}^{\infty} \sum_{k=0}^{\infty} B_n^{(-k)} {{x^n \over n !}}
{{z^k \over k !}} ={e^{x+z} \over {1-(1-e^x)(1-e^z)}}.
\end{gather*}

Note that the poly-Bernoulli numbers $B_{n}^{(-m)} (=B_{m}^{(-n)})$
have the following combinatorial interpretation\footnote{See for example, \cite[$A136126$]{SL}, \cite[$A099594$]{SL} or~\cite[Theorem~3.1]{CE},
and the literature quoted therein. Recall, that the {\it excedance set} of a~permutation $\pi \in \mathbb{S}_n$ is the set of indices~$i$, $1 \le i \le n$, such that $\pi(i) > i$.}, namely, the number $B_{n}^{(-m)}$, and therefore
the dimension of the algebra $3T^{(0)}(K_{n,m})$ is
equal to
\begin{gather*}
 B_{n}^{(-m)} =T(n-1,m)+T(n,m-1),
\end{gather*}
 where~\cite{CE}
\begin{gather*}
T(n,m):= \sum_{j=0}^{\min(n,m)} j ! (j+1) ! {n+1 \brace j+1} {m+1 \brace j+1}
\end{gather*}
 is equal to the number of permutations $w \in
\mathbb{S}_{n+m}$~having the {\it excedance set} $\{1,2,\ldots,m \}$.

\begin{Exercise} \label{exercise4.1} Show that polynomial ${\rm Hilb}(3T^{(0)}(K_{n,m}),t)$ has
degree $n+m-1$, and
\begin{gather*}
{\rm Coef\/f}_{t^{n+m-1}} \bigl({\rm Hilb}\big(3T^{(0)}(K_{n,m}),t\big)\bigr) = T(n-1,m-1).
\end{gather*}
\end{Exercise}

\begin{Problem}\label{problem4.0}
 To find a bijective proof of the identity~\eqref{equation4.1}.
\end{Problem}

Final remark, the explicit expression for the chromatic polynomial of the
complete tripartite graph $K_{n,n,n}$ can be found in~\cite[$A212220$]{SL}.
\end{Comments}

Now let $\theta_i^{(n+m)} = \sum\limits_{j \not=i} u_{ij}$, $1 \le i \le n+m$, be the
Dunkl elements in the algebra $3T^{(0)}(K_{n+m})$, def\/ine the following elements
the in the algebra $3T^{(0)}(K_{n,m})$
\begin{gather*}
c_k:=e_k\big(\theta_1^{(n+m)},\ldots,\theta_{n}^{(n+m)}\big), \qquad 1 \le k \le n,
\\
{\overline{c}}_r:=e_r\big(\theta_{n+1}^{(n+m)},\ldots,\theta_{n+m}^{(n+m)}\big),\qquad
 1 \le r \le m.
 \end{gather*}
Clearly,
\begin{gather*}
\left(1+\sum_{k=1}^{ n} c_k t^k\right)\left(1+\sum_{r=1}^{m} {\overline{c}}_r t^r\right) =
\prod_{j=1}^{n+m} \big(1+\theta_{j}^{(n+m)}\big) = 1.
\end{gather*}
Moreover, there exist the natural isomorphisms of algebras
\begin{gather*}
H^{*}(\operatorname{Gr}(n,n+m),\Z) \cong \Z[c_1,\ldots,c_n] \bigg/ \left\langle \left(1+\sum_{k=1}^{n} c_k t^k\right)
\left(1+\sum_{r=1}^{m} {\overline{c}}_r t^r\right) -1 \right\rangle,\\
QH^{*}(\operatorname{Gr}(n,n+m)) \cong \Z[q][c_1,\ldots,c_n] \bigg/ \left\langle
\left(1+\sum_{k=1}^{n} c_k t ^k\right)\left(1+\sum_{r=1}^{m} {\overline{c}}_r t^r\right) -1 - q t^{n+m} \right\rangle.
\end{gather*}
Let us recall, see Section~\ref{section2}, footnote~\ref{footnote26}, that for a commutative ring~$R$
and a~polynomial $p(t)=\sum\limits_{j=1}^{s}g_j t^{j}
 \in R[t]$, we denote by $\langle p(t) \rangle$ the ideal in the ring~$R$
generated by the coef\/f\/icients $g_1,\ldots,g_s$.

These examples are illustrative of the similar results valid for the
{\it general complete multipartite graphs} $K_{n_{1},\ldots,n_{r}}$, i.e.,
for the partial f\/lag varieties~\cite{K}.

To state our results for {\it partial flag varieties}
 we need a bit of notation. Let $N:=n_1+\cdots+n_r$, $n_j > 0$, $\forall\, j$, be a~composition of size~$N$.
 We set $N_j:=n_1+ \cdots +n_j$, $j=1,\ldots,r$, and
$N_0=0$. Now, consider the commutative subalgebra in the algebra
$3T_{N}^{(0)}(K_{N})$ generated by the set of Dunkl elements
$\big\{\theta_1^{(N)},\ldots,\theta_{N}^{(N)} \big\}$, and def\/ine elements
$\big\{ c_{k_{j}}^{(j,N)} \in 3T_N^{(0)}(K_{n_1,\ldots,n_r})\big\}$ to be the degree~$k_j$ elementary symmetric polynomials of the Dunkl
elements $\theta_{N_{j-1}+1}^{(N)},\ldots,\theta_{N_{j}}^{(N)}$, namely,
\begin{gather*}
c_k^{(j)}:= c_{k_{j}}^{(j,N)}= e_{k}\big(\theta_{N_{j-1}+1}^{(N)},\ldots,\theta_{N_{j}}^{(N)}\big),
\qquad 1 \le k_j \le n_j, \qquad j=1,\ldots,r,\\
 c_{0}^{(j)}=1, \qquad \forall\, j.
 \end{gather*}
Clearly
\begin{gather*}
 \prod_{j=1}^{r} \left(\sum_{a=0}^{n_j} c_{a}^{(j)} t^{a}\right) = \prod_{j=1}^{N}
\big(1+\theta_j^{(N)}t^{j}\big) = 1.
\end{gather*}

\begin{Theorem}\label{theorem4.2}
The commutative subalgebra generated by the elements
 $ \{ c_{k_{j}}^{(j)},\,1 \le k_j \le n_j,\,1 \le j \le r-1\}$, in the algebra
$3T_N^{(0)}(K_{n_1,\ldots,n_r})$ is isomorphic to the cohomology ring
$H^{*}({\cal{F}}l_{n_1,\ldots,n_r}, \Z)$ of the partial flag variety
${\cal{F}}l_{n_1,\ldots,n_r}$.
\end{Theorem}

In other words, we treat the Dunkl elements $\big\{
\theta_{N_{j-1}+a}^{(N)},\,1 \le a \le n_j\big\}$, $j=1,\ldots,r$, as the
{\it Chern roots} of the vector bundles $\{ \xi_{j}:=
{\cal{F}}_j / {\cal{F}}_{j-1} \}$, $j=1,\ldots,r$,
 over the partial f\/lag variety ${\cal{F}}l_{n_1,\ldots,n_r}$.

Recall that a point ${\boldsymbol{F}}$ of the partial f\/lag variety ${\cal{F}}l_{n_1,
\ldots,n_r}$, $n_1+\cdots+n_r=N$, is a sequence of embedded subspaces
\begin{gather*}
{\boldsymbol{F}}=\big\{0=F_{0} \subset F_1 \subset F_2 \subset \cdots \subset F_r=\C^{N}
 \big\}
 \end{gather*}
 such that
 \begin{gather*}
 \dim (F_i / F_{i-1})=n_i, \qquad i=1,\ldots, r.
\end{gather*}
By def\/inition, the f\/iber of the vector bundle~$\xi_i$ over a point ${\boldsymbol{F}} \in
 {\cal{F}}l_{n_1,\ldots,n_r}$ is the $n_i$-dimensional vector space
$F_{i} / F_{i-1}$.

To conclude, let us describe the set of (def\/ining) relations among the
elements $\big\{c_{a}^{(j)}\big\}$, $1 \le a \le n_{j}$, $1 \le j \le r-1$. To proceed,
let us introduce the set of variables $\big\{x_{a}^{(j)}\, |\, 1 \le a \le n_j, \, 1 \le j \le r-1 \big\}$, and def\/ine polynomials $b_0=1,b_{k}:=b_{k}\big(\big\{x_{a}^{(j)}\big\}\big)$, $k \ge 1$ by the use of generating function
\begin{gather*}
 \frac{1}{\prod\limits_{j=1}^{r-1} \prod\limits_{a=1}^{n_{j}}\big(1+x_{a}^{(j)}\big) t^{a}} =
\sum_{k \ge 0} b_{k} t^{k}.
\end{gather*}
Now let us introduce matrix $M_{m}\big(\big\{x_{a}^{(j)}\big\}\big):= (m_{ij})$, where
\begin{gather*}
m_{ij}:= \begin{cases}
b_{i+j-1} & \text{if} \ \ j > i, \\
1 & \text{if} \ \ j=i-1, \ i \ge 2 , \\
0 & \text{if} \ j < i-1.
\end{cases}
\end{gather*}

\begin{Claim}\label{claim_page58} $\det M_{m}\big(\big\{c_{a}^{(i)}\big\}\big) = 0$, $N_{r-1} < m \le N$. Moreover,
\begin{gather*}
H^{*}({\cal{F}}l_{n_1,\ldots,n_r},\Z)
\cong \Z[\{x_{a}^{j} \}] / \langle M_{N_{r-1}+1},\ldots, M_{N} \rangle.
\end{gather*}
\end{Claim}

A meaning of the algebra $3T_n^{(0)}(\Gamma)$ and
the corresponding commutative subalgebra inside it for a general graph
$\Gamma$ is still unclear.

\begin{Conjecture}\footnote{Part $(1)$ of this conjecture has been proved in~\cite{Li}.
In~\cite{Li} the author has used notation ${\rm OT}(\Gamma)$ for the Orlik--Terao
algebra associated with (simple) graph~$\Gamma$. In our paper we prefer to
denote the corresponding Orlik--Terao algebra by ${\rm OS}^{+}(\Gamma)$.} \label{conjecture4.2}
\begin{enumerate}\itemsep=0pt
\item[$(1)$] Let $\Gamma=(V,E)$ be a connected subgraph of the complete graph
$K_n$ on $n$ vertices. Then
\begin{gather*}
{\rm Hilb}\big(3T_n^{(0)}(\Gamma)^{ab},t\big)= t^{|V|-1} {\rm Tutte}\big(\Gamma;1+t^{-1},0\big).
\end{gather*}

\item[$(2)$] Let $\Gamma=(V,E,\{m_{ij}, (ij) \in E \})$ be a connected subgraph of
the complete graph $K^{({\boldsymbol{m})}}_n$ with multiple edges such that an edge
$(ij) \in K_n$ has the multiplicity~$m_{ij}$. Let $3T_n^{(0)}(\Gamma,{\boldsymbol{m}})$ denotes the subalgebra in the algebra $3T_n^{(0)}({\boldsymbol{m}})$ generated by
elements $\{u_{ij}^{(\alpha_{(ij)})},\, (ij) \in E,\, 1 \le \alpha_{(ij)} \le m_{ij} \}$, see Section~{\rm \ref{section2.2}}. Let ${\cal{A}}(\Gamma,\{m_{ij}\})$ denotes the
graphic arrangement cor\-res\-pon\-ding to the graph $(\Gamma,\{m_{ij}\})$,
that is the set of hyperplanes $\{ H_{(ij),a}= (x_i -x_j= a),\, 0 \le a \le
m_{ij}-1,\, (ij) \in E \}$.
Then
\begin{gather*}
3T_n^{(0)} (\Gamma,{\boldsymbol{m}})^{ab} = {\rm OS}^{+}({\cal{A}}(\Gamma,\{m_{ij} \})),
\end{gather*}
where for any arrangements of hyperplanes ${\cal{A}}$, ${\rm OS}^{+}({\cal{A}})$
denotes the even Orlik--Solomon algebra of the arrangement~$\cal{A}$~{\rm \cite{OT}}.
In the case when $m_{ij}=1$, $\forall\, 1 \le i < j \le n$,
$3T_n^{(0)}(\Gamma)^{\rm anti} = {\rm OS}({\cal{A}}(\Gamma)) $.
\end{enumerate}
\end{Conjecture}

\begin{Examples}\label{examples4.16}\quad
\begin{enumerate}\itemsep=0pt
\item[(1)] Let $G=K_{2,2}$ be complete bipartite graph of type $(2,2)$. Then
\begin{gather*}
{\rm Hilb}\big(3T_4^{0}(2,2)^{ab},t\big)=(1,4,6,3)=t^2(1+t)+t(1+t)^2+(1+t)^3,
\end{gather*}
and the Tutte polynomial for the graph $K_{2,2}$ is equal to $x+x^2+x^3+y$.

\item[(2)] Let $G=K_{3,2}$ be complete bipartite graph of type $(3,2)$. Then
\begin{gather*}
{\rm Hilb}\big(3T_5^{0}(3,2)^{ab},t\big)=(1,6,15,17,7)\\
\hphantom{{\rm Hilb}\big(3T_5^{0}(3,2)^{ab},t\big)}{}
=t^3 (1+t)+3 t^2~(1+t)^2+2t (1+t)^3+(1+t)^4,
\end{gather*}
and the Tutte polynomial for the graph $K_{3,2}$ is equal to
\begin{gather*}
x+3 x^2+2 x^3+x^4 +y+3 x y+y^2.
\end{gather*}

\item[(3)] Let $G=K_{3,3}$ be complete bipartite graph of type $(3,3)$. Then
\begin{gather*}
{\rm Hilb}\big(3T_6^{0}(3,3)^{ab},t\big)=(1,9,36,75,78,31)=(1+t)^5+4t(1+t)^4\\
\hphantom{{\rm Hilb}\big(3T_6^{0}(3,3)^{ab},t\big)=}{}+10t^{2}(1+t)^3+11t^{3}(1+t)^2+5t^{4}(1+t),
\end{gather*}
and the Tutte polynomial of the bipartite graph $K_{3,3}$ is equal to
\begin{gather*}
5x+11x^2+10x^3+4x^4+x^5+15xy+9x^2y+6xy^2+5y+9y^2+5y^3+y^4.
\end{gather*}

\item[(4)] Consider complete multipartite graph $K_{2,2,2}$. One can show
that
\begin{gather*}
{\rm Hilb}\big(3T_6^{(0)}(K_{2,2,2})^{ab},t\big)=(1,12,58,137,154,64)=
11t^4(1+t)+25t^3(1+t)^2\\
\hphantom{{\rm Hilb}\big(3T_6^{(0)}(K_{2,2,2})^{ab},t\big)=}{}+20 t^2(1+t)^3+7 t (1+t)^4+(1+t)^5,
\end{gather*}
and
\begin{gather*}
{\rm Tutte}(K_{2,2,2},x,y)= x(11,25 ,20,7,1)_{x}+ y (11,46,39,8)_{x} +
y^2(32,52,12)_{x}\\
\hphantom{{\rm Tutte}(K_{2,2,2},x,y)=}{}
+y^3(40,24)_{x}+ y^4(29,6)_{x}+15y^{5}+5 y^{6}+y^{7} .
\end{gather*}
\end{enumerate}
\end{Examples}

The above examples show that the Hilbert polynomial ${\rm Hilb}(3T_n^{0}(G)^{ab},t)$
appears to be a~certain specialization of the Tutte polynomial of the
corresponding graph~$G$.

Instead of using the Hilbert polynomial of the
algebra $3T_n^{0}(G)^{ab}$ one can consider the graded Betti numbers (over a
f\/ield $\boldsymbol{k}$) polynomial ${\rm Betti}_{\boldsymbol{k}}(3T_n^{0}(G)^{ab},x,y)$. For example,
\begin{gather*}
{\rm Betti}_{\Q}\big(3T_3^{0}(K_3)^{ab},x,y\big)= 1+ x y(4,2)_{x} + x^2 y^2 (3,2)_{x},\\
{\rm Betti}_{\Q}\big(3T_4^{0}(K_{2,2})^{ab},x,y\big)= 1+ 4 x y+ x y^2(1,9,3)_{x}+ x^2 y^3 (1,6,3)_{x},\\
{\rm Betti}_{\Q}\big(3T_5^{0}(K_{3,2})^{ab},x,y\big)= 1+ 6 x y+ x y^2(3,25,9)+x^2 y^3(6,45,34,7)+ x^3 y^4(3,25,25,7),\\
{\rm Betti}_{\Q}\big(3T_4^{0}(K_4)^{ab},x,y\big)=
1+ x y (10,10)+ x^2 y^2(25,46,26,6)+ x^3 y^3 (15,36,25,6),\\
{\rm Betti}_{\Z/2\Z}\big(3T_4^{0}(K_4)^{ab},x,y\big)=
1+ x y (10,10,{\bf 1})_{x}+ x^2 y^2(25,46,26,6)+ x^3 y^3 ({\bf 16},36,25,6), \\
{\rm Betti}_{\Q}\big(3T_5^{0}(K_5)^{ab},x,y\big)= 1+ x y(20,30)+ x^2 y^2(109,342,315,72)+ x^3 y^3(195,852,1470,\\
\hphantom{{\rm Betti}_{\Q}\big(3T_5^{0}(K_5)^{ab},x,y\big)=}{} 1232,639,190,24)+ x^4 y^4 (105,540,1155,1160,639,190,24),\\
{\rm Betti}_{\Q}\big(3T_5^{0}(K_5)^{ab},1,1\big)= 9304, \\
{\rm Betti}_{\Z/3\Z}\big(3T_5^{0}(K_5)^{ab},x,y\big)= 1+x y (20,30)+x^2 y^2(109,342,315,72,{\bf 1}) \\
\hphantom{{\rm Betti}_{\Z/3\Z}\big(3T_5^{0}(K_5)^{ab},x,y\big)=}{} +x^3 y^3 (195,852,1471,1232,640,190,24)\\
\hphantom{{\rm Betti}_{\Z/3\Z}\big(3T_5^{0}(K_5)^{ab},x,y\big)=}{}
+x^4 y^4 (105,540,{\bf 1156},1160,639,190,24), \\
{\rm Betti}_{\Z/3\Z}\big(3T_5^{0}(K_5)^{ab},1,1\big)= 9308,
 \\
{\rm Betti}_{\Z/2\Z}\big(3T_5^{0}(K_5)^{ab},x,y\big)= 1+x y(20,30,{\bf 5})+x^2 y^2({\bf 114},342,{\bf 340},{\bf 131},{\bf 10}) \\
\hphantom{{\rm Betti}_{\Z/2\Z}\big(3T_5^{0}(K_5)^{ab},x,y\big)=}{}+
x^3 y^3({\bf 220},{\bf 911},{\bf 1500},1291,649,190,24)\\
\hphantom{{\rm Betti}_{\Z/2\Z}\big(3T_5^{0}(K_5)^{ab},x,y\big)=}{}
+x^4 y^4 ({\bf 125},{\bf 599},{\bf 1165},1160,639,190,24), \\
 {\rm Betti}_{\Z/2\Z}\big(3T_5^{0}(K_5)^{ab},1,1\big)= 9680, \\
{\rm Betti}_{\Z/2\Z}\big(3T_6^{0}(K_{3,3})^{ab},x,y\big)= 1+9 xy+x y^2(9,69,27)+ x^2 y^3(40,285,257,{\bf 52}) \\
\hphantom{{\rm Betti}_{\Z/2\Z}\big(3T_6^{0}(K_{3,3})^{ab},x,y\big)=}{}+
x^3 y^4(59,526,866,563,201,31)\\
\hphantom{{\rm Betti}_{\Z/2\Z}\big(3T_6^{0}(K_{3,3})^{ab},x,y\big)=}{}
+ x^4 y^5(28,{\bf 311},636,520,201,31),\\
 {\rm Betti}_{\Z/2\Z}\big(3T_6^{0}(K_{3,3})^{ab},1,1\big)= 4740,\\
{\rm Betti}_{\Q}\big(3T_6^{0}(K_{3,3})^{ab},x,y\big)= 1+ 9 xy + x y^2(9,69,27)+ x^2 y^3(40,285,257,43)\\
\hphantom{{\rm Betti}_{\Q}\big(3T_6^{0}(K_{3,3})^{ab},x,y\big)=}{}
+ x^3 y^4(59,
526,866,563,201,31)\\
\hphantom{{\rm Betti}_{\Q}\big(3T_6^{0}(K_{3,3})^{ab},x,y\big)=}{}
+ x^4 y^5(28,302,636,520,201,31),\\
 {\rm Betti}_{\Q}\big(3T_6^{0}(K_{3,3})^{ab},1,1\big)= 4704.
\end{gather*}
Let us observe that in all examples displayed above, the Betti polynomials are
divisible by $1+x y$.

It should be emphasize that in the literatute one can f\/ind def\/initions of
big variety of (graded) Betti's numbers associated with a~given simple graph $\Gamma$, depending on choosing an algebra/ideal has been attached to graph~$\Gamma$. For example, to def\/ine Betti's numbers, one can start with {\it edge graph ideal/algebra} associated with a graph in question, the {\it Stanley--Reisner} ideal/ring and so on and so far. We refer the reader to carefully written
book by E.~Miller and B.~Sturmfels~\cite{MiS} for def\/initions and results
concerning combinatorial commutative algebra graded Betti's numbers. As far as
I'm aware, the graded Betti numbers we are looking for in the present paper,
are dif\/ferent from those treated in~\cite{MiS}, and more close to those studied in~\cite{B}.

It is not dif\/f\/icult to see (A.K.) that for a simple connected graph
$\Gamma$ the coef\/f\/icient just before the (unique!) monomial of the maximal degree in
${\rm Betti}_{\boldsymbol{k}}\big(3T^{0}(\Gamma)^{ab},x,y\big)$ is equal to
${\rm Tutte}(\Gamma;1,0)$. It is known \cite{BCP} that the number
${\rm Tutte}(\Gamma;1,0)$ counts that of acyclic orientations of the edges of
$\Gamma$ with a unique source at a vertex $v \in \Gamma$, or equivalently~\cite{BCP}, the number of maximum $\Gamma$-parking functions relative to vertex~$v$.

\begin{Claim}\label{claim4.0}
Let $G=(V,E)$ be a connected graph without loops. Then over any field $\textbf{k}$
\begin{gather*}
{\rm Betti}_{\textbf{k}} \big(3T_n^{0}(G)^{ab},-x,x\big)=(1-x)^{e} {\rm Hilb}\big(3T_n^{0}(G)^{ab},x\big),
\end{gather*}
where $n=|V(G)| = \text{number of vertices}$, $e=|E(G)|=\text{number of edges}$.
\end{Claim}

 \begin{Question} \label{question4.1}\quad
\begin{itemize}\itemsep=0pt
\item Let $G$ be a connected subgraph of the complete graph~$K_n$. Does the graded Betti polynomial ${\rm Betti}_{\Q}(3T_n^{0}(G)^{ab},x,y)$ is
 a~certain specialization of the Tutte polynomial $T(G,x,y)$? If not,
give example of two $($simple$)$ graphs such that their
Orlik--Terao algebras have the same Tutte polynomial, but different
 Betti polynomials over $\Q$ , and vice versa.

\item It is clear that for any graph $\Gamma$ $($or matroid$)$ one has
${\rm Tutte}(\Gamma,x,y)= a(\Gamma) (x+y) + (\text{higher degree terms})$ for some
integer $a(\Gamma) \in \N $. Does the number $a(\Gamma)$ have a~simple
combinatorial interpretation?
\end{itemize}
\end{Question}

\begin{Proposition}\label{proposition4.2}
 Let ${\boldsymbol{n}}=(n_1,\ldots,n_r)$ be a composition of $n \in \Z_{\ge 1}$, then
\begin{gather*}
 {\rm Hilb}\big(3T^{(0)} (K_{n_1,\ldots,n_r})^{ab},t\big) =\sum_{{\boldsymbol{k}}=(k_1,\ldots,k_r)
\atop 0 < k_j \le n_j } (-t)^{|{\boldsymbol{n}}|-|{\boldsymbol{k}}|}
\prod_{j=1}^{r} {n_j \brace k_j }
\prod_{j=1}^{|{\boldsymbol{k}}| -1}(1+j t),
\end{gather*}
where we set~$|{\boldsymbol{k}}|:=k_1+\cdots+k_r$.
\end{Proposition}

\begin{Remark}\label{remark4.0}
This proposition is a consequence of Conjecture~\ref{conjecture4.2}(1), which
has been proved in~\cite{Li}.
\end{Remark}

\begin{Corollary} \label{corollary4.1} One has
\begin{alignat*}{3}
& (a)\quad && 1+t(t-1) \sum_{(n_1,\ldots,n_r) \in \Z_{\ge 0}^{r} \backslash 0^{r}}
{\rm Hilb}\big(3T^{(0)}(K_{n_1,\ldots,n_r}\big)^{ab},t) {{x_{1}^{n_1} \over n_1 !}} \cdots{{x_{r}^{n_r} \over n_r !}} &\\
&&& \qquad{} = \left(1 +t \sum_{j=1}^{r} (e^{-x_j}-1) \right)^{1-t},&\\
& (b)\quad && \sum_{(n_1,n_2,\ldots,n_r) \in \Z_{\ge 0} \backslash 0^{r}}
\dim \big(3T^{(0)}(K_{n_1,\ldots,n_r})^{ab} \big)
{{x^{n_1} \over n_1 !}} \cdots {{x^{n_r} \over n_r !}}= - \log \left(1-r+
\sum_{j=1}^{r} e^{-x_j} \right), & \\
& (c)\quad && {\rm Hilb}\big(3T^{(0)}(K_{n_1,\ldots,n_r})^{ab},t\big)= (-t)^{|{\boldsymbol{n}}|} {\rm Chrom} \big(K_{n_1,\ldots,n_r},-t^{-1}\big),& \\
& (d)\quad && \dim \big(3T^{(0)}(\Gamma)^{ab}\big) \ \ \text{is equal to the number of acyclic orientations of $\Gamma$}, &
\end{alignat*}
where $\Gamma$ stands for a simple graph.
\end{Corollary}

Recall that for any graph $\Gamma$ we denote by ${\rm Chrom}(\Gamma,x)$
the chromatic polynomial of that graph.

Indeed, one can show\footnote{If $r=1$, the complete unipartite graph $K_{(n)}$ consists of $n$
distinct points, and
\begin{gather*}
{\rm Chrom}(K_{(n)},x)= x^n = \sum_{k=0}^{n-1} {n \brace k } (x)_{k}.
\end{gather*}
Let us stress that to abuse of notation the complete unipartite graph~$K_{(n)}$ consists of $n$ disjoint points with the Tutte polynomial equals
to~$1$ for all $n \ge 1$, whereas the complete graph~$K_n$ is equal to the
complete multipartite graph $K_{(1^n)}$.}

\begin{Proposition}\label{proposition4.3}
 If $r \in \Z_{\ge 1}$, then
\begin{gather*}
{\rm Chrom} (K_{n_1,\ldots,n_r},t)= \sum_{{\boldsymbol{k}}=(k_1,\ldots,k_r)}
\prod_{j=1}^{r} {n_j \brace k_j } (t)_{|{\boldsymbol{k}}|},
\end{gather*}
where by definition $(t)_{m}:= \prod\limits_{j=1}^{m-1} (t-j)$, $(t)_{0}=1$,
$(t)_m =0$ if $m< 0$.
\end{Proposition}

Finally we describe explicitly the exponential generating function for the
{\it Tutte polynomials} of the weighted complete multipartite graphs. We
refer the reader to~\cite{MRR} for a def\/inition and a~list of basic properties
 of the Tutte polynomial of a~graph.

\begin{Definition} \label{definition4.4}
Let $r \ge 2 $ be a positive integer and $ \{S_1, \ldots,S_r \}$
be a~collection of sets of cardinalities $\#|S_j|= n_j$, $j=1,\ldots,r$. Let
${\boldsymbol{\ell}}:= \{\ell_{ij} \}_{1 \le i < j \le n}$ be a collection of non-negative integers.

 The ${\boldsymbol{\ell}}$-weighted complete multipartite graph
$K_{n_1,\ldots,n_r}^{({\boldsymbol{\ell}})} $ is a graph with the set of vertices equals
to the disjoint union $\coprod\limits_{j=1}^{r} S_i$ of the sets $S_1,\ldots,S_r$,
and the set of edges $\{ (\alpha_i, \beta_j ), \, \alpha_i \in S_i, \, \beta_j \in
S_j \}_{1 \le i < j \le r}$ of multiplicity $\ell_{ij}$ each edge~$(\alpha_i,
\beta_j)$.
\end{Definition}

\begin{Theorem} \label{theorem4.3}
Let us fix an integer $r \ge 2$ and a collection of
non-negative integers ${\boldsymbol{\ell}}:= \{\ell_{ij} \}_{1 \le i < j \le r}$.
Then
\begin{gather*}
1 + \sum_{{\boldsymbol{n}}=(n_1,\ldots,n_r) \in \Z_{\ge 0}^{r} \atop {\boldsymbol{n}}
\not= {\boldsymbol{0}}} (x-1)^{\kappa({\boldsymbol{\ell}},{\boldsymbol{n}})} {\rm Tutte}\big(K_{n_1,\ldots,n_r}^{({\boldsymbol{\ell}})},x,y\big)
{{t_1^{n_1} \over n_1 !}} \cdots {{t_r^{n_r} \over n_r !}} \\
\qquad{}= \left( \sum_{{\boldsymbol{m}}=(m_1,\ldots,m_r) \in \Z_{\ge 0}^r}
y^{ \sum\limits_{1 \le i < j \le r} \ell_{ij} m_i m_j}
(y-1)^{- |{\boldsymbol{m}}|} {{t_1^{m_1} \over m_1 !}} \cdots {{t_r^{m_r} \over m_r !}}
\right)^{(x-1)(y-1)},
\end{gather*}
where $\kappa({\boldsymbol{\ell}},{\boldsymbol{n}})$ denotes the number of connected components
of the graph $K_{n_1,\ldots,n_r}^{({\boldsymbol{\ell}})}$.
\end{Theorem}


\begin{Comments}\quad
\begin{enumerate}\itemsep=0pt
\item[(a)] Clearly the condition $\ell_{ij}=0$ means that there are no edges
between vertices from the sets~$S_i$ and~$S_j$. Therefore Theorem~\ref{theorem4.3} allows
to compute the Tutte polynomial of {\it any} (f\/inite) graph. For
example,
\begin{gather*}
{\rm Tutte}\big(K_{2,2,2,2}^{(1^{6})},x,y\big) = \big\{(0,362,927,911,451,121,17,1)_{x},\\
\hphantom{{\rm Tutte}\big(K_{2,2,2,2}^{(1^{6})},x,y\big) = }{}
 (362,2154,2928,1584,374,32)_{x},
(1589,4731,3744,1072,96)_{x},\\
\hphantom{{\rm Tutte}\big(K_{2,2,2,2}^{(1^{6})},x,y\big) = }{}
 (3376,6096,2928,448,16)_{x},
 (4828,5736,1764,152)_{x}, \\
\hphantom{{\rm Tutte}\big(K_{2,2,2,2}^{(1^{6})},x,y\big) =}{}
 (5404,4464,900,32)_{x},
 (5140,3040,380)_{x}, (4340,1840,124)_{x},\\
\hphantom{{\rm Tutte}\big(K_{2,2,2,2}^{(1^{6})},x,y\big) =}{}
 (3325,984,24)_{x},
 (2331,448)_{x},
(1492,168)_{x}, (868,48)_{x}, (454,8)_{x},\\
\hphantom{{\rm Tutte}\big(K_{2,2,2,2}^{(1^{6})},x,y\big) =}{}
 210, 84, 28, 7, 1\big\}_{y}.
\end{gather*}

\item[(b)] One can show that a formula for the chromatic polynomials from
Proposition~\ref{proposition4.2} corresponds to the specialization $y=0$ ({\it but} not
direct substitution!) of the formula for generating function for the Tutte
 polynomials stated in Theorem~\ref{theorem4.3}.

\item[(c)] The Tutte polynomial ${\rm Tutte}\big(K_{n_1,\ldots,n_{r}}^{({\boldsymbol{\ell}})},x,y\big)$ does
 not symmetric with respect to parameters $\{ \ell_{ij} \}_{1 \le i < j
\le r}$.
For example, let us write ${\boldsymbol{\ell}}= (\ell_{12},\ell_{23},\ell_{13},\ell_{14},\ell_{24},\ell_{34})$, then
\begin{gather*}
{\rm Tutte}\big(K_{2,2,2,2}^{(6,3,4,5,2,4)},1,1\big)=
2^8 \cdot 3 \cdot 5 \cdot 11^3 \cdot 241= 1231760640 .
\end{gather*}
 On the other hand,
\begin{gather*}
{\rm Tutte}\big(K_{2,2,2,2}^{(6,4,3,5,2,4)},1,1\big)=2^{13} \cdot 3 \cdot 7 \cdot 11^2
\cdot 61 = 1269768192.
\end{gather*}

\subsubsection[Universal Tutte polynomials]{Universal Tutte polynomials}\label{section4.1.3add}

Let ${\boldsymbol{m}}=(m_{ij}, \, 1 \le i < j \le n)$ be a collection of non-negative
integers. Def\/ine {\it generalized Tutte polynomial}
${\widetilde{T}}_n({\boldsymbol{m}},x,y)$ as
follows
\begin{gather*}
 (x-1)^{\kappa(n,{\boldsymbol{m}})} {\widetilde{T}}_n({\boldsymbol{m}},x,y)\\
\qquad {} =
 {\rm Coef\/f}_{[t_1 \cdots t_n]} \left[ \left( \sum_{\ell_1,\ldots,\ell_n
\atop \ell_i \in \{0,1 \}, \, \forall\, i} y^{\sum\limits_{1 \le i < j \le n} m_{ij} \ell_i \ell_j}
 (y-1)^{- (\sum_{j} \ell_{j})} t_1^{\ell_{1}} \cdots t_n^{\ell_{n}} \right)^{(x-1)(y-1)} \right] ,
\end{gather*}
where as before, $\kappa(n, {\boldsymbol{m}})$ denotes the number of connected components of the graph $K_n^{({\boldsymbol{m}})}$.
\end{enumerate}

Clearly that if $\Gamma \subset K_n^{({\boldsymbol{\ell}})}$ is a subgraph of the
weighted complete graph $K_n^{({\boldsymbol{\ell}})}\stackrel{\rm def}{=}
K_{1^n}^{({\boldsymbol{\ell}})}$, then the Tutte polynomial of graph $\Gamma$
multiplied by $(x-1)^{\kappa(\Gamma)}$ is equal to the following
specialization
\begin{gather*}
 m_{ij}=0 \quad \text{if edge} \ \ (i,j) \notin \Gamma,\qquad m_{ij}= \ell_{ij} \quad \text{if
edge} \ \ (i,j) \in \Gamma
\end{gather*}
of the generalized Tutte polynomial
\begin{gather*}
 (x-1)^{\kappa(\Gamma)} {\rm Tutte}(\Gamma,x,y)= {\widetilde{T}}_n({\boldsymbol{m}},x,y)
\Big |_{m_{ij} =0 \ \ \text{if} \ (i,j) \notin \Gamma \atop m_{ij}=\ell_{ij} \ \text{if} \ (i,j) \in \Gamma}.
\end{gather*}
For example,
\begin{enumerate}\itemsep=0pt
\item[(a)] Take $n=6$ and $\Gamma= K_6 \backslash \{ 15,16,24,25,34,36 \}$,
then
\begin{gather*}
{\rm Tutte}(\Gamma,x,y)= \{(0,4,9,8,4,1)_{x}, (4,13,9)_{x}, (8,7)_{x}, 5, 1 \}_{y}.
\end{gather*}

\item[(b)] Take $n=6$ and $\Gamma =K_6 \backslash \{15,26,34 \}$, then
\begin{gather*}
{\rm Tutte}(\Gamma,x,y)= \{(0,11,25,20,7,1)_{x}, (11,46,39,8)_{x},
 (32,52,12)_{x},\\
 \hphantom{{\rm Tutte}(\Gamma,x,y)=}{}
 (40,24)_{x}, (29,6)_{x}, 15, 5, 1 \}_{y}.
 \end{gather*}

\item[(c)] Take $n=6$ and $\Gamma=K_6 \backslash \{12.34.56 \} = K_{2,2,2}$. As a~result one obtains an expression for the Tutte polynomial of the graph $K_{2,2,2}$ displayed in Example~\ref{examples4.16}(4).
\end{enumerate}
\end{Comments}

Now set us set
\begin{gather*}
 q_{ij}:=\frac{y^{m_{ij}}-1}{y-1}.
 \end{gather*}
\begin{Lemma}\label{lemma4.1}
 The generalized Tutte polynomial ${\widetilde{T}}_n({\boldsymbol{m}},x,y)$
 is a~polynomial in the variables $\{ q_{ij} \}_{1 \le i < j \le n}$, $x$ and~$y$.
\end{Lemma}

\begin{Definition} \label{definition4.5}
The universal Tutte polynomial $T_n(\{q_{ij}\},x,y)$ is def\/ined to be the polynomial in the variables $\{q_{ij} \}$, $x$, and~$y$ def\/ined in
Lemma~\ref{lemma4.1}.
\end{Definition}

Explicitly,
\begin{gather*}
(x-1) T_n(\{q_{ij} \},x,y)\\
{}=
 {\rm Coef\/f}_{[t_1 \cdots t_n]} \! \left[\Bigg( \sum_{\ell_1,\ldots,\ell_n
\atop \ell_i \in \{0,1 \}, \, \forall\, i} \prod_{1 \le i < j \le n}
(q_{ij} (y-1)+1)^{\ell_i \ell_j}
 (y-1)^{- (\sum_{j} \ell_{j})} t_1^{\ell_{1}} \cdots t_n^{\ell_{n}} \!\Bigg)^{\!(x-1)(y-1)} \right] .
\end{gather*}

\begin{Corollary}\label{corollary4.2}
 Let $\{ m_{ij} \}_{1 \le i < j \le n}$ be a collection of
positive integers. Then the specialization
\begin{gather*}
q_{ij} \longrightarrow [m_{ij}]_{y}:= \frac{y^{m_{ij}}-1}{y-1}
\end{gather*}
of the universal Tutte polynomial $T_n(\{q_{ij} \},x,y)$ is equal to the
Tutte polynomial of the complete graph~$K_n$ with each edge~$(i,j)$ of the
 multiplicity~$m_{ij}$.
\end{Corollary}

Further specialization $ q_{ij} \longrightarrow 0$ if edge $(i,j) \notin \Gamma$
allows to compute the Tutte polynomial for any graph
\begin{gather*}
{\rm Tutte}_{3}(\{q_{12},q_{13},q_{23} \},x,y) = (1 -
q[12])(1-q[13])(1-q[23]) +
y q[12] q[13] q[23])\\
\hphantom{{\rm Tutte}_{3}(\{q_{12},q_{13},q_{23} \},x,y) =}{} + x (q[12] + q[13] + q[23] -2 )+ x^2 .
\end{gather*}
\begin{center}
\framebox{\parbox[t]{6 in}{ It is not dif\/f\/icult to see that the
${\rm Tutte}_{n}(\{ q_{ij} \},x,y)$ is a {\it symmetric polynomial} with
respect to parameters $\{ q_{ij}\}_{1 \le i < j \le n}$. }}
\end{center}

For more compact expression, it is more
convenient to rewrite the universal chromatic polynomial in terms of parameters $p_{ij}:= 1-q_{ij},~1 \le i < j \le n$, and denote it by
$Ch_n(\{p_{ij}\},x)$.
 For example,
 \begin{gather*}
 {\rm Ch}_{4}(\{ p_{ij}\},x)= - p_{12} p_{13} p_{14} p_{23} p_{24} p_{34} + x \bigl(2 - p_{12} - p_{13} - p_{14} - p_{23}- p_{24}- p_{34}\\
 \hphantom{{\rm Ch}_{4}(\{ p_{ij}\},x)=}{} + p_{12} p_{34} + p_{14} p_{23} + p_{13} p_{24} +p_{12} p_{13} p_{23} + p_{12} p_{14} p_{24}+ p_{13} p_{14} p_{34}\\
\hphantom{{\rm Ch}_{4}(\{ p_{ij}\},x)=}{} + p_{23} p_{24} p_{34} \bigr) +
x^2 \bigl(3 - p_{12} - p_{13} - p_{14} - p_{23} - p_{24} - p_{34} \bigr) + x^3.
\end{gather*}
Note that $p_{12} p_{34} + p_{14} p_{23} + p_{13} p_{24}$
 is a symmetric polynomial of the variables $ p_{12}$, $p_{34}$, $p_{13}$, $p_{24}$,
$p_{14}$,
$p_{23}$. It is important to keep in mind that parameters $\{ m_{ij}\}$ and
 $\{ p_{ij}\}$ are connected by relations
\begin{gather*}
 p_{ij}= {\frac{y- y^{m_{ij}}}{y-1}}, \qquad 1 \le i < j \le n.
 \end{gather*}
Therefore, $p_{ij}=1$ if $(i,j) \notin {\rm Edge}(\Gamma)$, $p_{ij}=0$ if $m_{ij}=1$. We emphasize that the latter equalities are valid for arbitrary~$y$.
 It is not dif\/f\/icult to see that
\begin{gather*}
{\rm Ch}_n(\{q_{ij}=0, \, \forall\, i,j\}= {\rm Tutte}(K_n;x, 0), \qquad {\rm Ch}_n(\{q_{ij}=1, \,
\forall\, i,j\} = (x-1)^{n-1}.
\end{gather*}

\begin{center}
\framebox{\parbox[t]{6 in}{Def\/ine {\it universal chromatic}
polynomial to be ${\rm Ch}_{n}(\{p_{ij}\},x)={\rm Tutte}_{n}(\{p_{ij} \},x,0)$,
 where we treat $\{ p_{ij} \}_{1 \le i < j \le n}$ as a collection of a free parameters. }}
\end{center}

To state our result concerning the universal chromatic polynomial
${\rm Ch}_n(\{p_{ij} \},x)$, f\/irst we introduce a bit of notation. Let $n \ge 2$ be an integer, consider a partition ${\cal{B}}=\big\{B_{i}= \big(b_{1}^{(i)},\ldots,b_{r_{i}}^{(i)}\big)\big\}_{1 \le i \le k}$ of the set $[1,n]:=[1,2,\ldots,n]$. In other
words one has that $[1,n] = \cup_{i=1}^{k} B_{i}$ and $B_i \cap B_j = \varnothing$ if $i \not= j$. We assume that $b_{1}^{(1)} < b_{1}^{(2)} < \cdots < b_{1}^{(k)}$. We def\/ine $\kappa({\cal{B}}):=k$. To a~given partition
${\cal{B}}$ we associate a monomial $p_{\cal{B}}:= \prod\limits_{a=1}^{k}~p_{B_{a}}$, where $p_{B_{a}}=1$ if $\kappa({\cal{B}})=1$, and
\begin{gather*}
p_{B_{a}}= \prod_{i,j \in B_{a} \atop i < j} p_{ij}.
\end{gather*}
For a given partition $\lambda \vdash n$ denote by
${\cal{L}}_{\lambda}^{(\beta)}(\{p_{ij}\})$ the sum of all monomials
$p_{{\cal{B}}} \beta^{\kappa({\cal{B}}) -2 }$ such that $\lambda= \lambda({\cal{B}}) \stackrel{\rm def}{=} (|B_{1}|,\ldots,|B_{\kappa({\cal{B}})}|)^{+}$, where
 for any composition $\alpha \models n$, $\alpha^{+}$ denotes a unique
partition obtained from $\alpha$ by the reordering of its parts.

Def\/ine $\beta$-universal chromatic polynomial to be
\begin{gather*}
 {\rm Ch}_n^{(\beta)}(\{p_{ij}\},x) = \beta^{-1} {\cal{L}}_{(n)} + \sum_{\lambda \vdash n}
 {\rm Tutte}(K_{\ell(\lambda)-1};x,0) {\cal{L}}_{\lambda}^{(\beta)},
 \end{gather*}
where summation runs over all partitions $\lambda$ of $n$; we set
$K_{0}:= \varnothing$ and ${\rm Tutte}(\varnothing;x,y)=0$.
For the reader convenience we are reminded that for the complete graph~$K_n$,
$n > 0$, one has
\begin{gather*}
{\rm Tutte}(K_n,x,0)= \prod_{j=1}^{n-1} (x+j-1) = \sum_{k=0}^{n-1} s(k,n-1)x^{k},
\end{gather*}
where $s(k,n)$ denotes the Stirling number of the f\/irst kind\footnote{See, e.g., \cite[$A008275$]{SL} or
\url{https://en.wikipedia.org/wiki/Stirling_numbers_of_the_first_kind}.\label{footnote37}}.

\begin{Theorem}[formula for universal chromatic
polynomials]\label{theorem4.29}
\begin{gather*}
{\rm Ch}_{n}(\{p_{i,j}\},x) = {\rm Ch}_{n}^{(\beta = - 1)}(\{p_{ij}\},x).
\end{gather*}
\end{Theorem}

For a given partition $\lambda \vdash n$ denote by
${\cal{L}}_{\lambda}(\{p_{ij}\})$ the sum of all monomials $p_{{\cal{B}}}$
such that $\lambda= \lambda({\cal{B}}) \stackrel{\rm def}{=} (|B_{1}|,\ldots,|B_{\kappa({\cal{B}})}|)^{+}$, where for any composition
$\alpha \models n$, $\alpha^{+}$ denotes a unique partition obtained from~$\alpha$ by the reordering of its parts.

It is clear that for a graph $\Gamma \subset K_n$ and partition ${\cal{B}}$ the
 value of monomial $p_{{\cal{B}}}$ under the specialization $p_{ij}= 0$ if $(ij) \in {\rm Edge}(\Gamma)$ and $p_{ij}=1$ if $(ij) \notin {\rm Edge}(\Gamma)$, is equal to~$1$ i\/f\/f the complementary graph $K_n {\setminus} \Gamma$ contains a
subgraph which is isomorphic to the disjoint union of complete graphs
$K(\lambda):= \coprod\limits_{i=1}^{k} K_{\lambda_{i}}$, where $(\lambda_1,\ldots,\lambda_k)= \lambda({\cal{B}})$. Therefore the specialization
\begin{gather*}
{\cal{L}}_{\lambda} \big\vert_{p_{ij}=0, \ (ij) \in \Gamma, \atop p_{ij}=1, \ (ij) \notin \Gamma}
\end{gather*}
is equal to the number of non isomorphic subgraphs of the complementary graph
$ K_n {\setminus} \Gamma$ which are isomorphic to the graph $K(\lambda)$.

\begin{Example} Take $n=6$, then
\begin{gather*}
{\rm Ch}_{6}^{(\beta)}= \beta^{-1} {\cal{L}}_{(6)}+x(x+1)(x+2)(x+3)(x+4) {\cal{L}}_{(1^{6})} \beta^4\\
 \hphantom{{\rm Ch}_{6}^{(\beta)}=}{}
 +
 x(x+1)(x+2)(x+3){\cal{L}}_{(2,1^{4})} \beta^3 + x(x+1)(x+2) \bigl(
{\cal{L}}_{(2^{2},1^{2})} +{\cal{L}}_{(3,1^{3})} \bigr) \beta^2 \\
\hphantom{{\rm Ch}_{6}^{(\beta)}=}{}
+ x(x+1) \bigl( {\cal{L}}_{(2^3)} +{\cal{L}}_{(3,2,1)}+{\cal{L}}_{(4,1^2)} \bigr) \beta +
{\cal{L}}_{(3^2)}+{\cal{L}}_{(4,2)}+{\cal{L}}_{(5,1)}.
\end{gather*}
Since $p_{ij}$ is
equal to either~$1$ or $0$, one can see that ${\cal{L}}_{(n)}=0$ unless graph
$\Gamma$ is a collection of~$n$ distinct points and therefore~${\cal{L}}=1$.
\end{Example}

\begin{center}
\framebox{\parbox[t]{6 in}{The chromatic polynomial of any graph is a $\Z$-linear combination of the chromatic polynomials corresponding to a set of complete graphs. }}
\end{center}

 \begin{Corollary}[formula for universal $\beta$-Tutte polynomials\footnote{It should be remembered that ${\rm Tutte}(K_{1};x,y)=1 $ and
${\rm Tutte}(K_{0};x,y)=0$, since the graph $K_1:=\{pt\}$ and graph $K_{0}= \varnothing$.\label{footnote38}}]
\begin{gather*}
 (1-y)^{n-1}{\rm Tutte}_{n}^{(\beta)}(\{p_{i,j}\};x,y) \\
 \qquad {} = \prod_{1
\le i < j \le n} p_{ij} + \sum_{\lambda \vdash n}
 {\rm Tutte} (K_{\ell(\lambda)-1}; x+y+\beta x y,0) {\cal{L}}_{\lambda}^{(\beta)}(\{p_{ij}\}) .
\end{gather*}
 \end{Corollary}

\begin{center}
\framebox{\parbox[t]{6in}{The polynomial $(1-y)^{|V(\Gamma)|-1}{\rm Tutte}(\Gamma;x,y)$ is a
$\Z[y]$-linear combination of the chromatic polynomials
${\rm Tutte}(K_m;x+y-x y,0) $ corresponding to a family of complete graphs~$\{K_m\}$.}}
\end{center}
Here $V(\Gamma)$ denotes the set of vertices of graph~$\Gamma$.
\begin{Comments}\quad
\begin{enumerate}\itemsep=0pt
\item[$(i)$] Let us write
\begin{gather*}
{\rm Ch}_{n}^{(\beta)}(\{p_{ij} \},x) = - {\cal{L}}_{(n)} \beta^{-1} + \sum_{k=1}^{n-1} a_{n}^{(k)}(\{p_{ij}) x^{k}.
\end{gather*}
It follows from Theorem~\ref{theorem4.29} that
\begin{gather*}
 a_{n}^{(k)} = \sum_{\lambda \vdash n} s(\ell(\lambda)-1,k) {\cal{L}}_{\lambda}^{(\beta)},
 \end{gather*}
where as before, $s(k,n)$ denote the Stirling numbers of the f\/irst kind,
see, e.g., footnote~\ref{footnote37}. For example,
\begin{gather*}
a(\Gamma) = a_{n}^{(1)}(\{p_{ij}\}) \bigg\vert_{p_{ij}=0,\, (ij) \in \Gamma,
\atop p_{ij}=1, \, (ij) \notin \Gamma} (\ell(\lambda)-2) ! {\cal{N}}_{\lambda}(\Gamma)
\beta^{\ell(\lambda)-2},
\end{gather*}
where ${\cal{N}}_{\lambda}(\Gamma)$ denotes the number non isomorphic
subgraphs in the complementary graph $K_n {\setminus}\Gamma$, which are isomorphic
to the graph~$ K_{n}(\lambda)$.

\item[$(ii)$] It is clear that for a general set of parameters $\{p_{ij}\}$ the
number of dif\/ferent monomials which appear in ${\cal{L}}_{\lambda}^{(\beta)}
(\{p_{ij}\})$, where partition $\lambda= \sum\limits_{j=1}^{n} j m_j$, $\lambda
\vdash n$, is equal to
\begin{gather*}
\frac{n !}{\prod\limits_{j \ge 1}(j !)^{m_{j}} m_{j} !}.
\end{gather*}
\item[$(iii)$] For general set of parameters $\{p_{ij}\}$ one can show
that the number of dif\/ferent monomials which appear in polynomial
$a_{n}^{(1)}(\{p_{ij}\})$ is equal to ${\rm Bell}(n) -1$, where ${\rm Bell}(n)$ denotes
the $n$-th Bell number, see, e.g., \cite[$A000110$]{SL}.

\item[$(iv)$] In the limit $y \longrightarrow 1$ one has $q_{ij}= m_{ij}$ and
$p_{ij}= 1-m_{ij}$.

\item[$(v)$] Let us introduce a modi\/f\/ied universal Tutte polynomial, namely,
\begin{gather*}
{\rm Tutte}(\{q_{ij}\};x,y,z):=
 (-1)^{n-1} {\rm Coef\/f}_{[t_1 \cdots t_n]} \\
 \qquad {}\times \Biggl[\Biggl( \sum_{\ell_1,\ldots,
\ell_n \atop \ell_i \in \{0,1 \}, \, \forall \,i} \prod_{1 \le i < j \le n}
(z q_{ij} y+1)^{\ell_i \ell_j}
 y^{- (\sum_{j} \ell_{j})} t_1^{\ell_{1}} \cdots t_n^{\ell_{n}} \Biggr)^{ x y} x^{-1} \Bigg] .
 \end{gather*}
\end{enumerate}
\end{Comments}
We set $\deg(q_{ij})=1$,

\begin{Proposition}\quad
\begin{enumerate}\itemsep=0pt
\item[$(a)$] ${\rm Tutte}(\{q_{ij}\};x,y,z) \in \N[\{q_{ij}] [x,y,z]$.

\item[$(b)$] Degree $n-1$ monomials of the polynomial ${\rm Tutte}(\{q_{ij}\};0,y,z)$
are in one-to-one correspondence with the set of spanning trees of the
complete graph~$K_n$. Moreover, the polynomial ${\rm Tutte}(\{q_{ij}=1,\, \forall\, i,j \};x,0,1)$ is equal to the generating function of forests on~$n$ labeled
vertices, counting according to the number of connected components,~whereas
the polynomial ${\rm Tutte}(\{q_{ij}=1,\, \forall\, i,j \};1,0,z)$ is equal to the
Hilbert polynomial of the even Orlik--Solomon algebra\footnote{Known also as Orlik--Terao algebra.}
${\rm OS}^{+}(\Gamma_n)$ associated to the type $A_{n-1}$ {\it
generic hyperplane arrangement} $\Gamma_n$, see {\rm \cite[Section~5]{PSt}} or~{\rm \cite{K}}, namely,
\begin{gather*}
{\rm Tutte}(\{q_{ij}=1,\, \forall \, i,j \};1,0,z)= {\rm Hilb}({\rm OS}^{+}(\Gamma_{n}),z)=
\sum_{{\cal{F}}} z^{|{\cal{F}} |},
\end{gather*}
where the sum runs over all forests on the vertices $\{1,\ldots,n\}$, and $|{\cal{F}}|$ denotes the number of edges of~${\cal{F}}$.

\item[$(c)$] More generally, denote by $F_{n}(x,t):= \sum\limits_{{\cal{F}}} x^{|{\cal{F}}|} t^{{\rm inv} ({\cal{F}})}$ the generating function of statistics $|{\cal{F}}|$
 and ${\rm inv} ({\cal{F}})$ on the set $F(n)$ of forests on~$n$ labeled vertices.
Recall that the symbol $|{\cal{F}}|$ denotes the number of {\it edges} in a
forest ${\cal{F}} \in F(n)$ and that ${\rm inv} ({\cal{F}})$ its inversion index\footnote{For the readers convenience we recall def\/initions of statistics
${\rm inv} ({\cal{F}})$ and the major index ${\rm maj} ({\cal{F}})$.
 Given a forest ${\cal{F}}$ on $n$ labeled nodes, one can
construct a~tree ${\cal{T}}$ by adding a~new vertex (root) connected with
the maximal vertices in the connected components of~${\cal{F}}$.

The inversion index ${\rm inv} ({\cal{F}})$ is equal to the number of pairs~$(i,j)$
such that $1 \le i < j \le n$, and the vertex labeled by~$j$ lies on the
shortest path in~${\cal{T}}$ from the vertex labeled by~$i$ to the root.

The major index ${\rm maj} ({\cal{F}})$ is equal to $\sum\limits_{x \in {\rm Des}({\cal{F}})}
h(x)$; here for any vertex $x \in {\cal{F}}$, $h(x)$ is the size of the
subtree rooted at~$x$; the descent set ${\rm Des}({\cal{F}})$ of~${\cal{F}}$
consists of the vertices $x \in {\cal{F}}$ which have
the labeling strictly greater than the labeling of its child's.}.
\end{enumerate}
\begin{Lemma} One can show that
\begin{gather*}
F_{n}(x,t)= (x t)^{n-1} {\rm Tutte}(K_n; 1+ (x t)^{-1},t-1),\\
{\rm Coef\/f}_{(x t)^{n-1}} [ F_{n}(x,t) ] =I_{n}(t),
\end{gather*}
where $I_{n}(t):= \sum_{{\cal{F}} \in {\rm Tree}(n)} t^{{\rm inv}({\cal{F}})}$ denotes the
tree inversion polynomial, see, e.g., {\rm \cite{GS,ST1}}.
\end{Lemma}
\begin{enumerate}\itemsep=0pt
\item[$(d)$] Set
\begin{gather*}
DU_n(x):= (z t )^{n-1} {\rm Hilb}\big(K_n; 1+(z t)^{-1},z-1\big) \Big |_{t:=-1 \atop z:= -x} = F_{n}(-x,-1).
\end{gather*}
One $($A.K.$)$ can show that $(n \ge 2)$
\begin{gather*}
DU_n(x) \in \N [x], \qquad DU_{n}(1) =UD_{n+1}, \qquad {\rm Coef\/f}_{x^{n-1}}[DU_n(x)] = UD_{n-1},
\end{gather*}
where $UD_n$ denote the Euler or up/down numbers associated with the
exponential ge\-ne\-ra\-ting function $\sec(x) + \tan(x)$, see\footnote{The fact that $I_n(-1)= UD_{n-1}$ is due to G.~Kreweras~\cite{Kre}.},
 e.g., {\rm \cite[$A000111$]{SL}}.
\item[$(e)$] One has
\begin{gather*}
x^{{n\choose 2}} {\rm Tutte}\big(\{q_{ij}=1,\,\forall \, i,j \};x,x^{-1}-1,1\big)=
{\rm Hilb}({\cal{A}}_{n},x),
\end{gather*}
where ${\cal{A}}_n$ denotes the algebra generated by the curvature of $2$-forms of the standard Hermitian linear bundles over the flag variety ${\cal{F}}l_n$,
 see {\rm \cite{K,PSS,SS}} or Section~{\rm \ref{section4.2.2}}, Theorem~{\rm \ref{theorem4.4}$(B)$}.

\item[$(f)$] Write ${\rm Tutte}(\{q_{ij}\};0,y,z)=\sum\limits_{k=0}^{n-1} a_{n}^{(k)}(y,z)$,
then monomials which appear in polynomial $a_{n}^{(k)}(y,z)$ are in
one-to-one correspondence with the set of labeled graphs with n nodes having
exactly~$k$ connected components.
\item[$(g)$] One has ${\rm Tutte}((\{q_{ij}\};x,-1,1)= {\rm Tutte}(\{q_{ij}\},x+1,0)$.

\item[$(h)$] Recurrence relations for polynomials $F_n(x,t)$, cf.~{\rm \cite{Kre}},
\begin{gather*}
F_0(x,t)=F_1(x,t)=1, \qquad F_{n+1}(x,t)= \sum_{k=0}^{n} {n \choose k} (xt)^{k} I_k(t) F_{n-k}(x,t).
\end{gather*}
\end{enumerate}
\end{Proposition}

\begin{Example}
Take $n=5$, then
\begin{gather*}
{\rm Tutte}(K_5;x,y)= (0,6,11,6,1)+(6,20, 10) y+ 15(1,1) y^2+ 5(3,1) y^3 +10 y^4+4 y^5+y^6,\\
F_5(-x,- 1)= (1,10,25,20,5).
\end{gather*}
Write $F_n(x,t)= {\tilde{F}}_{n}(u,t) \big |_{u=xt}$,
then
\begin{gather*}
{\tilde{F}}_{5}(u,t) = 1+10 u+ u^2 (35+10 t)+ u^3 (50,40,15,5)_{t}+
 u^4 (24,36,30,20,10,4,1)_{t},\\
 {\tilde{F}}_{n}(u,0)= \prod_{j=1}^{n-1}(1+ j~u),\\
{\rm Hilb}({\cal A}_5,t)=(1,4,10,20,35,51,64,60,35,10,1)_{t},\\
{\rm Hilb}({\rm OS}^{+}(\Gamma_5),t)= (1,10,45,110,125)_{t}.
\end{gather*}
\end{Example}

\begin{Exercises}\label{exercises4.2}\quad
\begin{enumerate}\itemsep=0pt
\item[(1)] Assume that $\ell_{ij}=\ell$ for all $1 \le i < j \le r$. Based on
the above formula for the exponential generating function for the Tutte
polynomials of the complete multipartite graphs $K_{n_1,\ldots,n_r}$,
{\it deduce} the following well-known formula
\begin{gather*}
{\rm Tutte}\big(K_{n_1,\ldots,n_r}^{(\ell)},1,1\big) = \ell^{N-1} N^{r-2} \prod_{j=1}^{r} (N-n_j)^{n_j-1},
\end{gather*}
where $N:=n_1+\cdots + n_r$. It is well-known that the number
${\rm Tutte}(\Gamma,1,1)$ is equal to the number of {\it spanning trees} of a~connected graph $\Gamma$.

\item[(2)] Take $r=3$ and let $n_1$, $n_2$, $n_3$ and $\ell_{12}$, $\ell_{13}$, $\ell_{23}$ be positive integers. Set $N:= \ell_{12} \ell_{13} n_1 + \ell_{12}
\ell_{23} n_2 +\ell_{13} \ell_{23} n_3$. Show that
\begin{gather*}
{\rm Tutte}\big(K_{n_1,n_2,n_3}^{\ell_1,\ell_2,\ell_3}, 1,1\big) =
 N (\ell_{12} n_2\!+\ell_{13} n_3)^{n_{1}-1} (\ell_{12} n_1\!+\ell_{13} n_3)^{n_{2}-1} (\ell_{13} n_1\!+ \ell_{23} n_2)^{n_{3}-1}.
 \end{gather*}

\item[(3)] Let $r \ge 2$, consider weighted complete multipartite graph
$K_{ \underbrace{n,\ldots,n}_{r}}^{({\boldsymbol{\ell}})}$, where ${\boldsymbol{\ell}}= (\ell_{ij})$ such that $\ell_{1,j}=\ell$, $j=1,\ldots, r$ and
$\ell_{ij}= k$, $2 \le i < j \le r$. Show that
\begin{gather*}
{\rm Tutte}\big(K_{ \underbrace{n,\ldots,n}_{r}}^{({\boldsymbol{\ell}})},1,1\big) = k^{n}
(r-1)^{n-1} ((r-1) \ell + k )^{r-2} ((r-2) \ell +k )^{(r-1)(n-1)} n^{nr-1}.
\end{gather*}
\end{enumerate}
\end{Exercises}

Let $\Gamma_{n}(*)$ be a spanning star subgraph of the complete
graph~$K_n$. For example, one can take for a graph $\Gamma_{n}(*)$ the
subgraph $K_{1,n-1}$ with the set of vertices $V:=\{1,2,\ldots,n \}$
and that of edges $E:=\{(i,n),\, i=1,\ldots,n-1 \}$. The algebra
$3T_n^{(0)}(K_{1,n-1})$ can be treated as a~``noncommutative analog''
of the projective space $\mathbb{P}^{n-1}$.
We have $\theta_1 = u_{12}+u_{13}+\cdots+u_{1n}$. It is not dif\/f\/icult
to see that
${\rm Hilb}(3T_n^{(0)}(K_{1,n-1})^{ab},t)=(1+t)^{n-1}$, and $\theta_1^n=0$.
Let us observe that ${\rm Chrom}(\Gamma_n(\star),t)= t (t-1)^{n-1}$.

\begin{Problem} \label{problem4.1} Compute the Hilbert series of the algebra
$3T_n^{(0)}(K_{n_1,\ldots,n_r})$.
\end{Problem}

The f\/irst non-trivial case is that of {\it projective space}, i.e.,
the case $r=2$, $n_1=1$, $n_2=5$.

On the other hand, if $\Gamma_n= \{(1,2) \rightarrow (2,3) \rightarrow \cdots
\rightarrow (n-1,n) \}$ is the Dynkin graph of type~$A_{n-1}$, then the
algebra $3T_n^{(0)}(\Gamma_n)$ is isomorphic to the nil-Coxeter algebra of
type $A_{n-1}$, and if $\Gamma_n^{\rm (af\/f)}= \{ (1,2) \rightarrow (2,3)
\rightarrow \cdots \rightarrow (n-1,n) \rightarrow -(1,n) \}$ is the Dynkin
graph of type~$A_{n-1}^{(1)}$, i.e., {\it a~cycle}, then the algebra
$3T_n^{(0)}(\Gamma_n^{\rm (af\/f)})$ is isomorphic to a certain quotient of the
af\/f\/ine nil-Coxeter algebra of type $A_{n-1}^{(1)}$ by the two-sided ideal
which can be described explicitly~\cite{K}. Moreover~\cite{K},
\begin{gather*}
 {\rm Hilb}\big(3T_n^{0)}\big(\Gamma^{\rm (af\/f)}\big),t\big)= [n]_t \prod_{j=1}^{n-1} [j(n-j)]_t,
 \end{gather*}
see Theorem~\ref{theorem4.1}. Therefore, the dimension $\dim(3T^{(0)}(\Gamma^{\rm af\/f}))$ is
equal to $n ! (n-1) !$ and is equal also, as it was pointed out in
Section~\ref{section4.1.1}, to the number of (directed)
Hamiltonian cycles in the complete bipartite graph~$K_{n,n}$, see \cite[$A010790$]{SL}.

It is not dif\/f\/icult to see that
\begin{gather*}
{\rm Hilb}\big(3T_n^{(0)}(\Gamma_n)^{ab},t\big)=(t+1)^{n-1},
\qquad
{\rm Hilb}\big(3T^{(0)}\big(\Gamma_n^{\rm af\/f}\big)^{ab},t\big)=t^{-1} \big((t+1)^{n}-t-1\big),
\end{gather*}
whereas
\begin{gather*}
{\rm Chrom}(\Gamma_n,t)= t (t-1)^{n-1}, \qquad {\rm Chrom}\big (\Gamma_n^{\rm af\/f},t\big)=(t-1)^n+(-1)^n (t-1).
\end{gather*}
\begin{Exercise}\label{Exercise4.3}
 Let $K_{n_1,\ldots,n_r}$ be complete multipartite graph,
$N:= n_1+\cdots+n_r$.
Show that\footnote{It should be remembered that to abuse of notation, the complete
graph~$K_n$, by def\/inition, is equal to the complete multipartite graph
$K(\underbrace{(1,\ldots,1)}_{n})$, whereas the graph $K_{(n)}$ is a~collection of~$n$ distinct points.}
\begin{gather*}
{\rm Hilb}(3T_N(K_{n_1,\ldots,n_r}),t) = \frac{\prod\limits_{j=1}^{r} \prod\limits_{a=1}^{n_{j}-1} (1-a t)}{\prod\limits_{j=1}^{N-1}(1- jt)}.
\end{gather*}
\end{Exercise}

\subsubsection[Quasi-classical and associative classical Yang--Baxter algebras of type $B_n$]{Quasi-classical and associative classical Yang--Baxter algebras of type $\boldsymbol{B_n}$}\label{section4.1.3}

In this section we introduce an analogue of the algebra $3T_n(\beta)$ for the
classical root systems.

\begin{Definition}\label{definition4.6}\quad
\begin{enumerate}\itemsep=0pt
\item[(A)] {\it The quasi-classical Yang--Baxter algebra
$\widehat{{\rm ACYB}(B_n)}$ of type~$B_n$} is an associative algebra with the set of
generators $\{x_{ij},\,y_{ij},\,z_i,\,1 \le i \not= j \le n \}$ subject to the set
 of def\/ining relations
\begin{enumerate}\itemsep=0pt
\item[(1)] $x_{ij}+x_{ij}=0$, $y_{ij}=y_{ji}$ if $i \not= j$,

\item[(2)] $z_i z_j=z_j z_i$,

\item[(3)] $x_{ij} x_{kl}=x_{kl} x_{ij}$, $x_{ij} y_{kl}=y_{kl} x_{ij}$,
 $y_{ij}y_{kl}=y_{kl}y_{ij}$ if $i$, $j$, $k$, $l$ are distinct,

\item[(4)] $z_ix_{kl}=x_{kl}z_i$, $z_iy_{kl}=y_{kl}z_i$ if $i \not= k,l$,

\item[(5)] {\it three term relations}:
\begin{gather*}
 x_{ij} x_{jk}=x_{ik} x_{ij}+x_{jk} x_{ik} -\beta x_{ik}, \qquad x_{ij} y_{jk}
=y_{ik} x_{ij}+y_{jk} y_{ik}-\beta y_{ik},
\\
x_{ik} y_{jk}=y_{jk} y_{ij}+y_{ij} x_{ik}+\beta y_{ij},\qquad y_{ik}x_{jk}
=x_{jk}y_{ij}+y_{ij}y_{ik}+\beta y_{ij}
\end{gather*}
if $1 \le i < j < k \le n$,

\item[(6)] {\it four term relations}:
\begin{gather*}
x_{ij} z_j=z_{i} x_{ij}+y_{ij} z_i+z_j y_{ij}-\beta z_{i}
\end{gather*}
if $i < j$.
\end{enumerate}

\item[(B)]
{\it The associative classical Yang--Baxter algebra ${\rm ACYB}(B_n)$ of type~$B_n$} is the special case $\beta=0$ of the algebra
$\widehat{{\rm ACYB}(B_n)}$.
\end{enumerate}
\end{Definition}

\begin{Comments} \label{comments4.2}\quad
\begin{itemize}\itemsep=0pt
\item In the case $\beta=0$ the algebra ${\rm ACYB}(B_n)$ has a
rational representation
\begin{gather*}
 x_{ij} \longrightarrow (x_i-x_j)^{-1},\qquad
y_{ij} \longrightarrow (x_i+x_j)^{-1},\qquad
z_i \longrightarrow x_i^{-1}.
\end{gather*}

\item In the case $\beta=1$ the algebra $\widehat{{\rm ACYB}(B_n)}$ has
 a ``trigonometric'' representation
\begin{gather*}
 x_{ij} \longrightarrow \big(1-q^{x_i-x_j}\big)^{-1},\qquad
y_{ij} \longrightarrow \big(1-q^{x_i+x_j}\big)^{-1},\qquad
z_i \longrightarrow \big(1+q^{x_i}\big)\big(1-q^{x_i}\big)^{-1}.
\end{gather*}
\end{itemize}
\end{Comments}

\begin{Definition}\label{definition4.7} {\it The bracket algebra ${\cal E}(B_n)$ of type $B_n$}
is an associative algebra with the set of generators
$\{x_{ij},\, y_{ij},\, z_i,\, 1 \le i \not= j \le n \}$ subject to the set of
relations (1)--(6) listed in Def\/inition~\ref{definition4.6}, and the additional relations
\begin{alignat*}{3}
& (5a) \quad && x_{jk} x_{ij}=x_{ij} x_{ik}+x_{ik} x_{jk} -\beta x_{ik}, \qquad y_{jk} x_{ij}
=x_{ij} y_{ik}+y_{ik} y_{jk}-\beta y_{ik},&\\
&&& y_{jk} x_{ik}=y_{ij} y_{jk}+x_{ik} y_{ij}+\beta y_{ij}, \qquad x_{jk} y_{ik} =y_{ij} x_{jk}+y_{ik} y_{ij}+\beta y_{ij}&
\end{alignat*}
if $1 \le i < j < k \le n$,
\begin{alignat*}{3}
& (6a) \quad && z_j x_{ij}=x_{ij} z_{i}+z_i y_{ij}+y_{ij} z_j-\beta z_{i}&
\end{alignat*}
if $i < j$.
\end{Definition}

\begin{Definition}\label{definition4.8} The quasi-classical Yang--Baxter algebra
$\widehat{{\rm ACYB}(D_n)}$ of type $D_n$, as well as the algebras
${\rm ACYB}(D_n)$~and~${\cal E}(D_n)$ are def\/ined by putting
$z_i=0$, $i=1,\ldots, n$, in the corresponding $B_n$-versions of
algebras in question.
\end{Definition}

\begin{Conjecture}\label{conjecture4.3}
 The both algebras ${\cal E}(B_n)$ and ${\cal E}(D_n) $ are
Koszul, and
\begin{gather*}
{\rm Hilb}({\cal E}(B_n),t)= \left(\prod_{j=1}^{n}(1-(2j-1)t) \right)^{-1};
\end{gather*}
 if $n \ge 4$
 \begin{gather*}
 {\rm Hilb}({\cal E}(D_n),t)=
\left(\prod_{j=1}^{n-1}(1-2j t) \right)^{-1}.
\end{gather*}
\end{Conjecture}

\begin{Example}\label{example4.3}
\begin{gather*}
{\rm Hilb}({\rm ACYB}(B_2),t)=\big(1-4t+2t^2\big)^{-1}, \\
{\rm Hilb}({\rm ACYB}(B_3),t)=\big(1-9t+16t^2-4t^3\big)^{-1}, \\
{\rm Hilb}({\rm ACYB}(B_4),t)=\big(1-16t+64t^2-60t^3+9t^4\big)^{-1} , \\
{\rm Hilb}({\rm ACYB}(D_4),t)=\big(1-12t+18t^2-4t^3\big)^{-1}.
\end{gather*}
However,
\begin{gather*}
{\rm Hilb}({\rm ACYB}(B_5),t)=\big(1-25t+180t^2-400t^3+221t^4-31t^5\big)^{-1}.
\end{gather*}
\end{Example}

Let us introduce the following Coxeter type elements
\begin{gather}\label{equation4.2}
 h_{B_n}:= \prod_{a=1}^{n-1} x_{a,a+1} z_n \in {\cal E}(B_n) \qquad \text{and} \qquad
h_{D_n}:= \prod_{a=1}^{n-1} x_{a,a+1} y_{n-1,n} \in {\cal E}(D_n).
\end{gather}
Let us bring the element $h_{B_n}$ (resp.~$h_{D_n}$) to the reduced form in
the algebra ${\cal E}(B_n)$ that is, let us consecutively apply the def\/ining
rela\-tions~(1)--(6), (5a), (6a) to the element~$h_{B_n}$ (resp.\ apply to~$h_{D_n}$ the def\/ining relations for algebra ${\cal E}(D_n)$) in any order
until unable to do so. Denote the resulting (noncommutative) polynomial
by $P_{B_n}(x_{ij},y_{ij},z)$ (resp.\ $P_{D_n}(x_{ij},y_{ij})$). In
principal, this polynomial itself can depend on the order in which the
rela\-tions~(1)--(6), (5a), (6a) are applied.

{\samepage
\begin{Conjecture}[cf.~\protect{\cite[Exercise~8.C5(c)]{ST3}}] \label{conjecture4.4}\quad
\begin{enumerate}\itemsep=0pt
\item[$(1)$] Apart from applying the commutativity relations $(1)$--$(4)$, the
polynomial $P_{B_n}(x_{ij},y_{ij},z)$ $($resp.\ $P_{D_n}(x_{ij},y_{ij}))$
does not depend on the order in which the defining relations have been applied.

\item[$(2)$] Define polynomial $P_{B_n}(s,r,t)$ $($resp.\ $P_{D_n}(s,r))$ to be the
the image of that $P_{B_n}(x_{ij},y_{ij},z)$ $($resp.\ $P_{D_n}(x_{ij},y_{ij}))$
 under the specialization
\begin{gather*}
x_{ij} \longrightarrow s, \qquad y_{ij} \longrightarrow r,\qquad
z_i \longrightarrow t.
\end{gather*}
Then
$P_{B_n}(1,1,1)={1 \over 2} {2n \choose n}={1 \over 2} {\rm Cat}_{B_n}$.
\end{enumerate}
\end{Conjecture}}

Note that $P_{B_n}(1,0,1)={\rm Cat}_{A_{n-1}}$.

\begin{Problem} \label{problem4.2}
Investigate the $B_n$ and $D_n$ types reduced polynomials
corresponding to the Co\-xeter elements~\eqref{equation4.2}, and the reduced polynomials
corresponding to the longest elements
\begin{gather*}
 w_{B_{n}}:= \prod_{J=1}^{n} z_j \left(\prod_{1 \le i < j \le n} x_{ij} y_{ij}\right),\qquad
w_{D_{n}}= \prod_{1 \le i < j \le n} x_{ij} y_{ij}.
\end{gather*}
\end{Problem}

\subsection{Super analogue of 6-term relations and classical
Yang--Baxter algebras}\label{section4.2}

\subsubsection[Six term relations algebra $6T_n$, its quadratic dual $(6T_n)^{!}$, and algebra $6HT_n$]{Six term relations algebra $\boldsymbol{6T_n}$, its quadratic dual $\boldsymbol{(6T_n)^{!}}$, and algebra $\boldsymbol{6HT_n}$}\label{section4.2.1}

\begin{Definition}\label{definition4.9} {\it The $6$ term relations algebra} $6T_{n}$ is an associative
algebra (say over $\Q$) with the set of generators
$\{ r_{i,j}, 1 \le i \not= j < n \}$, subject to the following relations:
\begin{enumerate}\itemsep=0pt
 \item[1)] $r_{i,j}$ and $r_{k,l}$ commute if $\{i,j\} \cap \{k,l\}=\varnothing$,

\item[2)] {\it unitarity condition}: $r_{ij}+r_{ji}=0 $,

\item[3)] {\it classical Yang--Baxter relations}:
$[r_{ij},r_{ik}+r_{jk}]+[r_{ik},r_{jk}] =0$ if~$i$, $j$, $k$~are distinct.
\end{enumerate}
\end{Definition}

We denote by ${\rm CYB}_n$, named by {\it classical Yang--Baxter algebra}, an
associative algebra over $\Q$ generated by elements $\{r_{ij},\, 1 \le i \not= j \le n \}$ subject to relations~1) and~3).

Note that the algebra $6T_{n}$ is given by ${n \choose 2}$ generators
and ${n \choose 3}+3 {n \choose 4}$ quadratic relations.
\begin{Definition}\label{definition4.10} Def\/ine {\it Dunkl elements} in the algebra $6T_n$ to be
\begin{gather*}
\theta_i = \sum_{j \not=i} r_{ij}, \qquad i=1,\ldots,n .
\end{gather*}
\end{Definition}

It easy to see that the Dunkl elements $\{ \theta_i \}_{1 \le i \le n}$
generate a commutative subalgebra in the algebra~$6T_n$.

\begin{Example}[some ``rational and trigonometric'' representations
of the algebra~$6T_n$] \label{example4.4}
Let $A=U({\mathfrak{sl}}(2))$ be the universal enveloping algebra of the
Lie algebra~${\mathfrak{sl}}(2)$. Recall that the algebra ${\mathfrak{sl}}(2)$ is spanned by the
elements $e$, $f$, $h$, such that $[h,e]=2e$, $[h,f]=-2f$, $[e,f]=h$.

Let's search for solutions to the ${\rm CYBE}$ in the form
\begin{gather*}
r_{i,j}=
a(u_i,u_j) h \otimes h +b(u_i,u_j) e \otimes f+c(u_i,u_j) f \otimes e,
\end{gather*}
where $a(u,v),b(u,v) \not= 0,c(u,v) \not= 0$ are meromorphic functions
of the variables $(u,v) \in \C^2$, def\/ined in a neighborhood of
$(0,0)$, taking values in $A \otimes A$. Let $a_{ij}:=a(u_i,u_j)$
(resp.\ $b_{ij}:=b(u_i,u_j)$, $c_{ij}:=c(u_i,u_j)$).

\begin{Lemma}\label{lemma4.2} The elements $r_{i,j}:=a_{ij} h \otimes h
+b_{ij} e \otimes f+c_{ij} f \otimes e$ satisfy CYBE iff
\begin{gather*}
b_{ij} b_{jk} c_{ik}
=c_{ij} c_{jk} b_{ik} \qquad \text{and}\qquad 4 a_{ik}=b_{ij} b_{jk}/b_{ik}-
b_{ik} c_{jk}/b_{ij}-b_{ik} c_{ij}/b_{jk}
\end{gather*}
for $1 \le i < j < k \le n$.
\end{Lemma}

It is not hard to see that
\begin{itemize}\itemsep=0pt
\item there are three rational solutions:
\begin{gather*}
 r_1(u,v)= {1/2 h \otimes h+ e \otimes f+f \otimes e \over u-v}, \\
 r_2(u,v)= {u+v \over 4(u-v)} h \otimes h+{u \over u-v} e \otimes f+
{v \over u-v}f \otimes e,
\end{gather*}
and $r_3(u,v):=-r_2(v,u)$,

\item there is a trigonometric solution
\begin{gather*}
r_{\rm trig}(u,v)={1 \over 4} {q^{2u}+q^{2v} \over q^{2u}-q^{2v}} h \otimes
h+{q^{u+v} \over q^{2u}-q^{2v}} \big(e \otimes f+f \otimes e \big).
\end{gather*}
\end{itemize}

Notice that the {\it Dunkl element}
$\theta_j:= \sum\limits_{a \not= j} r_{\rm trig}(u_a,u_j)$
corresponds to the truncated (or level~$0$)
{\it trigonometric Knizhnik--Zamolodchikov operator}.

In fact, the ``${\mathfrak{sl}}_{n}$-Casimir element''
\begin{gather*}
\Omega=
{1 \over 2} \left(\sum_{i=1}^{n} E_{ii} \otimes E_{ii} \right) +
\sum_{1 \le i < j \le n} E_{ij} \otimes E_{ji}
\end{gather*}
 satisf\/ies the 4-term relations
\begin{gather*}
[\Omega_{12},\Omega_{13}+\Omega_{23}]=0
=[\Omega_{12}+\Omega_{13},\Omega_{23}],
\end{gather*}
and the elements $r_{ij}:= {\Omega_{ij } \over u_i-u_j }$,~
$1 \le i < j \le n$, satisfy the classical Yang--Baxter relations.

Recall that the set
$ \{ E_{ij}:=(\delta_{ik} \delta_{jl})_{1 \le k,l \le n},
\, 1 \le i,j \le n \}$, stands for the standard basis of the
 algebra $\operatorname{Mat}(n,\R)$.
\end{Example}

\begin{Definition}\label{definition4.11}
Denote by $6T_{n}^{(0)}$ the quotient of the algebra $6T_{n}$ by the
(two-sided) ideal generated by the set of elements
$\{r_{i,j}^2,\,1 \le i < j \le n \}$.
\end{Definition}

More generally, let $\{\beta, q_{ij},\,
1 \le i < j \le n \}$ be a set of parameters.
Let $R:=\Q[\beta][q_{ij}^{\pm 1}]$.

\begin{Definition}\label{definition4.12} Denote by $6HT_n$ the quotient of the algebra
$6T_{n} \otimes R$ by the
(two-sided) ideal generated by the set of elements
$\{r_{i,j}^2-\beta r_{i,j}-q_{ij}, \, 1 \le i < j \le n \}$.
\end{Definition}

All these algebras are naturally graded, with $\deg (r_{i,j})=1$,
$\deg(\beta)=1$, $\deg(q_{ij})=2$.
 It is clear that the algebra $6T^{(0)}_{n}$ can be considered as the
inf\/initesimal deformation $R_{i,j}:=1+\epsilon r_{i,j}$,
 $\epsilon \longrightarrow 0$, of the Yang--Baxter group ${\rm YB}_{n}$.

For the reader convenience we recall the def\/inition of the {\it
Yang--Baxter} group.

\begin{Definition}\label{definition4.13} The Yang--Baxter group~${\rm YB}_n$ is a group generated by
elements $\{R_{ij}^{\pm 1}, \,1 \le i < j \le n \}$, subject to the set
of def\/ining relations
\begin{itemize}\itemsep=0pt
\item $R_{ij} R_{kl}=R_{kl} R_{ij}$ if $i$, $j$, $k$, $l$ are distinct,

\item {\it quantum Yang--Baxter relations}:
\begin{gather*}
 R_{ij} R_{ik} R_{jk}=R_{jk} R_{ik} R_{ij} \qquad \text{if} \quad 1 \le i < j < k
\le n.
\end{gather*}
\end{itemize}
\end{Definition}

\begin{Corollary}\label{corollary4.3} Define $h_{ij}=1+r_{ij} \in 6HT_{n}$. Then the
following relations in the algebra $6HT_n$ are satisfied:
\begin{enumerate}\itemsep=0pt
\item[$(1)$] $r_{ij} r_{ik} r_{jk}=r_{jk} r_{ik} r_{ij}$ for all pairwise
distinct $i$, $j$ and $k$;

\item[$(2)$] Yang--Baxter relations: $h_{ij} h_{ik} h_{jk}=h_{jk} h_{ik} h_{ij}$
 if $1 \le i < j < k \le n$.
 \end{enumerate}
\end{Corollary}

Note, the item $(1)$ includes three relations in fact.
\begin{Proposition} \label{proposition4.4}\quad
\begin{enumerate}\itemsep=0pt
\item[$(1)$] The quadratic dual $(6T_n)^{!}$ of the algebra $6T_n$ is a~quadratic algebra generated by the elements $\{t_{i,j},
\, 1 \le i < j \le n \}$ subject to the set of relations
\begin{enumerate}\itemsep=0pt
\item[$(i)$] $t_{i,j}^2=0$ for all $i \not= j$;

\item[$(ii)$] anticommutativity: $t_{ij} t_{k,l}+t_{k,l} t_{i,j}=0$
for all $i \not= j$ and $k \not= l$;

\item[$(iii)$] $t_{i,j} t_{i,k}=t_{i,k} t_{j,k}=t_{i,j} t_{j,k}$ if $i$, $j$, $k$ are
distinct.
\end{enumerate}

\item[$(2)$] The quadratic dual $(6T_n^{(0)})^{!}$ of the algebra $6T_n^{(0)}$ is a quadratic algebra with generators $\{t_{i,j},\, 1 \le i < j \le n \}$
subject to the relations $(ii)$--$(iii)$ above only.
\end{enumerate}
\end{Proposition}

\subsubsection[Algebras $6T_n^{(0)}$ and $6T_n^{\bigstar}$]{Algebras $\boldsymbol{6T_n^{(0)}}$ and $\boldsymbol{6T_n^{\bigstar}}$}\label{section4.2.2}

We are reminded that the algebra $6T_n^{(0)}$ is the quotient of the six term
relation algebra $6T_n$ by the two-sided ideal generated by the elements
$\{r_{ij} \}_{1 \le i < j \le n}$. Important {\it consequence} of the
classical Yang--Baxter relations and relations $r_{ij}^2=0$, $\forall\, i \not=j$,
 is that the both additive Dunkl elements $\{ \theta_{i} \}_{1 \le i \le n}$
and multiplicative ones
\begin{gather*}
\left\{ \Theta_i = \prod_{\atop a=i-1}^{1} h_{ai}^{-1}
\prod_{\atop a=i+1}^{n} h_{ia} \right\}_{1\le i \le n}
\end{gather*}
 generate commutative
subalgebras in the algebra~$6T_n^{(0)}$ (and in the algebra $6T_n$ as well),
see Corollary~\ref{corollary4.3}. The problem we are interested in, is to describe
commutative subalgebras generated by additive (resp.\ multiplicative) Dunkl
elements in the algebra~$6T_{n}^{(0)}$. Notice that the subalgebra generated
by additive Dunkl elements in the abelianization\footnote{See, e.g., \url{http://mathworld.wolfram.com/Abelianization.html}.}
of the algebra $6T_n{(0)}$ has been studied in~\cite{PSS, SS}. In order
to state the result from~\cite{PSS} we need, let us introduce a bit of
notation. As before, let~${\cal{F}}l_n$ denotes the complete f\/lag variety, and
denote by~${\cal{A}}_n$ the algebra generated by the curvature of $2$-forms of
 the standard Hermitian linear bundles over the f\/lag variety ${\cal{F}}l_n$,
see, e.g.,~\cite{PSS}. Finally, denote by~$I_n$ the ideal in the ring of
polynomials $\Z[t_1,\ldots,t_n]$ generated by the set of elements
\begin{gather*}
(t_{i_{1}}+\cdots+t_{i_{k}})^{k(n-k)+1}
\end{gather*}
for all sequences of indices $1 \le i_1 < i_2 < \cdots < i_k \le n$, $k=1,
\ldots,n$.

\begin{Theorem}[\cite{PSS,SS}]\label{theorem4.4}\quad
\begin{enumerate}\itemsep=0pt
\item[$(A)$] There exists a natural isomorphism
\begin{gather*}
{\cal{A}}_n \longrightarrow \Z[t_1,\ldots,t_n] /I_n,
\end{gather*}
\item[$(B)$]
\begin{gather*}
{\rm Hilb}({\cal{A}}_n,t) =t^{{n \choose 2}} {\rm Tutte}\big(K_n,1+t,t^{-1}\big).
\end{gather*}
\end{enumerate}
\end{Theorem}

Therefore the dimension of ${\cal{A}}_n$ (as a $\Z$-vector space) is equal to
the number ${\cal{F}}(n)$ of forests on~$n$ labeled vertices. It is well-known that
\begin{gather*}
\sum_{n \ge 1} {\cal{F}}(n) \frac{x^n}{n !} =\exp \left( \sum_{n \ge 1}
n^{n-1} \frac{x^n}{n !} \right) -1.
\end{gather*}
For example,
\begin{gather*} {\rm Hilb}({\cal A}_3,t)=(1,2,3,1), \qquad {\rm Hilb}({\cal A}_4,t)=(1,3,6,10,11,6,1),\\
{\rm Hilb}({\cal A}_5,t)=(1,4,10,20,35,51,64,60,35,10,1),\\
{\rm Hilb}({\cal A}_6,t)=(1,5,15,35,70,126,204,300,405,490,511,424,245,85,15,1).
\end{gather*}

\begin{Problem} \label{problem4.3}
Describe subalgebra in $(6T_n^{(0)})^{ab}$ generated by the
 multiplicative Dunkl elements $\{\Theta_i \}_{1 \le i \le n}$.
\end{Problem}

On the other hand, the commutative subalgebra ${\cal{B}}_n$ generated by
the additive Dunkl elements in the algebra $6T_n^{(0)}$, $n \ge 3$, has
{\it infinite} dimension. For example,
\begin{gather*}
{\cal{B}}_3 \cong \Z[x,y]/ \langle xy(x+y) \rangle,
\end{gather*}
and the Dunkl elements $\theta_j^{(3)}$, $j=1,2,3$, have inf\/inite order.

\begin{Definition}\label{definition4.14}
 Def\/ine algebra $6T_n^{\bigstar}$ to be the quotient of that
$6T_n^{(0)}$ by the two-sided ideal generated by the set of ``cyclic
relations''
\begin{gather*}
 \sum_{j=2}^{m} \prod_{a=j}^{m} r_{i_{1},i_{a}} \prod_{a=2}^{j}
r_{i_{1},i_{a}} = 0
\end{gather*}
for all sequences $\{1 \le i_1,i_2,\ldots,i_m \le n \}$ of pairwise distinct
integers, and all integers $2 \le m \le n$.
\end{Definition}

For example,
\begin{itemize}\itemsep=0pt
\item ${\rm Hilb}(6T_{3}^{\bigstar},t)=(1,3,5,4,1)=(1+t)(1,2,3,1)$,

\item subalgebra (over~$\Z$) in the algebra $6T_{3}^{\bigstar}$
generated by Dunkl elements $\theta_{1}$ and $\theta_{2}$ has the Hilbert
polynomial equal to $(1,2,3,1)$, and the following presentation:
$\Z [x,y]/I_{3}$, where $I_{3}$ denotes the ideal in $\Z[x,y]$ generated
by $x^3$, $y^3$, and $(x+y)^3$,

\item ${\rm Hilb}(6T_{4}^{\bigstar},t) = (1,6,23,65,134,164,111,43,11,1)_{t}$.
\end{itemize}

As a consequence of the cyclic relations, one can check that for any integer
$n \ge 2$ the $n$-th power of the additive Dunkl element~$\theta_i$ is equal
to zero in the algebra $6T_{n}^{\bigstar}$ for all $i=1,\ldots,n$. Therefore,
 the Dunkl elements generate a f\/inite-dimensional commutative subalgebra in
the algebra $6T_{n}^{\bigstar}$. There exist natural homomorphisms
\begin{gather}\label{equation4.3}
 6T_n^{\bigstar} \longrightarrow 3T_n^{(0)}, \qquad
\begin{CD}
{\cal{B}}_n @>\tilde{\pi}>> {\cal{A}}_n \longrightarrow H^{*}({\cal{F}}l_n,\Z)
\end{CD}
\end{gather}
The f\/irst and third arrows in~\eqref{equation4.3} are epimorphism. We expect that the
map~$\tilde{\pi}$ is also epimorphism\footnote{Contrary to the case of the map ${\rm pr}_{n}\colon \Z[\theta_1,\ldots,
\theta_n] \longrightarrow (3T_n{(0)})^{ab}$, where the image ${\rm Im}({\rm pr}_n)$ has
dimension equals to the number of permutations in~${\mathbb{S}}_n$ with~$(n-1)$ inversions see~\cite[$A001892$]{SL}.}, and looking for a description of the kernel $\operatorname{ker}(\tilde{\pi})$.

\begin{Comments} \label{comments4.3}\quad
\begin{itemize}\itemsep=0pt
\item Let us denote by ${\cal{B}}_n^{\rm mult}$ and ${\cal{A}}_n^{\rm mult}$ the
subalgebras generated by {\it multiplicative} Dunkl elements in the algebras
$6T_n^{(0)}$ and $\big(6T_n^{(0)}\big)^{ab}$ correspondingly. One can def\/ine a~sequence of maps
\begin{gather}\label{equation4.4}
{\cal{B}}_n^{\rm mult} \longrightarrow {\cal{A}}_n^{\rm mult}
\overset{\tilde{\phi}}{\longrightarrow}
K^{*}({\cal{F}}l_n),
\end{gather}
which is a $K$-theoretic analog of that~\eqref{equation4.3}. It is an interesting problem to f\/ind a geometric interpretation of the algebra ${\cal{A}}_n^{\rm mult}$ and the
map~$\tilde{\phi}$.

\item ``Quantization''. Let $\beta$ and $\{ q_{ij}=q_{ji}, 1 \le i,j \le n \}$ be parameters.
\end{itemize}

\begin{Definition} \label{definition4.15}
Def\/ine algebra $6HT_n$ to be the quotient of the algebra $6T_n$ by
 the two sided ideal generated by the elements $\{r_{ij}^2 - \beta r_{ij} -q_{ij} \}_{1 \le i,j \le n}$.
\end{Definition}

\begin{Lemma} \label{lemma4.3}
The both additive $\{\theta_i \}_{1 \le i \le n}$ and
multiplicative $\{\Theta_{i} \}_{1 \le i \le n}$ Dunkl elements generate
commutative subalgebras in the algebra~$6HT_n$.
\end{Lemma}

Therefore one can def\/ine algebras $6{\cal{HB}}_n$ and $6{\cal{HA}}_n$ which
are a ``quantum deformation'' of algebras~${\cal{B}}_n$ and~${\cal{A}}_n$
respectively. We {\it expect} that in the case $\beta=0$ and a special choice
of ``arithmetic parameters'' $\{ q_{ij} \}$, the algebra ${\cal{HA}}_n$ is
connected with the arithmetic Schubert and Grothendieck calculi, cf.~\cite{SS, T}.
Moreover, for a ``general'' set of parameters $\{q_{ij}\}_{1 \le i,j \le n}$ and $\beta=0$, we {\it expect} an existence of a natural homomorphism
\begin{gather*}
 {\cal{HA}}_n^{\rm mult} \longrightarrow {\cal{QK}}^{*}({{\cal{F}}l}_{n}),
 \end{gather*}
where ${\cal{QK}}^{*}({{\cal{F}}l}_{n})$ denotes a {\it multiparameter
quantum deformation} of the $K$-theory ring $K^{*}({\cal{F}}l_n)$ \cite{K,KM}; see also Section~\ref{section3.1}. Thus, we treat the algebra
${\cal{HA}}_n^{\rm mult}$ as the $K$-theory version of a~multiparameter quantum
deformation of the algebra ${\cal{A}}_n^{\rm mult}$ which is generated by the
curvature of $2$-forms of the Hermitian linear bundles over the f\/lag variety~${\cal{F}}l_n$.
\begin{itemize}\itemsep=0pt
\item One can def\/ine an analogue of the algebras $6T_n^{(0)}$,
$6HT_n$ etc., denoted by $6T(\Gamma)$ etc., for any subgraph
$\Gamma \subset K_n$ of the complete graph~$K_n$, and in fact for any oriented
matroid. It is known that ${\rm Hilb}((6T_n(\Gamma)^{ab},t) = t^{e(\Gamma)}
{\rm Tutte}(\Gamma,1+t,t^{-1})$, see, e.g.,~\cite{B} and the literature quoted therein.
\end{itemize}
\end{Comments}

\subsubsection[Hilbert series of algebras ${\rm CYB}_n$ and $6T_n$]{Hilbert series of algebras $\boldsymbol{{\rm CYB}_n}$ and $\boldsymbol{6T_n}$\footnote{Results of this subsection have been obtained
independently in~\cite{BEE}. This paper contains, among other things, a~description of a basis in the algebra~$6T_n$, and much more.}}\label{section4.2.3}

\begin{Examples}\label{examples4.2}
\begin{gather*}
{\rm Hilb}(6T_{3},t)=\big(1-3t+t^2\big)^{-1},\qquad
{\rm Hilb}(6T_{4},t)=\big(1-6t+7t^2-t^3\big)^{-1},\\
{\rm Hilb}(6T_{5},t)=\big(1-10t+25t^2-15t^3+t^{4}\big)^{-1},\\
{\rm Hilb}(6T_{6},t)=\big(1-15t+65t^2-90t^3+31t^4-t^5\big)^{-1},\\
{\rm Hilb}\big(6T_{3}^{(0)},t\big)= [2][3](1-t)^{-1},\qquad
{\rm Hilb}\big(6T_{4}^{(0)},t\big)= [4](1-t)^{-2}\big(1-3t+t^2\big)^{-1}.
\end{gather*}
\end{Examples}

In fact, the following statements are true.
\begin{Proposition}[cf.~\cite{BEE}]\label{proposition4.5} Let $n \ge 2$, then
\begin{itemize}\itemsep=0pt
\item The algebras $6T_n$ and ${\rm CYB}_n$ are Koszul.

\item We have
\begin{gather*}
 {\rm Hilb}(6T_n,t)=
\left( \sum_{k=0}^{n-1}(-1)^k {n \brace n-k } t^{k}
 \right)^{-1},
 \end{gather*}
where ${n \brace k } $ stands for the Stirling
numbers of the second kind, i.e., the number of ways to partition a set
of~$n$ things into~$k$ nonempty subsets.

\item
\begin{gather*} {\rm Hilb}({\rm CYB}_n,t)=
\left(\sum_{k=0}^{n-1}(-1)^k (k+1)! N(k,n) t^k \right)^{-1},
\end{gather*}
where $N(k,n)={1 \over n}{n \choose k} {n \choose k+1}$ denotes the Narayana
number, i.e., the number of Dyck $n$-paths with exactly~$k$ peaks.
\end{itemize}
\end{Proposition}

\begin{Corollary} \label{corollary4.4}\quad
\begin{enumerate}\itemsep=0pt
\item[$(A)$] The Hilbert polynomial of the quadratic dual of the algebra
$6T_n$ is equal to
\begin{gather*}
{\rm Hilb}\big(6T_n^{!},t\big)=
\sum_{k=0}^{n-1} {n \brace n-k } t^{k}.
\end{gather*}
It is well-known that
\begin{gather*}
 \sum_{n \ge 0}\left( \sum_{k=0}^{n-1} {n \brace n-k } t^{k}
\right) {z^n \over n!}=\exp\left({\exp(zt)-1 \over t}\right).
\end{gather*}
Therefore,
\begin{gather*}
\dim (6T_n)^{!}={\rm Bell}_n,
\end{gather*}
 where ${\rm Bell}_n$ denotes the $n$-th Bell number, i.e., the number of
ways to partition $n$ things into subsets, see~{\rm \cite{SL}}.
Recall, that
\begin{gather*}
\sum_{n \ge 0} {\rm Bell}_n {z^n \over n!}=\exp(\exp(z)-1)).
\end{gather*}

\item[$(B)$]
 The Hilbert polynomial of the quadratic dual of the algebra
${\rm CYB}_n$ is equal to
\begin{gather*}
{\rm Hilb}\big(({\rm CYB}_n\big)^{!},t)=\sum_{k=0}^{n-1}(k+1)! N(k,n) t^k=
(n-1)! L_{n-1}^{(\alpha=1)}\big({-}t^{-1}\big) t^{n-1},
\end{gather*}
where
\begin{gather*}
L_n^{(\alpha)}(x)=
{x^{-\alpha} e^x \over n!}{d^n \over dx^n} \big(e^{-x} x^{n+\alpha}\big)
\end{gather*}
denotes the generalized Laguerre polynomial.
The numbers $(k+1)! N(n,k):= L(n,n-k)$ are known as Lah numbers, see,
e.g., {\rm \cite[$A008297$]{SL}}, moreover~{\rm \cite{SL}},
\begin{gather*}
\dim ({\rm CYB}_n)^{!} = A000262.
\end{gather*}
It is well-known that
\begin{gather*}
\sum_{n \ge 0}\left( \sum_{k \ge 0}^{n-1} (k+1)!
N(k,n) t^{k} \right) {z^n \over n!}=\exp\big(z(1-zt)^{-1}\big).
\end{gather*}
\end{enumerate}
\end{Corollary}

\begin{Comments} \label{comments4.4}
Let ${\cal{E}}_n(u)$, $u \not= 0,1$, be the {\it Yokonuma--Hecke} algebra, see, e.g.,~\cite{RH} and the literature quoted therein. It is
known that the dimension of the Yokonuma--Hecke algebra ${\cal{E}}_n(u)$ is
equal to $n! B_n$, where $B_n$ denotes as before the $n$-th Bell number.
 Therefore, $\dim({\cal{E}}_n(u)) =
\dim ((6T_n)^{!} \rtimes {\mathbb S}_n)$, where
$(6T_n)^{!} \rtimes {\mathbb S}_n$ denotes the semi-direct product of
the algebra $(6T_n)^{!}$ and the symmetric group~${\mathbb S}_n$. It
seems an interesting task to check whether or not the algebras
$(6T_n)^{!} \rtimes {\mathbb S}_n$ and ${\cal{E}}_n(u)$ are isomorphic.
\end{Comments}

\begin{Remark}\label{remark4.2}
 Denote by ${\cal M}{\rm YB}_n$ the group algebra over $\Q$
of the {\it monoid} corresponding to the Yang--Baxter group ${\rm YB}_n$, see, e.g.,
Def\/inition~\ref{definition4.10}. Let $P({\cal M}{\rm YB}_n,s,t)$ denotes the Poincar\'e
polynomial of the algebra ${\cal M}{\rm YB}_n$. One can show that
\begin{gather*}{\rm Hilb}(6T_n,s)=P({\cal M}{\rm YB}_n,-s,1)^{-1} .
\end{gather*}
 For example,
\begin{gather*}
P({\cal M}{\rm YB}_3,s,t)= 1+3s t+s^2 t^3,\\
P({\cal M}{\rm YB}_4,s,t)=1+6s t+s^2 \big(3t^2+4t^3\big)+s^3 t^6,\\
P({\cal M}{\rm YB}_5,s,t)=1+10s t+s^2\big(15t^2+10t^3\big)+s^3\big(10t^4+5t^6\big)+s^4 t^{10}.
\end{gather*}

Note that ${\rm Hilb}({\cal M}{\rm YB}_n,t)= P({\cal M}{\rm YB}_n,-1,t)^{-1}$ and
$P({\cal M}{\rm YB}_n,1,1) = {\rm Bell}_{n}$, the $n$-th Bell number.
\end{Remark}

\begin{Conjecture}\label{conjectute4.5}
\begin{gather*}
P({\cal M}{\rm YB}_n,s,t)=\sum_{\pi} s^{\#(\pi)} t^{n(\pi)},
\end{gather*}
where the sum runs over all partitions $\pi=(I_1,\ldots,I_k)$ of the set
$[n]:=[1,\ldots,n]$ into nonempty subsets $I_1,\ldots,I_k$, and we set by
definition, $\#(\pi):=n-k$, $n(\pi):=\sum\limits_{a=1}^{k} {|I_a| \choose 2}$.
\end{Conjecture}

\begin{Remark} \label{remark4.3}
 For any f\/inite Coxeter group $(W,S)$ one can def\/ine the
algebra ${\rm CYB}(W):={\rm CYB}(W,S)$ which is an analog of the algebra
${\rm CYB}_n={\rm CYB}(A_{n-1})$ for other root systems.
\end{Remark}

\begin{Conjecture}[A.N.~Kirillov, Yu.~Bazlov] \label{conjecture4.6}
Let $(W,S)$ be a finite Coxeter group
with the root system~$\Phi$. Then
\begin{itemize}\itemsep=0pt
\item the algebra ${\rm CYB}(W)$ is Koszul;
\item ${\rm Hilb}({\rm CYB}(W),t)= \left\{\sum\limits_{k=0}^{|S|}r_k(\Phi) (-t)^{k}
 \right\}^{-1}$,
 \end{itemize}
where $r_k(\Phi)$ is equal to the number of subsets in $\Phi^{+}$ which
constitute the positive part of a root subsystem of rank~$k$. For example,
$r_{1}(\Phi)=|\Phi^{+}|$, and $r_{2}(\Phi)$ is equal to the number of
defining relations in a representation of the algebra ${\rm CYB}(W)$.
\end{Conjecture}

\begin{Example}\label{example4.5}
\begin{gather*}
{\rm Hilb}\big({\rm CYB}(B_2)^{!},t\big)=(1,4,3), \qquad {\rm Hilb}\big({\rm CYB}(B_3)^{!},t\big)=(1,9,13,2), \\
{\rm Hilb}\big({\rm CYB}(B_4)^{!},t\big)=(1,16,46,28,5), \qquad
{\rm Hilb}\big({\rm CYB}(B_5)^{!},t\big)=(1,25,130,200,101,12),\\
{\rm Hilb}\big({\rm CYB}(D_4)^{!},t\big)=(1,12,34,24,4), \qquad
{\rm Hilb}\big({\rm CYB}(D_5)^{!},t\big)=(1,20,110,190,96,11).
\end{gather*}
\end{Example}

\begin{Definition}\label{definition4.16} The even generic Orlik--Solomon algebra
${\rm OS}^{+}(\Gamma_n)$ is def\/ined to be an associative algebra (say over~$\Z$)
generated by the set of {\it mutually commuting} elements~$y_{i,j}$, $1 \le i \not= j \le n$, subject to the set of cyclic relations
\begin{gather*}
 y_{i,j}=y_{j,i}, \qquad y_{i_1,i_2} y_{i_2,i_3} \cdots
y_{i_{k-1},i_k} y_{i_1,i_k}=0 \qquad \text{for} \quad k=2,\ldots, n,
\end{gather*}
and all sequences of pairwise distinct integers
$1 \le i_1, \ldots, i_k \le n$.
\end{Definition}

\begin{Exercises}\label{exercises4.4}\quad
\begin{enumerate}\itemsep=0pt
\item[(1)] Show that
\begin{gather*}
\exp\big(z (1-zt)^{-q}\big)=1+\sum_{n \ge 1} \left( 1+\sum_{k=1}^{n-1} {n-1 \choose k} \prod_{a=0}^{k-1}(a+(n-k) q) t^{k} \right) {z^n \over n!}.
\end{gather*}

\item[(2)] {\it The even generic Orlik--Solomon algebra.}
Show that the number of degree $k$, $k \ge 3$,
relations in the def\/inition of the Orlik--Solomon algebra ${\rm OS}^{+}(\Gamma_n)$
is equal to ${1 \over 2} (k-1)! {n \choose k}$
and also is equal to the maximal number of $k$-cycles in the complete graph~$K_n$.
\end{enumerate}

Note that if one replaces the commutativity condition in the above def\/inition
on the condition that ${y_{i,j}}$'s pairwise {\it anticommute}, then the
resulting algebra appears to be isomorphic to the Orlik--Solomon algebra
${\rm OS}(\Gamma_n)$ corresponding to the generic hyperplane arrangement~$\Gamma_n$, see~\cite{PSt}. It is known \cite[Corollary~5.3]{PSt}, that
\begin{gather*}
{\rm Hilb}({\rm OS}(\Gamma_n),t)= \sum_{F} t^{|F|},
\end{gather*}
where the sum runs over all forests~$F$ on the vertices $1, \ldots, n$, and~$|F|$ denotes the number of edges in a~forest~$F$.

It follows from Corollary~\ref{corollary4.4}, that
\begin{gather*}
\sum_{n \ge 1} {\rm Hilb}({\rm OS}(\Gamma_n),t) {z^n \over n!}
= \exp \left( \sum_{n \ge 1} n^{n-2} t^{n-1} {z^n \over n!} \right).
\end{gather*}
It is not dif\/f\/icult to see that
${\rm Hilb}({\rm OS}^{+}(\Gamma_n),t)={\rm Hilb}({\rm OS}(\Gamma_n),t)$. In particular,
$\dim {\rm OS}^{+}(\Gamma_n)= {\cal F}(n)$. Note also that a sequence
$\{ {\rm Hilb}({\rm OS}(\Gamma_n),-1) \}_{n \ge 2}$ appears in \cite[$A057817$]{SL}.
The polynomials ${\rm Hilb}({\cal A}_n,t)$, $F_n(x,t)$ and
${\rm Hilb}({\rm OS}^{+}(\Gamma_n),t)$ can be expressed, see, e.g.,~\cite{PSS},
 as certain {\it specializations} of the Tutte polynomial $T(G;x,y)$
corresponding to the complete graph $G:=K_n$. Namely,
\begin{gather*}
{\rm Hilb}({\cal A}_n,t)=t^{{n \choose 2}} T\big(K_n; 1+t, t^{-1}\big),\qquad
{\rm Hilb}\big({\rm OS}^{+}(\Gamma_n),t\big)=t^{n-1} T\big(K_n;1+t^{-1},1\big).
\end{gather*}
\end{Exercises}

\subsubsection{Super analogue of 6-term relations algebra}\label{section4.2.4}

Let $n$, $m$ be non-negative integers.
\begin{Definition}\label{definition4.17} {\it The super $6$-term relations algebra $6T_{n,m}$} is
 an associative algebra over $\Q$ generated by the elements $\{x_{i,j}, \,
1 \le i \not= j \le n \}$ and
$\{y_{\alpha,\beta}, \, 1 \le \alpha \not= \beta \le m \}$ subject to the
set of relations
\begin{enumerate}\itemsep=0pt
\item[(0)] $x_{i,j}+x_{j,i}=0$, $ y_{\alpha, \beta}=y_{\beta, \alpha}$;
\item[(1)] $x_{i,j}x_{k,l}=x_{k,l}x_{i,j}$, $x_{i,j}y_{\alpha,\beta}=
y_{\alpha,\beta}x_{i,j}$, $y_{\alpha,\beta}y_{\gamma,\delta}+
y_{\gamma,\delta}y_{\alpha,\beta}=0$,
if tuples $(i,j,k,l)$, $(i,j,\alpha,\beta)$, as well as
$( \alpha, \beta, \gamma, \delta)$ consist of pair-wise distinct integers;
\item[(2)] classical Yang--Baxter relations and theirs super analogue:
$[x_{i,k},x_{j,i}+x_{j,k}]+[x_{i,j},x_{j,k}]=0$
if $1 \le i,j,k \le n$ are distinct,
$[x_{i,k},y_{j,i}+y_{j,k}]+[x_{i,j},y_{j,k}]=0$
if $ 1 \le i,j,k \le \min(n,m)$ are distinct,
$[y_{\alpha,\gamma},y_{\beta,\alpha}+
y_{\beta,\gamma}]_{+}+[y_{\alpha,\beta},y_{\beta,\gamma}]_{+}=0$
if $1 \le \alpha, \beta,\gamma \le m$ are distinct.
\end{enumerate}
Recall that $[a,b]_{+}:=ab+ba$ denotes
the anticommutator of elements~$a$ and~$b$.
\end{Definition}

\begin{Conjecture} \label{conjecture4.7} The algebra $6T_{n,m}$ is Koszul.
\end{Conjecture}

\begin{Theorem} \label{theorem4.5}
Let $n,m \in \Z_{\ge 1}$, one has
\begin{gather*}
 {\rm Hilb}\big((6T_n)^{!},t\big) {\rm Hilb}\big((6T_m)^{!},t\big)=
\sum_{k=0}^{\min (n,m)-1}
{\min(n,m) \brace \min(n,m)-k}
{\rm Hilb}\big((6T_{n-k,m-k})^{!},t\big) t^{2k},
\end{gather*}
where as before ${n \brace n-k}$ denotes the Stirling numbers of the second
kind, see, e.g., {\rm \cite[$A008278$]{SL}}.
\end{Theorem}

\begin{Corollary}\label{corollary4.5} Let $n,m \in \Z_{ \ge 1}$. One has
\begin{enumerate}\itemsep=0pt
\item[$(a)$] Symmetry: ${\rm Hilb}(6T_{n,m},t)={\rm Hilb}(6T_{m,n},t)$.
\item[$(b)$] Let $n \le m$, then
\begin{gather*}
{\rm Hilb}\big((6T_{n,m})^{!},t\big)=
\sum_{k=0}^{n-1} s(n-1,n-k)
{\rm Hilb}\big((6T_{n-k})^{!},t\big) {\rm Hilb}\big((6T_{m-k})^{!},t\big) t^{2k},
\end{gather*}
where $s(n-1,n-k)$ denotes the Stirling numbers of the first kind, i.e.,
\begin{gather*}
\sum_{k=0}^{n-1} s(n-1,n-k) t^k= \prod_{j=1}^{n-1} (1-j t).
\end{gather*}

\item[$(c)$] $\dim (6T_{n,n})^{!}$ is equal to the number of pairs of partitions of
the set $\{1,2,\ldots,n \}$ whose meet is the partition
$\{\{1\}, \{2\},\ldots, \{n\}\}$, see, e.g., {\rm \cite[$A059849$]{SL}}.
\end{enumerate}
\end{Corollary}

\begin{Example} \label{examples4.77}
\begin{gather*}
{\rm Hilb}\big((6T_{3,2})^{!},t\big)={\rm Hilb}\big((6T_{2,3})^{!},t\big)=(1,4,3), \\
{\rm Hilb}\big((6T_{2,4})^{!},t\big)={\rm Hilb}\big((6T_{4,2})^{!},t\big)=(1,7,12,5),\qquad
{\rm Hilb}\big((6T_{3,3})^{!},t\big)=(1,6,8),\\
{\rm Hilb}\big((6T_{2,5})^{!},t\big)={\rm Hilb}\big((6T_{5,2})^{!},t\big)=(1,11,34,34,9), \\
{\rm Hilb}\big((6T_{3,4})^{!},t\big)={\rm Hilb}\big((6T_{4,3})^{!},t\big)=(1,9,23,16),\\
{\rm Hilb}\big((6T_{4,4})^{!},t\big)=(1,12,44,50,6), \\
{\rm Hilb}\big((6T_{3,5})^{!},t\big)={\rm Hilb}\big((6T_{5,3})^{!},t\big)=(1,13,53,79,34), \\
{\rm Hilb}\big((6T_{4,5})^{!},t\big)={\rm Hilb}\big((6T_{5,4})^{!},t\big)=(1,16,86,182,131,12), \\
{\rm Hilb}\big((6T_{5,5})^{!},t\big)=(1,20,140,410,462,120).
\end{gather*}
\end{Example}

Now let us def\/ine in the algebra $6T_{n,m}$ the Dunkl elements
$\theta_{i}:=\sum\limits_{j \not= i} x_{i,j}$, $1 \le i \le n$, and
${\bar \theta}_{\alpha}:=\sum\limits_{\beta \not= \alpha} y_{\alpha,\beta}$,
$1 \le \alpha \le m$.

\begin{Lemma}\label{lemma4.4} One has
\begin{itemize}\itemsep=0pt
\item $[\theta_{i},\theta_{j}]= 0$,
\item $[\theta_i, {\bar \theta}_{\alpha}]=[x_{i,\alpha},y_{i,\alpha}]$,
\item $[{\bar \theta}_{\alpha}, {\bar \theta}_{\beta}]_{+}= 2 y_{\alpha,\beta}^2
$ if $\alpha \not= \beta$.
\end{itemize}
\end{Lemma}

\begin{Remark}[``odd'' six-term relations algebra] \label{remark4.4}
In particular,
 one can def\/ine an ``odd'' analog $6T_n^{(-)}=6T_{0,n}$ of the six
term relations algebra~$6T_n$. Namely, the algebra $6T_n^{(-)}$ is given by the
set of generators $\{ y_{ij},\, 1 \le i < j \le n \}$, and that of relations:
\begin{enumerate}\itemsep=0pt
\item[1)] $y_{i,j}$ and $y_{k,l}$ {\it anticommute} if $i$, $j$, $k$, $l$ are pairwise
distinct;

\item[2)] $[y_{i,j},y_{i,k}+y_{j,k}]_{+}+[y_{i,k},y_{j,k}]_{+}=0$,
if $1 \le i < j \le k \le n$, where $[x,y]_{+}=xy+yx$ denotes the
anticommutator of~$x$ and~$y$.
\end{enumerate}

The ``odd'' three term relations algebra~$3T_n^{-}$ can be obtained as the
quotient of the algebra~$6T_n^{-}$ by the two-sided ideal generated by the
three term relations
$y_{ij} y_{jk} + y_{jk} y_{ki} +y_{ki} y_{ij} = 0$ if $i$, $j$, $k$ are pairwise
distinct.

One can show that the Dunkl elements $\theta_i$ and $\theta_j$, $i \not= j$,
given by formula
\begin{gather*}
\theta_i=\sum_{j \not= i} y_{ij}, \qquad i=1,\ldots,n,
\end{gather*}
 form an {\it anticommutative} family of elements in the algebra $6T_n^{(-)}$.

In a similar fashion one can def\/ine an ``odd'' analogue of the dynamical six
term relations algebra $6DT_n$, see Def\/inition~\ref{definition2.2} and Section~\ref{section2.1.1}, as well as
def\/ine an ``odd'~analogues of the algebra $3MT_n(\beta,{\boldsymbol{0}})$,
see Def\/inition~\ref{definition3.7}, the Kohno--Drinfeld algebra, the Hecke algebra and few
others considered in the present paper. Details are omitted in the present
paper.

More generally, one can ask what are natural $q$-analogues of the six term and
three term relations algebras? In other words to describe relations which
ensure the $q$-commutativity of Dunkl elements def\/ined above. First of all it
would appear natural that the ``$q$-locality and $q$-symmetry conditions''
hold among the set of generators $\{y_{ij}, \, 1 \le i \not= j \le n \}$, that is
$y_{ij}+ q y_{ji}=0$, $y_{ij} y_{kl} = q y_{kl} y_{ij}$ if $i < j$,
$k < l$, and $\{i,j\} \cap \{k,l\} =\varnothing$.

Another natural condition is the fulf\/illment of $q$-analogue of the classical
Yang--Baxter relations, namely,
$[y_{ik},y_{jk}]_{q} + [y_{ik},y_{ji}]_{q} + [y_{ij},y_{jk}]_{q} =0$ if
$i < j < k$, where $[x,y]_{q}:= x y -q y x$ denotes the $q$-commutator.
 However we are not able to f\/ind the $q$-analogue of the classical Yang--Baxter relation listed above in the mathematical and physical literature yet.
Only cases $q=1$ and $q= -1$ have been extensively studied.
\end{Remark}



\subsection{Four term relations algebras / Kohno--Drinfeld algebras}\label{section4.3}

\subsubsection[Kohno--Drinfeld algebra $4T_n$ and that ${\rm CYB}_n$]{Kohno--Drinfeld algebra $\boldsymbol{4T_n}$ and algebra $\boldsymbol{{\rm CYB}_n}$}\label{section4.3.1}

\begin{Definition}\label{definition4.18} {\it The $4$-term relations algebra} (or the Kohno--Drinfeld
 algebra, or inf\/initesimal pure braids algebra) $4T_{n}$ is an
associative algebra (say over $\Q$) with the set of generators
$y_{i,j}$, $1 \le i < j \le n$, subject to the following relations
\begin{enumerate}\itemsep=0pt
\item[1)] $y_{i,j}$ and $y_{k,l}$ are commute, if $i$, $j$, $k$, $l$ are all distinct;
\item[2)] $[y_{i,j},y_{i,k}+y_{j,k}]= 0$,
$[y_{i,j}+y_{i,k},y_{j,k}]= 0$ if $1 \le i < j \le k \le n$.
\end{enumerate}
\end{Definition}

Note that the algebra $4T_{n}$ is given by ${n \choose 2}$ generators and
$2 {n \choose 3}+3 {n \choose 4}$ quadratic relations, and the element
\begin{gather*}
c:=\sum_{1 \le i < j \le n} y_{i,j}
\end{gather*}
 belongs to the center of the Kohno--Drinfeld algebra.

\begin{Definition}\label{definition4.19}
Denote by $4T_{n}^{(0)}$ the quotient of the algebra $4T_{n}$ by the
(two-sided) ideal generated by by the set of elements
$\{y_{i,j}^2,\,1 \le i < j \le n \}$.

More generally, let $\beta$,
$\{q_{ij},\, 1 \le i < j \le n \}$ be the set of parameters, denote by $4HT_n$
the quotient of the algebra $4T_n$ by the two-sided ideal generated by the
set of elements $\{y_{ij}^2- \beta y_{ij}-q_{ij}, \, 1 \le i < j \le n \}$.
\end{Definition}

These algebras are naturally graded, with $\deg( y_{i,j})=1$, $\deg(\beta)=1$,
$\deg( q_{ij})=2$, as well as each of that algebras has a natural f\/iltration by
setting $\deg( y_{i,j})=1$, $\deg( \beta )= 0$, $\deg(q_{ij})=0$, $\forall\, i
\not= j$.

It is clear that the algebra $4T_{n}$ can be considered as the
inf\/initesimal deformation $g_{i,j}:=1+\epsilon y_{i,j}$, $\epsilon \longrightarrow 0$, of the pure braid group $P_{n}$.

There is a natural action of the
symmetric group ${\mathbb{S}}_{n}$ on the algebra $4T_{n}$ (and also on
$4T_{n}^{0}$) which preserves the grading: it is def\/ined by $w \cdot y_{i,j}=
y_{w(i),w(j)}$ for $w \in {\mathbb{S}}_{n}$. The semi-direct product
$\Q{\mathbb{S}}_{n} \ltimes 4T_{n}$ (and also that
$\Q{\mathbb{S}}_{n} \ltimes 4T_{n}^{0}$) is a~Hopf algebra
denoted by ${\cal B}_{n}$ (respectively~${\cal B}_{n}^{(0)})$.

\begin{Remark}\label{remark4.5}
There exists the natural map
\begin{gather*}
{\rm CYB}_{n} \longrightarrow 4T_{n} \qquad \text{given by} \ \ y_{i,j}:=u_{i,j}+u_{j,i}.
\end{gather*}
Indeed, one can easily check that
\begin{gather*}
[y_{ij},y_{ik}+y_{jk}] = w_{ijk}+w_{jik}-w_{kij}-w_{kji} ,
\end{gather*}
 see Section~\ref{section2.3.1}, Def\/inition~\ref{definition2.5} for a def\/inition of the classical
Yang--Baxter algebra $CYB_n$, and Section~\ref{section2}, equation~\eqref{equation2.3}, for a def\/inition of the
element~$w_{ijk}$.
\end{Remark}

\begin{Remark}\label{remark4.6}\quad
\begin{itemize}\itemsep=0pt
\item Much as the relations in the algebra $6T_n$ are chosen in a way to
 imply (and ``essentially''\footnote{Together with locality and factorization conditions a set of
 def\/ining relations in the algebra $6T_n$ is {\it equivalent} to the
commutativity property of Dunkl's elements.}
equivalent) the pair-wise commutativity of the Dunkl elements $\{\theta_{i} \}_{1 \le i \le n}$, the relations in the Kohno--Drinfeld algebra imply (and
``essentially'' equivalent) to pair-wise commutati\-vi\-ty of the Jucys--Murphy
elements (or, equivalently, {\it dual} JM-elements)
$d_j:= \sum\limits_{ 1 \le a < j} y_{aj}$, $2 \le j \le n$ (resp.\ $\overline{d}_i= \sum\limits_{1 \le a \le i} y_{n-i,n-a+1}$, $1 \le i \le n-1 $).

\item It follows from the classical 3-term identity
(``Jacobi identity'')
\begin{gather*}
{1 \over (a-b)(a-c)}-{1 \over (a-b)(b-c)}+{1 \over (a-c)(b-c)}=0,
\end{gather*}
that if elements $\{ y_{i,j} \,|\, 1 \le i < j \le n \}$ satisfy the 4-term
algebra relations, see Def\/ini\-tion~\ref{definition4.18}, and $t_1,\dots,t_n, $ a set of
(pair-wise) commuting parameters, then the elements
\begin{gather*}
r_{i,j}:={y_{i,j} \over t{_i}-t{_j}}
\end{gather*}
 satisfy the set of def\/in\/ing relations of the 6-term relations algebra $6T_n$,
see Section~\ref{section4.2.1},
Def\/inition~\ref{definition4.9}. In particular, the Knizhnik--Zamolodchikov
elements
\begin{gather*}
{\rm KZ}_{j}:=\sum_{i \not= j}{y_{i,j} \over t{_i}-t{_j}}, \qquad 1 \le j \le n,
\end{gather*}
 form a pair-wise commuting family (by def\/inition, we put $y_{i,j}=y_{j,i}$
if $i > j$).
\end{itemize}
\end{Remark}

\begin{Example}\label{example4.7}\quad
\begin{enumerate}\itemsep=0pt
\item[(1)] {\it Yang representation of the $4T_n$}.
Let ${\mathbb S}_n$ be the symmetric group acting
identically on the set of variables $\{x_1,\ldots,x_n \}$. Clearly that the
elements $ \{y_{i,j}: = s_{ij} \}_{1 \le i < j \le n}$, $y_{i,j}:=y_{j,i}$
if $i > j$, satisfy the Kohno--Drinfeld relations listed in
Def\/inition~\ref{definition4.18}. Therefore the operators $u_{ij}$ def\/ined by
\begin{gather*}
u_{ij}=(x_i-x_j)^{-1} s_{ij}
\end{gather*}
 give rise to a representation of the algebra $3T_n$ on the f\/ield of rational
 functions $\Q(x_1,\ldots$, $x_n)$. The Dunkl--Gaudin elements
\begin{gather*}
\theta_i= \sum_{j, j\not= i} y_{ij}, \qquad i=1,\ldots,n
\end{gather*}
correspond to the truncated Gaudin operators acting in the tensor
space $({\C})^{\otimes n}$. Cf.\ Section~\ref{section3.3}.

\item[(2)] Let $A=U({\mathfrak{sl}}(2))$ be the universal enveloping algebra of the
Lie algebra ${\mathfrak{sl}}(2)$. Recall that the algebra ${\mathfrak{sl}}(2)$ is spanned by the
elements~$e$, $f$, $h$, so that $[h,e]=2e$, $[h,f]=-2f$, $[e,f]=h$. Consider
the element $\Omega={1 \over 2} h \otimes h+e \otimes f+f \otimes e$.
Then the map $y_{i,j} \longrightarrow \Omega_{i,j} \in A^{ \otimes n}$
 def\/ines a representation of the Kohno--Drinfeld algebra~$4T_n$ on that~$A^{\otimes n}$. The element ${\rm KZ}_j$ def\/ined above, corresponds to the
truncated (or at critical level) rational Knizhnik--Zamolodchikov operator. Cf.\ Section~\ref{section4.2.1}, Example~\ref{example4.4}.
\end{enumerate}
\end{Example}

\begin{Proposition}[T.~Kohno, V.~Drinfeld]\label{proposition4.6}
\begin{gather*}
{\rm Hilb}(4T_{n},t)=\prod_{j=1}^{n-1}(1-jt)^{-1}=\sum_{k \ge 0}
 {n+k-1 \brace n-1} t^{k},
 \end{gather*}
where ${n \brace k }$ stands for the Stirling
numbers of the second kind, i.e., the number of ways to partition a~set of~$n$
things into~$k$ nonempty subsets.
\end{Proposition}

\begin{Remark} \label{remark4.7} It follows from~\cite{BN} that ${\rm Hilb}(4T_{n},t)$ is equal to
the generating function
\begin{gather*}
1+\sum_{ d \ge 1} v_{d}^{(n)} t^d
\end{gather*}
 for the number $v_{d}^{(n)}$ of {\it Vassiliev invariants of order~$d$}
 for {\it $n$-strand braids}. Therefore, one has the following equality:
 \begin{gather*}
 v_{d}^{(n)}= {n+d-1 \brace n-1},
 \end{gather*}
 i.e., the number of Vassiliev invariants of order~$d$ for $n$-strand braids is equal to the Stirling number of the second kind ${n+d-1 \brace n-1}$.

We {\it expect} that the generating function
\begin{gather*}
1+\sum_{ d \ge 1} {\widehat v}_{d}^{(n)} t^d
\end{gather*}
 for the number ${\widehat v}_{d}^{(n)}$ of {\it Vassiliev invariants of
order~$d$} for {\it $n$-strand virtual braids} is equal to the
Hilbert series ${\rm Hilb}(4NT_n,t)$ of the {\it nonsymmetric} Kohno--Drinfeld
 algebra~$4NT_n$, see Section~\ref{section4.3.2}.
\end{Remark}

\begin{Proposition}[cf.~\cite{BEE}] \label{proposition4.7}
The algebra $4NT_n,t)$ is Koszul, and
\begin{gather*}
{\rm Hilb} (4NT_n,t)= \left(\sum_{k=0}^{n-1} (k+1)! N(k,n) (-t)^{k} \right)^{-1}, \\
{\rm Hilb}\big((4NT_n)^{!},t\big)=(n-1) ! L_{n-1}^{(\alpha=1)}\big({-}t^{-1}\big) t^{n-1},
\end{gather*}
where $N(k,n):= \frac{1}{n} {n \choose k} {n \choose k+1}$ denotes the
Narayana number, i.e., the number of Dyck $n$-paths with exactly $k$ peaks,
\begin{gather*}
L_n^{(\alpha)}(x)=\frac{x^{\alpha} e^x}{n!} \frac{d^n}{dx^n}\big( e_x x^{n+\alpha} \big)
\end{gather*}
denotes the generalized Laguerre polynomial.
\end{Proposition}

See also Theorem~\ref{theorem4.6} below.

It is well-known that the quadratic dual $4T_n^{!}$ of the Kohno--Drinfeld
algebra $4T_n$ is isomorphic to the Orlik--Solomon algebra of type $A_{n-1}$,
as well as the algebra $3T_n^{\rm anti}$. However the algebra $4T_n^{0}$ is
{\it failed} to be Koszul.

\begin{Examples}\label{examples4.3}
\begin{gather*}
{\rm Hilb}\big(4T_{3}^{0},t\big)=[2]^{2}[3], \qquad
{\rm Hilb}\big(4T_{4}^{0},t\big)=(1,6,19,42,70,90,87,57,23,6,1),\\
{\rm Hilb}\big(\big(4T_{3}^{0}\big)^{!},t\big)(1-t)=(1,2,2,1),\qquad
{\rm Hilb}\big(\big(4T_{4}^{0}\big)^{!},t\big)(1-t)^2=(1,4,6,2,-4,-3),\\
{\rm Hilb}\big(\big(4T_{5}^{0}\big)^{!},t\big)(1-t)^2=(1,8,26,40,24,-3,-6).
\end{gather*}
\end{Examples}

We {\it expect} that ${\rm Hilb}\big((4T_{n}^{0})^{!},t\big)$ is a rational function with
the only pole at $t=1$ of order $[n/2]$, cf.\ Examples~\ref{examples4.77}.

\begin{Remark} \label{remark4.8}
 One can show that if $n \ge 4$, then ${\rm Hilb}(4T_{n}^{0},t) <
{\rm Hilb}(3T_{n}^{0},t)$ {\it contrary} to the statement of Conjecture~9.6
from~\cite{K3}.
\end{Remark}

\subsubsection[Nonsymmetric Kohno--Drinfeld algebra $4NT_n$, and McCool algebras ${\cal P}\Sigma_n$ and ${\cal P}\Sigma_n^{+}$]{Nonsymmetric Kohno--Drinfeld algebra $\boldsymbol{4NT_n}$,\\ and McCool algebras $\boldsymbol{{\cal P}\Sigma_n}$ and $\boldsymbol{{\cal P}\Sigma_n^{+}}$}\label{section4.3.2}

\begin{Definition} \label{definition4.20}
{\it The nonsymmetric $4$-term relations algebra} (or the
nonsymmetric Kohno--Drinfeld algebra) $4NT_{n}$ is an
associative algebra (say over~$\Q$) with the set of generators
$y_{i,j}$, $1 \le i \not= j \le n$, subject to the following relations
\begin{enumerate}\itemsep=0pt
\item[1)] $y_{i,j}$ and $y_{k,l}$ are commute if $i$, $j$, $k$, $l$ are all distinct;

\item[2)] $[y_{i,j},y_{i,k}+y_{j,k}]= 0$ if $i$, $j$, $k$ are all distinct.
\end{enumerate}
\end{Definition}

We denote by $4NT_{n}^{+}$ the quotient of the algebra $4NT_{n}$ by the
two-sided ideal generated by the elements
$\{y_{ij}+y_{ji}=0, \, 1 \le i \not= j \le n \}$.
\begin{Theorem} \label{theorem4.6}
One has
\begin{gather*}
{\rm Hilb}(4NT_n,t)={\rm Hilb}({\rm CYB}_n,t), \qquad {\rm Hilb}\big(4NT_n^{+},t\big)= {\rm Hilb}(6T_n,t)
\end{gather*}
for all $n \ge 2$.
\end{Theorem}
We expect that the both algebras $4NT_n$ and $4NT_n^{+}$ are {\it Koszul}.

\begin{Definition} \label{definition4.21}\quad
\begin{enumerate}\itemsep=0pt
\item[(1)] Def\/ine {\it the McCool algebra ${\cal P}\Sigma_n$} to be the
quotient of the nonsymmetric Kohno--Drinfeld algebra $4NT_n$ by the
two-sided ideal generated by the elements
$\{ y_{ik}y_{jk}-y_{jk}y_{ik} \}$
 for all pairwise distinct~$i$,~$j$ and~$k$.

\item[(2)] Def\/ine {\it the upper triangular McCool algebra
${\cal P}\Sigma_n^{+}$} to be the quotient of the McCool algebra
${\cal P}\Sigma_n$ by the two-sided ideal generated by the elements
$\{ y_{ij}+y_{ji} \}$,
$1 \le i \not= j \le n$.
\end{enumerate}
\end{Definition}

\begin{Theorem} \label{theorem4.7}
The quadratic duals of the algebras ${\cal P}\Sigma_n$ and
${\cal P}\Sigma_n^{+}$ have the following Hilbert polynomials
\begin{gather*}
 {\rm Hilb}\big({\cal P}\Sigma_n ^{!},t\big)=(1+n t)^{n-1},\qquad
 {\rm Hilb}\big(\big({\cal P}\Sigma_n^{+}\big)^{!},t\big)= \prod_{j=1}^{n-1} (1+j t).
 \end{gather*}
\end{Theorem}

\begin{Proposition}\label{proposition4.8}\quad
\begin{enumerate}\itemsep=0pt
\item[$(1)$] The quadratic dual ${\cal P}\Sigma_n^{!}$ of the algebra
${\cal P}\Sigma_n$ admits the
following description. It is gene\-ra\-ted over $\Z$ by the set of pairwise
anticommuting elements $\{y_{ij},\, 1 \le i \not= j \le n \}$,
 subject to the set of relations
\begin{enumerate}\itemsep=0pt
\item[$(a)$] $y_{ij}^2=0$, $y_{ij} y_{ji}=0$, $1 \le i \not= j \le n$,
\item[$(b)$] $y_{ik} y_{jk}=0$ for all distinct $i$, $j$, $k$,
\item[$(c)$] $y_{ij} y_{jk}+y_{ik} y_{ij}+y_{kj} y_{ik}=0$ for all distinct
$i$, $j$, $k$.
\end{enumerate}

\item[$(2)$] The quadratic dual $({\cal P}\Sigma_n^{+})^{!}$ of the algebra
${\cal P}\Sigma_n^{+}$ admits the
following description. It is generated over $\Z$ by the set of pairwise
anticommuting elements $\{z_{ij}, \,1 \le i < j \le n \}$, subject to
the set of relations
\begin{enumerate}\itemsep=0pt
\item[$(a)$] $z_{ij}^2=0$ for all $i < j$,
\item[$(b)$] $z_{ij} z_{jk}=z_{ij} z_{ik}$ for all $1 \le i< j < k \le n$.
\end{enumerate}
\end{enumerate}
\end{Proposition}

\begin{Comments}[the McCool groups and algebras]\label{comments4.5}
{\it The McCool group $P\Sigma_n$} is, by def\/inition, the group of pure
symmetric automorphisms of the free group $F_n$ consisting of all
automorphism that, for a~f\/ixed basis
$\{x_1,\ldots,x_n \}$, send each $x_i$ to a~conjugate of itself. This group is
generated by automorphisms $\alpha_{ij}$, $1 \le i \not= j \le n$, def\/ined by
\begin{gather*}
\alpha_{ij}(x_k)=\begin{cases}
x_j x_i x_j^{-1}, & k=i, \\
x_k, & k \not= i.
\end{cases}
\end{gather*}
McCool have proved that the relations
\begin{gather*}
\begin{cases}
[\alpha_{ij}, \alpha_{kl}]=1, & i,\,j,\, k,\, l \ \ \text{are distinct}, \\
[\alpha_{ij},\alpha_{ji}]=1, & i \not= j, \\
[\alpha_{ij},\alpha_{ik} \alpha_{jk}]=1,
& i,\,j,\,k \ \ \text{are distinct}.
\end{cases}
\end{gather*}
form the set of def\/ining relations for the group $P\Sigma_n$
 The subgroup of $P\Sigma_n$ generated by the~$\alpha_{ij}$ for
$1 \le i < j \le n$ is denoted by~$P\Sigma_n^{+}$ and is called by {\it upper
triangular McCool group}. It is easy to see that the McCool algebras
${\cal P}\Sigma_n$ and ${\cal P}\Sigma_n^{+}$ are the ``inf\/initesimal
deformations'' of the McCool groups $P\Sigma_n$ and $P\Sigma_n^{+}$
respectively.
\end{Comments}
\begin{Theorem} \label{theorem4.8}\quad
\begin{enumerate}\itemsep=0pt
\item[$(1)$] There exists a natural isomorphism
\begin{gather*}
 H^{*}(P\Sigma_n, \Z) \simeq {\cal P}\Sigma_n ^{!}
\end{gather*}
of the quadratic dual ${\cal P}\Sigma_n ^{!}$ of the McCool algebra
${\cal P}\Sigma_n$ and the cohomology ring $H^{*}(P\Sigma_n, \Z)$ of
the McCool group~$P\Sigma_n$, see~{\rm \cite{JMM}}.

\item[$(2)$] There exists a natural isomorphism
\begin{gather*}
 H^{*}(P\Sigma_n^{+}, \Z) \simeq ({\cal P}\Sigma_n^{+}) ^{!}
 \end{gather*}
of the quadratic dual $({\cal P}\Sigma_n^{+}) ^{!}$ of the upper triangular
McCool algebra ${\cal P}\Sigma_n^{+}$ and the cohomology ring
$H^{*}(P\Sigma_n^{+}, \Z)$ of the upper triangular McCool group
$P\Sigma_n^{+}$, see~{\rm \cite{CP}}.
\end{enumerate}
\end{Theorem}

\subsubsection[Algebras $4TT_n$ and $4ST_n$]{Algebras $\boldsymbol{4TT_n}$ and $\boldsymbol{4ST_n}$}\label{section4.3.3}

\begin{Definition}\label{definition4.22}\quad
\begin{enumerate}\itemsep=0pt
\item[(I)] {\it Algebra $4TT_n$} is generated over $\Z$ by the set of
 elements $\{x_{ij}, \,1 \le i \not= j \le n \}$, subject to the set of
relations
\begin{enumerate}\itemsep=0pt
\item[(1)] $x_{ij} x_{kl}=x_{kl} x_{ij}$ if all $i$, $j$, $k$, $l$ are distinct,

\item[(2)] $[x_{ij}+x_{jk},x_{ik}]=0$, $[x_{ji}+x_{kj},x_{ki}]=0$
 if $i < j < k$.
 \end{enumerate}

\item[(II)] {\it Algebra $4ST_n$} is generated over~$\Z$ by the set of
elements $\{x_{ij},\, 1 \le i \not= j \le n \}$, subject to the set of
relations
\begin{enumerate}\itemsep=0pt
\item[(1)] $[x_{ij},x_{kl}]=0$, $[x_{ij},x_{ji}]=0$ if $i$, $j$, $k$, $l$ are distinct,

\item[(2)] $[x_{ij},x_{ik}]=[x_{ik},x_{jk}]=[x_{jk},x_{ij}]$,
$[x_{ji},x_{ki}]=[x_{ki},x_{kj}]=[x_{kj},x_{ii}]$,

\item[(3)] $[x_{ij},x_{ki}]=[x_{kj},x_{ij}]=[x_{ji},x_{ik}]=[x_{ik},x_{kj}]=
[x_{ki},x_{jk}]=[x_{jk},x_{ji}]$ if $i <j < k$.
 \end{enumerate}
\end{enumerate}
\end{Definition}

\begin{Proposition}\label{proposition4.9}
 One has
\begin{gather*}
 t \sum_{n\ge 2} {\rm Hilb}\big((4TT_n)^{!},t\big) {z^n \over n!} =
{\exp(-t z) \over (1-z)^{2 t}}-1-t z.
\end{gather*}
Therefore, $\dim (4TT_n)^{!}$ is equal to the number of permutations of the
set $[1,\dots,n+1]$ having no substring $[k,k+1]$; also, for $n \ge 1$ equals
to the maximal permanent of a nonsingular $n \times n$
$(0,1)$-matrix, see {\rm \cite[$A000255$]{SL}}\footnote{See also a~paper by F.~Hivert, J.-C.~Novelli and J.-Y.~Thibon
\cite[Section~3.8.4]{Hivert2008} for yet another combinatorial interpretation of the dimension of the algebra~$(4TT_n)^{!}$.}. Moreover, one has
\begin{gather*}
{\rm Hilb}\big((4ST_n)^{!},t\big)=(1+t)^{n} (1+n t)^{n-2},
\end{gather*}
cf.\ Conjecture~{\rm \ref{conjecture4.9}}.
\end{Proposition}

We expect that The both algebras $4TT_n$ and $4ST_n$ are {\it Koszul}.

\begin{Problem}\label{problem-page79}
Give a combinatorial interpretation of polynomials
${\rm Hilb}((4TT_n)^{!},t)$ and construct a monomial basis in the algebras
$(4TT_n)^{!}$ and $4ST_n$.
\end{Problem}

\subsection[Subalgebra generated by Jucys--Murphy elements in $4T_n^{0}$]{Subalgebra generated by Jucys--Murphy elements in $\boldsymbol{4T_n^{0}}$}\label{section4.4}

\begin{Definition}\label{definition4.23}
 {\it The Jucys--Murphy elements} $d_j$,
$2 \le j \le n$, in the quadratic algebra $4T_{n}$ are def\/ined as follows
\begin{gather*}
d_j=\sum_{1 \le i < j} y_{i,j}, \qquad j=2,\dots,n.
\end{gather*}
\end{Definition}

It is clear that Jucys--Murphy's elements $d_{j}$ are the inf\/initesimal
deformation of the elements $D_{1,j} \in P_{n}$.
\begin{Theorem}\label{theorem4.9}\quad
\begin{enumerate}\itemsep=0pt
\item[$(1)$] The Jucys--Murphy elements $d_{j}$, $2 \le j \le n$,
commute pairwise in the algebra $4T_{n}$.

\item[$(2)$]In the algebra $4T_{n}^{0}$ the Jucys--Murphy elements~$d_{j}$,
$2 \le j \le n$, satisfy the following relations
\begin{gather*}
(d_2+\cdots+d_j) d_{j}^{2j-3}=0, \qquad 2 \le j \le n.
\end{gather*}

\item[$(3)$] Subalgebra (over $\Z$) in $4T_{n}^{0}$ generated by the
Jucys--Murphy elements $d_2,\dots, d_n$ has the following Hilbert polynomial
 $\prod\limits_{j=1}^{n-1} [2j]$.

\item[$(4)$] There exists an $($birational$)$ isomorphism
${\Z} [x_{1},\ldots,x_{n-1}]/J_{n-1} \longrightarrow {\Z} [d_2,\ldots,d_n]$
defined by $d_{j}:= \prod\limits_{i=1}^{n-j} x_{i}$, $2 \le j \le n$,
where $J_{n-1}$ is a $($two-sided$)$ ideal generated by
$e_{i}(x_{1}^{2},\ldots,x_{n-1}^{2})$, $1 \le i \le n-1$, and
$e_{i}(x_1,\ldots,x_{n-1})$ stands for the $i$-th elementary symmetric
polynomial in the variables $x_1,\ldots,x_{n-1}$.
\end{enumerate}
\end{Theorem}
\begin{Remark}\label{remark4.9}\quad
\begin{enumerate}\itemsep=0pt
\item[(1)] It is clearly seen that the commutativity of the Jucys--Murphy elements
 is equivalent to the validity of the Kohno--Drinfeld relations and the
locality relations among the generators $\{ y_{i,j} \}_{1 \le i < j \le n}$.

\item[(2)] Let's stress that $d_j^{2j-2} \not= 0$ in the algebra $4T_n^{0}$ for $j=3,
\ldots,n$. For example, $d_3^{4}=y_{13} y_{23} y_{13} y_{23}+y_{23} y_{13} y_{23} y_{13} \not= 0$
since $\dim(4T_3^{0})_{4}=1$ and it is generated by the element $d_3^{4}$.

\item[(3)] The map $\iota\colon y_{i,j} \longrightarrow y_{n+1-j,n+1-i}$
preserves the relations~$1)$ and~$2)$ in the def\/inition of the algebra $4T_n$,
and therefore def\/ines an involution of the Kohno--Drinfeld algebra. Hence the
elements
\begin{gather*} {\widehat d}_j:=\sum_{k=j+1}^{n}y_{j,k}=\iota(d_{n+1-j}),\qquad
 1 \le j \le n-1
 \end{gather*}
also form a pairwise commuting family.
\end{enumerate}
\end{Remark}

\begin{Problems}\label{problems4.1}\quad
\begin{enumerate}\itemsep=0pt
\item[$(a)$] Compute Hilbert series of the algebra $4T_{n}^{0}$ and
its quadratic dual algebra $(4T_{n}^{0})^{!}$.

\item[$(b)$] Describe subalgebra in the algebra $4HT_n$ generated by the
Jucys--Murphy elements $d_j$, $2 \le j \le n$.
\end{enumerate}
\end{Problems}

It is well-known that the Kohno--Drinfeld algebra $4T_n$ is {\it Koszul}, and
its quadratic dual $4T_n^{!}$ is isomorphic to the anticommutative quotient
$3T_n^{0,{\rm anti}}$ of the algebra $3T_n^{(-),0}$.

On the other hand,
if $n \ge 3$ the algebra $4T_{n}^{0}$ is not {\it Koszul}, and its quadratic
dual is isomorphic to the quotient of the ring of polynomials in the set of
anticommutative variables $\{t_{i,j} \,|\, 1 \le i < j \le n \}$, where
{\it we do not impose conditions} $t_{ij}^2=0$, modulo the ideal
generated by Arnold's relations
$\{ t_{i,j}t_{j,k}+t_{i,k}(t_{i,j}-t_{j,k})=0 \}$ for all pairwise
distinct~$i$,~$j$ and~$k$.

\subsection[Nonlocal Kohno--Drinfeld algebra $NL4T_n$]{Nonlocal Kohno--Drinfeld algebra $\boldsymbol{NL4T_n}$}\label{section4.5}

\begin{Definition}\label{definition4.24}
 {\it Nonlocal Kohno--Drinfeld algebra $NL4T_n$} is an associative algebra over~$\Z$ with the set of generators $\{y_{ij},\, 1 \le i < j \le n \}$
subject to the set of relations
\begin{enumerate}\itemsep=0pt
\item[(1)] $y_{ij} y_{kl}=y_{kl} y_{ij}$ if $(i-k)(i-l)(j-k)(j-l) > 0$,

\item[(2)] $\big[y_{ij},\sum\limits_{a=i}^{j} y_{ak}\big]=0$ if $i < j < k$,

\item[(3)] $\big[y_{jk},\sum\limits_{a=j}^{k} y_{ia}\big]=0$ if $i < j < k$.
\end{enumerate}
\end{Definition}

It's not dif\/f\/icult to see that relations (1)--(3) imply the following
 relations
\begin{enumerate}\itemsep=0pt
\item[(4)] $\big[x_{ij},\sum\limits_{a=i+1}^{j-1} (y_{ia}+y_{aj})\big]=0$ if~$i < j$.
\end{enumerate}

Let's introduce in the nonlocal Kohno--Drinfeld algebra $NL4T_n$ the
{\it Jucys--Murphy elements} (JM-elements for short) $d_{j}$
and the {\it dual JM-elements} ${\hat d}_j$ as follows
\begin{gather}\label{equation4.11}
d_{j}=\sum_{a=1}^{j-1} y_{aj}, \qquad {\hat d}_{j}=\sum_{a=n-j+2}^{n} y_{n-j+1,a}, \qquad j=2,\ldots,n.
\end{gather}
It follows from relations (1) and (2) (resp.~(1) and~(3)) that the
Jucys--Murphy elements $d_2,\ldots,d_n$ (resp.\ ${\hat d}_{2},\ldots,{\hat d}_{n}$)
form a commutative subalgebra in the algebra $NL4T_n$. Moreover, it follows from relations (1)--(3) that the element $c_1:=\sum\limits_{j=2}^{n}d_{j}= \sum\limits_{j=2}^{n} {\hat d}_{j}$ belongs to the center of the algebra $NL4T_n$.

\begin{Theorem} \label{theorem4.10}\quad
\begin{enumerate}\itemsep=0pt
\item[$(1)$] The algebra $NL4T_n$ is Koszul, and
\begin{gather*}
{\rm Hilb}\big((NL4T_n)^{!},t\big)=\sum_{k=0}^{n-1} C_k~{n+k-1 \choose 2k} t^{k},
\end{gather*}
where $C_k={1 \over k+1} {2k \choose k}$ stands for the $k$-th Catalan number.

\item[$(2)$] The quadratic dual $(NL4T_{n})^{!}$ of the nonlocal Kohno--Drinfeld
algebra $NL4T_n$ is an associative algebra generated by the set of
mutually anticommuting elements $\{ t_{ij},\, 1 \le i < j \le n \}$
subject to the set of relations
\begin{itemize}\itemsep=0pt
\item $t_{ij}^2=0$ if $1 \le i < j \le n$,

\item Arnold's relations: $t_{ij} t_{jk}+t_{ik} t_{ij}+t_{jk} t_{ik}=0$ if $i < j < k$,

\item disentanglement relations: $t_{ik} t_{jl}+t_{il} t_{ik}+ t_{jl} t_{il}=0$ if $i < j < k < l$.
\end{itemize}
\end{enumerate}
\end{Theorem}

Therefore the algebra $(NL4T_{n})^{!}$ is the quotient of the Orlik--Solomon algebra ${\rm OS}_n$ by the ideal generated by Disentanglement relations, and $\dim((NL4T_{n+1})^{!})$ is equal to the number of Schr\"oder paths,
i.e., paths from $(0,0)$ to $(2n,0)$ consisting of steps $U=(1,1)$,
$D=(1,-1)$, $H=(2,0)$ and never going below the $x$-axis. The Hilbert
polynomial ${\rm Hilb}((NL4T_n)^{!},t)$ is the generating function of such paths with respect to the number of $U's$, see \cite[$A088617$]{SL}.

\begin{Remark} \label{remark4.10}
Denote by $H_n(q)$ ``the normalized'' Hecke algebra of type
$A_n$, i.e., an associative algebra generated over $\Z[q,q^{-1}]$ by elements
$T_1,\ldots,T_{n-1}$ subject to the set of relations
\begin{enumerate}\itemsep=0pt
\item[(a)] $T_i T_j=T_j T_i$ if $|i-j| > 1$, $T_i T_j T_i=T_j T_i T_j$
if $|i-j| = 1$,
\item[(b)] $T_i^{2}=(q-q^{-1})~T_i+1$ for $i=1,\ldots,n-1$.
\end{enumerate}
\end{Remark}

If $1 \le i < j \le n-1$, let's consider elements $T_{(ij)}:=
T_i T_{i+1} \cdots T_{j-1} T_{j} T_{j-1} \cdots T_{i+1} T_i$.

\begin{Lemma}\label{lemma4.5}
The elements $\{ T_{(ij)}, \, 1 \le i < j < n-1 \}$ satisfy the
defining relations of the non-local Kohno--Drinfeld algebra $NL4T_{n-1}$, see
Definition~{\rm \ref{definition4.24}}.

Therefore the map $y_{ij} \rightarrow H_{(ij)}$ defines a epimorphism
$\iota_{n}\colon NL4T_n \longrightarrow H_{n+1}(q)$.
\end{Lemma}

\begin{Definition}\label{definition4.25}
 Denote by ${\cal {NL}}4T_n$ the quotient of the non-local
Kohno--Drinfeld algebra $NL4T_n$ by the two-sided ideal ${\cal I}_n$ generated
by the following set of degree three elements:
\begin{alignat*}{3}
& (1) \quad && z_{ij}:=y_{i,j+1} y_{ij} y_{j,j+1}-y_{j,j+1} y_{ij} y_{i,j+1} \qquad \text{if} \ \ 1 \le i < j \le n,&\\
& (2) \quad && u_{i}:= y_{i,i+1} \left( \sum_{a=1}^{i-1} \sum_{b=1,\, b
\not=a}^{i-1} y_{ai} y_{b,i+1} \right) - \left( \sum_{a=1}^{i-1} \sum_{b=1,\, b \not=a}^{i-1} y_{b,i+1} y_{ai} \right) y_{i,i+1}& \\
 &&& \text{if} \ \ 1 \le i \le n-1,& \\
& (3) \quad && v_{i}:=y_{i,i+1}\left( \sum_{a=i+1}^{n} \sum_{b=i+1,\, b \not= a}^{n} y_{i+1,a} y_{i,b} \right) -
\left( \sum_{a=i+1}^{n} \sum_{b=i+1,\, b \not= a}^{n}
y_{i+1,a} y_{i,b} \right) y_{i,i+1},& \\
&&& \text{if} \ \ 1 \le i \le n-1.&
\end{alignat*}
\end{Definition}

\begin{Proposition}\label{proposition4.10}\quad
\begin{enumerate}\itemsep=0pt
\item[$(1)$] The ideal ${\cal T}_n$ belongs to the kernel of the epimorphism
 $\iota_n\colon {\cal I}_n \subset \operatorname{Ker}(\iota_n)$,

\item[$(2)$] Let $d_{2},\ldots,d_n$ $($resp.\ ${\hat d}_2, \ldots, {\hat d}_n )$
be the Jucys--Murphy elements $($resp.\ dual JM-elements$)$ in the algebra
 ${\cal {NL}}4T_n$ given by the formula~\eqref{equation4.11}.
\end{enumerate}

Then the all elementary
symmetric polynomials $e_k(d_2,\ldots,d_n)$ $($resp.\
$e_k({\hat d}_2,\ldots,{\hat d}_n))$ of deg\-ree~$k$, $ 1 \le k < n$,
in the Jucys--Murphy elements $d_2,\ldots,d_n$, $($resp.\ in the dual JM-elements
${\hat d}_2,\ldots,{\hat d}_n)$, commute in the algebra
${\cal {NL}}4T_n$ with the all elements $y_{i,i+1}$, $i=1,\ldots,n-1$.
\end{Proposition}

Therefore, there exists an epimorphism of algebras
${\cal {NL}}4T_n \longrightarrow H_n(q)$, and images of the elements
$e_k(d_2,\ldots,d_n)$ (resp. $e_k({\hat d}_2,\ldots,{\hat d}_n$), $1 \le k < n$, belongs to the center of the
``normalized'' Hecke algebra $H_n(q)$, and in fact {\it generate}
the center of algebra~$H_n(q)$.

Few comments in order:
\begin{enumerate}\itemsep=0pt
\item[(A)] Let $N{\ell}4T_n$ be an associative algebra over~$\Z$ with the set of
generators $\{ y_{ij},\, 1 \le i < j \le n \}$ subject to the set of relations
\begin{enumerate}\itemsep=0pt
\item[(1)] $y_{ij} y_{kl}=y_{kl} y_{ij}$ if $(i-k)(i-l)(j-k)(j-l) > 0$,
\item[(2)] $\big[y_{ij},\sum\limits_{a=i}^{j} y_{ak}\big]=0$ if $i < j < k$.
\end{enumerate}
\end{enumerate}

\begin{Proposition} \label{proposition4.11}\quad
\begin{enumerate}\itemsep=0pt
\item[$(1)$] The algebra $N{\ell}4T_n$ is Koszul and has the Hilbert
series equals to
\begin{gather*}
 {\rm Hilb}(N{\ell}4T_n,t)= \left( \sum_{k=0}^{n-1} (-1)^{k} N(k,n) t^{k} \right)^{-1},
\end{gather*}
where $N(k,n):= {1 \over n} {n \choose k} {n \choose k+1}$ denotes the
Narayana number, i.e., the number of Dyck $n$-paths with exactly~$k$ peaks,
see, e.g., {\rm \cite[$A001263$]{SL}}.
Therefore, $\dim(N{\ell}4T_n)^{!}={1 \over n+1}{2 n \choose n}$, the $n$-th
Catalan number.
\item[$(2)$] Elementary symmetric polynomials $e_k(d_2,\ldots,d_n)$ of degree
$k$, $ 1 \le k < n$, in the Jucys--Murphy elements $d_2,\ldots,d_n$,
commute in the algebra ${N{\ell}}4T_n$ with the all elements $y_{i,i+1}$, $i=1,\ldots,n-1$.
\end{enumerate}
\end{Proposition}

\begin{enumerate}\itemsep=0pt
\item[(B)] The kernel of the epimorphism
${\cal {NL}}4T_n \longrightarrow H_n(q)$ contains the elements
\begin{gather*}
\{y_{i,i+1} y_{i+1,i+2} y_{i,i+1}-y_{i+1,i+2} y_{i,i+1} y_{i+1,i+2},\, i=1,
\ldots,n-2 \}, \\
\big\{T_{i,i+1}^2-\big(q-q^{-1}\big) T_{i,i+1}-1 \big\},
\end{gather*}
as well as the following set of commutators
\begin{gather*}
[y_{ij},e_{k}(d_{i},\ldots,d_j)], \qquad 1 \le k \le j-i+1.
\end{gather*}
It is an interesting {\it task} to f\/ind def\/ining relations among the
Jucys--Murphy elements $\{d_j, j=2,\ldots,n\}$ in the algebra $NL4T_n$ or that
$N{\ell}4T_n$. We {\it expect} that the Jucys--Murphy element~$d_k$ satisf\/ies
the following relation (= minimal polynomial) in the {\it Hecke algebra}
$H_n(q)$, $n \ge k$,
\begin{gather*}
 \prod_{a=1}^{k-1} \left(d_k-{q-q^{2a+1} \over 1-q^{2}}\right) \left(d_k+{q^{-1}-q^{-2a-1}
\over 1-q^{-2}}\right)=0.
\end{gather*}
\end{enumerate}

\subsubsection{On relations among JM-elements in Hecke algebras}\label{section4.5.1}

Let $H_n(q)$ be the ``normalized'' Hecke algebra of type $A_n$, see
Remark~\ref{remark4.10}. Let $\lambda \vdash n$ be a~partition of~$n$. For a~box
$x=(i,j) \in \lambda$ def\/ine
\begin{gather*}
c_{\lambda}(x;q):=q {1-q^{2 (j-i)} \over 1-q^2}.
\end{gather*}
It is clear that if $q=1$, $c_{q=1}(x)$ is equal to the content $c(x)$ of
a box $x \in \lambda$. Denote by
\begin{gather*}
\Lambda_q^{(n)} = \Z\big[q,q^{-1}\big] [z_1,\ldots,z_n]^{{\mathbb S}_n}
\end{gather*}
 the space of symmetric polynomials over the ring $\Z[q,q^{-1}]$ in variables
$\{z_1,\ldots,z_n \}$.

\begin{Definition} \label{definition4.26} Denote by $J_q^{(n)}$ the set of symmetric polynomials $f \in
\Lambda_q^{(n)}$ such that for any {\it partition} $\lambda \vdash n$
one has
\begin{gather*}
f(c_{\lambda}(x;q) \vert x \in \lambda) = 0.
\end{gather*}
\end{Definition}

For example, one can check that symmetric polynomial
\begin{gather*}
 e_1^2-\big(q^2+1+q^{-2}\big) e_2 -2\big(q-q^{-1}\big) e_1 -3
\end{gather*}
belongs to the set $J_q^{(3)}$.

Finally, denote by $ {\mathbb J}_q^{(n)}$~the ideal in the ring $\Z[q,q^{-1}]
[z_1, \ldots,z_n]$ generated by the set $J_q^{(n)}$.

\begin{Conjecture}\label{conjecture4.9}
 The algebra over $\Z[q,q^{-1}]$ generated by the Jucys--Murphy
elements $d_2$, $\ldots,d_n$ corresponding to the Hecke algebra $H_n(q)$ of
type $A_{n-1}$, is isomorphic to the quotient of the algebra $\Z[q,q^{-1}][z_1,\ldots,z_n]$ by the ideal ${\mathbb J}_q^{(n)}$.
\end{Conjecture}

It seems an interesting {\it problem} to f\/ind a minimal set of generators for
the ideal ${\mathbb J}_q^{(n)}$.

\begin{Comments}\label{comments4.6}
 Denote by $JM(n)$ the algebra over $\Z$ generated by the
JM-elements $d_2$, $\ldots,d_n$,
$\deg(d_i)=1$, $\forall\, i$, corresponding to the symmetric group
${\mathbb S}_n$. In this case one can check Conjecture~\ref{conjecture4.9} for $n <8$, and
compute the Hilbert polynomial(s) of the associated graded algebra(s)
${\rm gr}(JM(n))$. For example\footnote{I would like to thank Dr.\ S.~Tsuchioka for computation the Hilbert
 polynomials ${\rm Hilb}(JM(n),t)$, as well as the sets of def\/ining relations
among the Jucys--Murphy elements in the symmetric group $\mathbb{S}_n$ for
$n \le 7$.}
\begin{gather*}
{\rm Hilb}({\rm gr}(JM(2) ,t)=(1,1), \qquad {\rm Hilb}({\rm gr}(JM(3) ,t)=(1,2,1),\\
{\rm Hilb}({\rm gr}(JM(4) ,t)=(1,3,4,2), \qquad {\rm Hilb}({\rm gr}(JM(5) ,t)=(1,4,8,9,4),\\
{\rm Hilb}({\rm gr}(JM(6) ,t)=(1,5,13,21,21,12,3), \\
{\rm Hilb}({\rm gr}(JM(7) ,t)=(1,6,19,40,59,60,37,10).
\end{gather*}

It seems an interesting {\it task} to f\/ind a combinatorial interpretation of
the polynomials \linebreak
${\rm Hilb}({\rm gr}(JM(n)),t)$ in terms of standard Young tableaux of
size~$n$.
\end{Comments}

Let $\{ \chi^{\lambda},\, \lambda \vdash n \}$ be the characters of the
irreducible representations of the symmetric group~$\mathbb{S}_n$, which form
a basis of the center ${\cal{Z}}_n$ of the group ring $\Z[\mathbb{S}_n]$. The
famous result by A.~Jucys~\cite{J} states that for any symmetric polynomial
$f(z_1,\ldots,z_n)$ the character expansion of $f(d_2,\dots,d_n,0) $ $\in
{\cal{Z}}_n$ is
\begin{gather*}
f(d_2,\ldots,d_n,0)= \sum_{\lambda \vdash n} \frac{f(C_{\lambda})}{H_{\lambda}} \chi^{\lambda},
\end{gather*}
where $H_{\lambda}= \prod\limits_{x \in \lambda}~h_x$ denotes the product of all
{\it hook-lengths} of $\lambda$, and $C_{\lambda}:= \{c(x) \}_{x \in
\lambda}$ denotes the set of {\it contents} of all boxes of $\lambda$.

Recall that the Jucys--Murphy elements $\big\{ d_j^{H}\big\}_{2 \le j \le n}$ in the
(normalized) Hecke algebra $H_n(q)$ are def\/ined as follows: $d_{j}^{H}:= \sum\limits_{i < j} T_{(ij)}$, where $T_{(ij)}:= T_{i} \cdots T_{j-1} T_{j} T_{j-1} \cdots T_{i}$. Finally denote by $H_{\lambda}(q)$ and $C_{\lambda}^{(q)}$ the
hook polynomial and the set $\{c_{\lambda}{x;q) \}_x \in \lambda}$. Then for
any symmetric polynomial $f(z_1,\ldots,z_n)$ one has
\begin{gather*}
f\big(d_2^H,\ldots,d_n^H,0\big)= \sum_{\lambda \vdash n} \frac{f(C_{\lambda}^{(q)})}
{H_{\lambda}(q)} \chi_{q}^{\lambda},
\end{gather*}
where $\chi_{q}^{\lambda}$ denotes the $q$-character of the algebra~$H_{n(q)}$.

Therefore, if $f \in J_{q}^{(n)}$, then $f\big(d_{2}^H,\ldots,d_{n}^{H},0\big)=0$. It
is an open {\it problem} to prove/disprove that if
$f\big(d_{2}^H,\ldots,d_{n}^{H},0\big)=0$, then $f\big(C_{\lambda}^{(q)}\big) = 0$ for all
partitions of size $n$ (even in the case $q=1$).

\subsection{Extended nil-three term relations algebra and DAHA, cf.~\cite{Ch}}\label{section4.6}

Let $ A:=\{q,t, a,b,c,h,e,f, \ldots \}$ be a set of parameters.

\begin{Definition}\label{definition4.27} Extended nil-three term relations algebra
$3{\mathfrak{T}}_{n}$ is an associative algebra over ${\mathbb{Z}}[q^{\pm 1},
t^{\pm 1},a,b,c,h,e,\ldots]$ with the set of generators $\{u_{i,j},\, 1
\le i \not= j \le n, x_i,\, 1 \le i \le n, \pi \}$ subject to the set of relations
\begin{enumerate}\itemsep=0pt
\item[(0)] $u_{i,j}+u_{j,i}=0$, $u_{i,j}^2 =0$,

\item[(1)] $x_i x_j=x_j x_i$, $u_{i,j} u_{k,l}=u_{k,l} u_{i,j}$ if $i$, $j$, $k$, $l$ are
distinct,

\item[(2)] $ x_i u_{kl}= u_{k,l} x_i$ if $i \not= k,l$,

\item[(3)] $x_{i} u_{i,j}= u_{i,j} x_j +1$, $ x_{j} u_{i,j}=u_{i,j} x_{i} -1$,

\item[(4)] $u_{i,j} u_{j,k}+u_{k,i} u_{i,j}+u_{j,k} u_{k,i} =0$ if $i$, $j$, $k$ are
distinct,

\item[(5)] $\pi x_{i}= x_{i+1} \pi$ if $1 \le i < n$, $\pi x_{n} = t^{-1} x_{1}
 \pi$,

\item[(6)] $\pi u_{ij}=u_{i+1,j+1}$ if $1 \le i < j < n$, $\pi^{j} u_{n-j+1,n}=
t u_{1,j} \pi^{j}$, $2 \le j \le n$.
\end{enumerate}
\end{Definition}

Note that the algebra $3\mathfrak{T}_n$ contains also the set of elements
$\{ \pi^a u_{jn}, \, 1 \le a \le n-j \}$.
 \begin{Definition}[cf.\ \cite{LSZ}]\label{definition4.28} Let $1 \le i < j \le n$, def\/ine
\begin{gather*}
T_{i,j}= a+(b x_i+c x_j+h+e x_i x_j) u_{i,j}.
\end{gather*}
\end{Definition}

\begin{Lemma} \label{lemma4.6}\quad
\begin{enumerate}\itemsep=0pt
\item[$(1)$] $T_{i,j}^2 =(2a+b-c) T_{i,j} -a(a+b-c)$ if $a=0$, then $T_{ij}^2 =
(b-c) T_{ij}$.

\item[$(2)$] Coxeter relations
\begin{gather*}
T_{i,j}T_{j,k}T_{i,j} = T_{j,k}T_{i,j}T_{j,k},
\end{gather*}
are valid, if and only if
the following relation holds
\begin{gather}\label{equation4.16}
(a+b)(a-c)+h e =0.
\end{gather}
\item[$(3)$] Yang--Baxter relations
\begin{gather*}
T_{i,j} T_{i,k} T_{j,k}= T_{j,k} T_{i,k} T_{i,j}
\end{gather*}
are valid if and only if $b=c=e=0$, i.e., $T_{ij}=a+d u_{ij}$.

\item[$(4)$] $T_{ij}^2=1$ if and only if $a= \pm 1$, $c=b \pm 2$, $he= (b\pm 1)^2$.

\item[$(5)$] Assume that parameters $a$, $b$, $c$, $h$, $e$ satisfy the conditions~\eqref{equation4.16} and
 that $b c+1=h e$. Then
\begin{gather*}
T_{ij} x_i T_{ij} =x_j+ (h+(a+b)(x_i+x_j)+e x_i x_j) T_{ij}.
\end{gather*}

\item[$(6)$] Quantum Yang--Baxterization. Assume that parameters $a$, $b$, $c$, $h$, $e$
satisfy the conditions~\eqref{equation4.16} and that $\beta:=2 a+b-c \not= 0$. Then
$($cf.\ {\rm \cite{IK, LLT}} and the literature quoted therein$)$
the elements $R_{ij}(u,v):= 1+ \frac{\lambda-\mu}
{\beta \mu} T_{ij}$ satisfy the twisted quantum Yang--Baxter relations
\begin{gather*} R_{ij}(\lambda_i,\mu_j) R_{jk}(\lambda_i,\nu_k) R_{ij}(\mu_j,\nu_k)= R_{jk}(\mu_j,\nu_k) R_{ij}(\lambda_i,\nu_k) R_{jk}(\lambda_i,\mu_j),
 \qquad i < j < k,
 \end{gather*}
where $\{\lambda_i,\mu_i,\nu_i \}_{1 \le i \le n}$ are parameters.
\end{enumerate}
\end{Lemma}

\begin{Corollary}\label{corollary4.6}
 If $(a+b)(a-c)+h e=0$, then for any permutation
$w \in {\mathbb{S}}_n$ the element
\begin{gather*}
T_w:=T_{i_{1}} \cdots T_{i_{l}} \in 3 \mathfrak{T}_n,
\end{gather*}
 where $w=s_{i_{1}} \cdots s_{i_{l}}$ is any reduced decomposition of~$w$, is well-defined.
\end{Corollary}
\begin{Example} \label{example4.8}
Each of the set of elements
\begin{gather*}
s_i^{( h)}=1+(x_{i+1}-x_i + h) u_{i,i+1}
\end{gather*}
 and
 \begin{gather*}
 t_i^{(h)}=-1+(x_i-x_{i+1}+ h(1+x_i)(1+x_{i+1}) u_{ij}, \qquad i=1,\ldots,n-1,
 \end{gather*}
by itself generate the symmetric group ${\mathbb{S}}_n$.
\end{Example}

\begin{Comments} \label{comments4.7}
Let $A=(a,b,c,h,e)$ be a sequence of integers satisfying the conditions~\eqref{equation4.16}. Denote by $\partial_i^{A}$ the divided dif\/ference operator
\begin{gather*}
\partial_{i}^{A}= (a+(b x_i+c x_{i+1}+h+e x_i x_{i+1}) \partial_i, \qquad i=1,\ldots,n-1.
\end{gather*}
It follows from Lemma~\ref{lemma4.5} that the operators $\{ \partial_{i}^{A}\}_{1 \le i
\le n}$ satisfy the Coxeter relations
\begin{gather*}
 \partial_{i}^{A} \partial_{i+1}^{A} \partial_{i}^{A}=\partial_{i+1}^{A} \partial_{i}^{A} \partial_{i+1}^{A},\qquad i=1,\ldots,n-1.
 \end{gather*}
\end{Comments}

\begin{Definition}\label{definition4.29}\quad
\begin{enumerate}\itemsep=0pt
\item[(1)] Let $w \in \mathbb{S}_n$ be a permutation. Def\/ine the generalized
Schubert polynomial corresponding to permutation~$w$ as follows
\begin{gather*}
\mathfrak{S}_w^{A}(X_n) = \partial_{w^{-1} w_{0}}^{A} x^{\delta_n},
\end{gather*}
 where
\begin{gather*}
x^{\delta_n}:= x_1^{n-1} x_2^{n-2} \cdots x_{n-1},
\end{gather*}
and $w_0$ denotes the longest element in the symmetric group~$\mathbb{S}_{n}$.

\item[(2)] Let $\alpha$ be a composition with at most~$n$ parts, denote by
$w_{\alpha} \in \mathbb{S}_n$ the permutation such that $w_{\alpha}(\alpha)=
\overline{\alpha}$, where $\overline{\alpha}$ denotes a unique partition
corresponding to composition~$\alpha$.
\end{enumerate}
\end{Definition}

\begin{Proposition}[\cite{K5}] \label{proposition4.12} Let $w \in \mathbb{S}_n$ be a permutation.
\begin{itemize}\itemsep=0pt
\item If $A=(0,0,0,1,0)$, then $\mathfrak{S}_{w}^{A}(X_n)$ is equal to
the Schubert polynomial $\mathfrak{S}_{w}(X_n)$.

\item If $A=(-\beta,\beta,0,1,0)$, then $\mathfrak{S}_{w}^{A}(X_n)$ is
equal to the $\beta$-Grothendieck polynomial $\mathfrak{G}_w^{({\beta})}(X_n)$
introduced in~{\rm \cite{FK1}}.

\item If $A=(0,1,0,1,0)$ then $\mathfrak{S}_{w}^{A}(X_n)$ is equal to
the dual Grothendieck polynomial~{\rm \cite{K5, L}}.

\item If $A=(-1,2,0,1,1)$, then $\mathfrak{S}_{w}^{A}(X_n)$ is equal to
the Di~Francesco--Zinn-Justin polynomials and studied in~{\rm \cite{DZ0,DZ,DZ1}} and {\rm \cite{K5}}.
\end{itemize}

In all cases listed above the polynomials $\mathfrak{S}_{w}^{A}(X_n)$ have
non-negative integer coefficients.
\begin{itemize}\itemsep=0pt
\item If $A=(1,-1,1,-h,0)$, then $\mathfrak{S}_{w}^{A}(X_n)$ is equal to
the $h$-Schubert polynomials introduced in~{\rm \cite{K5}}.
\end{itemize}

Define the generalized key or Demazure polynomial corresponding to a
composition~$\alpha$ as follows
\begin{gather*}
 K_{\alpha}^{A}(X_n)= \partial_{w_{\alpha}} x^{\overline \alpha}.
\end{gather*}
\begin{itemize}\itemsep=0pt
\item If $A=(1,0,1,0,0)$, then $K_{\alpha}^{A}(X_n)$ is equal to key $($or
Demazure$)$ polynomial corresponding to~$\alpha$.

\item If $A=(0,0,1,0,0)$, then $K_{\alpha}^{A}(X_n)$ is equal to the
reduced key polynomial introduced in~{\rm \cite{K5}}.

\item If $A=(1,0,1,0,\beta)$, then $K_{\alpha}^{A}(X_n)$ is equal to the
key Grothendieck polynomial $KG_{\alpha}(X_n)$ introduced in~{\rm \cite{K5}}.

\item If $A=(0,0,1,0,\beta)$, then $K_{\alpha}^{A}(X_n)$ is equal to
the reduced key Grothendieck polynomial~{\rm \cite{K5}}.
\end{itemize}

 In all cases listed above the polynomials $\mathfrak{S}_{w}^{A}(X_n)$ have
non-negative integer coefficients.
\end{Proposition}

{\samepage
\begin{Exercises}\label{exercises4.5}\quad
\begin{enumerate}\itemsep=0pt
\item[(1)] Let $b$, $c$, $h$, $e$ be a collection of integers, def\/ine elements $P_{ij}:=
f_{ij} u_{ij} \in 3 \mathfrak{T}$, where $f_{ij}:= b x_i+c x_j+h+e x_i x_j$.
Show that
\begin{itemize}\itemsep=0pt
\item $P_{ij}^2= (b-c) P_{ij}$,

\item $P_{ij} P_{jk} P_{ij}=f_{ij} f_{ik} f_{jk} u_{ij} u_{jk} u_{ij}+(b c -e h) P_{ij}$,\\
$P_{jk} P_{ij} P_{jk} = f_{ij} f_{ik} f_{jk} u_{ij} u_{jk} u_{ij} -(b c- e h) P_{jk}$.
\end{itemize}

\item[$(2)$] Assume that $a=q$, $b=-q$, $c=q^{-1}$, $h=e=0$, and introduce elements
\begin{gather*}
e_{ij}:=\big(q x_i-q^{-1} x_j\big) u_{ij}, \qquad 1 \le i < j < k \le n.
\end{gather*}
\begin{enumerate}\itemsep=0pt
\item[(a)] Show that if $i$, $j$, $k$ are distinct, then
\begin{gather*}
 e_{ij} e_{jk} e_{ij} = e_{ij}+\big(q x_i-q^{-1} x_j\big)\big(q x_i - q^{-1} x_k\big)\big(q x_j -
q^{-1} x_k\big) u_{ij} u_{jk} u_{ij},\\
 e_{ij}^2= \big(q+q^{-1}\big)e_{ij}.
 \end{gather*}
\item[(b)] Assume additionally that
\begin{gather*}
u_{ij} u_{jk} u_{ij} = 0, \qquad \text{if $i$, $j$, $k$ are distinct}.
\end{gather*}
Show that the elements $\{e_{i}:=e_{i,i+1},\, i=1,\ldots,n-1 \}$,
generate a subalgebra in $3{\mathfrak{L}}_n$ which is isomorphic to the
 Temperley--Lieb algebra $TL_n(q+q^{-1})$.
 \end{enumerate}

\item[(3)] Let us set $T_i:=T_{i,i+1}$, $i=1,\ldots,n-1$, and def\/ine
\begin{gather*}
T_0:= \pi T_{n-1} \pi^{-1} .
\end{gather*}
Show that if $(a+b)(a-c)+e h=0$, then
\begin{gather*}
 T_1 T_0 T_1 =T_1 T_0 T_1, T_{n-1} T_0 T_{n-1}=T_0 T_{n-1} T_0,
 \end{gather*}
 Recall that $T_i^2= (2 a+b-c) T_i -a(a+b-c)$, $0 \le i \le n-1$.
\end{enumerate}
\end{Exercises}}

In what follows we take $a=q$, $b=-q$, $c=q^{-1}$, $h=e=0$. Therefore,
 $T_{i,j}^2=(q-q^{-1})T_{i,j}+1$. We denote by ${\cal{H}}_n(q)$ a subalgebra
in~$3 {\mathfrak{T}}_n$ generated by the elements $T_i:= T_{i,i+1}$, $i=1,\ldots, n-1$.

\begin{Remark} \label{remark4.11}
Let us stress on a dif\/ference between elements $T_{ij}$ as
a part of generators of the algebra $3{\mathfrak{T}}_n$, and the elements
\begin{gather*}
T_{(ij)}:= T_i \cdots T_{j-1} T_{j} T_{j-1} \cdots T_{i} \in {\cal{H}}_n(q).
\end{gather*}
Whereas one has $[T_{ij},T_{kl}]=0$ if $i$, $j$, $k$, $l$ are distinct, the
relation $[T_{(ij)},T_{(kl)}]=0$ in the algebra ${\cal{H}}_n(q)$ holds
(for general~$q$ and $i \le k$) if and only if either one has $i <j < k < l$ or $i< k < l <j$.
\end{Remark}

\begin{Lemma}\label{lemma4.7}\quad
\begin{enumerate}\itemsep=0pt
\item[$(1)$] $T_{ij} T_{kl}= T_{kl} T_{ij}$ if $i$, $j$, $k$, $l$ are distinct,

\item[$(2)$] $T_{i,j} x_i T_{i,j} =x_{j}$ if $1 \le i < j \le n$,

\item[$(3)$] $\pi T_{i,j} = T_{i+1,j+1}$ if $1 \le i < j < n$, $\pi^j T_{n-j+1,n}= T_{1,j} \pi^j$.
\end{enumerate}
\end{Lemma}
\begin{Definition}\label{definition4.30}
 Let $1 \le i < j \le n$, set
\begin{gather*}
Y_{i,j}= T_{i-1,j-1}^{-1} T_{i-2,j-2}^{-1} \cdots T_{1,j-i+1}^{-1} \pi^{j-i}
~T_{n-j+i,n} \cdots T_{i+1,j+1} T_{i,j}, \qquad 1 \le i < j \le n,
\end{gather*}
and $Y_n=T_{n-1,n}^{-1} \cdots T_{1,2}^{-1} \pi$.
\end{Definition}

For example,
\begin{gather*}
Y_{1,j}=\pi^{j-1} T_{n-j+1,n} \cdots T_{1,j}, \qquad j \ge 2,\\
Y_{2,j}=T_{1,j-1}^{-1}\pi^{j-2} T_{n-j+2,n}\cdots T_{2,j}, \qquad \text{and so on},\\
Y_{j-1,j} = T_{j-2,j-1}^{-1} \cdots T_{1,2}^{-1} \pi T_{n-1,n} \cdots T_{j-1,j}.
\end{gather*}

\begin{Proposition}\label{proposition4.13}\quad
\begin{enumerate}\itemsep=0pt
\item[$(1)$] $x_j x_j T_{ij}=T_{ij} x_i x_j$,

\item[$(2)$] $Y_{i,j}= T_{i,j} Y_{i+1,j+1} T_{i,j}$ if $1 \le i < j < n$,

\item[$(3)$] $Y_{i,j} Y_{i+k,j+k}=Y_{i+k,j+k} Y_{i,j}$ if $1 \le i < j \le n-k$,

\item[$(4)$] one has $x_{i-1}Y_{i,j}^{-1}=Y_{i,j}^{-1}x_{i-1}T_{i-1,j-1}^{2}$, $2 \le i < j \le n$,

\item[$(5)$] $Y_{i,j} x_1 x_2 \cdots x_n = t x_1 x_2 \cdots x_n Y_{i,j}$,

\item[$(6)$] $x_i Y_1 Y_2 \cdots Y_n = t^{-1} Y_1 Y_2 \cdots Y_n x_i$,
\end{enumerate}
where we set $Y_i:= Y_{i,i+1}$, $ 1 \le i < j < n$.
\end{Proposition}
\begin{Conjecture} \label{conjecture4.10}
Subalgebra of $3{\mathfrak {T}}_n$ generated by the
elements $\{T_i:= T_{i,i+1}, \, 1 \le i < n, \, Y_{1}, \ldots, Y_n$,
and $x_1,\ldots,x_n \}$, is isomorphic to the double affine Hecke algebra
${\rm DAHA}_{q,t}(n)$.
\end{Conjecture}

Note that the algebra $3{\mathfrak{T}}_n$ contains also two additional commutative subalgebras generated by {\it additive} $ \big\{ \theta_i=\sum\limits_{j \not= i} u_{ij} \big\}_{1 \le i \le n}$ and {\it multiplicative}
\begin{gather*}
\left\{ \Theta_i = \prod_{a=1}^{i-1} (1-u_{ai}) \prod_{a=i+1}^{n} (1+u_{ia})
\right\}_{1 \le i \le n}
\end{gather*}
 Dunkl elements correspondingly.

Finally we introduce (cf.~\cite{Ch}) a (projective) representation of the
modular group $SL(2,\Z)$ on the {\it extended affine Hecke algebra}
${\widehat{ \cal{H}}_n}$
over the ring $\Z[q^{\pm 1},t^{\pm 1}]$ generated by elements
\begin{gather*}
\{ T_{1},\ldots,T_{n-1} \}, \quad \pi , \quad \text{and}\quad \{ x_1,\ldots,x_n \}.
\end{gather*}
 It is well-known that the group ${\rm SL}(2,\Z)$ can be generated by two matrices
\begin{gather*} \tau_{+}= \left(
\begin{matrix}
1&1\\
0&1
\end{matrix}
\right),\qquad
\tau_{-}= \left(
\begin{matrix}
1&0\\
1&1
\end{matrix}
\right),
\end{gather*}
which satisfy the following relations
\begin{gather*}
 \tau_{+} \tau_{-}^{-1} \tau_{+}= \tau_{-}^{-1} \tau_{+} \tau_{-}^{-1},\qquad
 \big(\tau_{+} \tau_{-}^{-1} \tau_{+}\big)^6 = I_{2 \times 2}.
\end{gather*}
Let us introduce operators~$\tau_{+}$ and~$\tau_{-}$ acting on the extended
af\/f\/ine algebra ${\widehat{ \cal{H}}_n}$. Namely,
\begin{gather*}
\tau_{+}(\pi)= x_{1} \pi, \tau_{+}(T_{i}) = T_{i}, \tau_{+}(x_i) = x_{i},\qquad
\forall\, i,\\
\tau_{-}(\pi)=\pi,\qquad \tau_{-}(T_{i})=T_{i}, \qquad \tau_{-}(x_i)=
\left( \prod_{a=i-1}^{1} T_{a} \right) \pi \left( \prod_{a=n}^{i} T_{a} \right) x_{i}.
\end{gather*}

\begin{Lemma}\label{lemma4.8}\quad
\begin{gather*}
 \tau_{+}(Y_i)= \left(\prod_{a=i-1}^{1} T_{a}^{-1} \right) \left(
\prod_{a=1}^{i-1}T_{a}^{-1} \right) x_{i} Y_{i},\\
\tau_{-}(x_i)= \left(\prod_{a=i-1}^{1} T_{a} \right) \left(
\prod_{a=1}^{i-1}T_{a} \right) Y_{i} x_i, \\
\big(\tau_{+} \tau_{-}^{-1} \tau_{+}\big)(x_i)=Y_i^{-1}= \big(\tau_{-}^{-1} \tau_{+} \tau_{-}^{-1}\big)(x_i),\\
\big(\tau_{+} \tau_{-}^{-1} \tau_{+}\big)(Y_i) = t x_i \left
(\prod_{a=i-1}^{1} T_{a} \right) \left( T_{1} \cdots T_{n-1}\right) \left( \prod_{a= n-1}^{i} T_{a} \right), \qquad i=1,\ldots,n.
\end{gather*}
\end{Lemma}

In the last formula we set $T_{n}=1$ for convenience.

\subsection{Braid, af\/f\/ine braid and virtual braid groups}\label{section4.7}

The main objective of this section is to describe the distinguish abelian subgroup
in the braid group $B_n$, see Proposition~\ref{proposition4.130}$\big(2^{(0)}\big)$, and that in the
Yang--Baxter groups ${\widehat{{\rm YB}}}_n$ and ${\rm YB}_n$, see Proposition~\ref{proposition4.130}$\big(5^{(0)}\big)$ and $\big(6^{(0)}\big)$ correspondingly. As far as I'm aware, these constructions go back to E.~Artin in the case of braid groups, and to C.N.~Yang in the
case of Yang--Baxter group, and nowadays are widely use in the representation
theory of Hecke's type algebras and that of integrable systems. In a~few words,
 by choosing a suitable representation (f\/inite-dimensional or birational) of
 either $B_n$ or ${\rm YB}_n$, or ${\widehat{{\rm YB}}}_n$, one gives rise to a family of
mutually commuting operators acting in the space of a~representation selected. In the
case of braid groups one comes to Jucys--Murphy's type operators/elements, and
 in the case of Yang--Baxter groups one comes to Dunkl's type operators/elements. See, e.g., \cite{IK,IKT}, where it was used the so-called $R$-matrix
representation of the af\/f\/ine braid group of type $C_n^{(1)}$ to construct the
(two boundary) quantum Knizhnik--Zamolodchikov connections with values in the
af\/f\/ine Birman--Murakami--Wenzl algebras.

 To start with, let $n \ge 2$ be an integer.
\begin{Definition}\label{definition4.40} \quad
\begin{itemize}\itemsep=0pt
\item Denote by $\mathbb{S}_{n}$ {\it the symmetric group}
 on $n$ letters, and by $s_{i}$ the simple transposition
$(i,i+1)$ for $1 \le i \le n-1$.
The well-known Moore--Coxeter presentation of the symmetric group has the form
\begin{gather*}
\langle s_{1},\ldots,s_{n-1} \,|\,s_{i}^2=1,\, s_{i}s_{i+1}s_{i}=s_{i+1}s_{i}s_{i+1}
,\, s_{i}s_{j}=s_{j}s_{i}\, {\rm if}\, |i-j| \ge 2 \rangle.
\end{gather*}
Transpositions $s_{ij}:=s_i s_{i+1} \cdots s_{j-2} s_{j-1}
s_{j-2} \cdots s_{i+1} s_{i}$, $1 \le i < j < j \le n$, satisfy the following
set of (def\/ining) relations
\begin{gather*}
 s_{ij}^{2}=1, \qquad s_{ij} s_{kl}=s_{kl} s_{ij} \qquad {\rm if} \quad \{i,j\} \cap \{k,l \}= \varnothing,\\
s_{ij} s_{ik}=s_{jk} s_{ij}=s_{ik} s_{jk}, \qquad s_{ik} s_{ij}=s_{ij} s_{jk}= s_{jk} s_{ik}, \qquad i < j < k.
\end{gather*}

\item {\it The Artin braid group on $n$ strands $B_{n}$} is
def\/ined by generators $\sigma_{1},\ldots,\sigma_{n-1}$ and relations
\begin{gather}\label{equation4.7new}
\sigma_{i}\sigma_{i+1}\sigma_{i}=\sigma_{i+1}\sigma_{i}\sigma_{i+1}, \qquad
1 \le i \le n-2, \qquad \sigma_{i}\sigma_{j}=\sigma_{j}\sigma_{i} \qquad {\rm if} \quad |i-j|
\ge 2.
\end{gather}

\item {\it The monoid of positive braids on $n$ strands
$B_{n}^{+}$}
is a monoid generated by the elements $\sigma_{1},\ldots,\sigma_{n-1}$ subject
to the set of relations~\eqref{equation4.7new}.

\item {\it A new representation of the braid group} \cite{BKL}.
 The Birman--Ko--Lee representation of the braid group~$B_n$ has the set of
 generators $\{ \sigma_{i,j} \,|\, 1 \le i < j \le n \}$ subject to
the Birman--Ko--Lee (def\/ining) relations
\begin{gather*}
\sigma_{i,j} \sigma_{k,l}=\sigma_{k,l} \sigma_{i,j} \qquad {\rm if}\quad
(j-l)(j-k)(i-l)(i-k) > 0,\\
 \sigma_{i,j} \sigma_{i,k}=\sigma_{j,k} \sigma_{i,j}=
\sigma_{i,k} \sigma_{j,k} \qquad {\rm if} \quad 1 \le i < j < k \le n.
\end{gather*}
One can take $\sigma_{i,j}:=(\sigma_{j-1} \cdots \sigma_{i+1})\sigma_{i}
(\sigma_{i+1}^{-1} \cdots \sigma_{j-1}^{-1})$, see~\cite{BB} for details.
It would be well to note that as a~corollary of the Birman--Ko--Lee relations
 one can deduce the {\it $2D$ Coxeter relations} among the Birman--Ko--Lee generators
\begin{gather*}
\sigma_{i,j} \sigma_{j,k} \sigma_{i,j}=\sigma_{j,k} \sigma_{i,j} \sigma_{j,k}, \qquad 1 \le i < j <k \le n.
\end{gather*}

\item {\it The Birman--Ko--Lee monoid ${\rm BKL}_n$} is a monoid generated by
the elements $\sigma_{i,j}$, $1 \le i < j \le n$, subject to the Birman--Ko--Lee
relations.
We denote by ${\rm BKL}(n)$ and called it as
{\it the Birman--Ko--Lee algebra}, the group algebra ${\Q}[{\rm BKL}_n]$ of the Birman--Ko--Lee monoid.
The Hilbert series of the Birman--Ko--Lee algebra ${\rm BKL}(n)$ will
be computed in Section~\ref{section4.7.3}, Theorem~\ref{theorem4.132}.

\item {\it The pure braid group} ${\rm PB}_{n}$ is def\/ined to be
the kernel of the natural (non-split) projection $p\colon B_{n} \longrightarrow \mathbb{S}_{n}$
given by $p(\sigma_{i})=s_{i}$.
It is well-known that the pure braid group ${\rm PB}_{n}$
is generated by the elements
\begin{gather*}
g_{i,j}:= \sigma_{i,j}^2=\sigma_{j-1}\sigma_{j-2}\cdots\sigma_{i+1}
\sigma_{i}^2\sigma_{i+1}^{-1}\cdots\sigma_{j-2}^{-1}\sigma_{j-1}^{-1}
\qquad {\rm for} \quad 1 \le i < j \le n,
\end{gather*}
 subject to the following def\/ining relations
\begin{gather*}
g_{i,j} g_{k,l}=g_{k,l} g_{i,j} \qquad {\rm if} \quad (i-k) (i-l) (j-k) (j-l) > 0,\\
g_{i,j} g_{i,k} g_{j,k}=g_{i,k} g_{j,k} g_{i,j}=g_{j,k} g_{i,j} g_{i,k} \qquad {\rm if}\quad 1 \le i < j < k \le n,\\
g_{i,k} g_{i,l} g_{j,l} g_{k,l}=g_{i,l} g_{j,l} g_{k,l} g_{i,k} \qquad {\rm if} \quad 1 \le i < j <k < l \le n.
\end{gather*}
\end{itemize}

\begin{Comments} It is easy to see that the def\/ining relations for the pure braid group ${\rm PB}_n$ listed above are equivalent to the following list of def\/ining relations
\begin{gather*}
 g_{i,j}^{-1}g_{k,l}g_{i,j}= \begin{cases}
g_{k,l} & \text{if} \ (i-k)(i-l)(j-k)(j-l) > 0, \\
g_{i,l} g_{k,l} g_{i,l}^{-1} & \text{if} \ i < k=j < l , \\
g_{i,l} g_{j,l} g_{k,l} g_{i,l}^{-1} g_{j,l}^{-1} & \text{if} \ i=k < j < l, \\
g_{i,l} g_{j,l} g_{i,l}^{-1} g_{j,l}^{-1} g_{k,l} g_{j,l} g_{i,l} g_{j,l}^{-1} g_{i,l}^{-1} & \text{if} \ i < k < j < l,
\end{cases}
\end{gather*}
commonly used in the literature to describe the def\/ining relations among the generators $\{g_{ij} \}$ of the pure braid group $P_n$, see, e.g., \cite{BB}.
\end{Comments}
\begin{itemize}\itemsep=0pt
\item {\it The affine Artin braid group $B_n^{\rm af\/f}$}, cf.~\cite{OR}, is an extension of the Artin braid group on~$n$ strands $B_n$ by the
element~$\tau$ subject to the set of crossing relations
\begin{gather*}
 \sigma_1 \tau \sigma_1 \tau=\tau \sigma_1 \tau \sigma_1,
\qquad \sigma_i \tau=\tau \sigma_i \qquad {\rm for} \quad 2 \le i \le n-1.
\end{gather*}

\item {\it The virtual braid group ${\rm VB}_n$} is a group generated by
$\sigma_1,\ldots,\sigma_{n-1}$ and $s_1,\ldots,s_{n-1}$ subject to the relations:
\begin{enumerate}\itemsep=0pt
\item[$(1)$] {\it braid relations} $\sigma_1,\ldots,\sigma_{n-1}$ generate the Artin
braid group $B_n$;

\item[$(2)$] {\it Moore--Coxeter relations} $s_1,\ldots,s_{n-1}$ generate the symmetric
group ${\mathbb S}_n$;

\item[$(3)$] {\it crossing relations} $\sigma_i s_j=s_j \sigma_i$ if $|i-j| \ge 2$,
$s_is_{i+1}\sigma_i=\sigma_{i+1}s_is_{i+1}$ if $1 \le i \le n-2$.
\end{enumerate}

\item {\it The virtual pure braid group ${\rm VP}_n$} is def\/ined to
be the kernel of the natural map
\begin{gather*}
 \eta\colon \ {\rm VB}_n \longrightarrow {\mathbb S}_n, \qquad \eta(\sigma_i)=\eta(s_i)=s_i,\qquad
i= 1, \ldots,n-1.
\end{gather*}
\end{itemize}
\end{Definition}

\subsubsection{Yang--Baxter groups}\label{section4.7.1}
\begin{Definition}\quad
\begin{itemize}\itemsep=0pt
\item {\it The quasitriangular Yang--Baxter group $\widehat {{\rm YB}_{n}}$},
cf.~\cite{BEE}, is a group generated by the set of elements
$\{ Q_{i,j},\, 1 \le i \not= j \le n \}$, subject to the set of def\/ining
relations
\begin{enumerate}\itemsep=0pt
\item[$(1)$] $[Q_{i,j},Q_{k,l}]=0$ if $i$, $j$, $k$ and $l$
 are all distinct,

\item[$(2)$] {\it Yang--Baxter relations} $Q_{i,j}Q_{i,k}Q_{j,k}=Q_{j,k}Q_{i,k}Q_{i,j}$ if $i$, $j$, $k$ are distinct.
\end{enumerate}
According to \cite[Theorem~1]{Bar}, the
{\it quasitriangular Yang--Baxter group $\widehat {{\rm YB}_{n}}$} is isomorphic
 to the {\it virtual pure braid group} ${\rm VP}_n$.
\item {\it The Yang--Baxter monoid $\widetilde {{\rm YB}_{n}}$} is a~monoid
generated by the elements $Q_{i,j}$, $1 \le i \not= j \le n$.
Important particular case corresponds to the case when $Q_{i,j}Q_{j,i}=1$
 for all $1 \le i \not= j \le n$.

\item {\it The Yang--Baxter group ${\rm YB}_{n}$} is def\/ined by the set of
generators
$R_{i,j}$, $1 \le i < j \le n$, subject to the set of def\/ining relations
\begin{enumerate}\itemsep=0pt
\item[$(1)$] $R_{i,j}R_{k,l}=R_{k,l}R_{i,j}$ if $i$, $j$, $k$ and $l$ are pairwise
distinct,
\item[$(2)$] $R_{i,j}R_{i,k}R_{j,k}=R_{j,k}R_{i,k}R_{i,j}$ if
$1 \le i < j <k \le n$.
\end{enumerate}
\end{itemize}
\end{Definition}

\subsubsection{Some properties of braid and Yang--Baxter groups}\label{section4.7.2}

For the sake of convenience and future references, below we state some basic
properties of the groups $P_{n}$, ${\rm YB}_{n}$ and $B_n^{\rm af\/f}$.

\begin{Proposition}\label{proposition4.130}
Let $F_m$ denotes the free group with~$m$ generators.
\begin{enumerate}\itemsep=0pt
\item[$\big(1^{0}\big)$] The elements $g_{1,n},g_{2,n},\ldots,g_{n-1,n}$ generate
a~free normal subgroup $F_{n-1}$ in $P_{n},$ and $P_{n}=P_{n-1} \ltimes \langle
g_{1,n},g_{2,n},\dots,g_{n-1,n} \rangle$. Hence $P_{n}$ is an iterated
extension of free groups.

\item[$\big(2^{0}\big)$] Let us consider the following elements in the group
$B_n^{\rm af\/f}$:
\begin{gather*} \gamma_1=\tau,\qquad \gamma_i=
\prod_{j=i-1}^{1}\sigma_j \tau \prod_{j=1}^{i-1}\sigma_j,\qquad 2 \le i \le n.
\end{gather*}
Then
\begin{enumerate}\itemsep=0pt
\item[$(a)$] commutativity, $\gamma_i \gamma_j=\gamma_j \gamma_i$ for all $1 \le i,j \le n$;
\item[$(b)$] the elements $\gamma_1,\dots,\gamma_n$
generate a free abelian subgroup of rang $n$ in $B_n^{\rm af\/f}$.\footnote{We refer the reader to~\cite{OR} for more details about af\/f\/ine braid
groups.
Here we only remark that the type $A$ af\/f\/ine Weyl groups
$\widehat{\mathbb {S}}_n$,
the Hecke algebras~$H_{n,q}$, the af\/f\/ine Hecke
algebras $\widehat {H}_{n,q}$, the Ariki--Koike, or cyclotomic Hecke,
 algebras $H_{r,1,n}$, the af\/f\/ine and cyclotomic Birman--Murakami--Wenzl
algebras ${\cal Z}_{r,1,n}$, all can be obtained as certain quotients of the
group algebra $\C B_n^{\rm af\/f}$ of the af\/f\/ine braid group.}
\end{enumerate}

\item[$\big(3^{0}\big)$] Let us introduce elements
\begin{gather*}
D_{i,j}:=\sigma_{j-1}\sigma_{j-2}\cdots\sigma_{i+1}~
\sigma_{i}^2~\sigma_{i+1}\cdots\sigma_{j-2}\sigma_{j-1}=
\prod_{i \le a < j}g_{a,j} \in P_{n},\\
F_{i,j}:=
\sigma_{n-j} \sigma_{n-j+1} \cdots \sigma_{n-i-1} \sigma_{n-i}^{2}
 \sigma_{n-i-1} \cdots \sigma_{n-j+1} \sigma_{n-j}=
\prod_{a=j}^{i+1} g_{n-a,n-i} \in P_n,
\end{gather*}
where $1 \le i < j \le n$. For example,
\begin{gather*}
\begin{split}
& D_{i,i+1}=\sigma_{i}^2,\qquad D_{i,i+2}=
\sigma_{i+1}\sigma_{i}^2\sigma_{i+1},\qquad F_{i,i+1}=\sigma_{n-i}^{2},\\
& F_{i,i+2}=\sigma_{n-i-1} \sigma_{n-i}^2 \sigma_{n-i-1}
 \end{split}
 \end{gather*}
and etc.
Then
\begin{itemize}\itemsep=0pt
\item For each $j=3, \ldots,n,$ the element $D_{1,j}$ commutes with
 $\sigma_{1},\ldots,\sigma_{j-2}$.

\item The elements $D_{1,2},D_{1,3},\ldots,D_{1,n}$ $($resp.\
$F_{1,2}, F_{1,3},\ldots,F_{1,n})$ generate a free
abelian subgroup in~$P_{n}$.

\item If $n \ge 3$, the element
\begin{gather*}
 \prod_{2 \le j \le n}D_{1,j} = \prod_{2 \le j \le n} F_{1,j}=
(\sigma_{1}\cdots\sigma_{n-1})^n
\end{gather*}
generates the center of the braid group~$B_{n}$ and that of the pure braid
group $P_{n}$.

\item $D_{i,j}D_{i,j+1}D_{j,j+1}=D_{j,j+1}D_{i,j+1}D_{i,j}$
if $i < j$.

\item Consider the elements $s:=\sigma_1 \sigma_2 \sigma_1$,
$t:= \sigma_1 \sigma_2$ in the braid group $B_3$. Then $s^2=t^3$ and the
element $c:=s^2$ generates the center of the group~$B_3$. Moreover,
\begin{gather*}
 B_3 / \langle c \rangle \cong {\rm PSL}_2(\Z), \qquad
B_3 / \langle c^2 \rangle \cong {\rm SL}_2(\Z).
\end{gather*}
\end{itemize}

\item[$\big(4^{0}\big)$] Let us introduce the following elements in the
quasitriangular Yang--Baxter group $\widehat {{\rm YB}_n}$:
\begin{gather*}
 C_{i,j}= \left( \prod_{a=j-1}^{i} Q_{a,j} \right) \left( \prod_{b=i}^{j-1} Q_{j,b} \right),\qquad
f_{i,j}=\left(\prod_{a=j-1}^{i+1}Q_{a,j} \right) Q_{i,j}Q_{j,i} \left( \prod_ {b=i+1}^{j-1} Q_{b,j} \right)^{-1}.
\end{gather*}
Then
\begin{itemize}\itemsep=0pt
\item The elements $C_{1,2},C_{1,3},\ldots,C_{1,n}$ generate a free
abelian subgroup in~$\widehat {{\rm YB}_n}$.

\item The elements $ f_{i,j}$, $1 \le i < j \le n$, generate a subgroup in $\widehat {{\rm YB}_n} $, which
is isomorphic to the pure braid group~$P_n$.\footnote{It is enough to check that the elements $\{f_{i,j},\, 1 \le i < j \le n \}$ satisfy the def\/ining relations for the pure braid group $P_n$ only in the case $n=4$.
Let us prove that
\begin{gather*}
f_{1,4} f_{2,4} f_{3,4} f_{1,3}=f_{1,3} f_{1,4} f_{2,4} f_{3,4}.
\end{gather*}
Other relations are simple and can be checked in a similar fashion.

Let
\begin{gather*}
{\rm l.h.s.} =f_{1,4} f_{2,4} f_{3,4} f_{1,3}=Q_{34}Q_{24}Q_{14}Q_{41}Q_{42}Q_{43}
Q_{23}Q_{13}Q_{31}Q_{23}^{-1},\\
{\rm r.h.s.} =f_{1,3} f_{1,4} f_{2,4} f_{3,4}=Q_{23}Q_{13}Q_{31}{\boldsymbol{Q_{23}^{-1}Q_{34}Q_{24}}}Q_{14}Q_{41}Q_{42}Q_{43}.
\end{gather*}
Now we are going to apply the Yang--Baxter relations
\begin{gather*}
Q_{23}^{-1}Q_{34}Q_{24}=Q_{24}Q_{34}Q_{23}^{-1},\qquad Q_{23}^{-1}Q_{42}Q_{43}=Q_{43}Q_{42}Q_{23}^{-1},\qquad
Q_{31}Q_{34}Q_{14}=Q_{14}Q_{34}Q_{31}.
\end{gather*}
Therefore,
\begin{gather*}
{\rm r.h.s.} = Q_{23}Q_{13}Q_{31}{\boldsymbol{Q_{23}^{-1}Q_{34}Q_{24}}}Q_{14}Q_{41}Q_{42}Q_{23}^{-1}={\boldsymbol{Q_{23}Q_{24}Q_{34}}}Q_{14}Q_{13}
{\boldsymbol{Q_{31}Q_{41}Q_{43}}}Q_{42}Q_{23}^{-1}\\
\hphantom{{\rm r.h.s.}}{}
=Q_{34}Q_{24}Q_{14}Q_{23}{\boldsymbol{Q_{13}Q_{43}Q_{41}}}Q_{42}Q_{31}Q_{23}^{-1}=Q_{34}Q_{24}Q_{14}Q_{41}{Q_{23}Q_{43}Q_{42}}Q_{13}Q_{31}Q_{23}^{-1}\\
\hphantom{{\rm r.h.s.}}{}
=Q_{34}Q_{24}Q_{14}Q_{41}Q_{42}Q_{43}Q_{23}Q_{13}Q_{31}Q_{23}^{-1}={\rm l.h.s.}
\end{gather*}
}
\end{itemize}

\item[$\big(5^{0}\big)$] Assume that the following additional relations in $\widehat {{\rm YB}_n}$ are satisfied
\begin{gather*}
Q_{i,j}Q_{j,i}=Q_{j,i}Q_{i,j},\qquad Q_{k,l}Q_{i,j}Q_{j,i}=
Q_{j,i}Q_{i,j}Q_{k,l}
\end{gather*}
if $i \not= j$ and $k \not= l$. In other words, the elements $Q_{i,j}$ and
$Q_{j,i}$ commute, and the elements $Q_{i,j}Q_{j,i}=Q_{j,i} Q_{i,j}$ are central.
Under these assumptions, we have
that the elements
\begin{gather*}
\Theta_i:=\prod_{j=i-1}^{1}Q_{j,i}
\prod_{j=n}^{i+1}Q_{i,j}, \qquad \bar \Theta_{i}:= \prod_{j=i+1}^{n}Q_{j,i}
\prod_{j=1}^{i-1}Q_{i,j}, \qquad 1 \le i \le n,
\end{gather*}
 satisfy the following relations
\begin{gather*}
[\Theta_{i},\Theta_{j}]=0=[\bar\Theta_{i},\bar\Theta_{j}],\\
\Theta_{i} \bar\Theta_{i}=\prod_{j \not= i}Q_{i,j}Q_{j,i}=
\bar\Theta_{i} \Theta_{i},\qquad \prod_{i=1}^{n}\Theta_{i}=
\prod_{1 \le i \not= j \le n}Q_{i,j}Q_{j,i}= \prod_{i=1}^{n}\bar\Theta_{i}.
\end{gather*}

\item[$\big(6^{0}\big)$]
In the special case $Q_{i,j}Q_{j,i}=1$ for all $i \not= j$, the following statement holds:
the elements
\begin{gather*}
\Theta_{j}=\prod_{a=j-1}^{1}R_{a,j}^{-1}
\prod_{b=n}^{j+1}R_{j,b}, \qquad 1 \le j \le n-1,
\end{gather*}
 generate a subgroup
in the Yang--Baxter group ${\rm YB}_{n}$, which is isomorphic to the free abelian group of rang~$n-1$.
\end{enumerate}
\end{Proposition}

\subsubsection{Artin and Birman--Ko--Lee monoids}\label{section4.7.3}

Let $(W,S)$ be a f\/inite Coxeter group, $B(W)$ and $B^{+}(W)$ be
the corresponding braid group and monoid of positive braids.
Denote by $P_{W}(s,t)=
\sum\limits_{i \ge 0,\,j \ge 0}B_{{\Q}[B^{+}(W)]}(i,j)s^it^j$ the Poincar\'e
 polynomial of the group algebra over $\Q$ of the monoid $B^{+}(W)$.

\begin{Conjecture} $P_{W}(s,1)=(s+1)^{|S|}$.
\end{Conjecture}

It is known \cite{De,Sa}
 that the Hilbert series of the
group algebra of the monoid $B^{+}(W)$ is a~{\it rational} function of the form
${1 \over P(t)}$ for a some polynomial $P(t):=P_{W}(t) \in \Z[t]$.

\begin{Theorem} \label{theorem4.132}\quad
\begin{enumerate}\itemsep=0pt
\item[$(1)$] Some Betti numbers of the group algebra over $\Q$ of the
monoid
$B^{+}(A_{n-1})$:
\begin{gather*}
 B_{{\Q}[B^{+}(A_{n-1})]}(k,k)={n-k \choose k}, \\
B_{{\Q}[B^{+}(A_{n-1})]}\left(k,{k+1 \choose 2}\right)=n-k, \qquad 1 \le k \le n-1,\\
B_{{\Q}[B^{+}(A_{n-1})]}(k,k+1)= (k-1) {n-k \choose k-1},\\
B_{{\Q}[B^{+}(A_{n-1})]}(k,k+2)= {k-2 \choose 2} {n-k \choose k-2},\\
B_{{\Q}[B^{+}(A_{n-1})]}(k,k+3)= (k-2) {n-k \choose k-2}+
\max(3 k-17,0) {n-k \choose k-3} \quad \text{if} \ k \ge 3.
\end{gather*}

\item[$(2)$] The Birman--Ko--Lee algebra ${\rm BKL}(n)$ is Koszul, and the Hilbert
polynomial of its quadratic dual is equal to
\begin{gather*}
{\rm Hilb}\big({\rm BKL}(n)^{!},t\big)= \sum_{k=0}^{n-1}{1 \over {k+1}} {n-1 \choose k}
{n+k-1 \choose k} t^k.
\end{gather*}
\end{enumerate}
\end{Theorem}

\begin{Conjecture}[type $A_{n-1}$ case] Let $I \subset [1,n-1]$ be a subset of
vertices in the Dynkin diagram of type $A_{n-1}$, and $R_{I}$ denotes the
root system generated by the positive roots
$\{\alpha_{ij}=\epsilon_i-\epsilon_j, \, (i,j) \in I \times I \}$. Assume that
\begin{gather*}
R_I \cong A_{n_{1}} \coprod \cdots \coprod A_{n_{k}},\qquad n_1+\cdots+n_{k}=n-1
\end{gather*}
stands for the decomposition of the root system~$R_{I}$ into the disjoint
union of irreducible root subsystems of type~$A$. The numbers $n_1, \ldots,
 n_k$ are defined uniquely up to a permutation. Let us set $n(I)=
\sum\limits_{a=1}^{k} {n_a \choose 2}$. Then
\begin{gather*}
P_{A_{n-1}}(s,t)=\sum_{I} s^{|I|} t^{n(I)},
\end{gather*}
where the sum runs over the all subsets of vertices~$I$ in the Dynkin diagram
of type~$A_{n-1}$, including the empty set, and~${|I|}$ denotes the
cardinality of the set~$I$.
\end{Conjecture}

\begin{Comments} \quad
\begin{enumerate}\itemsep=0pt
\item[$(A)$] The Hilbert polynomial of the Birman--Ko--Li algebra ${\rm BKL}(n)$ has been computed also by M.~Albenque and P.~Nadeau, see~\cite{AN}.

\item[$(B)$] Let's consider the {\it truncated theta function}
$\theta^{+}(z,t)=
\sum\limits_{n \ge 0}t^{n(n+1)/2} z^n$. Then
\begin{gather*}
\sum_{n \ge 1} P_{A_{n-1}}(s,t) z^{n-1}=
\theta^{+}(t,s z)/(1-z(\theta^{+}(t,s z))).
\end{gather*}

\item[$(C)$] It is well known that the number
\begin{gather*}
T(n,k)= {1 \over {k+1}} {n \choose k}~{n+k \choose k}
\end{gather*}
counts the number of Schr\"oder paths (i.e., consisting of
steps $(1,1)$, $(1,-1)$ and $(2,0)$ and never going below $x$-axis) from
$(0,0)$ to $(2n,0)$, having exactly~$k$ $(1,1)$ steps. In particular,
\begin{gather*}
\dim({\rm BKL}(n)^{!})= {\rm Sch}(n),
\end{gather*}
 is the $n$-th (large) Schr\"oder number, see \cite[$A006318$]{SL}.
It is a classical result that
\begin{gather*}
\sum_{k=0}^{n} T(n,k) x^k (1-x)^{n-k}=\sum_{k=0}^{n-1} N(n,k) x^k,
\end{gather*}
where $N(n,k):={1 \over n} {n \choose k} {n \choose k+1}$ denotes the
Narayana number.
Some explicit combinatorial interpretations of the values of
the above polynomials for $x=0,1,2,4$ can be found in~\cite[$A088617$]{SL}.
Note that ${\rm Hilb}({\rm BKL}^{!},t)= (1+t) {\rm Ass}_{n-2}(t)$, where
\begin{gather*}
{\rm Ass}_{n}(t):= \sum_{k=0}^{n} {1 \over k+1} {n \choose k} {n+k \choose k}
t^{k}
\end{gather*}
denotes the $f$-vector polynomial corresponding to the associahedron of type~$A_n$.

\item[$(D)$] The polynomials
$F(n,t):=\sum\limits_{k \ge 0} B_{{\Q}[B^{+}(A_{n-1})]}(k,k) t^{k}$ appear to be
equal to the so-called {\it Fibonacci} polynomials, see, e.g., \cite[$A011973$]{SL}. It is well-known that
\begin{gather*}
 \sum_{n \ge 0} F(n,t) y^{n}={1+t y \over 1-y-t y^2 }.
 \end{gather*}
Moreover, the coef\/f\/icient $B_{{\Q}[B^{+}(A_{n-1})]}(k,k)$ is equal to the
number of compositions of $n+2$ into $k+1$ parts, all $ \ge 2$, see~\cite[$A011973$]{SL}.

\item[$(E)$] {\it Monoid of positive pure braids.}
The monoid of positive pure braids ${\rm PB}_n^{+}$ (of the type $A_{n-1}$) is the
monoid generated by the set $\{g_{i,j},\, 1 \le i < j \le n \}$ of the Artin
generators of the pure braid group ${\rm PB}_n$.
\end{enumerate}
\end{Comments}

\begin{Conjecture} The following list of relations is the defining set of relations in
the monoid ${\rm PB}_n$:
\begin{gather*}
(a)\quad [g_{i,j},g_{k,l}]=0, \qquad [g_{i,l},g_{j,k}]=0, \qquad \text{if} \quad 1 \le i < j < k < l \le n,\\
(b)\quad \left[g_{\overline {j_{1}+m}, \overline{j_{k-1}+m}}, \prod_{a=1}^{k-1}g_{\overline {j_{a}+m}, \overline{j_{k}+m}} \right]=0,
\end{gather*}
for all sequences of integers $1 \le j_1 < j_2 < \cdots < j_k \le n$ of the
length $k \ge 4$ and $m= 0, \ldots, n-1$.
Here we assume that $g_{i,j}=g_{j,i}$
 for all $i \not= j$, and for any non-negative integer~$a$ we denote by~$\overline {a}$ a unique integer $1 \le \overline{a} \le n$ such that
$ a \equiv \overline{a}$ $({\rm mod} \, n+1)$.
\end{Conjecture}

It is worth noting that the def\/ining relations in the pure braid group $P_n$
 are that listed in~$(a)$ and the part of that listed in~$(b)$
corresponding to $k=3$, $m=0$ and~$1$, and that for $k=4$, $m=0$.

\section{Combinatorics of associative Yang--Baxter algebras}\label{section5}

Let $\alpha$~and~$\beta$ be parameters.

\begin{Definition}[\cite{Kir,K2,K}, cf.\ \cite{Ag,Pol}]\label{definition5.1}\quad
\begin{enumerate}\itemsep=0pt
\item[(1)] The associative quasi-classical Yang--Baxter algebra
of weight $(\alpha,\beta)$, denoted by\linebreak $\widehat{{\rm ACYB}}_n(\alpha,\beta)$, is
an associative algebra, over the ring of polynomials $\Z[\alpha,\beta]$,
generated by the set of elements $\{ x_{ij}, \, 1 \le i < j \le n\}$, subject to
the set of relations
\begin{enumerate}\itemsep=0pt
\item[(a)] $x_{ij} x_{kl}=x_{kl} x_{ij}$ if $ \{i,j\} \cap \{k,l \}=\varnothing$,
\item[(b)] $x_{ij} x_{jk}=x_{ik} x_{ij}+x_{jk} x_{ik}+\beta x_{ik} + \alpha$ if
$1 \le 1 <i < j \le n$.
\end{enumerate}

\item[(2)] Def\/ine {\it associative quasi-classical Yang--Baxter algebra of weight
 $\beta$}, denoted by $\widehat{{\rm ACYB}}_n(\beta)$, to be
$\widehat{{\rm ACYB}}_n(0,\beta)$.
\end{enumerate}
\end{Definition}

\begin{Comments}\label{comments5.1}
The algebra $3T_n(\beta)$, see Def\/inition~\ref{definition3.1}, is the
quotient of the algebra \linebreak $\widehat{{\rm ACYB}}_n(- \beta)$, by the ``dual relations''
\begin{gather*}
x_{jk}x_{ij}-x_{ij} x_{ik}-x_{ik} x_{jk}+\beta x_{ik}=0, \qquad i < j < k.
\end{gather*}
The (truncated) Dunkl elements $\theta_i=\sum\limits_{j \not= i} x_{ij}$, $i=1,\ldots,n$, do not commute in the algebra $\widehat{{\rm ACYB}}_n(\beta)$. However a~certain version of noncommutative elementary polynomial of degree $k \ge 1$,
 still is equal to zero after the substitution of Dunkl elements instead of
variables~\cite{K}. We state here the corresponding result only ``in classical case'', i.e., if
$\beta=0$ and $q_{ij}= 0$ for all~$i$,~$j$.
\end{Comments}

\begin{Lemma}[\cite{K}]\label{lemma5.1} Define noncommutative elementary polynomial
$L_k(x_1,\ldots,x_n)$ as follows
\begin{gather*}
L_k(x_1,\ldots,x_n)= \sum_{I=(i_1 < i_2 < \cdots < i_k) \subset [1,n]}
x_{i_{1}} x_{i_{2}} \cdots x_{i_{k}}.
\end{gather*}
Then $L_{k}(\theta_1,\theta_2,\ldots,\theta_n)=0$.

Moreover, if $1 \le k \le m \le n$, then one can show that the value of
the noncommutative polynomial $L_k(\theta_1^{(n)},\ldots,\theta_m^{(n)})$ in
the algebra $\widehat{{\rm ACYB}}_n(\beta)$ is given by the Pieri formula,
see~{\rm \cite{FK, P}}.
\end{Lemma}

\subsection{Combinatorics of Coxeter element}\label{section5.1}
Consider the ``Coxeter element'' $w \in \widehat{{\rm ACYB}}_n(\alpha,\beta)$ which
is equal to the ordered product of ``simple generators'':
\begin{gather*}
w:=w_n=\prod_{a=1}^{n-1} x_{a,a+1}.
\end{gather*}
Let us bring the element~$w$ to the reduced form in the algebra
$\widehat{{\rm ACYB}}_n(\alpha,\beta)$, that is, let us consecutively apply the
def\/ining relations~$(a)$ and~$(b)$ to the element~$w$ in any order until unable
 to do so. Denote the resulting (noncommutative) polynomial by
$P_n(x_{ij};\alpha,\beta)$. In principal, the polynomial itself can depend on
the order in which the relations~$(a)$ and~$(b)$ are applied. We set
$P_n(x_{ij};\beta):=P_n(x_{ij};0,\beta)$.

\begin{Proposition}[cf.~\protect{\cite[Exercise~6.C5(c)]{ST3}}, \cite{Me,Me-b}] \label{proposition5.1}\quad
\begin{enumerate}\itemsep=0pt
\item[$(1)$] Apart from applying the relation $(a)$ $($commutativity$)$, the
polynomial $P_n(x_{ij};\beta)$
does not depend on the order in which
relations~$(a)$ and~$(b)$ have been applied, and can be written in a~unique way as a linear combination:
\begin{gather*}
P_{n}(x_{ij};\beta)= \sum_{s=1}^{n-1} \beta^{n-s-1}
\sum_{\{i_a \} } \prod_{a=1}^{s} x_{i_a,j_a},
\end{gather*}
where the second summation runs over all sequences of integers
$\{i_a \}_{a=1}^{s}$ such that $n-1 \ge i_1 \ge i_2 \ge \cdots \ge i_s =1$,
and $i_a \le n-a$ for $a=1,\ldots,s-1$; moreover, the corresponding
sequence $\{j_a \}_{a=1}^{n-1}$ can be defined uniquely by
that $\{i_a \}_{a=1}^{n-1}$.
\begin{itemize}\itemsep=0pt
\item It is clear that the polynomial $P(x_{ij}; \beta)$ also can be
written in a unique way as a linear combination of monomials
$\prod\limits_{a=1}^{s} x_{i_a,j_a}$ such that $j_1 \ge j_2 \cdots \ge j_s$.
\end{itemize}

\item[$(2)$] Let us set $\deg(x_{ij})=1$, $\deg(\beta)=0$. Denote by $T_{n}(k,r)$
the number of degree~$k$ monomials in the
polynomial $P(x_{ij};\beta)$ which contain exactly~$r$ factors of the form
$x_{*,n}$. $($Note that $1 \le r \le k \le n-1.)$ Then
\begin{gather*}
T_{n}(k,r)={r \over k} {n+k-r-2 \choose n-2} {n-2 \choose k-1}.
\end{gather*}
In other words,
\begin{gather*}
P_n(t,\beta)= \sum_{1 \le r \le k < n} T_n(k,r) t^{r} \beta^{n-1-k},
\end{gather*}
where $P_n(t,\beta)$ denotes the following specialization
\begin{gather*}
x_{ij} \longrightarrow 1 \qquad \text{if}\quad j < n,\qquad x_{in} \longrightarrow t, \quad \forall\, i=1,\ldots, n-1
\end{gather*}
of the polynomial $P_n(x_{ij} ; \beta)$.

In particular, $T_{n}(k,k)={n-2 \choose k-1}$ and $T_{n}(k,1)=T(n-2,k-1)$,
 where
\begin{gather*}
T(n,k):={1 \over k+1} {n+k \choose k} {n \choose k}
\end{gather*}
is equal to the number of Schr\"oder paths $($i.e., consisting of steps
$U=(1,1)$, $D=(1,-1)$, $H=(2,0)$ and never going below the $x$-axis$)$ from
$(0,0)$ to $(2n,0)$, having~$k$ $U$'s, see {\rm \cite[$A088617$]{SL}}.

Moreover, $T_{n}(n-1,r)={\rm Tab}(n-2,r-1)$, where
\begin{gather*}
{\rm Tab}(n,k):={k+1 \over n+1} {2n-k \choose n} = F^{(2)}_{n-k}(k)
\end{gather*}
is equal to the number of standard Young tableaux of the shape $(n,n{-}k)$,
see {\rm \cite[$A009766$]{SL}}. Recall that $F^{(p)}_{n}(b) =
{1+b \over n} {np+b \choose n-1}$ stands for the generalized Fuss--Catalan
number.

\item[$(3)$] After the specialization $x_{ij} \longrightarrow 1$ the polynomial
$P(x_{ij})$ is transformed to the polynomial
\begin{gather*}
 P_n(\beta):=\sum_{k=0}^{n-1}N(n,k) (1+\beta)^k,
\end{gather*}
where $N(n,k):={1 \over n}~{n \choose k} {n \choose k+1}$, $k=0,\ldots,n-1$,
 stand for the Narayana numbers.

Furthermore,
$P_n(\beta)=\sum\limits_{d=0}^{n-1}s_n(d) \beta^d$, where
\begin{gather*}
s_n(d)={1 \over n+1}{2n-d \choose n} {n-1 \choose d}
\end{gather*}
 is the number of ways to draw $n-1-d$ diagonals in a convex $(n+2)$-gon,
such that no two diagonals intersect their interior.

Therefore, the number of $($nonzero$)$ terms
in the polynomial $P(x_{ij};\beta)$ is equal to the $n$-th little
Schr\"oder number $s_n:=\sum\limits_{d=0}^{n-1}s_n(d)$, also known as the
$n$-th super-Catalan number, see, e.g., {\rm \cite[$A001003$]{SL}}.

\item[$(4)$] Upon the specialization $x_{1j} \longrightarrow t$,
$1 \le j \le n$, and that $x_{ij} \longrightarrow 1$ if
$2 \le i < j \le n$, the polynomial~$P(x_{ij};\beta)$ is transformed to the
polynomial
\begin{gather*}
P_n(\beta,t)= t \sum_{k=1}^{n} (1+\beta)^{n-k} \sum_{\pi} t^{p(\pi)},
\end{gather*}
where the second summation runs over the set of Dick paths~$\pi$ of length
$2n$ with exactly~$k$ picks $($UD-steps$)$, and~$p(\pi)$ denotes the number of
valleys $($DU-steps$)$ that touch upon the line $x=0$.

\item[$(5)$] The polynomial $P(x_{ij} ; \beta)$ is invariant under the
action of anti-involution $\phi \circ \tau$, see {\rm \cite[Section~5.1.1]{K}} for
definitions of~$\phi$ and~$\tau$.

\item[$(6)$] Follow {\rm \cite[Exercise~6.C8(c)]{ST3}} consider the
specialization
\begin{gather*}
 x_{ij} \longrightarrow t_{i}, \qquad 1 \le i < j \le n,
\end{gather*}
and def\/ine $P_n(t_1,\ldots,t_{n-1};\beta)= P_n(x_{ij}=t_i;\beta)$.

One can show, cf.\ {\rm \cite{ST3}}, that
\begin{gather*}
P_n(t_1,\ldots,t_{n-1}; \beta)= \sum \beta^{n-k} t_{i_{1}} \cdots t_{i_{k}},
\end{gather*}
where the sum runs over all pairs $\{(a_1,\ldots,a_k),(i_1,\ldots,i_k) \in
\Z_{\ge 1} \times \Z_{\ge 1} \}$ such that $ 1 \le a_1 < a_2 < \cdots < a_k$,
$ 1 \le i_1 \le i_2 \cdots \le i_k \le n$ and $i_{j} \le a_j$ for all~$j$.
\end{enumerate}
 \end{Proposition}

Now we are ready to state our main result about polynomials
$P_n(t_1,\ldots,t_n;\beta)$.
Let $\pi:=\pi_{n} \in {\mathbb S}_n$ be the permutation
\begin{gather*}
\pi= \begin{pmatrix}1& 2& 3& \ldots& n \\
1& n& {n-1}& \cdots& 2 \end{pmatrix}.
\end{gather*}
Then
\begin{gather}\label{equation5.2}
P_n(t_1,\ldots,t_{n-1};\beta)= \left(\prod_{i= 1}^{n-1} t_i^{n-i} \right)
\mathfrak{G}_{\pi}^{(\beta)}\big(t_1^{-1},\dots,t_{n-1}^{-1}\big) = \sum_{{\cal{T}}}
 wt({\cal{T}}),
\end{gather}
where $\mathfrak{G}_{w}^{(\beta)}(x_1,\ldots,x_{n-1})$ denotes the
$\beta$-Grothendieck polynomial corresponding to a permutation
$w \in {\mathbb S}_n$, see~\cite{FK1} or Appendix~\ref{appendixA.1}; summation in the right hand
side of the second equality runs over the set of all {\it dissections}
${\cal{T}}$ of a convex $(n+2)$-gon, and $wt({\cal{T}})$ denotes
{\it weight} of a~dissection ${\cal{T}}$, namely,
\begin{gather*}
wt({\cal{T}}) = \prod_{ d \in {\cal{T}}} x_{d} \beta^{n-3- |{\cal{T}}|},
\end{gather*}
where the product runs over diagonals in ${\cal{T}}$, $x_{d} = x_{ij}$, if
diagonal~$d$ connects vertices~$i$ and~$j$, $i < j$, and $|{\cal{T}}|$ denotes the number of diagonals in dissection~${\cal{T}}$.

In particular,
\begin{gather*}
\mathfrak{ G}_{\pi}^{(\beta)}(x_1=1,\ldots,x_{n-1}=1)=\sum_{k=0}^{n-1}
N(n,k) (1+\beta)^{k},
\end{gather*}
where $N(n,k)$ denotes the Narayana numbers, see item~(3) of
Proposition~\ref{proposition5.1}.

More generally, write $P_n(t,\beta)= \sum\limits_{k} P_{n}^{ (k)}(\beta) t^k$. Then
\begin{gather*}
\mathfrak{ G}_{\pi}^{(\beta)}(x_1=t,\,x_i=1,\, \forall\, i \ge 2)=
 \sum_{k=0}^{n-1} P_{n-1}^{(k)}\big(\beta^{-1}\big) \beta^{k} t^{n-1-k}.
\end{gather*}

\begin{Comments}\label{comments5.2}\quad
\begin{itemize}\itemsep=0pt
\item Note that if $\beta=0$, then one has
$\mathfrak{G}_{w}^{(\beta=0)}(x_1,\ldots,x_{n-1})=
\mathfrak{S}_{w}(x_1,\ldots,x_{n-1})$, that is the $\beta$-Grothendieck
polynomial at $\beta=0$, is equal to the Schubert polynomial corresponding to
the same permutation~$w$. Therefore, if
\begin{gather*}
\pi= \begin{pmatrix} 1& 2& 3& \ldots& n \\
1& n& {n-1}& \ldots& 2 \end{pmatrix},
\end{gather*}
 then
\begin{gather}\label{equation5.3}
\mathfrak{S}_{\pi}(x_1=1,\ldots,t_{n-1}=1)= C_{n-1},
\end{gather}
where $C_m$ denotes the $m$-th Catalan number. Using the formula~\eqref{equation5.3} it is not dif\/f\/icult to check that the following formula for the
principal specialization of the Schubert polynomial $\mathfrak{S}_{\pi}(X_n)$
is true
\begin{gather}\label{equation5.4}
\mathfrak{S}_{\pi}\big(1,q, \ldots,q^{n-1}\big)=q^{n-1 \choose 3} C_{n-1}(q),
\end{gather}
where $C_m(q)$ denotes the Carlitz--Riordan $q$-analogue of the Catalan
numbers, see,\linebreak e.g.,~\cite{ST1}. The formula~\eqref{equation5.4} has been proved in~\cite{FK3} using the observation that~$\pi$ is a~{\it vexillary} permutation, see~\cite{M} for the a def\/inition of the latter. A~combinatorial/bijective proof
of the formula~\eqref{equation5.3} is {\it due} to A.~Woo~\cite{W}.

\item The Grothendieck polynomials, had been def\/ined originally by
A.\,Lascoux and M.-P.\,Sch\"ut\-zen\-berger, see, e.g.,~\cite{LS}, correspond to
the case $\beta=-1$. In this case $P_n(-1)=1$ if~$n \ge 0$, and therefore the
specialization $\mathfrak{G}_{w}^{(-1)}(x_1=1,\ldots,x_{n-1}=1) =1$ for all
$ w \in \mathbb{S}_n$.

\item
 In Section~\ref{section5.2.2}, Theorems~\ref{theorem5.5} and~\ref{theorem5.6}, we state a generalization of
the second equality in the formula~\eqref{equation5.2} to the case of Richardson's
permutations of the form $ 1^{k} \times w_{0}^{(n-k)}:= \pi_{k}^{(n)}$, and
relate monomials which appear in a combinatorial formula\footnote{See~\cite{BJS, FK3, KMi} for example.}
for the corresponding $\beta$-Grothendieck polynomial, and/with the
set of $k$-dissections and $k$-triangulations of a~convex $(n+k+1)$-gon,
and the Lagrange inversion formula, see Section~\ref{section5.4.2} for more details.
\end{itemize}
\end{Comments}

 Clearly, the Richardson permutations $\pi_{k}^{(0)}$ are special subset of
permutations of the form $1^{k} \times w_{\lambda}:= w_{k}^{(\lambda)}$,
where $w_{\lambda}$ stands for the dominant permutation of shape $\lambda$.
 An analogue and extension of the f\/irst equality in the formula~\eqref{equation5.2} for
permutations of the form $w_{1}^{(\lambda)}$ has been proved in~\cite[Theorem~5.4]{EM}. We state here a particular case of that result related with the
Fuss--Catalan numbers obtained independently by the author of the present paper
 as a~generalization of \cite[Exercise~8C5(c)]{ST3} and~\cite{W} to
the case of Fuss--Catalan numbers. Namely, let $\lambda =(\lambda_1,\ldots,\lambda_k=1)$ be a~Young diagram such that $\lambda_i-\lambda_{i+1} \le 1$.
Therefore, the boundary~$\partial(\lambda)$ of~$\lambda$, that is the set of
the last boxes in each row of~$\lambda$, is a disjoint union of vertical
intervals. To the last box of the
lowermost interval we attach the generator~$x_{23}$. To the next box of that
 interval, if exists, we attach the generator~$u_{24}$ and so on, up to the
top box of that interval is equipped with the generator, say~$x_{2,k_{1}}$.
It is clear that $k_1= \lambda'_1-\lambda'_2+2$. Now let us consider the
next vertical interval. To the bottom box of that interval
we attach the variable $x_{k_1,k_{1}+1}$, to the next box we attach the
variable $x_{k_1,k_1+2}$ and so on. Let the top of that vertical interval is
equipped with the generator~$x_{k_1,k_2}$; it is clear that $k_2= \lambda'_1-\lambda'_3+2$. Applying this procedure successively step by step to each
vertical interval, we attach the variable~$u_{b}$ to each box~$b$ in the boundary of Young diagram~$\lambda$. Finally we attach the monomial
\begin{gather*}
M_{\lambda} = x_{12} \prod_{ b \in \partial(\lambda)} x_{b}.
\end{gather*}

\begin{Theorem}[\cite{EM}]\label{theorem5.1}
 Let $\lambda$ be a partition such that $\lambda_i-\lambda_{i+1} \le 1$, $\forall\, i \ge 1$, and set $N:=\lambda'_{1}+2$. Let
 $w_{\lambda} \in {\mathbb{S}}_N$ be a unique dominant partition of shape~$\lambda$, and $M_{\lambda} \in {\widehat{{\rm ACYB}}}_{N}(\beta)$ be the monomial
associated with the boundary $\partial(\lambda)$ of partition~$\lambda$.
Then
\begin{gather*}
P_{M_{\lambda}}(x_{ij}=t_i,\, \forall \, i,j; \beta) = t^{\lambda}
{\mathfrak{G}}_{1 \times w_{\lambda}}^{(\beta)}\big(t_1^{-1}, \ldots,t_{N}^{-1}\big),
\end{gather*}
where $t^{\lambda}:=t_{1}^{\lambda'_{1}} \cdots t_{N}^{\lambda'_{N}}$. In
other words, after the specialization $x_{ij} \longrightarrow t_i^{-1}$, $\forall\, i,j$, the spe\-ciali\-zed reduced polynomial corresponding to the monomial~$M_{\lambda}$ is equal to $t^{- \lambda}$ multiplied by the $\beta$-Grothendieck polynomial associated with the permutation $1 \times w_{\lambda}$.
\end{Theorem}

Let us illustrate the above theorem by example. We take $\lambda= 43221$. In this case $N=7=5+2$ and $w:=w_{\lambda}= [1,6,5,4,7,3,2]$. The monomial
corresponding to the boundary of $\lambda$ is equal to
\begin{gather*}
x_{12} x_{23} x_{34} x_{35} x_{56} x_{67} \in {\widehat{{\rm ACYB}}}_{7}.
\end{gather*}
Since the both reduced and $\beta$-Grothendieck polynomials appearing in this
example are huge, we display only its specialized values at $x_{ij}=1$, $\forall\, i,j$ and $t_i=1$, $\forall\, i$. We set also $d:= \beta-1$. It is not dif\/f\/icult to
check that the reduced polynomial corresponding to monomial $ x_{12} x_{23} x_{34} x_{35} x_{56}$ after the specialization $x_{ij}=1$, $\forall\, 1 \le i < j \le 5$, and the identif\/ication $x_{i,6}=x_{1,6}$, $1 \le i \le 5$, is equal to
\begin{gather*}
(9,20,14,3)_{\beta} x_{16}+(9,15,6)_{\beta} x_{16}^2+(4,4)_{\beta} x_{16}^3+
x_{16}^4.
\end{gather*}
Finally after multiplication of the above expression by~$x_{67}$, applying
$3$-term relations $(b)$ in the algebra ${\widehat{{\rm ACYB}}}_7$ to the result obtained,and and taking the specialization $x_{i,7}=1$, $\forall\,i$, we will come to
the following expression
\begin{gather*}
(9,20,14,3)_{\beta} (2+\beta)+(9,15,6)_{\beta} (3+2 \beta)^2+(4,4)_{\beta} (4+3 \beta)+ (5+ 4 \beta)\\
\qquad{} =(66,144,108,32,3)_{\beta}.
\end{gather*}
One can check that the latter polynomial is equal to ${\mathfrak{G}}_{w}^{\beta}(1)$.
 \begin{Corollary}[monomials and Fuss--Catalan numbers ${\rm FC}_{n}^{(p+1)}$]
Let $p$, $n$, $b$ be integers, consider diagram $\lambda=(n^{b},(n-1)^{p},(n-2)^{p},\ldots,2^{p},1^{p})$
and dominant permutation $w \in {\mathbb{S}}_{(n-1)p +b+2}$ of shape~$\lambda$. Let us define monomial
\begin{gather*}
M_{n,p,b} =x_{12} \prod_{j=0}^{n-2} \left( \prod_{a=3}^{p+2} x_{j p+2, jp+a}
\right) \prod_{a=3}^{b+2} x_{(n-1) p +2, (n-1) p +a}.
\end{gather*}
Then
\begin{gather*}
P_{M_{n,p,b}}(x_{ij}=1, \, \forall\, i,j)(\beta) = \sum_{k=1}^{n} {\frac{1}{k}}
{\binom{n-1}{k-1}}{\binom{p n -\overline{b}}{k-1}} (\beta +1)^{k-1}.
\end{gather*}
Moreover,
\begin{gather*}
P_{M_{n,p,b}}(x_{ij}=1,\, \forall\, i,j)(\beta=0)= {\frac{1}{n p - \overline{b}+1}} {\binom{n(p+1) - \overline{b}}{n}} = {\frac{1}{n}} {\binom{n(p+1)-\overline{b}}{n-1}},
\end{gather*}
where $\overline{b}:= b-\frac{1-(-1)^{b}}{2}$.
\end{Corollary}

For $b=0$ the right hand side of the above equality is equal to the Fuss--Narayana polynomial, see Theorem~\ref{theorem5.10} and Proposition~\ref{proposition5.5}; a combinatorial interpretation of the value $P_{M_{n,p,b}}(x_{ij}=1$, $\forall\, i,j)(\beta=1)$ one can f\/ind in~\cite{NT}.
Note that reduced expressions for monomial $M_{n,p,b}$ in the (noncommutative)
algebra ${\widehat{{\rm ACYB}}}_n(\beta)$ up to applying the commutativity rules~$(a)$, Def\/inition~\ref{definition5.1}, is unique.

It seems an interesting {\it problem} to construct a natural bijection between the set of monomials in the (noncommutative) reduced expression associated with monomials $M_{n,p,0}$ and the set of $(p+1)$-gulations\footnote{That is the set of dissections of a convex $pk$-gon by (maximal)
 collection of non-crossing diagonals such that the all regions obtained are a
convex $(p+2)$-gons of a convex $kp$-gon.}
Finally we remark that there are certain connections of the
$\beta$-Grothendieck
polynomials corresponding to shifted dominance permutations (i.e., permutations
 of the form $1^{k} \times w_{\lambda}$) and some generating functions for the set of bounded by~$k$ plane
partitions of shape~$\lambda$, see, e.g.,~\cite{FK3}. In the case
of a staircase shape partition $\lambda = (n-1,\dots, 1)$ one can envision
(cf.\ \cite{SSt, St}) a connection/bijection between the set of $k$-bounded
 plane partitions of that shape and $k$-dissections of a convex $(n+k+1)$-gon.
 However in the case $k \ge 2$ it is not clear does there exist a monomial
$M$ in the algebra ${\widehat{{\rm ACYB}}}_n$ such that the value of the corresponding reduced polynomial at $x_{ij}= 1$, $\forall\, i,j$ is equal to the number
of $k$-dissections ($k \ge 2$) of a convex $(n+k+1)$-gon.

\begin{Exercises}\label{Exercises5.1}\quad

$(1)$

$(a)$ Let as before,
\begin{gather*}
\pi = \begin{pmatrix} 1& 2& 3& \ldots& n \cr
1& n& {n-1}& \ldots& 2 \end{pmatrix}.
\end{gather*}
{\it Show} that
\begin{gather*} \s_{\pi}(x_1=q,\,x_j=1, \,\forall\, j \not= i) =\sum_{a=0}^{n-2} {n-a-1
\over n-1} {n+a-2 \choose a} q^{a}.
\end{gather*}
Note that the number
\begin{gather*}
{n-k+1 \over n+1} {n+k \choose k}
\end{gather*}
 is equal to the
dimension of irreducible representation of the symmetric group
${\mathbb S}_{n+k}$ that corresponds to partition $(n+k,k)$.

$(b)$ Big Schr\"{o}der numbers, paths and polynomials $\mathfrak{G}_{1 \times w_{0}^{(n-1)}}^{(\beta)}(x_1=q, \,x_i=1,\,\forall\, i \ge 2)$.
Let $n \ge 1$ and $k \ge 0$ be integers, denote by $S_{k,n}$ the number of
{\it big Schr\"{o}der paths of type $(k,n)$}, that is lattice paths from the point $(0,0)$ and
 ending at point $(2n+k,k)$, using only the steps $U =(1,1)$, $H= (2 ,0)$
and $D=(1,-1)$ and never going below the line $x= 0$. The numbers $S(n):=
S_{0,n}$ commonly known as {\it big Schr\"{o}der numbers}, see, e.g.,
 \cite[$A001003$]{SL}. It is well-known that
\begin{gather*}
S_{k,n} = {\frac{k+1}{n}} \sum_{a=0}^{n} {\binom{n}{a}} {\binom{n+k+a}{n-1}}.
\end{gather*}
{\it Show} that
\begin{gather*}
\mathfrak{G}_{1 \times w_{0}^{(n-1)}}^{(\beta)}(x_1=q, \, x_i=1,\, \forall\, i \ge 2) =
\sum_{k=0}^{n-2} S_{k,n-2-k}(\beta) q^{n-k-2},
\end{gather*}
where $S_{k,n}(\beta)$ is the generating functions of the big Schr\"{o}der
paths of type $(k,n)$ according to the number of horizontal steps~$H$.

$(c)$ {\it Show} that the polynomial $\mathfrak{G}_{1 \times w_{0}^{(n-1-1)}}^{(\beta)}(x_1=q,\, x_i=1,\,\forall \,i \ge 2)$ belongs to the ring $\N [q,\beta+1]$. For example, for $n=8$ one has
\begin{gather*}
\mathfrak{G}_{1 \times w_{0}^{(7)}}^{(\beta)}(x_1=q, \,x_i=1,\,\forall\, i \ge 2) = (0,1,15,50,50,15,1)_{\beta+1} t^6 +(0,2,24,60,40,6)_{\beta+1}t^5\\
 \qquad{} + (0,3,27,45,15)_{\beta+1} t^4+ (0,4,24,20)_{\beta+1} t^3+(0,5,15)_{\beta+1} t^2+ 6(\beta+1) t+1.
 \end{gather*}

{\it Show} that
\begin{gather*}
S_{k,n}(\beta) ={\frac{k+1}{n}} \sum_{a=0}^{n} {\binom{n}{a}}{\binom{n+k+a}{n-1}} \beta^{n-a}= {\frac{k+1}{k+1+n}} {\binom{2n+k}{n}} + \cdots+{\binom{n+k}{n}} \beta^{n}.
\end{gather*}

$(d)$ Write
\begin{gather*}
\mathfrak{G}_{1^{k} \times w_{0}^{(n-k)}}^{(\beta)}(x_1=q, \, x_i=1,\,\forall\, i \ge 2) =A_{k,n}(\beta) q^{n-k-1} +\cdots+ B_{n,k}(\beta).
\end{gather*}
{\it Show} that
\begin{gather*}
A_{k,n}= (1+\beta)^{k} \mathfrak{G}_{k,n-1}^{(\beta)}(x_i=1,\, \forall\, i \ge 1),\qquad
B_{k,n}= \mathfrak{G}_{k-1,n-1}^{(\beta)}(x_i=1,\, \forall\, i \ge 1).
\end{gather*}

(2)~Consider the commutative quotient ${\widetilde{{\rm ACYB}}}_n^{ab}(\alpha,
\beta)$ of the algebra ${\widetilde{{\rm ACYB}}}_n(\alpha,\beta)$,
i.e., assume that
the all generators $\{x_{ij} \,| \,1 \le i < j \le n$ are mutually commute. Denote
 by ${\overline{P}}_n(x_{ij};\alpha,\beta)$ the image of polynomial the
$P_n(x_{ij};\alpha,\beta) \in {\widetilde{{\rm ACYB}}}_n(\alpha,\beta)$ in the
algebra ${\widetilde{{\rm ACYB}}}_n^{ab}(\alpha,\beta)$. Finally, def\/ine polynomials
 $P_n(t,\alpha,\beta)$~to be the specialization
\begin{gather*}
 x_{ij} \longrightarrow 1 \qquad \text{if}\quad j < n, \qquad x_{in} \longrightarrow t \qquad \text{if}
 \quad 1 \le i < n.
 \end{gather*}
{\it Show} that

$(a)$ Polynomial $P_n(t,\alpha,\beta)$ does not depend on on order in which
relations~$(a)$ and~$(b)$, see Def\/inition~\ref{definition5.1}, have been applied.

$(b)$
\begin{gather*}
P_n(1,\alpha=1,\beta=0)= \sum_{k \ge 0} {(2n-2k) ! \over k ! (n+1-k) !
(n-2k) !},
\end{gather*}
see \cite[$A052709(n)$]{SL} for combinatorial
interpretations of these numbers.

For example,
\begin{gather*}
P_7(t,\alpha,\beta)= t^7+6(1+\beta) t^6+ [(0,5,15)_{\beta+1}+
6 \alpha ] t^5 + [(0,4,24,20)_{\beta+1 }+
\alpha (5,29)_{\beta+1} ] t^4\\
\hphantom{P_7(t,\alpha,\beta)=}{} + [(0,3,27,45,15)_{\beta +1}+
\alpha (4,45,55)_{\beta +1}+14 \alpha^2 ] t^3 \\
\hphantom{P_7(t,\alpha,\beta)=}{} +
[(0,2,24,60,40,6)_{\beta +1} +\alpha (3,48,115,50)_{\beta +1}+
\alpha^2 (21,49)_{\beta +1} ] t^2 \\
\hphantom{P_7(t,\alpha,\beta)=}{} +
[(0,1,15,50,50,15,1)_{\beta +1}+\alpha (2,38,130,110,20)_{\beta +1}+
\alpha^2 (21,91,56)_{\beta +1}\\
\hphantom{P_7(t,\alpha,\beta)=}{}
+ 14 \alpha^3 ]t+\alpha (1,15,50,50,15,1)_{\beta +1} +
\alpha^2(14,70,70,14)_{\beta +1} +\alpha^3 (21,21)_{\beta + 1}.
\end{gather*}

$(c)$ {\it Show} that in fact
\begin{gather*}
P_n(1,\alpha,0)=\sum_{k \ge 0} {1 \over n+1} {2n-2k \choose n} {n+1
\choose k} \alpha^k = \sum_{k \ge 0} {T_{n+2}(n-k,k+1) \over 2n-1-2k}
\alpha^{k},
\end{gather*}
see Proposition~\ref{proposition5.1}(2), for def\/inition of numbers~$T_n(k,r)$. As for a~combinatorial interpretation of the polynomials $P_n(1,\alpha,0)$, see
\cite[$A117434$, $A085880$]{SL}.

$(3)$ Consider polynomials $P_n(t,\beta)$ as it has been def\/ined in
Proposition~\ref{proposition5.1}(2).
{\it Show} that
\begin{gather*}
P_n(t,\beta)= P_n(t,\alpha=0,\beta)= t^n+ \sum_{r=1}^{n-1} t^{r} \left( \sum_{k=0}^{n- 1-r} {r \over n} {n \choose k+r} {n-r-1 \choose k} (1+\beta)^{n-r-k} \right),
\end{gather*}
cf., e.g., \cite[$A033877$]{SL}.

 A few comments in order. Several combinatorial interpretations of the
integer numbers
\begin{gather*}
U_n(r,k):= {r \over n+1} {n+1 \choose k+r} {n-r \choose k}
\end{gather*}
are well-known. For example,
if $r=1$, the numbers $U_n(1,k) = {1 \over n}
{n \choose k+1} {n \choose k}$ are equal to the Narayana numbers, see, e.g.,
\cite[$A001263$]{SL};
if $r=2$, the number $U_n(2,k)$ counts the number of Dyck $(n+1)$-paths
whose last descent has length~$2$ and which contain $n-k$ peaks, see \cite[$A108838$]{SL} for details.

Finally, it's easily seen, that $P_n(1, \beta) = A127529(n)$, and
$P_n(t,1)= A033184 (n)$, see~\cite{SL}.

$(4)$ {\it Show} that
\begin{gather*}
P_n(t,\alpha,\beta) \in \N [t,\alpha] [\beta+1],
\end{gather*}
that is the polynomial $P_n(t,\alpha,\beta)$ is a polynomial of $\beta +1$
with coef\/f\/icients from the ring~$\N[t,\alpha]$.

{\it Show} that
\begin{gather*}
 P_n(0,1,\beta) \in \N [\beta +2].
 \end{gather*}
For example,
\begin{gather*}
P_7(0,1,\beta)=(0,3,8,14,10,1)_{\beta+2}, \qquad P_8(0,1,\beta)=(1,3,11,25,35,15,1)_{\beta+2}.
\end{gather*}

{\it Show} that~\cite{SL}
\begin{gather*}
P_n(1,1,0)= A052709(n+1),
\end{gather*}
that is the number of underdiagonal lattice paths from $(0,0)$ to $(n-1,n-1)$
 and such that each step is either $(1,0)$, $(0,1)$, or $(2,1)$.
For example, $P_7(1,1,0)= 1697 = A052709(8)$. Cf.\ with the next exercise.

{\it Show} that~\cite{SL}
\begin{gather*}
 P_n(0,1,0) = A052705(n),
\end{gather*}
namely, the number of underdiagonal paths from (0,0) to the line $x=n-2$, using
 only steps $(1,0)$, $(0,1)$ and $NE= (2,1)$. For example,
\begin{gather*}
P_7(0,1,0)= 36+106+120+64+15+1 = 342 = A052705(7).
\end{gather*}

{\it Show} that~\cite{SL}
\begin{gather*}
{\frac{\partial}{\partial a}} P_n(a,{\boldsymbol{b}} ={\boldsymbol{1}},{\boldsymbol{\beta}}={\boldsymbol{0}},{\boldsymbol{\alpha}}={\boldsymbol{1}},
{\boldsymbol{y}} = {\boldsymbol{z}} ={\boldsymbol{1}})= A005775,
\end{gather*}
that is the number number of paths in the half-plane $x \ge 0$ from $(0,0)$
 to $(n-1,2)$ or $(n-1,-3)$, and consisting of steps $U=(1,1)$, $D=(1,-1)$
 and $H=(1,0)$ that contain at least one $UUU$ but avoid $UUU's$ starting
above level~$0$.
\end{Exercises}

\subsubsection{Multiparameter deformation of Catalan, Narayana and Schr\"oder numbers}\label{section5.1.1}

Let $\mathfrak{b} = (\beta_1,\ldots,\beta_{n-1})$ be a set of mutually
commuting parameters. We def\/ine a multiparameter analogue of the associative
quasi-classical Yang--Baxter algebra $\widehat{{\rm MACYB}}_n(\mathfrak{b})$ as
follows.

\begin{Definition}[cf.\ Def\/inition~\ref{definition2.4}]\label{definition5.2}
 The multiparameter associative quasi-
classical Yang--Baxter algebra of weight ${\mathfrak{b}}$, denoted by
$\widehat{{\rm MACYB}}_n(\mathfrak{b})$, is an associative
algebra, over the ring of polynomials $\Z[\beta_1,\ldots,\beta_{n-1}]$,
generated by the set of elements $\{ x_{ij}, \, 1 \le i < j \le n\}$,
subject to the set of relations
\begin{enumerate}\itemsep=0pt
\item[$(a)$] $x_{ij} x_{kl}=x_{kl} x_{ij}$ if $ \{i,j\} \cap \{k,l \}=\varnothing$,

\item[$(b)$] $x_{ij} x_{jk}=x_{ik} x_{ij}+x_{jk} x_{ik}+\beta_{i} x_{ik}$ if
$1 \le 1 <i < j \le n$.
\end{enumerate}
\end{Definition}

Consider the ``Coxeter element'' $w_n \in \widehat{{\rm MACYB}}_n(\mathfrak{b})$
which is equal to the ordered product of ``simple generators'':
\begin{gather*}
w_n:=\prod_{a=1}^{n-1} x_{a,a+1}.
\end{gather*}
Now we can use the same method as in~\cite[Exercise~8.C5(c)]{ST3}, see
Section~\ref{section5.1}, to def\/ine the {\it reduced form} of the Coxeter element~$w_n$.
Namely, let us bring the element~$w_n$ to the reduced form in the algebra~
$\widehat{{\rm MACYB}}_n({\mathfrak{b}})$, that is, let us consecutively apply the
def\/ining relations $(a)$~and~$(b)$ to the element~$w_n$ in any order until
unable to do so. Denote the resulting (noncommutative) polynomial by
$P(x_{ij};\mathfrak{b})$. In principal, the polynomial itself can depend on
the order in which the relations~$(a)$ and~$(b)$ are applied.

\begin{Proposition}[cf.~\protect{\cite[Exercise~8.C5(c)]{ST3}, \cite{Me,Me-b}}]\label{proposition5.2}
 The specialized
polynomial $P(x_{ij} =1$, $\forall\, i,j, \, \mathfrak{b})$ does not depend on the
order in which relations~$(a)$ and~$(b)$ have been applied.
\end{Proposition}

To state our main result of this subsection, let us def\/ine polynomials
\begin{gather*}
Q(\beta_1,\ldots,\beta_{n-1}):=
P(x_{ij}=1, \,\forall \,i,j ; \,\beta_1-1,\beta_2-1,\ldots,\beta_{n-1}-1).
\end{gather*}

\begin{Example}\label{example5.1}
\begin{gather*}
Q(\beta_1,\beta_2)=1+2 \beta_1 +\beta_2 +\beta_1^2,\\
Q(\beta_1,\beta_2,\beta_3)=1+ 3 \beta_1+ 2 \beta_2+\beta_3 +3 \beta_1^2+ \beta_1 \beta_2 +\beta_1 \beta_3+ \beta_2^2+\beta_1^3,\\
Q(\beta_1,\beta_2,\beta_3,\beta_4)= 1+4 \beta_1+3 \beta_2+ 2 \beta_3+ \beta_4+
\beta_1 (6 \beta_1 + 3 \beta_2+3 \beta_3 +2 \beta_4) \\
\hphantom{Q(\beta_1,\beta_2,\beta_3,\beta_4)=}{} + \beta_2 (3 \beta_2+\beta_3 + \beta_4) +\beta_3^2 +
\beta_1^2 (4 \beta_1+\beta_2+\beta_3+\beta_4)\\
\hphantom{Q(\beta_1,\beta_2,\beta_3,\beta_4)=}{}
 +\beta_1 (\beta_2^2+\beta_3^2)+
\beta_2^3 +\beta_1^4.
\end{gather*}
\end{Example}
\begin{Theorem}\label{theorem5.2}
Polynomial $Q(\beta_1,\ldots,\beta_{n-1})$ has non-negative
integer coefficients.
\end{Theorem}

It follows from \cite{ST3} and Proposition~\ref{proposition5.1}, that
\begin{gather*}
Q(\beta_1,\ldots,\beta_{n-1})\bigl|_{\beta_1=1,\ldots,\beta_{n-1}=1} =
{\rm Cat}_n.
\end{gather*}
Polynomials $Q(\beta_1,\ldots,\beta_{n-1})$ and
$Q(\beta_1+1,\ldots,\beta_{n-1}+1)$ can be considered as a multiparameter
deformation of the Catalan and (small) Schr\"oder numbers
correspondingly, and the homogeneous degree~$k$ part of
$Q(\beta_1,\ldots,\beta_{n-1})$ as a multiparameter analogue of Narayana
numbers.

\subsection[Grothendieck and $q$-Schr\"oder polynomials]{Grothendieck and $\boldsymbol{q}$-Schr\"oder polynomials}\label{section5.2}

\subsubsection{Schr\"oder paths and polynomials}\label{section5.2.1}

\begin{Definition} \label{definition5.3}
A Schr\"oder path of the length $n$ is an over diagonal
path from $(0,0)$ to $(n,n)$ with steps $(1,0)$, $(0,1)$ and steps
$D=(1,1)$ {\it without} steps of type $D$ on the diagonal $x=y$.
\end{Definition}

If $p$ is a Schr\"oder path, we denote by $d(p)$ the number of the
diagonal steps resting on the path~$p$, and by $a(p)$ the number of unit
squares located between the path~$p$ and the diagonal $x=y$. For each (unit)
 diagonal step $D$ of a path $p$ we denote by $i(D)$ the $x$-coordinate of
the column which contains the diagonal step~$D$. Finally, def\/ine the index~$i(p)$ of a path~$p$ as the some of the numbers~$i(D)$ for all diagonal steps
of the path~$p$.

\begin{Definition}\label{definition5.4}
 Def\/ine $q$-Schr\"oder polynomial $S_n(q;\beta)$ as follows
\begin{gather}\label{equation5.5}
S_n(q;\beta)= \sum_{p} q^{a(p)+i(p)} \beta^{d(p)},
\end{gather}
where the sum runs over the set of all Schr\"oder paths of length~$n$.
\end{Definition}

\begin{Example}\label{example5.2}
\begin{gather*}
S_1(q;\beta)=1, \qquad S_2(q;\beta)=1+q+\beta q,\\
S_3(q;\beta)=1+2 q+q^2+q^3+\beta \big(q+2q^2+2q^3\big )+\beta^2 q^3,\\
S_4(q;\beta)=1+3q+3q^2+3q^3+2q^4+q^5+q^6+\beta\big(q+3q^2+5q^3+6q^4+3q^5+3q^6\big)\\
\hphantom{S_4(q;\beta)=}{} + \beta^2\big(q^3+2q^4+3q^5+3q^6\big)+\beta^3 q^6.
\end{gather*}
\end{Example}

\begin{Comments}\label{comments5.3}
The $q$-Schr\"oder polynomials def\/ined by the
formula \eqref{equation5.5} are {\it different} from the $q$-analogue of
Schr\"oder polynomials which has been considered in~\cite{BK}. It seems
that there are no simple connections between the both.
\end{Comments}

\begin{Proposition}[recurrence relations for $q$-Schr\"oder polynomials]\label{proposition5.3}
The $q$-Schr\"oder polynomials satisfy the following relations
\begin{gather*}
S_{n+1}(q;\beta)= \big(1+q^{n}+\beta q^{n}\big) S_n(q;\beta)+\sum_{k=1}^{k=n-1}
\big(q^k+\beta q^{n-k}\big) S_{k}(q;q^{n-k} \beta) S_{n-k}(q;\beta),
\end{gather*}
and the initial condition $S_1(q;\beta)=1$.
\end{Proposition}

Note that $P_n(\beta)=S_n(1;\beta)$ and in particular, the polynomials
$P_n(\beta)$ satisfy the following recurrence relations
\begin{gather}\label{equation5.7}
P_{n+1}(\beta)=(2+\beta)~P_n(\beta)+(1+\beta)~\sum_{k=1}^{n-1}~P_k(\beta)~
P_{n-k}(\beta).
\end{gather}

\begin{Theorem}[evaluation of the Schr\"oder--Hankel determinant]\label{theorem5.3}
 Consider permutation
\begin{gather*}
\pi_{k}^{(n)}= \begin{pmatrix} 1& 2& \ldots& k& k+1& k+2&\ldots& n \\
1& 2& \ldots& k& n& {n-1}& \ldots& k+1 \end{pmatrix}.
\end{gather*}
Let as before
\begin{gather}\label{equation5.8}
 P_n(\beta)=\sum_{j=0}^{n-1} N(n,j) (1+\beta)^{j},\qquad n \ge 1,
\end{gather}
be Schr\"oder polynomials. Then
\begin{gather*}
 (1+\beta)^{{k \choose 2}} \mathfrak{G}_{\pi_{k}^{(n)}}^{(\beta)}(x_1=1,\ldots,
x_{n-k}=1) =\operatorname{Det} \vert P_{n+k-i-j}(\beta) \vert_{1 \le i,j \le k}.
\end{gather*}
\end{Theorem}

Proof is based on an observation that the permutation $\pi_{k}^{(n)}$ is a~{\it vexillary} one and the recurrence relations~\eqref{equation5.7}.

\begin{Comments} \label{comments5.4}

{\bf (1)} In the case $\beta=0$, i.e., in the case of {\it Schubert
polynomials}, Theorem~\ref{theorem5.3} has been proved in~\cite{FK3}.

{\bf (2)} In the cases when $\beta=1$ and $ 0 \le n-k \le 2$, the value of
the determinant in the r.h.s.\ of~\eqref{equation5.8} is known\footnote{See, e.g.,~\cite{BK}, or
M.~Ichikawa talk ``Hankel determinants of Catalan, Motzkin and
 Schr\"oder numbers and its $q$-analogue'',
\url{http://www.uec.tottori-u.ac.jp/~mi/talks/kyoto07.pdf}.}. One can check
that in the all cases mentioned above, the formula~\eqref{equation5.8}
gives the same results.

{\bf (3)} {\it Grothendieck and Narayana polynomials.}
It follows from the expression~\eqref{equation5.7} for the Narayana--Schr\"oder
polynomials that $P_n(\beta-1) = \mathfrak{N}_n(\beta)$,
where
\begin{gather*}
 \mathfrak{N}_n(\beta):=\sum_{j=0}^{n-1} {1\over n} {n \choose j} {n \choose
j+1} \beta^{j},
\end{gather*}
denotes the $n$-th Narayana polynomial. Therefore, $P_n(\beta -1)=
\mathfrak{N}_n(\beta)$ is a
symmetric polynomial in $\beta$ with non-negative integer coef\/f\/icients.
Moreover, the value of the polynomial $P_n(\beta-1)$ at $\beta=1$ is equal to
the $n$-th Catalan number $C_n:= {1 \over n+1} {2n \choose n}$.

It is well-known, see, e.g.,~\cite{Su}, that the Narayana polynomial
$\mathfrak{N}_n(\beta)$ is equal to the ge\-ne\-ra\-ting function of the statistics
$\pi(\mathfrak{p})= (\text{number of peaks of a Dick path $\mathfrak{p}$})-1$ on the set ${\rm Dick}_n$ of Dick paths of the length~$2n$
\begin{gather*}
\mathfrak{N}_n(\beta)=\sum_{\mathfrak{p}} \beta^{\pi(\mathfrak{p})}.
\end{gather*}
Moreover, using the Lindstr\"om--Gessel--Viennot lemma\footnote{See, e.g.,
\url{https://en.wikipedia.org/wiki/Lindstrom-Gessel-Viennot_lemma}.},
 one can see that
\begin{gather}\label{equation5.10}
\operatorname{DET} | \mathfrak{N}_{n+k-i-j}(\beta) |_{1 \le i,j \le k} =
\beta^{{k \choose 2}} \sum_{(\mathfrak{p}_1,\ldots,\mathfrak{p}_k)} \beta^{\pi(\mathfrak{p}_1) + \cdots+\pi(\mathfrak{p}_k)},
\end{gather}
where the sum runs over $k$-tuple of non-crossing Dick paths $(\mathfrak{p}_1,
\ldots,\mathfrak{p}_k)$ such that the path $\mathfrak{p}_i$ starts from the
point $(i-1,0)$ and has length $2(n-i+1)$, $i=1,\ldots,k$.

We {\it denote} the sum in the r.h.s.\ of~\eqref{equation5.10} by
$\mathfrak{N}_n^{(k)}(\beta)$. Note that $\mathfrak{N}_{k-1}^{(k)}(\beta)= 1$
 for all $k \ge 2$.

Thus, $\mathfrak{N}_n^{(k)}(\beta)$ is a symmetric polynomial in $\beta$ with
non-negative integer coef\/f\/icients, and
\begin{gather*}
\mathfrak{N}_n^{(k)}(\beta=1)= C_n^{(k)}= \prod_{1 \le i \le j \le n-k}
{2k+i+j \over i+j} = \prod_{2 a \le n-k-1} {{2n-2a \choose 2k} \over {2k+2a+1
\choose 2k}}.
\end{gather*}

As a corollary we obtain the following statement
\begin{Proposition}\label{proposition5.4}
Let $n \ge k$, then
\begin{gather*}
\mathfrak{G}_{\pi_{k}^{(n)}}^{(\beta-1)}(x_1=1,\ldots,x_n=1)=
\mathfrak{N}_n^{(k)}(\beta).
\end{gather*}
\end{Proposition}

Summarizing, the specialization $\mathfrak{G}_{\pi_{k}^{(n)}}^{(\beta-1)}(x_1=1,
\ldots,x_n=1)$ is a symmetric polynomial in $\beta$ with non-negative
integer coef\/f\/icients, and coincides with the generating function of the
statistics $\sum\limits_{i=1}^{k} \pi(\mathfrak{p}_i)$ on the set $k$-${\rm Dick}_n$ of
$k$-tuple of non-crossing Dick paths $(\mathfrak{p}_1,\ldots,\mathfrak{p}_k)$.

\begin{Example} \label{example5.3}
Take $n=5$, $k=1$. Then $\pi_{1}^{(5)}=(15432)$ and one has
\begin{gather*}
\mathfrak{G}_{\pi_{1}^{(5)}}^{(\beta)} \big(1,q,q^2,q^3\big)=q^4(1,3,3,3,2,1,1)+q^5
(1,3,5,6,3,3) \beta+q^7(1,2,3,3) \beta^2+q^{10} \beta^3.
\end{gather*}
It is easy to compute the Carlitz--Riordan $q$-analogue of the Catalan number
$C_5$, namely,
\begin{gather*}
C_5(q)=(1,3,3,3,2,1,1).
\end{gather*}
\end{Example}

\begin{Remark} \label{remark5.1}
 The value $\mathfrak{N}_n(4)$ of the Narayana polynomial at $\beta = 4$ has the
following combinatorial interpretation:
 $\mathfrak{N}_{n}(4)$ is equal to the number of dif\/ferent lattice paths
from the point $(0,0)$ to that $(n,0)$ using steps from the set
$\Sigma = \{ (k,k) \, \text{or}\, (k,-k), \, k \in \Z_{ > 0} \}$, that never go below the
$x$-axis, see \cite[$A059231$]{SL}.
\end{Remark}

\begin{Exercises}\label{exercises5.2}\quad
\begin{enumerate}\itemsep=0pt
\item[$(a)$] Show that
\begin{gather*} \gamma_{k,n}:= { C_{n}^{(k+1)} \over C_{n}^{(k)}} = {(2n-2k) ! (2k+1) !
\over (n-k) ! (n+k+1) !}.
\end{gather*}
\item[$(b)$] Show that
$\gamma_{k,n} \le 1$ if $k \le n \le 3k+1$, and $\gamma_{k,n} \ge
2^{n-3k-1}$ if $n > 3k+1$.
\end{enumerate}
\end{Exercises}

{\bf (4)} {\it Polynomials $\mathfrak{F}_{w}(\beta)$,
$\mathfrak{H}_{w}(\beta)$, $\mathfrak{H}_{w}(q,t;\beta)$ and $\mathfrak{R}_{w}(
q;\beta)$.}
Let $w \in \mathbb{S}_n$ be a permutation,
$\mathfrak{G}^{(\beta)}_{w}(X_n)$ and~$\mathfrak{G}^{(\beta)}_{w}(X_n,Y_n)$
be the corresponding $\beta$-Grothendieck and double
$\beta$-Grothendieck polynomials. We denote by
$\mathfrak{G}^{(\beta)}_{w}(1)$ and by
$\mathfrak{G}^{(\beta)}_{w}(1;1)$ the specializations
$X_n:=(x_1=1,\ldots$, $x_n=1)$,
$Y_n:=(y_1=1,\ldots,y_n=1)$ of the $\beta$-Grothendieck polynomials introduced
above.

\begin{Theorem} \label{theorem5.4} Let $w \in \mathbb{S}_n$ be a permutation.
Then
\begin{enumerate}\itemsep=0pt
\item[$(i)$] The polynomials $\mathfrak{F}_{w}(\beta):= \mathfrak{G}_{w}^{(\beta-1)}(1)$ and $\mathfrak{H}_{w}(\beta):=\mathfrak{G}_{w}^{(\beta-1)}(1;1)$
have both non-negative integer coefficients.

\item[$(ii)$] One has
\begin{gather*}
\mathfrak{H}_{w}(\beta)=(1+\beta)^{\ell(w)} \mathfrak{F}_{w}\big(\beta^2\big).
\end{gather*}

\item[$(iii)$] Let $w \in \mathbb{S}_n$ be a permutation, define polynomials
\begin{gather*}
\mathfrak{H}_{w}(q,t;\beta):=
\mathfrak{G}_{w}^{(\beta)}(x_1=q,x_2=q, \ldots,x_n=q,y_1=t,y_2=t,\ldots,
y_n=t)
\end{gather*}
to be the specialization $\{x_i=q,\,y_i=t,\, \forall\, i \}$ of the double
 $\beta$-Grothendieck polynomial $\mathfrak{G}_{w}^{(\beta)}(X_n,Y_n)$.
Then
\begin{gather*}
 \mathfrak{H}_{w}(q,t;\beta) =(q+t+\beta q t)^{\ell(w)} \mathfrak{F}_{w}((1+\beta q)(1+\beta t)).
\end{gather*}
In particular, $\mathfrak{H}_{w}(1,1;\beta) = (2+\beta)^{\ell(w)}
\mathfrak{F}_{w}((1+\beta)^2)$.

\item[$(iv)$] Let $w \in \mathbb{S}_n$ be a permutation, define polynomial
\begin{gather*}
{\cal{R}}_{w}(q;\beta) := \mathfrak{G}_{w}^{(\beta-1)}(x_1=q,x_2=1,x_3=1,\ldots)
\end{gather*}
to be the specialization $\{x_1=q,\, x_i=1,\, \forall\, i \ge 2 \}$, of the
 $(\beta-1)$-Grothendieck polynomial $\mathfrak{G}_{w}^{(\beta -1)}(X_n)$.
Then
\begin{gather*} {\cal{R}}_{w}(q;\beta) = q^{w(1)-1} \mathfrak{R}_{w}(q;\beta),
\end{gather*}
where $\mathfrak{R}_{w}(q;\beta)$ is a polynomial in $q$ and $\beta$ with
non-negative integer coefficients, and $\mathfrak{R}_{w}(0;\beta=0)=1$.

\item[$(v)$] Consider permutation $w_n^{(1)}:= [1,n,n-1,n-2,\dots, 3,2]
\in \mathbb{S}_n$.
Then $\mathfrak{H}_{w_{n}^{(1)}}(1,1; 1) =
3^{{n-1 \choose 2}} \mathfrak{N}_n(4)$.
\end{enumerate}
\end{Theorem}

In particular, if $w_n^{(k)}= (1,2,\ldots,k,n,n-1,\ldots,k+1) \in {\mathbb{S}}_n$, then
\begin{gather*}
 \s_{w_{n}^{(k)}}^{(\beta-1)}(1;1) =(1+\beta)^{{n-k \choose 2}} \s_{w_{n}^{(k)}}^{(\beta-1)}\big(\beta^2\big).
\end{gather*}
See Remark~\ref{remark5.1} for a combinatorial interpretation of the number
$\mathfrak{N}_{n}(4)$.

\begin{Example}\label{example5.4}
 Consider permutation $v= [2,3,5,6,8,9,1,4,7] \in
\mathbb{S}_{9}$ of the length $12$, and set
 $x:=(1+\beta q)(1+\beta t)$. One can check that
\begin{gather*}
\mathfrak{H}_{v}(q,t ; \beta) = x^{12} (1+2 x)\big(1+6 x+19 x^2 + 24 x^3 + 13 x^4\big),
\end{gather*}
and $\mathfrak{F}_{v} (\beta)
= (1+2 \beta)(1 +6 \beta+ 19 \beta^2 + 24 \beta^3 + 13 \beta^4)$.

Note that $\mathfrak{F}_{v}( \beta=1) = 27 \times 7$, and $7 = {\rm AMS}(3)$,
$26={\rm CSTCTPP}(3)$, cf.\ Conjecture~\ref{conjecture5.4}, Section~\ref{section5.2.4}.
\end{Example}

\begin{Remark}\label{remark5.2}
One can show, cf.~\cite[p.~89]{M}, that if
$w \in \mathbb{S}_{n}$, then ${\cal{R}}_w(1,\beta)=
{\cal {R}}_{w^{-1}}(1,\beta)$.
 However, the equality ${\mathfrak{R}}_w(q,\beta) =
{\mathfrak{R}}_{w^{-1}}(q,\beta)$ can be violated, and it seems that
in general, there are no simple connections between polynomials
${\mathfrak{R}}_w(q,\beta)$ and ${\mathfrak{R}}_{w^{-1}}(q,\beta)$, if so.
\end{Remark}

From this point we shell use the notation $(a_{0},a_{1},\ldots,a_{r})_{\beta} := \sum\limits_{j=0}^{r} a_{j} \beta^{j}$, etc.

\begin{Example}\label{example5.5}
 Let us take $w = [1,3,4,6,7,9,10,2,5,8]$. Then
\begin{gather*}
\mathfrak{R}_{w}(q,\beta)= (1, 6, 21, 36, 51, 48, 26)_{\beta} +
q \beta (6, 36, 126, 216, 306, 288, 156)_{\beta}\\
\hphantom{\mathfrak{R}_{w}(q,\beta)=}{}+
q^2 \beta^3 (20, 125, 242, 403, 460, 289)_{\beta} +
q^3 \beta^5 (6, 46, 114, 204, 170)_{\beta}.
\end{gather*}
Moreover,
$\mathfrak{R}_{w}(q,1)= (189,1134,1539,540)_{q}$.
 On the other hand,
$w^{-1}\! = [1,8,2,3,9,4,5,10,6,7]$, and
\begin{gather*}
\mathfrak{R}_{w^{-1}}(q,\beta) =
(1, 6, 21, 36, 51, 48, 26)_{\beta}+
q \beta (1, 6, 31, 56, 96, 110, 78)_{\beta}\\
\hphantom{\mathfrak{R}_{w^{-1}}(q,\beta) =}{} +
 q^2 \beta (1, 6, 27, 58, 92, 122, 120, 78)_{\beta}+
q^3 \beta (1, 6, 24, 58, 92, 126, 132, 102, 26)_{\beta}\\
\hphantom{\mathfrak{R}_{w^{-1}}(q,\beta) =}{}
+
 q^4 \beta (1, 6, 22, 57, 92, 127, 134, 105, 44)_{\beta}\\
\hphantom{\mathfrak{R}_{w^{-1}}(q,\beta) =}{}
 +
 q^5 \beta (1, 6, 21, 56, 91, 126, 133, 104, 50)_{\beta}\\
\hphantom{\mathfrak{R}_{w^{-1}}(q,\beta) =}{}
 +
 q^6 \beta (1, 6, 21, 56, 91, 126, 133, 104, 50)_{\beta}.
\end{gather*}
Moreover, $\mathfrak{R}_{w^{-1}}(q,1)=(189,378,504,567,588, 588,588)_{q}$.

Notice that $w= 1 \times u$, where $u=[2,3,5,6,8,9,1,4,7]$. One can show that
\begin{gather*}
\mathfrak{R}_{u}(q,\beta)= (1, 6, 11, 16, 11)_{\beta}+
q \beta^2 (10, 20, 35, 34)_{\beta}+
 q^2 \beta^4 (5, 14, 26)_{\beta}.
\end{gather*}

On the other hand,
$u^{-1}= [7,1,2,8,3,4,9,5,6]$ and
\begin{gather*}
\mathfrak{R}_{u^{-1}}(1,\beta)= (1,6,21,36,51,48,26)_{\beta} = \mathfrak{R}_{u}(1,\beta).
\end{gather*}
\end{Example}

Recall that by our def\/inition $(a_{0},a_{1},\ldots,a_{r})_{\beta} :=
\sum\limits_{j=0}^{r} a_{j} \beta^{j}$.
\end{Comments}

\subsubsection[Grothendieck polynomials and $k$-dissections]{Grothendieck polynomials and $\boldsymbol{k}$-dissections}\label{section5.2.2}

Let $ k \in \N$ and $n \ge k-1$,~be a integer, def\/ine {\it a $k$-dissection} of
a convex $(n+k+1)$-gon to be a~collection $\cal{E}$ of diagonals in $(n+k+1)$-gon not containing $(k+1)$-subset
of pairwise crossing diagonals and such that at least $2(k-1)$ diagonals are
coming from each vertex of the $(n+k+1)$-gon in question. One can show that
the number of diagonals in any $k$-dissection~$\cal{E}$ of a convex
$(n+k+1)$-gon contains at least $(n+k+1)(k-1)$ and at most $n(2k-1)-1$
diagonals. We def\/ine the {\it index} of a $k$-dissection $\cal{E}$ to be
$i({\cal{E}})= n(2k-1) -1- \# | {\cal{E}}|$. Denote by
\begin{gather*}
{\cal{T}}_n^{(k)}(\beta)=\sum_{\cal{E}} \beta^{i(\cal{E})}
\end{gather*}
 the generating function for the number of $k$-dissections with a f\/ixed
index, where the above sum runs over the set of all $k$-dissections of a~convex $(n+k+1)$-gon.
\begin{Theorem}\label{theorem5.5}
\begin{gather*}
 \mathfrak{G}_{\pi_{k}^{(n)}}^{(\beta)}(x_1=1,\ldots,x_n=1)= {\cal{T}}_{n}^{
(k)}(\beta).
\end{gather*}
\end{Theorem}

 Mopre generally, let $ n \ge k > 0$ be integers, consider a convex $(n+k+1)$-gon $P_{n+k+1}$ and a
vertex $v_0 \in P_{n+k+1}$. Let us label clockwise the vertices of
$P_{n+k+1}$ by the numbers $1,2,\ldots,n+k+1$ starting from the vertex~$v_0$.
Let $\operatorname{Dis}(P_{n+k+1})$ denotes the set of all $k$-dissections of the $(n+k+1)$-gon $P_{n+k+1}$. We denote by $D_0:=\operatorname{Dis}_{0}(P_{n+k+1})$ the ``minimal''
$k$-dissection of the $(n+k+1)$-gon $P_{n+k+1}$ in question
consisting of the set of diagonals connecting vertices~$v_{a}$ and
$v_{\overline{a+r}}$, where $ 2 \le r \le k$, $1 \le a \le n+k+1$, and for
any positive integer $a$ we denote by $\overline{a}$ a unique integer such
that $1 \le \overline{a} \le n+k+1$ and $a \equiv \overline{a} ({\rm mod}\, ( n+k+1))$.
For example, if $k=1$, then $\operatorname{Dis}_0(P_{n+2}) =\varnothing$; if $k=3$ and $n=4$,
 in other words, $P_8$ is a octagon, the minimal $3$-dissection consists of~$16$ diagonals connecting vertices with the following labels
 \begin{gather*}
 1 \rightarrow 3 \rightarrow 5 \rightarrow 7 \rightarrow \overline{9}=1, \qquad
2 \rightarrow 4 \rightarrow 6 \rightarrow 8 \rightarrow \overline{10}=2, \\
1 \rightarrow 4 \rightarrow 7 \rightarrow \overline{10}=2 \rightarrow 5
\rightarrow 8 \rightarrow \overline{11}=3 \rightarrow 6 \rightarrow \overline{9}=1.
\end{gather*}

Now let $D \in \operatorname{Dis}(P_{n+k+1})$ be a dissection. Consider a
diagonal $d_{ij} \in (D {\setminus} D_{0})$, $i < j$ which connects vertex $v_i$
with that~$v_j$. We attach variable $x_i$ to the diagonal~$d_{ij}$ in
question and consider the following expression
\begin{gather*}
 {\cal{T}}_{P_{n+k+1}}(X_{n+k+1}) = \sum_{D \in \operatorname{Dis}(P_{n+k+1})} \beta^{\#|D {\setminus} D_0|} \sum_{d_{ij} \in (D {\setminus} D_{0}) \atop i < j} \prod x_{i}.
 \end{gather*}

\begin{Theorem} \label{theorem5.6}
One has
\begin{gather*}
 {\cal{T}}_{P_{n+n+1}}(X_{n+k+1}) = \beta^{k(n-k)} \prod_{a=1}^{n} x_a^{\min(n-a+1,n-k)} \mathfrak{G}_{w_{k}^{n}}^{\beta^{-1}}\big(x_1^{-1},\dots,x_{n}^{-1}\big).
 \end{gather*}
\end{Theorem}

\begin{Exercises} \label{exercises5.3}
It is not dif\/f\/icult to check that
\begin{gather*}
 {\mathfrak{G}}_{15432}^{\beta}(X_5) = \beta^3 x_1^3 x_2^3 x_3^2 x_4
 + \beta^2 (x_1^3 x_2^3 x_3
+ 2 x_1^3 x_2^3 x_3 x_4
+ 3 x_1^3 x_2^2 x_3^2 x_4
+ 3 x_1^2 x_2^3 x_3^2 x_4) \\
\hphantom{{\mathfrak{G}}_{15432}^{\beta}(X_5) =}{}
 + \beta (x_1^3 x_2^3 x_3
+ x_1^3 x_2^3 x_4 + 2 x_1^3 x_2^2 x_3
+ 2 x_1^2 x_2^3 x_3^2 + 3 x_1^3 x_2^2 x_3 x_4
+ 3 x_1^3 x_2 x_3^2 x_4\\
\hphantom{{\mathfrak{G}}_{15432}^{\beta}(X_5) =}{}
+ 3 x_1^2 x_2^3 x_3 x_4
+ 3 x_1^2 x_2^2 x_3^2 x_4
+ 3 x_1 x_2^3 x_3^2 x_4)
+
 x_1^3 x_2^2 x_3
+ x_1^3 x_2^2 x_4
+ x_1^3 x_2 x_3^2\\
\hphantom{{\mathfrak{G}}_{15432}^{\beta}(X_5) =}{}
+x_1^3 x_2 x_3 x_4
+x_1^3 x_3^2 x_4
+x_1^2 x_2^3 x_3
+x_1^2 x_2^3 x_4
+x_1^2 x_2^2 x_3^2
+x_1^2 x_2^2 x_3 x_4
+x_1^2 x_2 x_3^2 x_4\\
\hphantom{{\mathfrak{G}}_{15432}^{\beta}(X_5) =}{}
+x_1 x_2^3 x_3^2
+x_1 x_2^3 x_3 x_4
+x_1 x_2^2 x_3^2 x_4
+x_2^3 x_3^2 x_4.
\end{gather*}
Describe
bijection between dissections of hexagon $P_{6}$ (the
case $k=1$, $n=4$) and the above listed monomials involved in the $\beta$-Grothendieck polynomial ${\mathfrak{G}}_{15432}^{\beta}(x_1,x_2,x_3,x_4)$.
\end{Exercises}

A $k$-dissection of a convex $(n+k+1)$-gon with the maximal number of
diagonals (which is equal to $n(2k-1)-1$) is called {\it $k$-triangulation}.
It is well-known that the number of $k$-triangulations of a convex
$(n+k+1)$-gon is equal to the Catalan--Hankel number $C_{n-1}^{(k)}$.
Explicit bijection between the set of $k$-triangulations of a convex
$(n+k+1)$-gon and the set of $k$-tuple of non-crossing Dick paths
$(\gamma_1, \ldots,\gamma_k)$ such that the Dick path~$\gamma_i$ connects
points $(i-1,0)$ and $(2n-i-1,0)$, has been constructed in~\cite{SSt,St}.

\subsubsection[Grothendieck polynomials and $q$-Schr\"oder polynomials]{Grothendieck polynomials and $\boldsymbol{q}$-Schr\"oder polynomials}\label{section5.2.3}

Let $\pi_{k}^{(n)}=1^{k} \times w_{0}^{(n-k)} \in \mathbb{S}_n$ be the
vexillary permutation as before, see Theorem~\ref{theorem5.3}. Recall that
\begin{gather*}
\pi_{k}^{(n)}= \begin{pmatrix} 1& 2& \ldots& k& k+1& k+2&\ldots& n \\
1& 2& \ldots& k& n& {n-1}& \ldots& k+1 \end{pmatrix}.
\end{gather*}

$({\bf A})$ {\bf Principal specialization of the Schubert polynomial
$\boldsymbol{\mathfrak{S}_{\pi_{k}^{(n)}}}$.}
Note that $\pi_{k}^{(n)}$ is a~vexillary permutation of the staircase shape
$\lambda =(n-k-1,\ldots,2,1)$ and has the staircase f\/lag $\phi=(k+1,k+2,\ldots,n-1)$. It is known, see, e.g., \cite{M,Wa}, that for a vexillary
permutation $w \in \mathbb{S}_n$ of the shape $\lambda$ and f\/lag $\phi=(\phi_1,
\ldots,\phi_r)$, $r= \ell(\lambda)$, the corresponding Schubert polynomial
$\mathfrak{S}_{w}(X_n)$ is equal to the multi-Schur polynomial
$s_{\lambda}(X_{\phi})$, where $X_{\phi}$ denotes the f\/lagged set of variables, namely, $X_{\phi}=(X_{\phi_{1}}, \ldots,X_{\phi_{r}})$ and $X_m=
(x_1,\ldots,x_m)$. Therefore we can write the following determinantal formula
for the principal specialization of the Schubert polynomial corresponding to
the vexillary permutation~$\pi_{k}^{(n)}$
\begin{gather*}
\mathfrak{S}_{\pi_{k}^{(n)}}\big(1,q,q^2,\ldots\big)= \operatorname{DET} \left({n-i+j-1 \brack k+i-1 }_{q} \right)_{1 \le i,j \le n-k},
\end{gather*}
where ${n \brack k }_{q}$ denotes the $q$-binomial
coef\/f\/icient.

Let us observe that the Carlitz--Riordan $q$-analogue $C_n(q)$ of the Catalan
number $C_n$ is equal to the value of the $q$-Schr\"oder polynomial at
$\beta=0$, namely, $C_n(q)=S_n(q,0)$.

\begin{Lemma} \label{lemma5.2}
Let $k$, $n$ be integers and $n > k$, then
\begin{alignat*}{3}
& (1) \quad && \operatorname{DET} \left({n-i+j-1 \brack k+i-1}_{q} \right)_{1 \le i,j \le n-k} = q^{{n-k \choose 3}} C_{n}^{(k)}(q),&\\
& (2) \quad && \operatorname{DET} \big( C_{n+k-i-j}(q) \big)_{1 \le i,j \le k} =
q^{k(k-1)(6n-2k-5
)/6}~C_{n}^{(k)}(q).&
\end{alignat*}
\end{Lemma}

$({\bf{B}})$ {\bf Principal specialization of the Grothendieck
polynomial $\boldsymbol{\mathfrak{G}_{\pi_{k}^{(n)}}^{(\beta)}}$.}

\begin{Theorem}\label{theorem5.7}
\begin{gather*}
q^{{n-k+1 \choose 3}-(k-1) {n-k \choose 2}}
\operatorname{DET}\big|S_{n+k-i-j}\big(q; q^{i-1} \beta\big) \big|_{1 \le i,j \le k}\\
\qquad {}=
q^{k(k-1)(4k+1)/6} \prod_{a=1}^{k-1}\big(1+q^{a-1} \beta\big) \mathfrak{G}_{\pi_{k}^{(n)}}\big(1,q,q^2,\ldots\big).
\end{gather*}
\end{Theorem}
\begin{Corollary}\label{corollary5.2}\quad
\begin{enumerate}\itemsep=0pt
\item[$(1)$] If $k=n-1$, then
\begin{gather*}
\operatorname{DET} |S_{2n-1-i-j}\big(q; q^{i-1} \beta\big) |_{1 \le i,j \le n-1}=q^{(n-1)(n-2)(4n-3)/6} \prod_{a=1}^{n-2}\big(1+q^{a-1} \beta\big)^{n-a-1},
\end{gather*}

\item[$(2)$] If $k=n-2$, then
\begin{gather*}
 q^{n-2} \operatorname{DET} \big| S_{2n-2-i-j}\big(q; q^{i-1} \beta\big) \big|_{1 \le i,j \le n-2}\\
 \qquad{}=
 q^{(n-2)(n-3)(4n-7)/6} \prod_{a=1}^{n-3}\big(1+q^{a-1} \beta\big)^{n-a-2}
\left\{{(1+ \beta)^{n-1}-1 \over \beta} \right\}.
\end{gather*}
\end{enumerate}
\end{Corollary}

{\bf Generalization.}
Let ${\boldsymbol{n}}=(n_1,\ldots,n_p) \in \N^{p}$ be a composition of $n$ so that $n=n_1+ \cdots + n_p$. We set $n^{(j)}=n_1+ \cdots+n_j$, $j=1,\ldots,p$, $n^{(0)}=0$.

Now consider the permutation $w^{({\boldsymbol{n}})}=w_{0}^{(n_1)} \times w_{0}^{(n_2)} \times \cdots \times w_{0}^{(n_p)} \in \mathbb{S}_n$,
where $w_{0}^{(m)} \in \mathbb{S}_m$ denotes the longest permutation in the
symmetric group $\mathbb{S}_m$. In other words,
\begin{gather*}
w^{({\boldsymbol{n}})} = \begin{pmatrix} 1& 2& \ldots& n_1& n^{(2)}& \ldots &
n_1+1& \ldots& n^{(p-1)}& \ldots n \\
n_{1}& {n_1-1}& \ldots& 1& {n_1+1}& \ldots & n^{(2)}& \ldots& n& \ldots
n^{(p-1)+1} \end{pmatrix}.
\end{gather*}
For the permutation $w^{({\boldsymbol{n}})}$ def\/ined above, one has the following
factorization formula for the Grothendieck polynomial corresponding to
$w^{({\boldsymbol{n}})} $~\cite{M}
\begin{gather*}
\mathfrak{G}_{w^{({\boldsymbol{n}})}}^{(\beta)} =\mathfrak{G}_{w_{0}^{(n_1)}}^{(\beta)} \times
\mathfrak{G}_{1^{n_1} \times w_{0}^{(n_2)}}^{(\beta)} \times
\mathfrak{G}_{1^{n_1+n_2} \times w_{0}^{(n_3)}}^{(\beta)} \times \cdots \times
\mathfrak{G}_{1^{n_{1}+ \ldots n_{p-1}} \times w_{0}^{(n_{p})}}^{(\beta)}.
\end{gather*}
In particular, if
\begin{gather}\label{equation5.11}
 w^{({\boldsymbol{n}})}= w_{0}^{(n_1)} \times w_{0}^{(n_2)} \times \cdots \times w_{0}^{(n_p)} \in \mathbb{S}_n,
\end{gather}
then the principal specialization $\mathfrak{G}_{w^{({\boldsymbol{n}})}}^{(\beta)}$ of the Grothendieck polynomial corresponding to the permutation $w$, is the product
of $q$-Schr\"oder--Hankel polynomials. Finally, we observe that from
discussions in Section~\ref{section5.2.1}(3), {\it Grothendieck and Narayana polynomials}, one can deduce that
\begin{gather*}
 \mathfrak{G}_{w^{({\boldsymbol{n}})}}^{(\beta-1)}(x_1=1,\ldots,x_n=1)= \prod_{j=1}^{p-1} \mathfrak{N}_{n^{(j+1)}}^{(n^{(j)})} (\beta).
 \end{gather*}
In particular, the polynomial $\mathfrak{G}_{w^{({\boldsymbol{n}})}}^{(\beta-1)}(x_1, \ldots,x_n)$ is a symmetric polynomial in $\beta$ with non-negative integer coef\/f\/icients.

\begin{Example}\label{example5.6}\quad

 $(1)$ Let us take (non vexillary) permutation $w=2143=s_1 s_3$.
One can check that
\begin{gather*}
\mathfrak{G}_{w}^{(\beta)}(1,1,1,1)=3+3 \beta+\beta^2=1+(\beta+1)+(\beta+1)^2,
 \end{gather*}
and
\begin{gather*}
\mathfrak{N}_4(\beta)=(1,6,6,1) , \qquad \mathfrak{N}_3(\beta)=(1,3,1), \qquad \mathfrak{N}_2(\beta)=(1,1).
\end{gather*}
 It is easy to see that
\begin{gather*}
 \beta \mathfrak{G}_{w}^{(\beta)}(1,1,1,1)= \operatorname{DET} \left | \begin{matrix}
\mathfrak{N}_4(\beta) & \mathfrak{N}_3(\beta) \\
\mathfrak{N}_3(\beta) & \mathfrak{N}_2(\beta)
\end{matrix} \right |.
\end{gather*}
 On the other hand,
\begin{gather*}
\operatorname{DET} \left | \begin{matrix}
P_4(\beta) & P_3(\beta) \\
P_3(\beta) & P_2(\beta)
\end{matrix} \right | = (3,6,4,1)=\big(3+3 \beta+ \beta^2\big) (1+\beta).
\end{gather*}
It is more involved to check that
\begin{gather*}
 q^5 (1+\beta)~\mathfrak{G}_{w}^{(\beta)}\big(1,q,q^2,q^3\big)= \operatorname{DET} \left |
\begin{matrix}
S_4(q;\beta) & S_3(q;\beta) \\
S_3(q; q \beta) & S_2(q; q \beta)
\end{matrix} \right | .
\end{gather*}

$(2)$~Let us illustrate Theorem~\ref{theorem5.7} by a few examples. For the sake of
simplicity, we consider the case $\beta=0$, i.e., the case of {\it Schubert polynomials}. In this case $P_n(q;\beta=0)=C_n(q)$ is equal to the Carlitz--Riordan $q$-analogue of Catalan numbers. We are reminded that the $q$-Catalan--Hankel polynomials are def\/ined as follows
\begin{gather*}
C_n^{(k)}(q)= q^{k(1-k)(4k-1)/6} \operatorname{DET} | C_{n+k-i-j}(q) |_{1 \le i,j \le n} .
\end{gather*}
In the case $\beta =0$ the Theorem~\ref{theorem5.7} states that if ${\boldsymbol{n}} =(n_1, \ldots,
n_p) \in \N^{p}$ and the permutation $w_{({\boldsymbol{n}})} \in \mathbb{S}_n$ is def\/ined
by the use of~\eqref{equation5.10}, then
\begin{gather*}
\mathfrak{S}_{w^{({\boldsymbol{n}})}}\big(1,q,q^2,\ldots\big)
=q^{\sum {n_i \choose 3}} C_{n_{1}+n_{2}}^{(n_1)}(q) \times C_{n_{1}+n_{2}+n_{3}}^{(n_1+n_2)}(q) \times C_{n}^{(n-n_{p})}(q).
\end{gather*}
Now let us consider a few examples for $n=6$.
\begin{itemize}\itemsep=0pt
\item ${\boldsymbol{n}}=(1,5)$ $\Longrightarrow$ $\mathfrak{S}_{w^{({\boldsymbol{n}})}}(1,q,\ldots)= q^{10} C_{6}^{(1)}(q)=C_{5}(q)$.

\item ${\boldsymbol{n}}=(2,4)$ $\Longrightarrow$ $\mathfrak{S}_{w^{({\boldsymbol{n}})}}(1,q, \ldots)=q^{4} C_{6}^{(2)}(q) = \operatorname{DET} \left | \begin{matrix}
C_{6}(q) & C_{5}(q) \\
C_5(q) & C_4(q)
\end{matrix} \right |$.
\end{itemize}

Note that $\mathfrak{S}_{w^{(2,4)}}(1,q,\ldots)=\mathfrak{S}_{w^{(1,1,4)}}(1,q,\ldots)$.
\begin{itemize}\itemsep=0pt
\item ${\boldsymbol{n}}=(2,2,2)$ $\Longrightarrow$
$\mathfrak{S}_{w^{({\boldsymbol{n}})}}(1,q,\ldots)=C_{4}^{(2)}(q) C_{6}^{(4)}(q)$.

\item ${\boldsymbol{n}}=(1,1,4)$ $\Longrightarrow$ $\mathfrak{S}_{w^{({\boldsymbol{n}})}}(1,q,\ldots)= q^4 C_{2}^{(1)}(q) C_{4}^{(2)}(q)=q^4 C_{4}^{(2)}(q)$,
the last equality follows from that $C_{k+1}^{(k)}(q)=1$ for all $k \ge 1$.

\item ${\boldsymbol{n}}=(1,2,3)$ $\Longrightarrow$ $\mathfrak{S}_{w^{({\boldsymbol{n}})}}(1,q,\ldots)=q C_{3}^{(1)}(q) C_{6}^{(3)}(q)$. 

\item ${\boldsymbol{n}}=(3,2,1)$ $\Longrightarrow$ $\mathfrak{S}_{w^{({\boldsymbol{n}})}}(1,q,\ldots)=q C_{5}^{(3)}(q) C_{6}^{(5)}(q)=q C_{5}^{(3)}(q)=q(1,1,1,1)$.
Note that $C_{k+2}^{(k)}(q)= {k+1 \brack 1 }_{q}$.
\end{itemize}
\end{Example}

\begin{Exercises}\label{exercises5.4} Let $1 \le k \le m \le n$ be integers, $n \ge 2k+1$. Consider permutation
\begin{gather*}
 w = \begin{pmatrix} 1& 2& \ldots& k& {k+1}& \ldots& &n \\
m& {m-1}& \ldots& {m-k+1}& n& \ldots& \ldots& 1 \end{pmatrix} \in
\mathbb{S}_n.
\end{gather*}
Show that
\begin{gather*}
\mathfrak{S}_{w}(1,q,\ldots)= q^{n(D(w))} C_{n-m+k}^{(m)}(q),
\end{gather*}
where for any permutation $w$, $n(D(w))= \sum {d_i(w) \choose 2}$ and $d_i(w)$
denotes the number of boxes in the $i$-th column of the ({\it Rothe}) diagram~$D(w)$ of the permutation~$w$, see \cite[p.~8]{M}.
\end{Exercises}

({\bf C}) {\bf A determinantal formula for the Grothendieck polynomials
$\boldsymbol{\mathfrak{G}_{\pi_{k}^{(n)}}^{(\beta)}}$.}
Def\/ine polynomials
\begin{gather*}
\Phi_{n}^{(m)}(X_n)=\sum_{a=m}^{n}e_a(X_n) \beta^{a-m},\\
A_{i,j}(X_{n+k-1})={1 \over (i-j)!} \left({\partial \over \partial \beta}
\right)^{j-1} \Phi_{k+n-i}^{(n+1-i)}(X_{k+n-i}) \qquad \text{if}\quad 1 \le i \le j \le n,
\end{gather*}
and
\begin{gather*}
A_{i,j}(X_{k+n-1})=\sum_{a=0}^{i-j-1} e_{n-i-a}(X_{n+k-i})~{i-j-1 \choose a} \qquad \text{if} \quad 1 \le j < i \le n.
\end{gather*}

\begin{Theorem}\label{theorem5.8}
\begin{gather*}
\operatorname{DET} |A_{i,j}|_{1 \le i,j \le n}=\mathfrak{G}_{\pi_{k+n}^{(k)}}^{(\beta)}(X_{k+n-1}).
\end{gather*}
\end{Theorem}

\begin{Comments} \label{comments5.5}\quad

 $(a)$~One can compute the Grothendieck polynomials for yet
another interesting family of permutations. namely, {\it grassmannian}
permutations
\begin{gather*}
\sigma_{k}^{(n)}= \begin{pmatrix} 1& 2& \ldots& {k-1}& k& {k+1}& {k+2}& \ldots& n+k \\
1& 2& \ldots& {k-1}& {n+k}& {k}& {k+1}& \ldots & n+k-1 \end{pmatrix}\\
\hphantom{\sigma_{k}^{(n)}}{} =
s_k s_{k+1} \cdots s_{n+k-1} \in \mathbb{S}_{n+k}.
\end{gather*}
Then
\begin{gather*}
\mathfrak{G}_{{\sigma_{k}}^{(n)}}^{(\beta)}(x_1,\ldots,x_{n+k})=
\sum_{j=0}^{k-1} s_{(n,1^{j})}(X_k) \beta^{j},
\end{gather*}
where $s_{(n,1^{j})}(X_k)$ denotes the Schur polynomial corresponding to the
hook shape partition $(n,1^{j})$ and the set of variables
$X_k:=(x_1,\ldots,x_k)$. In particular,
\begin{gather*}
\mathfrak{G}_{{\sigma_{k}}^{(n)}}^{(\beta)}(x_j=1,\, \forall\, j)=
{n+k-1 \choose k} \left(\sum_{j=0}^{k-1} {k \over n+j} {k-1 \choose j}
 \beta^{j} \right) = \sum_{j=0}^{k-1} {n+j-1 \choose j} (1+\beta)^{j}.
 \end{gather*}

$(b)$ {\it Grothendieck polynomials for grassmannian permutations.}
In the case of a {\it grassmannian} permutation
$w:=\sigma_{\lambda} \in {\mathbb{S}}_{\infty}$ of the shape
$\lambda=(\lambda_1 \ge \lambda_2 \ge \cdots \ge \lambda_n )$ where $n$
is a unique descent of~$w$, one can prove the following
formulas for the $\beta$-Grothendieck polynomial
\begin{gather}\label{equation5.12}
 \mathfrak{G}_{\sigma_{\lambda}}^{(\beta)}(X_n)= {\operatorname{DET} \big| x_i^{\lambda_j +n-j}
(1+\beta x_i)^{j-1} \big|_{1 \le i,j \le n} \over \prod\limits_{1 \le i < j \le n}
(x_i-x_j) },
\\
\operatorname{DET}\big| h_{\lambda_j +i,j}^{(\beta)}(X_{[i,n]})\big|_{1 \le i,j \le n} =
\operatorname{DET}\big|h_{\lambda_j +i,j}^{(\beta)}(X_n)\big|_{1 \le i,j \le n},\nonumber
\end{gather}
where $X_{[i,n]}=(x_i,x_{i+1},\ldots,x_n)$, and for any set of variables~$X$
\begin{gather*}
h_{n,k}^{(\beta)}(X) = \sum_{a=0}^{k-1} {k-1 \choose a} h_{n-k+a}(X)
\beta^{a},
\end{gather*}
and $h_k(X)$ denotes the complete symmetric polynomial of degree~$k$ in the variables from the set~$X$.

A proof is a straightforward adaptation of the proof of special case $\beta =0$ (the case of {\it Schur} polynomials) given by I.~Macdonald~\cite[Section~2, equation~(2.10) and Section~4, equation~(4.9)]{M}.

Indeed, consider $\beta$-divided dif\/ference operators $\pi_{j}^{(\beta)}$,
$j=1,\ldots,n-1$, and $\pi_{w}^{(\beta)}$, $w \in {\mathbb{S}}_n$,
introduced in~\cite{FK1}. For example,
\begin{gather*}
\pi_{j}^{(\beta)}(f) = {1 \over x_j-x_{j+1}}\big((1+\beta x_{j+1}) f(X_n)-
(1+\beta x_j) f(s_{j}(X_n) \big).
\end{gather*}

Now let $w_{0}:= w_{0}^{(n)}$ be the longest element in the symmetric group
${\mathbb{S}}_n$. The same proves of the State\-ments~2.10,~2.16 from~\cite{M} show that{\samepage
\begin{gather*}
\pi_{w_{0}}^{(\beta)}= a_{\delta}^{-1} w_{0} \left( \sum_{\sigma \in {\mathbb{S}}_n}
(-1)^{\ell(\sigma)} \prod _{j=1}^{n-1} (1+\beta x_{j})^{n-j} \sigma \right),
\end{gather*}
where $a_{\delta} = \prod\limits_{1 \le i < j \le n} (x_i - x_j)$.}

On the other hand, the same arguments as in the proof of Statement~4.8
from~\cite{M} show that
\begin{gather*}
 \mathfrak{G}_{\sigma_{\lambda}}^{(\beta)}(X_n) = \pi_{w^{(0)}}^{(\beta)}
\big(x^{\lambda+\delta_{n}}\big).
\end{gather*}
Application of the formula for operator $\pi_{w_{n}^{(0)}}^{(\beta)}$
displayed above to the monomial $x^{\lambda +\delta_{n}}$ f\/inishes the proof
of the f\/irst equality in~\eqref{equation5.11}. The statement that the right hand side of
the equality~\eqref{equation5.12} coincides with determinants displayed in the identity~\eqref{equation5.12} can be checked by means of simple transformations.
\end{Comments}

\begin{Problems}\label{problems5.1}\quad
\begin{enumerate}\itemsep=0pt
\item[$(1)$] Give a bijective prove of Theorem~{\rm \ref{theorem5.5}}, i.e., construct a
bijection between
\begin{itemize}\itemsep=0pt
\item the set of $k$-tuple of mutually non-crossing
Schr\"oder paths $(\mathfrak{p}_1,\ldots,\mathfrak{p}_k)$ of lengths
$(n,n-1,\ldots,n-k+1)$ correspondingly, and

\item the set of pairs $(\mathfrak{m},{\cal{T}})$, where $\cal{T}$ is a $k$-dissection of a convex $(n+k+1)$-gon, and~$\mathfrak{m}$ is a upper triangle
$(0,1)$-matrix of size $(k-1) \times (k-1)$,
which is compatible with natural statistics on the both sets.
\end{itemize}

\item[$(2)$] Let $w \in \mathbb{S}_n$ be a permutation, and~${\rm CS}(w)$ be the set of
compatible sequences corresponding to~$w$, see, e.g.,~{\rm \cite{BJS}}.
Define statistics $c(\bullet)$ on the set ${\rm CS}(w)$ such that
\begin{gather*}
\mathfrak{G}_{w}^{(\beta-1)}(x_1=1,x_2=1,\ldots)= \sum_{a \in {\rm CS}(w)} \beta^{c(a)}.
\end{gather*}

\item[$(3)$] Let $w$ be a vexillary permutation.
Find a determinantal formula for the $\beta$-Grothendieck
polynomial ${\mathfrak G}_{w}^{(\beta)}(X)$.

\item[$(4)$] Let $w$ be a permutation.
Find a geometric interpretation of coefficients of the
polynomials $\s_{w}^{(\beta)}(x_i=1)$ and $\s_{w}^{(\beta)}(x_i=q, \,x_j=1,\,
\forall\, j \not= i)$.
\end{enumerate}
\end{Problems}

For example, let $w \in {\mathbb S}_n$ be an involution, i.e., $w^2=1$,
and $w' \in {\mathbb S}_{n+1}$ be the image of $w$ under the natural
embedding ${\mathbb S}_n \hookrightarrow {\mathbb S}_{n+1}$ given by
$ w \in {\mathbb S}_n \longrightarrow (w,n+1) \in {\mathbb S}_{n+1}$.
It is well-known, see, e.g., \cite{KMi, W}, that the multiplicity
$m_{e,w}$ of the $0$-dimensional Schubert cell $\{pt \}=Y_{w_{0}^{(n+1)}}$ in
the Schubert variety $\overline{Y}_{w'}$ is equal to the specialization
$\s_{w}(x_i=1)$ of the Schubert polynomial $\s_{w}(X_{n})$. Therefore one can
consider the polynomial $\s_{w}^{(\beta)}(x_i=1)$ as a $\beta$-deformation of
the multiplicity~$m_{e,w}$.

\begin{Question}\label{question-page105}
What is a geometrical meaning of the coefficients
of the polynomial \linebreak $\s_{w}^{(\beta)}(x_i=1) \in \N [\beta]$?
\end{Question}

\begin{Conjecture} \label{conjecture5.1}
 The polynomial $\s_{w}^{(\beta)}(x_i=1)$ is a unimodal
polynomial for any permutation~$w$.
\end{Conjecture}

\subsubsection{Specialization of Schubert polynomials}\label{section5.2.4}

Let $n$, $k$, $r$ be positive integers and $p$, $b$ be
non-negative integers such that $r \le p+1$. It is well-known~\cite{M} that
in this case there exists a unique {\it vexillary} permutation
$\varpi:= \varpi_{\lambda,\phi} \in \mathbb{S}_{\infty}$ which has the
{\it shape} $\lambda=(\lambda_1,\ldots,\lambda_{n+1})$ and the
{\it flag} $\phi =(\phi_1,\ldots,\phi_{n+1})$, where
\begin{gather*}
 \lambda_i= (n-i+1) p+b, \qquad \phi_i=k+1+r (i-1), \qquad 1 \le i \le n+1-
\delta_{b,0}.
\end{gather*}
According to a theorem by M.~Wachs~\cite{Wa}, the Schubert polynomial $\mathfrak{S}_{\varpi}(X)$ admits the following determinantal representation
\begin{gather*}
\mathfrak{S}_{\varpi}(X)= \operatorname{DET} \big( h_{\lambda_i-i+j}(X_{\phi_{i}}) \big)_{1 \le i,j \le n+1}.
\end{gather*}
Therefore we have
\begin{gather*}
\mathfrak{S}_{\varpi}(1):= \mathfrak{S}_{\varpi}(x_1=1,x_2=1,\ldots) \\
\hphantom{\mathfrak{S}_{\varpi}(1)}{} =\operatorname{DET} \left( {(n-i+1)p+b-i+j+k+(i-1)r \choose k+ (i-1)r} \right) _{1 \le i,j \le n+1}.
\end{gather*}
We denote the above determinant by $D(n,k,r,b,p)$.

\begin{Theorem}\label{theorem5.9}
\begin{gather*}
D(n,k,r,b,p)= \!\!
\prod_{(i,j) \in {\cal{A}}_{n,k,r}}\!\!\! {i+b+jp \over i}\!\!
\prod_{(i,j) \in {\cal{B}}_{n,k,r}}\!\!\!
{(k-i+1)(p+1)+(i+j-1)r+r(b+np) \over k-i+1+(i+j-1)r},
\end{gather*}
where
\begin{gather*}
{\cal{A}}_{n,k,r}= \big\{ (i,j) \in \Z_{\ge 0}^2 \,|\, j \le n, j < i \le k+(r-1)(n-j) \big\},\\
{\cal{B}}_{n,k,r}= \big\{ (i,j) \in \Z_{\ge 1}^2 \,|\, i+j \le n+1, \, i \not= k+1+r s,\, s\in \Z_{\ge 0} \bigr\}.
\end{gather*}
\end{Theorem}

It is convenient to re-write the above formula for $D(n,k,r,b,p)$ in the
following form
\begin{gather*}
 D(n,k,r,b,p)= \prod_{j=1}^{n+1} {((n-j+1)p +b+k+(j-1)(r-1) )! (n-j+1)! \over
(k+(j-1)r )!((n-j+1)(p+1)+b )!}\\
\hphantom{D(n,k,r,b,p)=}{} \times
 \prod_{1 \le i \le j \le n} ((k-i+1)(p+1)+jr+(np+b)r).
\end{gather*}

\begin{Corollary}[some special cases] \label{corollary5.3} $(A)$ The case $r=1$.

We consider below some special cases of Theorem~{\rm \ref{theorem5.9}} in the case
$r=1$. To simplify notation, we set $D(n,k,b,p):= D(n,k,r=1,b,p)$. Then we can rewrite the
above formula for $D(n,k,r,b,p)$ as follows
\begin{gather*}
D(n,k,b,p)=
\prod_{j=1}^{n+1} {((n+k-j+1)(p+1)+b ) ! ((n-j+1)p+b+k ) ! (j-1) ! \over ((n-j+1)(p+1)+b )! ((k+n-j+1)p+b+k)! (k+j-1)! }.
\end{gather*}

$(1)$ If $ k \le n+1$, then
\begin{gather*}
D(n,k,b,p)=
 \prod_{j=1}^{k} {(n+k+1-j)(p+1)+b \choose n-j+1} {(k-j)p+b+k \choose j}
{j! (k-j)! (n-j+1)! \over (n+k-j+1)!}.
\end{gather*}
In particular,
\begin{itemize}\itemsep=0pt
\item if $k=1$, then
\begin{gather*}
D(n,1,b,p)={1+b \over 1+b+(n+1)p}~{(p+1)(n+1)+b \choose n+1}:=F_{n+1}^{(p+1)}(b),
\end{gather*}
where $F_{n}^{p}(b):={1+b \over 1+b+(p-1)n} {pn+b \choose n}$ denotes the generalized Fuss--Catalan number,

\item if $k=2$, then
\begin{gather*}
D(n,2,b,p)={(2+b)(2+b+p) \over (1+b)(2+b+(n+1)p)(2+b+(n+2)p)} F_{n+1}^{(p+1)}(b) F_{n+2}^{(p+1)}(b),
\end{gather*}
in particular,
\begin{gather*}
 D(n,2,0,1)= \frac{6}{(n+3)(n+4)} {\rm Cat}_{n+1} {\rm Cat}_{n+2}.
 \end{gather*}
See {\rm \cite[$A005700$]{SL}} for several combinatorial interpretations of these
numbers.
\end{itemize}

$(2)$ Consider the Young diagram $($see R.A.~Proctor~{\rm \cite{Pr3})}
\begin{gather*}
\lambda:=\lambda_{n,p,b}= \big\{ (i,j) \in \Z_{\ge 1} \times \Z_{\ge 1} \,|\, 1 \le i \le n+1, \,1 \le j \le (n+1-i)p+b \}.
\end{gather*}
For each box $(i,j) \in \lambda$ define the numbers $c(i,j):=n+1-i+j$, and
\begin{gather*}
l_{(i,j)}(k)= \begin{cases}\dfrac{k+c(p,j)}{c(i,j)} & \text{if} \ \ j \le (n+1-i)(p-1)+b, \vspace{1mm}\\
\dfrac{(p+1)k+c(i,j)}{c(i,j)} & \text{if} \ \ (n+1-i)(p-1) < j-b \le (n+1-i)p. \end{cases}
\end{gather*}
Then
\begin{gather}\label{equation5.14}
 D(n,k,b,p)=\prod_{(i,j) \in \lambda} l_{(i,j)}(k).
\end{gather}
Therefore, $D(n,k,b,p)$ is a polynomial in $k$ with rational coefficients.

$(3)$ If $p=0$, then
\begin{gather*}
D(n,k,b,0)=\dim V_{(n+1)^{k}}^{\mathfrak{gl}(b+k)}=\prod_{j=1}^{n+k}
\left({j+b \over j}\right)^{\min(j,n+k+1-j)},
\end{gather*}
where for any partition $\mu$, $\ell(\mu) \le m$, $V_{\mu}^{\mathfrak{gl}(m)}$ denotes the irreducible $\mathfrak{gl}(m)$-module with the highest weight~$\mu$.
 In particular,
\begin{gather*}
 D(n,2,b,0)= {1 \over n+2+b}{n+2+b \choose b} {n+2+b \choose b+1}
 \end{gather*}
is equal to the Narayana number $N(n+b+2,b)$,
\begin{gather*}
 D(1,k,b,0) = {(b+k)! (b+k+1)! \over k! b! (k+1)! (b+1)!}:=
N(b+k+1,k),
\end{gather*}
and therefore the number $D(1,k,b,0)$ counts the number of pairs of
non-crossing lattice paths inside a~rectangular of size $(b+1) \times (k+1)$, which go from the point $(1,0)$
$($resp.\ from that $(0,1))$ to the point $(b+1,k)$ $($resp.\ to that $(b,k+1))$,
 consisting of steps $U=(1,0)$ and $R=(0,1)$, see~{\rm \cite[$A001263$]{SL}}, for some
list of combinatorial interpretations of the Narayana numbers.

$(4)$ If $p=b=1$, then
\begin{gather*}
D(n,k,1,1)= C^{(k)}_{n+k+1}:= \prod_{1 \le i \le j \le n+1} {2k+i+j \over i+j} .
\end{gather*}

$(5)$ If $p=1$ and $b$ is odd integer, then $D(n,k,b,1)$ is
equal to the dimension of the irreducible representation of the symplectic Lie
algebra ${\rm Sp}(b+2n+1)$ with the highest weight $k\omega_{n+1}$
$($R.A.~Proctor~{\rm \cite{Pr1,Pr2})}.

$(6)$ If $p=1$ and $b=0$, then
\begin{gather*}
 D(n,k,1,0) = D(n-1,k,1,1) = \prod_{ 1 \le i \le j \le n} {2k +i +j
\over i+j} =C_{n+k}^{(k)},
\end{gather*}
see section on Grothendieck and Narayana polynomials.

$(7)$ Let $\varpi_{\lambda}$ be a unique dominant
permutation of shape $\lambda:= \lambda_{n,p,b}$ and $\ell:= \ell_{n,p,b} = {1 \over 2}(n+1)(np+2b)$ be its length $($cf.~{\rm \cite{FK3})}. Then
\begin{gather*}
 \sum_{{\boldsymbol{a}} \in R(\varpi_{\lambda})} \prod_{i=1}^{\ell} (x+a_i) =
 \ell ! B(n,x,p,b).
 \end{gather*}
Here for any permutation~$w$ of length~$l$, we denote by~$R(w)$ the set
$\{ {\boldsymbol{a}}=(a_1,\ldots,a_l) \}$ of all reduced decompositions of~$w$.
\end{Corollary}

\begin{Exercises}\label{exercises5.5}
Show that
\begin{gather*}
\operatorname{DET} \big | F_{n+i+j-2}^{(2)}(0) \big |_{1 \le i,j \le k} =
\prod_{j=1}^{k} F_{n+j-1}^{(2)}(0) \frac{{k+1 \choose 2} !}{ \prod\limits_{1 \le i
\le k-1 \atop 1 \le j \le k}(n+i+j)},\\
D(n,k,b,1)= \prod_{j=1}^{k} F_{n+j}^{(2)}(b) \frac{\prod\limits_{1 \le i \le j \le k}(b+i+j-1)}{\prod\limits_{1 \le i \le k-1 \atop 1 \le j \le k}(n+b+i+j+1)}.
\end{gather*}
Clearly that if $b=0$, then $F_{n}^{(2)}(0)=C_n$, and $D(n,k,0,1)$ is equal to
the Catalan--Hankel determinant~$C_n^{(k)}$.

Finally we recall that the generalized
Fuss--Catalan number $F_{n+1}^{(p+1)}(b)$ counts the number of lattice paths
from $(0,0)$ to $(b+np,n)$ that do not go above the line $x=py$, see,
e.g.,~\cite{Kr}.
\end{Exercises}

\begin{Comments}\label{comments5.6} It is well-known, see, e.g., \cite{Pr3} or \cite[Vol.~2, Exercise~7.101.b]{ST1}, that the number $D(n,k,b,p)$ is equal
to the total number $pp^{\lambda_{n,p,b}}(k)$ of plane
partitions\footnote{Let $\lambda$ be a partition. A plane (ordinary) partition bounded
by $d$ and shape $\lambda$ is a f\/illing of the shape~$\lambda$ by the numbers
from the set $\{0,1,\ldots,d \}$ in such a way that the numbers along
columns and rows are weakly {\it decreasing}.
A~{\it reverse} plane partition bounded by~$d$ and shape~$\lambda$ is a~f\/illing of the shape~$\lambda$ by the numbers
from the set $\{0,1,\ldots,d \}$ in such a way that the numbers along columns
and rows are weakly {\it increasing}.}
bounded by $k$ and contained in the shape $\lambda_{n,b,p}$.

More generally, see, e.g.,~\cite{FK3}, for any partition $\lambda$ denote by
$w_{\lambda} \in \mathfrak{S}_{\infty}$ a unique {\it dominant} permutation of
shape~$\lambda$, that is a unique permutation with the code $c(w)=\lambda$.
Now for any non-negative integer $k$ consider the so-called {\it shifted
dominant} permutation~$w_{\lambda}^{(k)}$ which has the shape $\lambda$ and the
 f\/lag $\phi=(\phi_i=k+i-1,\, i=1,\ldots,\ell(\lambda))$. Then
\begin{gather*}
\s_{w_{\lambda}^{(k)}}(1) = pp^{\lambda}( \le k),
\end{gather*}
where $pp^{\lambda}(\le k)$ denotes the number of all plane partitions bounded by~$k$ and contained in~$\lambda$. Moreover,
\begin{gather*}
 \sum_{\pi \in PP^{\lambda}(\le k)} q^{|\pi|} = q^{n(\lambda)}
\s_{w_{\lambda}^{(k)}}\big(1,q^{-1},q^{-2},\ldots\big),
\end{gather*}
where $PP^{\lambda}(\le k)$ denotes the set of all plane partitions bounded by~$k$ and contained in~$\lambda$.
\end{Comments}

\begin{Exercises}\label{exercises5.6}\quad

$(1)$ Show that
\begin{gather*} \lim _{k \rightarrow \infty}~\s_{w_{\lambda}^{(k)}}\big(1,q,q^2,\ldots\big) =
\frac{q^{n(\lambda)}}{H_{\lambda}(q)},
\end{gather*}
where $H_{\lambda}(q) = \prod\limits_{x \in \lambda} (1-q^{h(x)})$
denotes the {\it hook} polynomial corresponding to a given partition~$\lambda$.

$(2)$ Let $\lambda = ((n+\ell)^{\ell}, \ell^{n})$ be a {\it fat hook}.
Show that
\begin{gather*}
 \lim_{k \rightarrow \infty} q^{n(\lambda)} \s_{w_{\lambda}^{(k)}}\big(1,q^{-1},q^{-2},\ldots\big)
 = q^{s(\ell,n)} \frac{K_{\lambda}(q)}{M_{\ell}(2 n+2 \ell - 1 ;q)},
\end{gather*}
where $a(\ell,n)$ is a certain integer we don't need to specify in what
follows,
\begin{gather*}
M_{\ell}(N;q)= \prod_{j=1}^{N} \left(\frac{1}{1- q^{j}} \right)^{\min(j,N+1-j,\ell)}
\end{gather*}
denotes the MacMahon generating function for the number of plane partitions
f\/it inside the box $N \times N \times \ell$, $K_{\lambda}(q)$ is a
polynomial in~$q$ such that $K_{\lambda}(0)=1$.

$(a)$ Show that
\begin{gather*}
(1-q)^{|\lambda|} \frac{K_{\lambda}(q)}{M_{\ell}(2 n+2 \ell - 1 ;q)}
\bigg |_{q=1} = {1 \over \prod\limits_{x \in \lambda} h(x)}.
\end{gather*}

$(b)$ Show that
\begin{gather*}
 K_{\lambda}(q) \in \N [q] \qquad \text{and}\qquad K_{\lambda}(1)= M(n,n,\ell),
\end{gather*}
where $M(a,b,c)$ denotes the number of plane partitions f\/it inside the box
$a \times b \times c$. It is well-known, see, e.g., \cite[p.~81]{M1}, that
\begin{gather*}
 M(a,b,c) = \prod_{\substack{1 \le i \le a \\1 \le j \le b\\ 1 \le k \le c}}
\frac{i+j+k-1}{i+j+k-2} = \prod_{i=1}^{c}\frac{(a+b+i-1) ! (i-1) !}{(a+i-1) ! (b+1-1) !} =\dim V_{(a^c)}^{{\mathfrak{gl}}_{b+c}}.
\end{gather*}
 Show that
\begin{gather*}
 K_{\lambda}(q) = \sum_{\pi \in B_{n,n,\ell}} q^{wt_{\ell}(\pi)},
\end{gather*}
where the sum runs over the set of plane partitions $\pi=(\pi_{ij})_{1 \le i,j \le n}$ f\/it inside the box
$B_{n,n,\ell}:=n \times n \times \ell$, and
\begin{gather*}
wt_{\ell}(\pi)= \sum_{i,j} \pi_{ij} + \ell \sum_i \pi_{ii}.
\end{gather*}

$(c)$ Assume as before that $\lambda:=((n+\ell)^{\ell},\ell^n)$.
Show that
\begin{gather*}
 \lim_{ n \rightarrow \infty} K_{\lambda}(q) = M_{\ell}(q) \sum_{\mu \atop \ell(\mu) \le \ell} q^{|\mu|} \left(\frac{q^{n(\mu)}}{\prod_{x \in \mu} (1-q^{h(x)})} \right)^2,
\end{gather*}
where the sum runs over the set of partitions $\mu$ with the number of parts
at most $\ell$, and $n(\mu)= \sum_{i} (i-1) \mu_i$,
\begin{gather*}
M_{\ell}(q) :=
\prod_{j \ge 1} \big(1-q^{j}\big)^{\min(j,\ell)}.
\end{gather*}
Therefore the generating function $PP^{(\ell,0)}(q):= \sum\limits_{\pi \in
 PP^{(\ell,0)}} q^{|\pi|}$
is equal to
\begin{gather*}
 \sum_{\mu \atop \ell(\mu) \le \ell} q^{|\mu|} \left(\frac{q^{n(\mu)}}{\prod\limits_{x \in \mu} (1-q^{h(x)})} \right)^2,
 \end{gather*}
where $PP^{(\ell,k)}:= \{ \pi= (\pi_{ij})_{i,j \ge 1} \,|\, \pi_{ij}
\ge 0, \pi_{\ell+1,\ell+1} \le k \}$, $|\pi|= \sum\limits_{i,j} \pi_{ij}$.

$(d)$ Show that
\begin{gather*}
PP^{(\ell,0)}(q) = \frac{1}{M_{\ell}(q)^2} \sum_{\mu, \atop \ell(\mu) \le \ell} (- q)^{|\mu|} q^{n(\mu)+n(\mu')} \big( \dim_{q} V_{\mu}^{{\mathfrak{gl}}(\ell)} \big)^2,
\end{gather*}
where $\mu'$ denotes the {\it conjugate} partition of $\mu$, therefore
$n(\mu') =\sum\limits_{i \ge 1} {\mu_i \choose 2}$.

The formula \eqref{equation5.14} is the special case $n=m$ of \cite[Theorem~1.2]{FM}.
 In particular, if $\ell=1$ then one come to following identity
\begin{gather*}
\frac{1}{(q;q)_{\infty}^2} \sum_{k \ge 0} (-1)^{k} q^{{k+1 \choose 2}} =
\sum_{k \ge 0} q^{k} \left(\frac{1}{(q;q)_{k}} \right)^2.
\end{gather*}

$(e)$ Let $k \ge 0$, $\ell \ge 1$ be integers.
Show that the ({\it fermionic}) generating function for the
number of plane partitions $\pi = (\pi_{ij}) \in PP^{(\ell,k)}$ is equal to
\begin{gather*}
 \sum_{\pi \in PP^{(\ell,k)}} q^{|\pi|} = \sum_{\mu \atop \mu_{\ell+1} \le k} q^{|\mu|} \left(\frac{q^{n(\mu)}}{\prod\limits_{x \in \mu} (1-q^{h(x)})} \right)^2.
 \end{gather*}
\end{Exercises}

$(B)$ {\it The case $k=0$.}
\begin{enumerate}\itemsep=0pt
\item[(1)] $D(n,0,1,p,b) =1$ for all nonnegative $n$, $p$, $b$.

\item[(2)] $D(n,0,2,2,2) ={\rm VSASM}(n)$, i.e., the number of alternating sign
$(2n+1) \times (2n+1)$ matrices symmetric about the vertical axis, see, e.g., \cite[$A005156$]{SL}.

\item[$(3)$] $D(n,0,2,1,2)= {\rm CSTCPP}(n)$, i.e., the number of cyclically symmetric
transpose complement plane partitions, see, e.g., \cite[$A051255$]{SL}.
\end{enumerate}

\begin{Theorem} \label{theorem5.10}
Let $\varpi_{n,k,p}$ be a unique vexillary permutation
of the shape $\lambda_{n.p}:=(n,n-1, \ldots,2,1)p$ and flag
$\phi_{n,k} := (k+1,k+2,\ldots,k+n-1,k+n)$. Then
\begin{gather*}
\mathfrak{G}_{\varpi_{n,1,p}}^{(\beta-1)}(1)= \sum_{j=1}^{n+1} {1 \over n+1} {n+1 \choose j} {(n+1)p \choose j-1} \beta^{j-1}.
\end{gather*}
If $k \ge 2$, then $G_{n,k,p}(\beta) := \mathfrak{G}_{\varpi_{n,k,p}}^{(\beta -1)}(1)$ is a~polynomial of deg\-ree~$nk$ in~$\beta$, and
\begin{gather*}
{\rm Coef\/f}_{[\beta^{nk}]}(G_{n,k,p}(\beta)) = D(n,k,1,p-1,0).
\end{gather*}
\end{Theorem}

The polynomial
\begin{gather*}
\sum_{j=1}^{n} {1 \over n} {n \choose j} {p n \choose j-1} t^{j-1}:=\mathfrak{FN}_{n}(t)
\end{gather*}
 is known as the Fuss--Narayana polynomial and can be considered as a $t$-deformation of the Fuss--Catalan number ${\rm FC}_{n}^{p}(0)$.

Recall that the number ${1 \over n} {n \choose j} {p n \choose j-1}$ counts
paths from $(0,0)$ to $(np,0)$ in the f\/irst quadrant, consisting of steps
$U=(1,1)$ and $D=(1,-p)$ and have~$j$ peaks (i.e., $UD$'s), cf.~\cite[$A108767$]{SL}.

For example, take $n=3$, $k=2$, $p=3$, $r=1$, $b=0$. Then
\begin{gather*}
\varpi_{3,2,3} = [1,2,12,9,6,3,4,5,7,8,10,11] \in \mathbb{S}_{12},
\\
G_{3,2,3}(\beta)= (1,18,171,747,1767,1995,1001).
\end{gather*}
 Therefore,
 \begin{gather*}
 G_{3,2,3}(1)=5700=D(3,2,3,0) \qquad \text{and} \qquad
{\rm Coef\/f}_{[\beta^{6}]}(G_{3,2,3}(\beta))= 1001= D(3,2,2,0).
\end{gather*}

\begin{Proposition}[\cite{NT}]\label{proposition5.5}
 The value of the Fuss--Catalan polynomial at $t=2$, that is the
number
\begin{gather*}
\sum_{j=1}^{n} {1 \over n} {n \choose j} {p n \choose j-1} 2^{j-1}
\end{gather*}
is equal to the number of hyperplactic classes of $p$-parking functions of
length~$n$, see~{\rm \cite{NT}} for definition of $p$-parking functions, its
properties and connections with some combinatorial Hopf algebras.
\end{Proposition}

Therefore, the value of the {\it Grothendieck polynomial}
$\mathfrak{G}_{\varpi_{n,1,p}}^{(\beta =1)}(1)$ at $\beta=1$ and $x_i =1$,
$\forall\, i$, is equal to the number of $p$-parking functions of length $n+1$.
It is an open problem to f\/ind combinatorial interpretations of the
polynomials $\mathfrak{G}_{\varpi_{n,k,p}}^{(\beta )}(1)$ in the case
$k \ge 2$. Note f\/inally, that in the case $p=2,~k=1$ the values of the Fuss--Catalan polynomials at $t=2$ one can f\/ind in \cite[$A034015$]{SL}.

\begin{Comments}\label{comments5.7} {\it $(\Longrightarrow)$ The case $r=0$.}
It follows from Theorem~\ref{theorem5.7} that in the case $r=0$ and
$k \ge n$, one has
\begin{gather*}
D(n,k,0,p,b)= \dim V_{\lambda_{n,p,b}}^{\mathfrak{gl}(k+1)}=
(1+p)^{{n+1 \choose 2}} \prod_{j=1}^{n+1} {{(n-j+1)p+b+k-j+1 \choose k-j+1}
\over {(n-j+1)(p+1)+b \choose n-j+1}} .
\end{gather*}
Now consider the conjugate $\nu:=\nu_{n,p,b}:=((n+1)^b,n^{p},(n-1)^{p},\ldots,
1^{p})$ of the partition $\lambda_{n,p,b}$, and a rectangular shape partition $\psi=(\underbrace{k,\ldots,k}_{np+b})$. If $k \ge np+b$, then there exists a
unique {\it grassmannian} permutation
$\sigma:=\sigma_{n,k,p,b}$ of the {\it shape} $\nu$ and the
{\it flag}~$\psi$~\cite{M}. It is easy to see from the above formula
for $D(n,k,0,p,b)$, that
\begin{gather*}
\mathfrak{S}_{\sigma_{n,k,p,b}}(1) = \dim V_{\nu_{n,p,b}}^{\mathfrak{gl}(k-1)}\\
\hphantom{\mathfrak{S}_{\sigma_{n,k,p,b}}(1)}{}=
 (1+p)^{{n \choose 2}} {k+n-1 \choose b} \prod_{j=1}^{n}{(p+1)(n-j+1) \over (n-j+1)(p+1)+b} \prod_{j=1}^{n}{ {k+j-2 \choose (n-j+1)p+b} \over {(n-j+1)(p+1)+b-1 \choose n-j}}.
\end{gather*}
After the substitution $k:=np+b+1$ in the above formula we will have
\begin{gather*}
 \mathfrak{S}_{\sigma_{n,np+b+1,p,b}}(1)= (1+p)^{{n \choose 2}} \prod_{j=1}^{n} {{{np+b+j-1 \choose (n-j+1)p }} \over {{j(p+1)-1 \choose j-1}}}.
\end{gather*}

In the case $b=0$ some simplif\/ications are happened, namely,
\begin{gather*}
\mathfrak{S}_{\sigma_{n,k,p,0}}(1) = (1+p)^{{n \choose 2}} \prod_{j=1}^{n}
{{k+j-2 \choose (n-j+1)p} \over {(n-j+1)p+n-j \choose n-j}}.
\end{gather*}
Finally we observe that if $k=np+1$, then
\begin{gather*}
\prod_{j=1}^{n} {{{ np+j-1 \choose (n-j+1)p}} \over {{(n-j+1)p+n-j \choose n-j}}}=
{\prod_{j=2}^{n} {{np+j-1 \choose (p+1)(j-1)} \over {j(p+1)-1 \choose j-1}} =
\prod_{j=1}^{n-1} {j ! (n(p+1) -j -1) ! \over ((n-j)(p+1)) ! ((n-j)(p+1)-1) !}}:=A_{n}^{(p)},
\end{gather*}
where the numbers $A_{n}^{(p)}$ are {\it integers} that generalize the
numbers of {\it alternating sign matrices} (ASM) of size $n \times n$, recovered in the case
$p=2$, see \cite{DZ, O} for details.
\end{Comments}

\begin{Examples}\label{examples5.1} \quad

$(1)$ Let us consider polynomials $\mathfrak{G}_{n}(\beta) :=
\mathfrak{G}_{{\sigma_{n,2n,2,0}}}^{(\beta-1)}(1)$.

If $n=2$, then
\begin{gather*}
\sigma_{2,4,2,0}=235614 \in \mathbb{S}_6, \qquad
 \mathfrak{G}_2(\beta)= (1,2,{\bf 3}):=1+ 2 \beta+ {\bf 3} \beta^2.
\end{gather*}
Moreover,
\begin{gather*}
\mathfrak{R}_{\sigma_{2,4,2,0}}(q;\beta)= (1,{\bf 2})_{\beta}+ {\bf 3} q \beta^2.
\end{gather*}

If $n=3$, then
\begin{gather*}
\sigma_{3,6,2,0}=235689147 \in \mathbb{S}_9, \qquad
\mathfrak{G}_{3}(\beta)=(1,6,21,36,51,48,{\bf 26}).
\end{gather*}
Moreover,
\begin{gather*}
\mathfrak{R}_{\sigma_{3,6,2,0}}(q;\beta) = (1,6,11,16,{\bf 11})_{\beta} +
q \beta^2 (10,20,35,34)_{\beta}+q^2 \beta^4 (5,14,{\bf 26})_{\beta},\\
\mathfrak{R}_{\sigma_{3,6,2,0}}(q;1) = (45,99,45)_{q}.
\end{gather*}

If $n=4$, then
\begin{gather*}
\sigma_{4,8,2,0}= [2,3,5,6,8,9,11,12,1,4,7,10] \in
\mathbb{S}_{12},\\
\mathfrak{G}_{4}(\beta)= (1,12,78,308,903,2016,3528,4944,5886,5696,4320,2280,{\bf 646}).
\end{gather*}
Moreover,
\begin{gather*}
\mathfrak{R}_{\sigma_{4,8,2,0}}(q;\beta)= (1,12,57,182,392,602,763,730,493,{\bf 170})_{\beta}\\
\hphantom{\mathfrak{R}_{\sigma_{4,8,2,0}}(q;\beta)=}{} +
q \beta^2 (21,126,476,1190,1925,2626,2713,2026,804)_{\beta}\\
\hphantom{\mathfrak{R}_{\sigma_{4,8,2,0}}(q;\beta)=}{}+
q^2 \beta^4 (35,224,833,1534,2446,2974,2607,1254)_{\beta}\\
\hphantom{\mathfrak{R}_{\sigma_{4,8,2,0}}(q;\beta)=}{}
+q^3 \beta^6 (7,54,234,526,909,1026,{\bf 646})_{\beta},\\
\mathfrak{R}_{\sigma_{4,8,2,0}}(q;1) = (3402,11907,11907,3402)_{q}= 1701~(2,7,7,2)_{q}.
\end{gather*}

$\bullet$ If $n=5$, then
\begin{gather*}
\sigma_{5,10,2}=[2,3,5,6,8,9,11,12,14,15,1,4,7,10,13] \in \mathbb{S}_{15}, \\
\mathfrak{G}_{5}(\beta)= (1,20,210,1420,7085,27636,87430,230240,516375,997790,1676587,2466840,\\
\hphantom{\mathfrak{G}_{5}(\beta)= (}{}
3204065,3695650,3778095,3371612,2569795,1610910,782175,262200,{\bf 45885}).
\end{gather*}
Moreover,
\begin{gather*}
\mathfrak{R}_{\sigma_{5,10,2,0}}(q;\beta)= (1, 20, 174, 988, 4025,
12516, 31402, 64760, 111510, 162170,\\
\hphantom{\mathfrak{R}_{\sigma_{5,10,2,0}}(q;\beta)= (}{}
 202957,220200, 202493, 153106, 89355,35972, {\bf 7429})_{\beta}\\
\hphantom{\mathfrak{R}_{\sigma_{5,10,2,0}}(q;\beta)=}{}
 +
 q \beta^2 (36, 432, 2934, 13608, 45990, 123516, 269703, 487908, 738927,
\\
\hphantom{\mathfrak{R}_{\sigma_{5,10,2,0}}(q;\beta)= (}{}
956430, 1076265,
1028808, 813177, 499374, 213597, 47538)_{\beta}\\
\hphantom{\mathfrak{R}_{\sigma_{5,10,2,0}}(q;\beta)=}{}
+
q^2 \beta^{4} (126, 1512, 9954, 40860, 127359, 314172, 627831,1029726,
 1421253,\\
\hphantom{\mathfrak{R}_{\sigma_{5,10,2,0}}(q;\beta)= (}{}
 1711728,
1753893, 1492974, 991809, 461322, 112860)_{\beta}\\
\hphantom{\mathfrak{R}_{\sigma_{5,10,2,0}}(q;\beta)=}{}
+
q^3 \beta^6 (84, 1104, 7794, 33408, 105840, 255492, 486324, 753984, 1019538,\\
\hphantom{\mathfrak{R}_{\sigma_{5,10,2,0}}(q;\beta)= (}{}
 1169520, 1112340,
825930, 428895, 117990)_{\beta}\\
\hphantom{\mathfrak{R}_{\sigma_{5,10,2,0}}(q;\beta)=}{}
 +
q^4 \beta^8 (9, 132, 1032, 4992, 17730, 48024, 102132, 173772, 244620, 276120,\\
\hphantom{\mathfrak{R}_{\sigma_{5,10,2,0}}(q;\beta)= (}{}
 240420, 144210,
{\bf 45885})_{\beta},\\
\mathfrak{R}_{\sigma_{5,10,2,0}}(q;1) = (1299078,6318243,10097379,6318243,1299078)_{q}\\
\hphantom{\mathfrak{R}_{\sigma_{5,10,2,0}}(q;1)}{} =
59049 (22,107,171,107,22)_{q}.
\end{gather*}
We are reminded that over the paper we have used the notation
\begin{gather*}
(a_{0},a_{1},\ldots,a_{r})_{\beta} := \sum\limits_{j=0}^{r} a_{j} \beta^{j},
\end{gather*}
 etc.

One can show that $\deg_{[\beta]} \mathfrak{G}_n(\beta)=n(n-1)$,
$\deg_{[q]} \mathfrak{R}_{\sigma_{n,2n,2,0}}(q,1) =n-1$, and looking on the
numbers $3$, $26$, $646$, $45885$ we made
\begin{Conjecture}\label{conjecture5.2}
 Let $a(n):= {\rm Coef\/f}[\beta^{n(n-1)}]~
(\mathfrak{G}_n(\beta) )$. Then
\begin{gather*}
a(n)={\rm VSASM}(n)={\rm OSASM}(n)= \prod_{j=1}^{n-1} { (3j+2)(6j+3)! (2j+1)! \over (4j+2)! (4j+3)!},
\end{gather*}
where
${\rm VSASM}(n)$ is the number of alternating sign $(2n+1) \times (2n+1)$ matrices
symmetric about the vertical axis,
${\rm OSASM}(n)$ is the number of $2n \times 2n$ off-diagonal symmetric alternating
sign matrices.
See {\rm \cite[$A005156$]{SL}}, {\rm \cite{O}} and references therein, for details.
\end{Conjecture}

\begin{Conjecture}\label{conjecture5.3}
Polynomial $\mathfrak{R}_{\sigma_{n,2n,2,0}}(q;1)$ is symmetric
and
\begin{gather*}
\mathfrak{R}_{\sigma_{n,2n,2,0}}(0;1) = A20342(2n-1),
\end{gather*}
 see~{\rm \cite{SL}}.
\end{Conjecture}

$(2)$ Let us consider polynomials $\mathfrak{F}_{n}(\beta) :=\mathfrak{G}_{{\sigma_{n,2n+1,2,0}}}^{(\beta-1)}(1)$.

If $n=1$, then
\begin{gather*}
\sigma_{1,3,2,0}=1342 \in \mathbb{S}_4, \qquad \mathfrak{F}_2(\beta)= (1,{\bf 2}):=1+{\bf 2} \beta.
\end{gather*}

If $n=2$, then
\begin{gather*}
\sigma_{2,5,2,0}=1346725 \in \mathbb{S}_7, \qquad \mathfrak{F}_3(\beta)= (1,6,11,16,{\bf 11}).
\end{gather*}
Moreover,
\begin{gather*}
\mathfrak{R}_{\sigma_{2,5,2,0}}(q;\beta)= (1,2, {\bf 3})_{\beta} +q \beta (4,8,12)_{\beta}+q^2 \beta^3 (4,{\bf 11})_{\beta}.
\end{gather*}

If $n=3$, then
\begin{gather*}
\sigma_{3,7,2,0}=[1,3,4,6,7,9,10,2,5,8] \in \mathbb{S}_{10}, \\
 \mathfrak{F}_4(\beta)=
(1,12,57,182,392,602,763,730,493,{\bf 170}).
\end{gather*}
Moreover,
\begin{gather*}
\mathfrak{R}_{\sigma_{3,7,2,0}}(q;\beta)=
(1,6,21,36,51,48,{\bf 26})_{\beta} + q
\beta (6,36,126,216,306,288,156)_{\beta}\\
\hphantom{\mathfrak{R}_{\sigma_{3,7,2,0}}(q;\beta)=}{}+ q^2 \beta^3 (20,125,242,403,460,289)_{\beta} + q^3 \beta^5 (6,46,114,204,
{\bf 170})_{\beta},\\
\mathfrak{R}_{\sigma_{3,7,2,0}}(q;1) = (189,1134,1539,540)_{q}=27 (7,42,57,20)_{q}.
\end{gather*}

If $n=4$, then
\begin{gather*}
\sigma_{4,9,2,0} = [1,3,4,6,7,9,10,12,13,2,5,8,11] \in \mathbb{S}_{13},\\
\mathfrak{F}_5(\beta)= (1,20,174,988,4025,12516,31402,64760,111510,162170,202957,\\
\hphantom{\mathfrak{F}_5(\beta)= (}{} 220200,202493,153106,89355,35972,{\bf 7429}).
\end{gather*}
Moreover,
\begin{gather*}
\mathfrak{R}_{\sigma_{4,9,2,0}}(q;\beta)=
(1,12,78,308,903,2016,3528,4944,5886,5696,4320,2280,{\bf 646})_{\beta} \\
\hphantom{\mathfrak{R}_{\sigma_{4,9,2,0}}(q;\beta)=}{}
+ q \beta (8,96,624,2464,7224,16128,28224,39552,47088,45568,\\
\hphantom{\mathfrak{R}_{\sigma_{4,9,2,0}}(q;\beta)=(}{}
34560,18240,5168)_{\beta}\\
\hphantom{\mathfrak{R}_{\sigma_{4,9,2,0}}(q;\beta)=}{}
 +q^2 \beta^3(56,658,3220,11018,27848,53135,78902,100109,103436,\\
\hphantom{\mathfrak{R}_{\sigma_{4,9,2,0}}(q;\beta)=(}{}
 84201,47830,14467)_{\beta}\\
\hphantom{\mathfrak{R}_{\sigma_{4,9,2,0}}(q;\beta)=}{}
 +q^3 \beta^5(56,728,3736,12820,29788,50236,72652,85444,78868,\\
\hphantom{\mathfrak{R}_{\sigma_{4,9,2,0}}(q;\beta)=(}{}
 50876,17204)_{\beta}\\
\hphantom{\mathfrak{R}_{\sigma_{4,9,2,0}}(q;\beta)=}{}
 +q^4 \beta^7(8,117,696,2724,7272,13962,21240,24012,18768,{\bf 7429})_{\beta},\\
\mathfrak{R}_{\sigma_{4,9,2,0}}(q;1) = (30618,244944,524880,402408,96228)_{q} = 4374 (7,56,120,92,22)_{q} .
\end{gather*}

One can show that $\mathfrak{F}_n(\beta)$ is a polynomial in $\beta$ of degree~$n^2$, and looking on the num\-bers~$2$, $11$, $170$, $7429$ we made
\begin{Conjecture}\label{conjecture5.4}
 Let $b(n):= {\rm Coef\/f}_{[\beta^{(n-1)^2}]}
(\mathfrak{F}_n(\beta))$. Then
$b(n) = {\rm CSTCPP}(n)$. In other words, $b(n)$ is equal to the number of
cyclically symmetric transpose complement plane partitions in an
$2n \times 2n \times 2n$~box. This number is known to be
\begin{gather*}
\prod_{j}^{n-1} {{(3j+1)(6j)! (2j)! \over (4j+1)! (4j)! }},
\end{gather*}
 see {\rm \cite[$A051255$]{SL}}, {\rm \cite[p.~199]{Br}}.
\end{Conjecture}

It ease to see that polynomial $\mathfrak{R}_{\sigma_{n,2n+1,2,0}}(q;1)$ has degree~$n$.

\begin{Conjecture}\label{conjecture5.5}
\begin{gather*}
 {\rm Coef\/f}_{[\beta^n]} \big(\mathfrak{R}_{\sigma_{n,2n+1,2,0}}(q;1) \big) =
A20342(2n),
\end{gather*}
 see {\rm \cite{SL}};
\begin{gather*}
 \mathfrak{R}_{\sigma_{n,2n+1,2,0}}(0;1)= A_{\rm QT}^{(1)}(4n;3) = 3^{n(n-1)/2} {\rm ASM}(n),
 \end{gather*}
see {\rm \cite[Theorem~5]{Ku}} or {\rm \cite[$A059491$]{SL}}.
\end{Conjecture}

\begin{Proposition}\label{proposition5.6}
 One has
\begin{gather*}
 \mathfrak{R}_{\sigma_{4,2n+1,2,0}}(0;\beta) = \mathfrak{G}_n( \beta) = \mathfrak{G}_{\sigma_{n,2n,2,0}}^{(\beta-1)}(1),\qquad
\mathfrak{R}_{\sigma_{n,2n,2,0}}(0,\beta)= \mathfrak{F}_n(\beta)= \mathfrak{G}_{\sigma_{n,2n+1,2,0}}^{(\beta-1)}(1).
\end{gather*}
\end{Proposition}
\end{Examples}

Finally we def\/ine $(\beta,q)$-deformations of the numbers ${\rm VSASM}(n)$ and
${\rm CSCTPP}(n)$. To accomplish these ends, let us consider permutations
\begin{gather*}
w_k^{-} = (2,4,\ldots, 2k,2k-1,2k-3,\ldots,3,1),\\
 w_k^{+} = (2,4,\ldots,2k,2k+1,2k-1,\ldots,3,1).
\end{gather*}

\begin{Proposition}\label{proposition5.7} One has
\begin{gather*}
 \mathfrak{S}_{{w_{k}^{-}}}(1)={\rm VSAM}(k), \qquad \mathfrak{S}_{{w_{k}^{+}}}(1) = {\rm CSTCPP}(k) .
\end{gather*}
\end{Proposition}

Therefore the polynomials $\mathfrak{G}_{{w_{k}^{-}}}^{(\beta-1)}(x_1=q,\,x_j=1,\, \forall\, j \ge 2)$ and $\mathfrak{G}_{{w_{k}^{+}}}^{(\beta-1)}(x_1=q,\,x_j=1$, $\forall \,j \ge 2)$ def\/ine $(\beta,q)$-deformations of the numbers ${\rm VSAM}(k)$ and
${\rm CSTCPP}(k)$ respectively. Note that the inverse permutations
\begin{gather*}
(w_{k}^{-})^{-1} = (\underbrace{2k,1},\ldots,\underbrace{2k+1-i,i},\ldots,
\underbrace{k+1,k}),\\
(w_{k}^{+})^{-1}= (\underbrace{2k+1,1},\ldots,\underbrace{2k+2-j,j},
\ldots, \underbrace{k+2,k},k+1)
\end{gather*}
 also def\/ine a $(\beta,q)$-deformation of the
 numbers considered above.

\begin{Problem}\label{problem5.1}
It is well-known, see, e.g., {\rm \cite[p.~43]{ER}}, that the set
${\cal{VSASM}}(n)$ of alternating sign $(2n+1) \times (2n+1)$ matrices
symmetric about the vertical axis has the same cardinality as the set
${\rm SYT}_{2}( \lambda(n), \le n)$ of semistandard Young tableaux of the shape
$\lambda(n):= (2n-1,2n-3,\ldots,3,1)$ filled by the numbers from the set
$\{1,2,\ldots,n \}$, and such that the entries are weakly increasing down the
anti-diagonals.

On the other hand, consider the set ${\cal{CS}}(w_{k}^{-})$
 of compatible sequences, see, e.g., {\rm \cite{BJS,FK1}}, corresponding to
the permutation $w_{k}^{-} \in {\mathbb{S}}_{2k}$.
\end{Problem}

\begin{Challenge}
Construct bijections between the sets
${\cal{CS}}(w_k^{-})$, ${\rm SYT}_{2}( \lambda(k), \le k)$ and\linebreak ${\cal{VSASM}}(k)$.
\end{Challenge}

\begin{Remark}\label{remarks5.1} One can compute the principal specialization of the Schubert polynomial corresponding to the transposition $t_{k,n}:=
(k,n-k) \in \mathbb{S}_n$ that interchanges $k$ and $n-k$, and f\/ixes all
other elements of~$[1,n]$.
\end{Remark}

\begin{Proposition}\label{proposition5.8}
\begin{gather*}
q^{(n-1)(k-1)} \mathfrak{S}_{t_{k,n-k}}\big(1,q^{-1},q^{-2},q^{-3},\ldots\big)\\
\qquad{}= \sum_{j=1}^{k} (-1)^{j-1} q^{{j \choose 2}} {n-1 \brack k-j }_{q}~{n -2+j
\brack k+j-1}_{q} =
\sum_{j=1}^{n-2} q^{j} \left( {j+k-2 \brack k-1}_{q}
\right)^{2}.
\end{gather*}
\end{Proposition}

\begin{Exercises} \label{exercises5.7}\quad

(1) Show that if $k \ge 1$, then
\begin{gather*}
{\rm Coef\/f}_{[q^k \beta^{2k}]} ( {\mathfrak{R}}_{{\sigma_{n,2n,2,0}}}(q;t)
 ) ={2n-1 \choose 2k},\\
{\rm Coef\/f}_{[q^k \beta^{2k-1}]} ( {\mathfrak{R}}_{{\sigma_{n,2n+1,2,0}}}(q;t) ) ={2n \choose 2k-1}.
\end{gather*}

(2) Let $n \ge 1$ be a positive integer, consider ``zig-zag'' permutation
\begin{gather*}
w = \begin{pmatrix} 1& 2& 3& 4&\ldots& 2k+1& 2k+2& \ldots& {2n-1}& 2n \\
2& 1& 4& 3&\ldots& {2k+2}& {2k+1}& \ldots& 2n& {2n-1} \end{pmatrix}
\in \mathbb{S}_{2n}.
\end{gather*}
Show that
\begin{gather*}
\mathfrak{R}_{w}(q,\beta)= \prod_{k=0}^{n-1} \left({1- \beta^{2k} \over 1-\beta} +q \beta^{2k} \right).
\end{gather*}

(3) Let $\sigma_{k,n,m}$ be grassmannian permutation with shape
$\lambda =(n^m)$ and f\/lag $\phi =(k+1)^m$, i.e.,
\begin{gather*}
\sigma_{k,n,m} = \begin{pmatrix} 1& 2& \ldots& k& {k+1}& \ldots& {k+n}& {k+n+1}& \ldots& {k+n+m} \\
1& 2& \ldots& k& {k+m+1}& \ldots& {k+m+n}& {k+1}& \ldots& {k+m} \end{pmatrix}.
\end{gather*}
Clearly $\sigma_{k+1,n,m}= 1 \times \sigma_{k,n,m}$.

Show that the coef\/f\/icient
 ${\rm Coef\/f}_{\beta^m} ( \mathfrak{R}_{\sigma_{k,n,m}}(1,\beta) )$
 is equal to the Narayana number $ N(k+n+m,k)$.

(4) Consider permutation $w:= w^{(n)}=(w_1,\ldots,w_{2n+1})$, where
$w_{2k-1}=2k+1$ for $k=1,\ldots, n$, $w_{2n+1}= 2 n$, $w_{2}=1$ and
$w_{2k}= 2 k -2$ for $k=2,\ldots,n$. For example, $w^{(3)}=(3152746)$. We
set $w^{(0)}=1$.
Show that
the polynomial $\s_{w}^{(\beta)}(x_i=1, \forall i)$ has degree
$n(n-1)$ and the coef\/f\/icient
${\rm Coef\/f}_{\beta^{n(n-1)}} (\s_{w}^{(\beta)}(x_i=1,\, \forall \, i) )$
is equal to the $n$-th Catalan number $C_n$.

Note that the specialization
$ \s_{w}^{(\beta)}(x_i=1)|_{\beta=1}$ is equal to the $2n$-th Euler
(or up/down) number, see~\cite[$A000111$]{SL}.

More generally, consider permutation $w^{(n)}_{k}:= 1^k \times w^{(n)} \in
{\mathbb{S}}_{k+2n+1}$, and polynomials
\begin{gather*}
 P_k(z)=\sum_{j \ge 0} (-1)^{j} \s_{w^{(j)}_{k-2j}}(x_i=1) z^{k-2j},
 \qquad k \ge 0.
\end{gather*}
Show that
\begin{gather*}
\sum_{k \ge 0} P_k(z) {t^k \over k!} = \exp (t z)\operatorname{sech}(t).
\end{gather*}
The polynomials $P_{k}(z)$ are well-known as {\it Swiss--Knife} polynomials,
see \cite[$A153641$]{SL}, where one can f\/ind an overview of some properties of
the Swiss--Knife polynomials.

(5) Assume that $n=2k+3$, $k \ge 1$, and consider permutation
$v_n=(v_1,\ldots,v_n) \in \mathbb{S}_n$, where $v_{2a+1}=2a+3$, $a=0,\ldots,n-1$, $w_2=1$ and $w_{2a}=2 a-2$, $a=2,\ldots,k+1$. For example, $v_4=[31527496,11,8,10]$ and $\s_{v_{4}}(1)=50521 =E_{10}$.

Show that
\begin{gather*}
\s_{v_{n}}(q,\,x_i=1,\,\forall\, i \ge 2) = (n-2) E_{n-3} q^2+ \cdots+ 2^{k-1} (k-1) ! q^{k+2}, \\
\s_{v_{n}}(x_i=1,\,\forall\, i \ge 1) = E_{n-1}.
\end{gather*}

Set $\beta=d-1$, consider polynomials ${\cal{E}}_n(q,d)=
{\mathfrak{G}}_{v_{n}}^{(\beta)}(x_1=q, \,x_i =1,\, \forall\, i \ge 2)$.~Clearly,
see the latter formula, ${\cal{E}}_{n}(1,1)= E_{n-1}$. Give a~combinatorial prove that ${\cal{E}}_n(q,d) \in \N [q,d]$, that is to give combinatorial interpretation(s) of coef\/f\/icients of the polynomial
${\cal{E}}_{n}(q,d)$.

Show that $\deg_{d} {\cal{E}}_{n}(1,d) =n(n+1)$ and the
leading coef\/f\/icient is equal to the Catalan number~$C_{n+1}$.

(6) Consider permutation $u:=u_{n}=(u_1,\ldots,u_{2n}) \in
{\mathbb S}_{2n}$, $n \ge 2$, where
$u_1=2$, $u_{2k+1}=2k-1$, $k=1,\ldots,n$,
$u_{2k}=2k+2$, $k=1,\ldots,n-1$, $u_{2n}=2n-1$. For example, $u_{4}=(24163857)$.

Now consider polynomial
\begin{gather*}
R_n^{(k)}(q) =\s_{1^{k} \times u_{n}}(x_1=q,\,x_i=1, \,\forall \,i \ge 2).
\end{gather*}

Show that
$R_{n}^{(k)}(1) ={2n+k-1 \choose k} E_{2n-1}$, where $E_{2k-1}$,
$k \ge 1$, denotes the Euler number, see \cite[$A00111$]{SL}. In particular,
$R_{n}^{(1)}(1) =2^{2n-1} G_n$, where~$G_n$ denotes the {\it unsigned Genocchi} number, see \cite[$A110501$]{SL}.

Show that $\deg_{q} R_{n}^{(k)}(q) =n$ and ${\rm Coef\/f}_{q^{n}} \big(
R_{n}^{(0)}(q) \big) =(2n-3) ! !$.

(7) Consider permutation $w_n \in \mathbb{S}_{2n+2}$, where $w_2=1$, $w_4=2$, and
\begin{gather*}
w_{2k-1}=2k+2, \quad 1 \le k \le n, \qquad w_{2k}=2k-3, \quad 3 \le k \le n,\\
 w_{2n+1}=2n-3,\qquad w_{2n+2}=2n-1.
\end{gather*}
For example, $w_{5}= [4,1,6,2,8,3,10,5,12,7,9,11]$.

Show that
\begin{gather*}
 \s_{w_{n}}(x_i=1, \, \forall\, i) =(2n+1) !! \big(2^{2n} -2\big) |B_{2n} |,
\end{gather*}
where $B_{2n}$ denotes the {\it Bernoulli} numbers\footnote{See, e.g., \url{https://en.wikipedia.org/wiki/Bernoulli_number}.}.

(8) Consider permutation $w_k:= (2k+1,2k-1,\ldots,3,1,2k,2k-2,\ldots,4,2)
\in {\mathbb{S}}_{2k+1}$. Show that
\begin{gather*}
\s_{w_{k}}^{(\beta-1)}(x_1=q,\,x_j=1,\, \forall \, j \ge 2) = q^{2k} (1+\beta)^{{n \choose 2}}.
\end{gather*}

(9) Consider permutations $\sigma_k^{+}=(1,3,5,\ldots,2k+1,2k+2,2k,\ldots,4,2)$ and $\sigma_{k}^{-}=(1,3,5,\ldots$, $2k+1,2k,2k-2,\ldots,4,2)$, and def\/ine polynomials
\begin{gather*}
S_k^{\pm}(q)= \s_{\sigma_{k}^{\pm}} (x_1=q,\,x_j=1, \,\forall\, j \ge 2).
\end{gather*}
Show that
\begin{gather*}
S_k^{+}(0)={\rm VSASM}(k), \qquad S_k^{+}(1)={\rm VSASM}(k+1), \\
{\partial \over \partial q} S_k^{+}(q) \vert_{q=0}= 2k S_{k}^{+}(0),\qquad
{\rm Coef\/f}_{q^{k}}(S_k^{+}(q) )={\rm CSTCPP}(k+1),\\
S_k^{-}(0)={\rm CSTCPP}(k), \qquad S_k^{-}(1)={\rm CSTCPP}(k+1), \\
{\partial \over \partial q} S_k^{-}(q) \vert_{q=0}= (2k-1) S_{k}^{-}(0),\qquad
{\rm Coef\/f}_{q^{k}} (S_k^{-}(q) )={\rm VSASM}(k).
\end{gather*}

Let's observe that $\sigma_{k}^{\pm}= 1 \times \tau_{k-1}^{\pm}$, where
$\tau_k^{+}= (2,4,\ldots,2k,2k+1,2k-1,\ldots,3,1)$ and
$\tau_{k}^{-} = (2,4,\ldots,2k,2k-1,2k-3,\ldots,3,1)$. Therefore,
\begin{gather*}
 \s_{\tau_{k}^{\pm}}(x_1=q,\, x_j=1,\,
 \forall j \ge 2) = q S_{k-1}^{\pm}(q).
 \end{gather*}
Recall that ${\rm CSTCPP}(n)$ denotes the number of cyclically symmetric transpose
compliment plane partitions in a $2 n \times 2 n$ box, see, e.g., \cite[$A051255$]{SL},
 and ${\rm VSASM}(n)$ denotes the number of alternating sign $(2n+1)
\times (2n+1)$ matrices symmetric the vertical axis, see, e.g., \cite[$A005156$]{SL}.

It might be well to point out that
\begin{gather*}
\s_{\sigma_{n-1}^{+}}(x_1=x,\,x_i=1,\, \forall\, i \ge 2)= G_{2n-1,n-1}(x,y=1),\\
\s_{\sigma_{n}^{-}}(x_1=x, \,x_i =1,\, \forall\, i \ge 2)= F_{2n,n-1}(x,y=1),
\end{gather*}
where (homogeneous) polynomials $G_{m,n}(x,y)$~and $F_{m,n}(x,y)$ are
def\/ined in~\cite{Py}, and related with integral solutions to Pascal's
hexagon relations
\begin{gather*}
f_{m-1,n} f_{m+1,n} +f_{m,n-1} f_{m,n+1}= f_{m-1,n-1} f_{m+1,n+1}, \qquad (m,n) \in \Z^2.
\end{gather*}

(10) Consider permutation
\begin{gather*}
u_n = \begin{pmatrix} 1& 2& \ldots& n& n+1& n+2& n+3 &\ldots & 2n \\
2& 4& \ldots& 2n& 1& 3& 5&\ldots& 2n-1 \end{pmatrix},
\end{gather*}
and set $u_n^{(k)}:= 1^{2k+1} \times u_{n}$.
Show that
\begin{gather*}
 {\mathfrak G}_{u_{n}^{(k)}}^{(\beta-1)}(x_i=1,\,\forall\, i \ge 1) =
(1+\beta)^{{n+1 \choose 2}} {\mathfrak G}_{1^{k} \times w_{0}^{(n+1)}}^{((\beta)^2-1)}(x_i=1,\,\forall\, i \ge 1),
\end{gather*}
where $w_{0}^{(n+)}$ denotes the permutation $(n+1,n,n-1,\ldots,2,1)$.

(11) Let $n \ge 0$ be an integer.
Consider permutation $u_n= 1^n \times 321 \in {\mathbb{S}}_{3+ n}$. Show that
\begin{gather*} \s_{u_{n}}(x_1=t,\,x_i=1,\, \forall\, i \ge 2) = {\frac{1}{4}}
{2 n+2 \choose 3}
+ {\frac{n}{2}} {2 n+2 \choose 1} t + {\frac{1}{2}}{2 n+2 \choose 1} t^2.
\end{gather*}

Consider permutation $v_n := 1^n \times 4321 \in {\mathbb{S}}_{n+4}$.
Show that
\begin{gather*}
\s_{v_{n}}(x_1=t,\,x_i=1,\,\forall\, i \ge 2) \\
\qquad{}=
 {\frac{1}{24}}{2n+4 \choose 5}
{2 n+2 \choose 1} +{\frac{1}{2}} {2n+4 \choose 5} t + {\frac{n}{4}} {2 n+4 \choose 3} t^2+ {\frac{1}{4}} {2n +4 \choose 3} t^3.
\end{gather*}

(12) Show that
\begin{gather*}
\sum_{(a,b,c) \in (\Z_{\ge 0})^3 } q^{a+b+c}
{a+b \brack b}_{q}
{a+c \brack c }_{q} {b+c \brack b }_{q} = {1 \over (q;q)_{\infty}^3} \left( \sum_{k \ge 2} (-1)^k {k \choose 2} q^{{k \choose 2}-1} \right).
\end{gather*}
It is not dif\/f\/icult to see that the left hand side sum of the above identity
counts the weighted number of plane partitions $\pi=(\pi_{ij})$ such that
\begin{gather*}
 \pi_{i,j} \ge 0, \qquad \pi_{ij} \ge \max(\pi_{i+1,j},\pi_{i,j+1}), \qquad \pi_{ij} \le 1
 \quad \text{if}\quad i \ge 2 \quad \text{and} \quad j \ge 2,
\end{gather*}
and the weight $wt(\pi):=\sum\limits_{i,j} \pi_{ij}$.

$(13)$ Let $\lambda=(\lambda_1 \ge \lambda_2 \ge \cdots \ge \lambda_p > 0)$
be a partition of size $n$. For an integer~$k$ such that $1 \le k \le n-p$
def\/ine a grassmannian permutation
\begin{gather*}
w_{\lambda}^{(k)} = [1, \ldots, k, \lambda_p+k+1, \lambda_{p-1}+k+2,\ldots,
\lambda_1+k+p, a_1,\ldots,a_{n-p-k}],
\end{gather*}
where we denote by $(a_1 < a_2 < \cdots < a_{n-k-p})$ the complement
$[1,n] \backslash (1, \ldots, k, \lambda_p+k+1, \lambda_{p-1}+k+2,\ldots,
\lambda_1+k+p)]$.

Show that the Grothendieck polynomial
\begin{gather*}
G_{\lambda}(\beta):={\mathfrak G}_{{w_{\lambda}}^{k}}^{\beta-1}(1^n)
\end{gather*}
 is a polynomial of~$\beta$ with {\it nonnegative} coef\/f\/icients.
Clearly, $G_{\lambda}(1) = \dim V_{\lambda}^{\mathfrak{Gl}(k+\ell(\lambda))}$.

Find a combinatorial interpretations of polynomial $G_{\lambda}(\beta)$.
\end{Exercises}

Final remark, it follows from the seventh exercise listed above, that the
polynomials $\s_{\sigma_{k}^{\pm}}^{(\beta)}(x_1=q,\,x_j=1,\, \forall\, j \ge 2)$ def\/ine a $(q,\beta)$-deformation of the number ${\rm VSASM}(k)$ (the case $\sigma_{k}^{+}$)
and the number ${\rm CSTCPP}(k)$ (the case $\sigma_{k}^{-}$), respectively.

\subsubsection{Specialization of Grothendieck polynomials}\label{section5.2.5}

Let $p$, $b$, $n$ and $i$, $2i < n $ be positive integers. Denote by
${\cal{T}}_{p,b,n}^{(i)}$ the {\it trapezoid}, i.e., a convex quadrangle having
vertices at the points
\begin{gather*}
 (ip,i), \quad (ip,n-i), \quad (b+ip,i) \quad \text{and} \quad (b+(n-i)p,n-i).
 \end{gather*}

\begin{Definition}\label{definition5.5} Denote by ${\rm FC}_{b,p,n}^{(i)}$ the set of lattice path from the
point $(ip,i)$ to that $(b+(n-i)p,n-i)$ with east steps
$E = (0,1)$ and north steps $N = (1,0)$, which are located inside of the
trapezoid ${\cal{T}}_{p,b,n}^{(i)}$.

If $\mathfrak{p} \in {\rm FC}_{b,p,n}^{(i)}$ is a path, we denote by
$p(\mathfrak{p})$ the number of {\it peaks}, i.e.,
\begin{gather*}
p(\mathfrak{p})= NE(\mathfrak{p})+E_{\rm in}(\mathfrak{p})+N_{\rm end}(\mathfrak{p}),
\end{gather*}
where $NE(\mathfrak{p})$ is equal to the number of steps $NE$ resting on
path $\mathfrak{p}$; $E_{\rm in}(\mathfrak{p})$~is equal to~$1$, if the
path $\mathfrak{p}$
{\it starts} with step~$E$ and~$0$ otherwise;
$N_{\rm end}(\mathfrak{p})$ is equal to $1$, if the path $\mathfrak{p}$
{\it ends} by the step~$N$ and~$0$ otherwise.
\end{Definition}

Note that the equality $N_{\rm end}(\mathfrak{p})=1$ may happened only in the case
$b=0$.

\begin{Definition} \label{definition5.6}
Denote by ${\rm FC}_{b,p,n}^{(k)}$ the set of $k$-tuples
${\mathfrak{P}}=(\mathfrak{p}_1,\ldots,\mathfrak{p}_{k})$ of
{\it non-crossing}
lattice paths, where for each $i=1,\ldots,k$, $\mathfrak{p}_i \in
{\rm FC}_{b,p,n}^{(i)}$.
\end{Definition}

Let
\begin{gather*}
{\rm FC}^{(k)}_{b,p,n}(\beta):= \sum_{\mathfrak{P} \in {\rm FC}_{b,p,n}^{(k)}}
\beta^{p(\mathfrak{P})}
\end{gather*}
denotes the generating function of the statistics $p(\mathfrak{P}):=
\sum\limits_{i=1}^{k} p(\mathfrak{p_{i}}) -k$.

\begin{Theorem}\label{theorem5.11} The following equality holds
\begin{gather*}
\mathfrak{G}_{\sigma_{n,k,p,b}}^{(\beta)}(x_1=1,x_2=1,\ldots)={\rm FC}_{p,b,n+k}^{(k)}(\beta +1),
\end{gather*}
where $\sigma_{n,k,p,b}$ is a unique grassmannian permutation with shape
$((n+1)^b,n^p,(n-1)^p,\ldots,1^{p})$ and flag $(\underbrace{k,\ldots,k)}_{np+b}$.
\end{Theorem}

\subsection[The ``longest element'' and Chan--Robbins--Yuen polytope]{The ``longest element'' and Chan--Robbins--Yuen polytope\footnote{Some results of this section, e.g., Theorems~\ref{theorem5.11} and~\ref{theorem5.12},
has been proved independently and in greater generality in~\cite{Me3}.}}\label{section5.3}

\subsubsection[The Chan--Robbins--Yuen polytope ${\cal{CRY}}_n$]{The Chan--Robbins--Yuen polytope $\boldsymbol{{\cal{CRY}}_n}$}\label{section5.3.1}

 Assume additionally, cf.~\cite[Exercise~6.C8(d)]{ST3}, that the
condition~$(a)$ in Def\/inition~\ref{definition5.1} is replaced by that
\begin{enumerate}\itemsep=0pt
\item[$(a')$] $x_{ij}$ and $x_{kl}$ {\it commute} for all $i$, $j$, $k$
and~$l$.
\end{enumerate}

 Consider the element $w_{0}^{(n)}:= \prod\limits_{1 \le i < j \le n} x_{ij}$. Let
us bring the element $w_{0}^{(n)}$ to the reduced form, that is, let us
consecutively apply the def\/ining relations~$(a')$~and~$(b)$ to the element
$w_{0}^{(n)}$ in {\it any} order until unable to do so. Denote the
resulting polynomial by $Q_n(x_{ij};\alpha, \beta)$.
Note that the polynomial itself {\it depends} on the order in
which the relations~$(a')$ and~$(b)$ are applied.

We denote by $Q_n(\beta)$ the specialization
\begin{gather*}
x_{ij}=1 \qquad \text{for all $i$ and $j$},
\end{gather*}
of the polynomial $Q_n(x_{ij}; \alpha=0,\beta)$.

\begin{Example}\label{example5.7}
\begin{gather*}
Q_3(\beta)=(2,1)=1+(\beta +1),\qquad
Q_4(\beta)=(10,13,4)=1+5(\beta +1)+ 4 (\beta +1)^2, \\
Q_5(\beta)=(140,336,280,92,9) =1+16 (\beta+1)+58 (\beta +1)^2+56 (\beta +1)^3 + 9(\beta +1)^4, \\
Q_6(\beta)=1+42 (\beta+1)+448 (\beta+1)^2+1674 (\beta+1)^3+2364 (\beta+1)^4\\
\hphantom{Q_6(\beta)=}{} +
1182 (\beta+1)^5+169 (\beta+1)^6, \\
Q_7(\beta)=(1,99,2569,25587,114005,242415,248817,118587,22924,1156)_{\beta+1}, \\
Q_8(\beta)=(1,219,12444,279616,2990335,16804401,52421688,93221276,94803125, \\
\hphantom{Q_8(\beta)=(}{}
 53910939, 16163947, 2255749, 108900)_{\beta+1}.
\end{gather*}
\end{Example}

What one can say about the polynomial
$Q_n(\beta):= Q_n(x_{ij};\beta) \arrowvert_{x_{ij}=1, \,\forall \, i,j}$?

It is known, \cite[Exercise~6.C8(d)]{ST3}, that the constant term of the
polynomial $Q_n(\beta)$ is equal to the product of Catalan numbers
$\prod\limits_{j=1}^{n-1}C_{j}$. It is not dif\/f\/icult to see that if $n \ge 3$, then
 ${\rm Coef\/f}_{[\beta+1]}(Q_n(\beta))= 2^n-1-{n+1 \choose 2} $, see \cite[$A002662$]{SL}, for a number of combinatorial interpretations of the numbers $2^n-1-{n+1 \choose 2} $.

\begin{Theorem}\label{theorem5.12}
One has
\begin{gather*}
 Q_n(\beta -1) = \left(\sum_{m \ge 0} \iota( {\cal{CRY}}_{n+1}, m)
\beta ^{m}\right)
 (1-\beta )^{{n+1 \choose 2}+1},
 \end{gather*}
where ${ \cal{CRY}}_m$ denotes the Chan--Robbins--Yuen polytope~{\rm \cite{CR,CRY}}, i.e., the convex polytope given by the following conditions:
\begin{gather*}
{\cal{CRY}}_m = \bigl\{ (a_{ij}) \in \operatorname{Mat}_{m \times m}(\Z_{\ge 0}) \bigr\}
\end{gather*}
 such that
\begin{enumerate}\itemsep=0pt
\item[$(1)$] $\sum_{i} a_{ij} =1$, $\sum_{j} a_{ij} =1$,

\item[$(2)$] $a_{ij}=0$ if $j > i+1$.
\end{enumerate}
Here for any integral convex polytope ${\cal P} \subset \Z^{d}$,
$\iota( {\cal P}, n)$ denotes the number of integer points in the set
$n {\cal P} \cap \Z^{d}$.
\end{Theorem}

In particular, the polynomial $Q_n(\beta)$ does not depend on the
order in which the relations $(a')$ and $(b)$ have been applied.

Now let us denote by ${\widehat{Q}}_n(q,t; \alpha,\beta)$ the specialization
\begin{gather*}
 x_{ij}=1, \quad i < j < n, \qquad \text{and}\qquad x_{i,n}=q \quad \text{if}\quad i=2,\ldots, n-1, \qquad x_{1,n}=t
 \end{gather*}
of the (reduced) polynomial $Q_n(x_{ij};\alpha,\beta)$ obtained by applying
the relations $(a')$ and $(b)$ in a~{\it certain} order. The
polynomial $Q_n(x_{ij};\alpha,\beta)$ itself {\it depends} on
the order selected. To def\/ine polynomials which are frequently appear in
Section~\ref{section5}, we apply the rules $(a)$ and $(b)$ stated in Def\/inition~\ref{definition5.1} to a
given monomial $x_{i_{1},j_{1}} \cdots x_{i_{p},j_{p}} \in
{\widehat{{\rm ACYB}}}_n(\alpha, \beta)$ consequently according to the order in
 which the monomial taken has been written. We set $Q_n(t,\alpha,\beta):=
{\widehat{Q}}_n(q=t,t;\alpha,\beta)$.

\begin{Conjecture}\label{conjecture5.6}
Let $n \ge 3$ and write
\begin{gather*}
 Q_n(t=1; \alpha,\beta)=\sum_{k \ge 0}(1+\beta)^k c_{k,n}( \alpha),
 \end{gather*}
then $c_{k,n}( \alpha) \in \Z_{\ge 0}[\alpha]$.

The polynomial $Q_n(t,\beta,\alpha=0)$ has degree $d_n:=[\frac{(n-1)^2}{4}]$ with respect to~$\beta$.
Write
\begin{gather*}
Q_n(t,\beta):= Q_n(t;\alpha=0,\beta)= t^{n-2} \sum_{k=0}^{d_{n}} c_{n}^{(k)}(t) \beta^{k}.
\end{gather*}
Then
$c_n^{(d_n)}(1) =a_{n}^2$ for some non-negative integer $a_n$.
Moreover, there exists a polynomial $a_n(t) \in \N[t] $ such that
\begin{gather*}
 c_n^{(d_n)}(t) =a_n(1) a_n(t), \qquad a_{n}(0)= a_{n-1}.
 \end{gather*}
The all roots of the polynomial $Q_n(\beta)$ belong to the set $\R_{< -1}$.
\end{Conjecture}

For example,
\begin{gather*}
Q_4(t=1;\alpha,\beta)=(1,5,4)_{\beta +1}+ \alpha (5,7)_{\beta +1} +3 \alpha^2, \\
Q_5(t=1; \alpha,\beta)= (1,16,58,56,9)_{\beta+1} + \alpha (16,109,146,29)_{\beta+1}\\
\hphantom{Q_5(t=1; \alpha,\beta)=}{} + \alpha^2 (51,125,34)_{\beta+1}+\alpha^3(35,17)_{\beta+1},\\
c_6^{(6)}=13(2,3,3,3,2), \qquad c_{7}^{(9)}(t)= 34 (3,5,6,6,6,5,3), \\
c_{8}^{(12)}(t)=330(13,27,37,43,45,45,43,37,27,13),\\
Q_4(t,\beta,\alpha=0) t^{-1} = t^2+(\beta+1)\big(3t+2t^2\big)+(\beta+1)^2(t+1)^2,\\
{\widehat{Q}}_4(q,t; \alpha=0,\beta)= \big(q t^2 + t^3 + 2qt^3 + q^2t^3 + q^3t^3 + t^4 + 2qt^4 + q^2t^4\big)\\
\hphantom{{\widehat{Q}}_4(q,t; \alpha=0,\beta)=}{} + \big(2qt^2 + 2t^3 + 3qt^3 + 2q^2t^3
+ 2t^4 + 2qt^4\big) \beta +\big(t^2+t^3\big)(q+t) \beta^2,\\
{\widehat{Q}}_5(q,t; \alpha=0,\beta)= \big(3 q^2 t + q^3 t + 5 q t^2 + 6 q^2 t^2 + 2 q^3 t^2 + 2 t^3 + 10 q t^3 +
 10 q^2 t^3 + 6 q^3 t^3\\
 \hphantom{{\widehat{Q}}_5(q,t; \alpha=0,\beta)=}{}
 + 3 q^4 t^3 + 3 q^5 t^3 + 2 q^6 t^3 + 3 t^4 +
 11 q t^4 + 11 q^2 t^4 + 8 q^3 t^4 + 5 q^4 t^4+ 3 q^5 t^4 \\
 \hphantom{{\widehat{Q}}_5(q,t; \alpha=0,\beta)=}{}
 + 3 t^5 +
 9 q t^5 + 9 q^2 t^5 + 6 q^3 t^5 + 3 q^4 t^5 + 2 t^6 + 6 q t^6 + 6 q^2 t^6 +
 2 q^3 t^6\big) \\
\hphantom{{\widehat{Q}}_5(q,t; \alpha=0,\beta)=}{}
+\big(9 q^2 t + 2 q^3 t + 17 q t^2 + 18 q^2 t^2 + 4 q^3 t^2 + 7 t^3 +
 31 q t^3 + 29 q^2 t^3 \\
\hphantom{{\widehat{Q}}_5(q,t; \alpha=0,\beta)=}{}
 + 15 q^3 t^3 + 10 q^4 t^3 + 7 q^5 t^3 + 10 t^4 +
 31 q t^4 + 29 q^2 t^4 + 18 q^3 t^4 \\
\hphantom{{\widehat{Q}}_5(q,t; \alpha=0,\beta)=}{}
 + 10 q^4 t^4 + 10 t^5 + 24 q t^5 +
 21 q^2 t^5 + 10 q^3 t^5 + 6 t^6 + 12 q t^6 + 6 q^2 t^6\big) \beta \\
\hphantom{{\widehat{Q}}_5(q,t; \alpha=0,\beta)=}{}
+ \big(9 q^2 t + q^3 t + 21 q t^2 + 18 q^2 t^2 + 2 q^3 t^2 + 9 t^3 + 34 q t^3 +
 28 q^2 t^3 \\
\hphantom{{\widehat{Q}}_5(q,t; \alpha=0,\beta)=}{}
 + 14 q^3 t^3 + 9 q^4 t^3 + 12 t^4 + 30 q t^4 + 24 q^2 t^4 +
 12 q^3 t^4 + 12 t^5 + 21 q t^5 \\
\hphantom{{\widehat{Q}}_5(q,t; \alpha=0,\beta)=}{}
 + 12 q^2 t^5 + 6 t^6 + 6 q t^6\big) \beta^2 +
 \big(3 q^2 t + 11 q t^2 + 6 q^2 t^2 + 5 t^3 + 15 q t^3 \\
\hphantom{{\widehat{Q}}_5(q,t; \alpha=0,\beta)=}{}
 + 10 q^2 t^3 + 5 q^3 t^3 +
 6 t^4 + 11 q t^4 + 6 q^2 t^4 + 6 t^5 + 6 q t^5 + 2 t^6\big) \beta^3\\
\hphantom{{\widehat{Q}}_5(q,t; \alpha=0,\beta)=}{}
 +
 \big(2 q t^2 + t^3 + 2 q t^3 + q^2 t^3 + t^4 + q t^4 + t^5\big) \beta^4.
\end{gather*}

Note that polynomials ${\widehat{Q}}_n(q,t; \alpha=0,\beta=0)$ give rise to
a two parameters deformation of the product of Catalan's numbers $C_1C_2 \cdots C_{n-1}$. Are there combinatorial interpretations of these polynomials and
 polynomials ${\widehat{Q}}_n(q,t; \alpha=0,\beta)$?

\begin{Comments} \label{comments5.8}
We {\it expect} that for each integer $n \ge 2$ the set
\begin{gather*}
\Psi_{n+1} := \left\{ w \in \mathbb{S}_{2n-1} \,|\, \mathfrak{S}_{w}(1)=
\prod_{j=1}^{n} {\rm Cat}_j \right\}
\end{gather*}
is {\it non empty}, whereas the set
\begin{gather*}
\left\{ w \in
\mathbb{S}_{2n-2} \,|\, \mathfrak{S}_{w}(1)= \prod_{j=1}^{n} {\rm Cat}_j \right\}
\end{gather*} is
{\it empty}. For example,
\begin{gather*}
 \Psi_4 =\{ [1,5,3,4,2] \}, \qquad \Psi_5 =
\{ [1,5,7,3,2,6,4],~[1,5,4,7,2,6,3] \},\\
\Psi_6 = \big\{ w:=[1,3,2,8,6,9,4,5,7], w^{-1}, \dots \big\} , \qquad \Psi_7 = \{??? \},
\end{gather*}
but one can check that for $w=[2358,10,549,12,11] \in {\mathbb{S}}_{12}$
\begin{gather*}
\s_{w}(1)= 776160 = \prod_{j=2}^6{\rm Cat}_j.
\end{gather*}

More generally, for any positive integer $N$ def\/ine
\begin{gather*}
\kappa(N) = \min \{ n \,|\, \exists \, w \in {\mathbb{S}}_n \ \text{such that} \
\s_{w}(1)=N \}.
\end{gather*}
It is clear that $\kappa(N) \le N+1$.

\begin{Problem}\label{problem-page120} Compute the following numbers
\begin{gather*}
\kappa( n !), \quad \kappa\left(\prod_{j=1}^{n} {\rm Cat}_j\right), \quad \kappa({\rm ASM}(n)), \quad \kappa\big((n+1)^{n-1}\big).
\end{gather*}
\end{Problem}

For example, $ 10 \le \kappa({\rm ASM}(6)=7436) \le 12$. Indeed, take
$w=[716983254,10,12,11] \in \mathbb{S}_{12}$. One can show that
\begin{gather*}
\s_{w}(x_1=t, \,x_i=1, \,\forall\, i \ge 2) = 13 t^6(t+10)(15 t +37),
\end{gather*}
so that $\s_{w}(1) =ASM(6)$; $\kappa(6^4) =9$,~namely, one can take $w=[157364298]$.

\begin{Question}\label{question-page121}
 Let $N$ be a positive integer. Does there exist a vexillary
$($grassmannian?$)$ permutation $w \in \mathbb{S}_{n}$ such that
$ n \le 2 \kappa(N)$ and $\mathfrak{S}_{w}(1) = N$?
\end{Question}

For example, $w=[1,4,5,6,8,3,5,7] \in \mathbb{S}_8$ is a grassmannian
permutation such that $\mathfrak{S}_{w}(1) =140$, and
$\mathfrak{R}_{w}(1,\beta)=(1,9,27,43,38,18,4)$.
\begin{Remark} \label{remark5.3}
We expect that for $n \ge 5$ there are {\it no} permutations
$w \in {\mathbb S}_{\infty}$ such that $Q_{n}(\beta) =\mathfrak{S}_{w}^{(\beta)}(1)$.
\end{Remark}

The numbers $\mathfrak{C}_n:= \prod\limits_{j=1}^{n} {\rm Cat}_j$ appear also
as the values of the {\it Kostant partition function} of the type $A_{n-1}$ on
some special vectors. Namely,
\begin{gather*}
 \mathfrak{C}_n = K_{\Phi(1^{n})}(\gamma_n), \qquad \text{where}\quad \gamma_n=
\left(1,2,3,\ldots,n-1, -{n \choose 2}\right),
\end{gather*}
see, e.g., \cite[Exercise~6.C10]{ST3}, and \cite[pp.~173--178]{K1}. More generally
\cite[Exercise~g, p.~177, (7.25)]{K1}, one has
\begin{gather*}
K_{\Phi(1^{n})}(\gamma_{n,d}) = pp^{\delta_{n}}(d) \mathfrak{C}_{n-1} =
\prod_{j=d}^{n+d-2} {1 \over 2j+1}~{n+d+j \choose 2j},
\end{gather*}
where $\gamma_{n,d}=(d+1,d+2,\ldots,d+n-1, - n(2d+n-1)/2)$,
$pp^{\delta_n}(d)$ denotes the set of reversed (weak) plane partitions
bounded
 by~$d$ and contained in the shape $\delta_n=(n-1,n-2,\ldots,1)$. Clearly,
$pp^{\delta_{n}}(1)= \prod\limits_{{1 \le i < j \le n}} {i+j+1 \over
 i+j-1} =C_n$, where $C_n$ is the $n$-th Catalan number\footnote{For example,~if $n=3$, there exist~$5$ reverse (weak) plane
partitions of shape $\delta_3 =(2,1)$ bounded by~$1$, namely reverse plane
partitions $\left\{
\begin{pmatrix} 0& 0\\ 0 \end{pmatrix},\,
\begin{pmatrix} 0& 0 \\ 1 \end{pmatrix},\,
\begin{pmatrix} 0& 1 \\ 0& \end{pmatrix},\,
\begin{pmatrix} 0& 1\\ 1 \end{pmatrix},\,
\begin{pmatrix} 1& 1 \\ 1 \end{pmatrix} \right\}$.}.

\begin{Conjecture} \label{conjecture5.7}
For any permutation $w \in \mathbb{S}_n$ there exists a graph $\Gamma_{w}=
(V,E)$, possibly with multiple edges, such that the reduced volume
$\widetilde{{\rm vol}} ({\cal{F}}_{\Gamma_{w}})$ of the flow polytope
${\cal{F}}_{\Gamma_{w}}$, see, e.g.,~{\rm \cite{ST2}} for a definition of the former,
is equal to $\mathfrak{S}_{w}(1)$.
\end{Conjecture}

For a family of vexillary permutations $w_{n,p}$ of the shape $\lambda =
p \delta_{n+1}$ and f\/lag $\phi=(1,2,\ldots$, $n-1,n)$ the corresponding graphs
$\Gamma_{n,p}$ have been constructed in \cite[Section~6]{Me2}. In this case
the reduced volume of the f\/low polytope ${\cal{F}}_{\Gamma_{n,p}}$ is equal
to the Fuss--Catalan number
\begin{gather*}
{1 \over 1+(n+1)p} {(n+1)(p+1) \choose n+1}=
\mathfrak{S}_{w_{n,p}}(1),
\end{gather*}
cf.\ Corollary~\ref{corollary5.2}.
\end{Comments}

\begin{Exercises}\label{exercises5.8}\quad
\begin{enumerate}\itemsep=0pt
\item[$(a)$] Show that
 the polynomial $R_n(t):= t^{1-n} Q_n(t;0,0)$ is symmetric
({\it unimodal?}), and $R_n(0)= \prod\limits_{k=1}^{n-2} {\rm Cat}_{k}$.
For example,
\begin{gather*}
R_4(t)=(1+t)\big(2+t+2t^2\big), \qquad R_5(t)= 2 (5,10,13,14,13,10,5)_{t},\\
R_6(t)= 10 (2,3,2)_{t} (7,7,10,13,10,13,10,7,7)_{t}, \qquad R_7(t)=30\big(196+\cdots+196 t^{15}\big).
\end{gather*}
Note that $R_n(1) =\prod\limits_{k=1}^{n-1} {\rm Cat}_k$.

\item[$(b)$] More generally, write as before,
\begin{gather*}
Q_n(t; 0,\beta)= t^{n-2}
\sum_{k \ge 0} c_{n}^{(k)}(t) \beta^{k}.
\end{gather*}
Show that the polynomials $c_n^{(k)}(t)$ are symmetric (unimodal?) for all~$k$ and~$n$.

\item[$(c)$] Consider a reduced polynomial ${\overline{R}}_n(\{x_{ij} \})$ of the
element
\begin{gather*}
\prod_{1 \le i < j \le n \atop (i,j) \not= (n-1,n)} x_{ij} \in
{\widehat{{\rm ACYB}}}(\alpha=\beta=0)^{ab},
\end{gather*}
see Def\/inition~\ref{definition5.1}. Here we assume additionally, that all elements
$\{x_{ij} \}$ are mutually commute. Def\/ine polynomial ${\widetilde{R}}_n(q,t)$ to be the following specialization
\begin{gather*}
x_{ij} \longrightarrow 1 \quad \text{if}\quad i < j <n-1, \qquad x_{i,n-1} \longrightarrow q, \qquad
x_{i,n} \longrightarrow t, \quad \forall\, i
\end{gather*}
of the polynomial ${\overline{R}}_n(\{x_{ij} \})$ in question.
Show that polynomials ${\widetilde{R}}_n(q,t)$ are well-def\/ined, and
\begin{gather*}
 {\widetilde{R}}_n(q,t) ={\widetilde{R}}_n(t,q).
 \end{gather*}
\end{enumerate}
\end{Exercises}

\begin{Examples}\label{examples5.2}
\begin{gather*}
R_4(t,\beta)= (2,3,3,2)_{t}+(4,5,4)_{t} \beta + (2,2)_{t} \beta^{2} ,\\
 R_5(t,\beta)= (10,20,26,28,26,20,10)_{t} +(33,61,74,74,61,33)_{t} \beta +(39,65,72,65,39)_{t} \beta^2\\
 \hphantom{R_5(t,\beta)=}{} +
 (19,27,27,19)_{t} \beta^3 +(3,3,3)_{t} \beta^4 , \\
 R_6(t,\beta)=
 (140,350,550,700,790,820,790,700,550,350,140)_{t}\\
 \hphantom{R_6(t,\beta)=}{} +
 (686,1640,2478,3044,3322,3322,3044,2478,1640,686)_{t} \beta \\
 \hphantom{R_6(t,\beta)=}{}
 +
 (1370, 3106,4480,5280,5537,5280,4480, 3106,1370)_{t} \beta^2\\
\hphantom{R_6(t,\beta)=}{}
 +
 (1420,3017,4113,4615,4615,4113,3017,1420)_{t} \beta^3\\
\hphantom{R_6(t,\beta)=}{}
 +
 (800,1565,1987,2105,1987,1565,800)_{t} \beta^4 +
 (230,403,465,465, 403,230)_{t} \beta^5 \\
\hphantom{R_6(t,\beta)=}{}
 +
 (26,39,39,39,26)_{t} \beta^6,\\
R_{6}(1,\beta)= (5880, 22340, 34009, 26330, 10809, 2196, 169)_{\beta},\\
R_7(t,\beta)=(5880,17640,32340,47040,59790,69630,76230,79530,79530,76230 ,69630, \\
\hphantom{R_7(t,\beta)=(}{}
59790,47040,32340,17640,5880)_{t} +
 (39980,116510,208196,295954,368410,\\
\hphantom{R_7(t,\beta)=(}{}
420850, 452226,462648,452226,420850,368410 , 295954,208196,116510,\\
\hphantom{R_7(t,\beta)=(}{}
39980)_{t} \beta
 + (118179,333345,578812,802004,975555,1090913,1147982, \\
\hphantom{R_7(t,\beta)=(}{}
1147982,1090913, 975555, 802004,578812,333345,118179)_{t} \beta^2 \\
\hphantom{R_7(t,\beta)=}{}
+ (198519,539551,906940,1221060,1447565,1580835,1624550,1580835,\\
\hphantom{R_7(t,\beta)=(}{}
1447565,1221060,
 906940,539551,198519)_{t} \beta^3 \\
\hphantom{R_7(t,\beta)=}{} +
 (207712,540840,875969,1141589,1314942,1398556,1398556,1314942,\\
\hphantom{R_7(t,\beta)=(}{}
 1141589,875969,
 540840,207712)_{t} \beta^4 \\
\hphantom{R_7(t,\beta)=}{}
+ (139320,344910,535107,671897,749338,773900,749338,671897,535107,\\
\hphantom{R_7(t,\beta)=(}{}
344910 ,
 139320)_{t} \beta^5 +(59235,137985,203527,244815,263389,263389,244815,\\
\hphantom{R_7(t,\beta)=(}{}
 203527,137985,59235)_{t} \beta^6 + (15119,32635,45333,51865,53691,51865, 45333,\\
 \hphantom{R_7(t,\beta)=(}{}
 32635,15119)_{t} \beta^7 +
 (2034,3966,5132,5532,5532,5132,3966,2034) \beta^8\\
 \hphantom{R_7(t,\beta)=}{}
 +
(102,170,204,204,204,170,102)_{t} \beta^9,\\
R_{7}(1,\beta)= (776160,4266900,10093580,13413490,10959216,5655044, 1817902,\\
\hphantom{R_{7}(1,\beta)= (}{}
343595,33328,1156)_{\beta}.
\end{gather*}
\end{Examples}

\subsubsection[The Chan--Robbins--M\'{e}sz\'{a}ros polytope ${\cal{P}}_{n,m}$]{The Chan--Robbins--M\'{e}sz\'{a}ros polytope $\boldsymbol{{\cal{P}}_{n,m}}$}\label{section5.3.2}

Let $m \ge 0$ and $n \ge 2$ be integers, consider the reduced polynomial $Q_{n,m}(t, \beta)$ corresponding to the element
\begin{gather*}
 M_{n.m}:=\left( \prod_{j=2}^{n} x_{1j} \right)^{m+1} \prod_{j=2}^{n-2}
 \prod_{k=j+2}^{n} x_{jk}.
 \end{gather*}
For example,
\begin{gather*}
 Q_{2,4}(t,\beta)= (4,7,9,10,10,9,7,4)_{t} + (10,17,21,22,21,17,10)_{t} \beta \\
 \hphantom{Q_{2,4}(t,\beta)=}{}
 +(8,13,15,15,13,8 )_{t} \beta^2 + (2,3,3,3,2)_{t} \beta^3 , \\
 Q_{2,4}(1,\beta)=
(60, 118, 72, 13)_{\beta} ,\\
Q_{2,5}(t,\beta)=(60,144,228,298,348,378,388,378,348,298,228,144,60)_{t} \\
\hphantom{Q_{2,5}(t,\beta)=}{}
 +(262,614,948,1208,1378,1462,1462,1378,1208,948,614,262)_{t} \beta \\
\hphantom{Q_{2,5}(t,\beta)=}{}
 +(458,1042,1560,1930,2142,2211,2142,1930,1560,1042,458)_{t} \beta^2 \\
\hphantom{Q_{2,5}(t,\beta)=}{}
 +(405,887,1278,1526,1640,1640,1526,1278,887,405)_{t} \beta^4 \\
 \hphantom{Q_{2,5}(t,\beta)=}{}
 +(187,389,534,610,632,610,534,389,187)_{t} \beta^4 \\
\hphantom{Q_{2,5}(t,\beta)=}{}
+(41,79,102,110,110,102,79,41)_{t} \beta^5
+(3,5,6,6,6,5,3)_t \beta^6 , \\
 Q_{2,5}(1,\beta) = (3300, 11744, 16475, 11472, 4072, 664, 34)_{\beta} , \\
 Q_{2,6}(1,\beta)= (660660,3626584,8574762,11407812,9355194,4866708, 1589799 , \\
 \hphantom{Q_{2,6}(1,\beta)= (}{}
 310172, 32182, 1320)_{\beta} ,\\
 Q_{2,7}(\beta)= (1,213,12145, 279189, 3102220,
 18400252, 61726264, 120846096, 139463706, \\
 \hphantom{Q_{2,7}(\beta)= (}{}
93866194, 5567810, 7053370, 626730, 16290 )_{\beta+1}.
\end{gather*}

\begin{Theorem}\label{theorem5.13}
 One has
\begin{gather*}
 Q_{m,n}(1,1) = \prod_{k=1}^{n-2} {\rm Cat}_{k} \prod_{1\le i < j \le n-1}
 \frac{2(m+1)+i+j-1}{i+j-1},\\
\sum_{k \ge 0} \iota ({\cal{P}}_{n,m};k) \beta^k = \frac{Q_{m,n}(1,
\beta-1)}{(1-\beta)^{{n+1 \choose 2}+1}},
\end{gather*}
where $ {\cal{P}}_{n,m}$ denotes the generalized Chan--Robbins--Yuen polytope
defined in {\rm \cite{Me2}}, and for any integral convex polytope~$\cal{P}$,
$\iota({\cal{P}},k)$ denotes the Ehrhart polynomial of polytope~$\cal{P}$.
\end{Theorem}

\begin{Conjecture}\label{conjecture5.8} Let $n \ge 3$, $m \ge 0$ be integers, , write
\begin{gather*}
 Q_{m,n}(t,\beta)=\sum_{k \ge 0} c_{m,n}^{(k)}(t) \beta^k, \qquad \text{and set}\qquad b(m,n):=
 \max\big( k \,|\, c_{m,n}^{(k)}(t) \not= 0 \big).
 \end{gather*}
Denote by ${\tilde{c}}_{m,n}(t)$ the polynomial obtained from that $c_{m,n}^{(b(m,n)}(t)$
by dividing the all coefficients of the latter on their GCD. Then
\begin{gather*}
{\tilde{c}}_{n,m}(t)=a_{n+m}(t),
\end{gather*}
where the polynomials $a_n(t):=c_{0,n}(t)$ have been defined in Conjecture~{\rm \ref{conjecture5.6}}. 
\end{Conjecture}

For example,
\begin{gather*}
c_{2,5}(t)=4 a_7(t), \qquad c_{2,6}(t) = 10 a_8(t), \qquad c_{3,5}(t)= a_8(t), \\
c_{2,7}(t)= 10 (34,78,118,148,168,
178,181,178,168,148,118,78,34) \stackrel{?}{=} 10 a_9(t).
\end{gather*}

It is known \cite{K1, Me,Me-b} that
\begin{gather*}
\prod_{k=1}^{n-2} {\rm Cat}_{k} \prod_{1\le i < j \le n-1}
 \frac{2(m+1)+i+j-1}{i+j-1}=\prod_{j=m+1}^{m+n-2} \frac{1}{2j+1} {n+m+j \choose 2 j}\\
\qquad{}
 =
K_{A_{n+1}}\left(m+1,m+2,\ldots,n+m,-m n- {n \choose 2}\right) .
\end{gather*}

\begin{Conjecture}\label{conjecture5.9}
Let $ {\boldsymbol{a}}=(a_2,a_3,\ldots,a_n)$ be a sequence of non-negative
integers, consider the following element
\begin{gather*}
 M_{({\boldsymbol{a}})} = \left( \prod_{j=2}^{n} x_{1j}^{a_{j}} \right)
\prod_{j=2}^{n-1} \left( \prod_{k=j+1}^{n} x_{jk} \right) .
\end{gather*}

Let $R_{{\boldsymbol{a}}}(t_1,\ldots,t_{n-1}, \alpha,\beta)$ be the following specialization
$x_{ij} \longrightarrow t_{j-1}$ for all $1 \le i < j \le n $
of the reduced polynomial $R_{{\boldsymbol{a}}}(x_{ij}) $ of monomial $M_{{\boldsymbol{a}}} \in
{ \widehat{{\rm ACYB}}}_n(\alpha,\beta)$.
Then the polynomial $R_{{\boldsymbol{a}}}(t_1,\ldots,t_{n-1}, \alpha,\beta)$ is well-defined, i.e., does not depend on an order in
which relations~$(a')$ and~$(b)$, Definition~{\rm \ref{definition5.1}}, have been applied.
\begin{gather*}
Q_{M_{\boldsymbol{a}}}(1,\beta=0)= K_{A_{n+1}}\left(a_{2}+1,a_{3}+2,\ldots,a_{n}+n-1,- {n \choose 2} - \sum_{j=2}^{n} a_j\right) .
\end{gather*}
Write
\begin{gather*}
Q_{M_{\boldsymbol{a}}}(t,\beta) =\sum_{k \ge 0} c_{ {\boldsymbol{a}}}^{(k)}(t) \beta^k.
\end{gather*}
The polynomials $c_{ {\boldsymbol{a}}}^{(k)}(t)$ are symmetric $($unimodal?$)$ for all~$k$.
\end{Conjecture}

\begin{Example} \label{example5.8}
Let's take $n=5$, ${\boldsymbol{a}}=(2,1,1,0)$. One can show that the value
of the Kostant partition function $K_{A_{5}}(3,3,4,4,-14)$ is equal to~$1967$. On the other hand, one has
\begin{gather*}
 Q_{(2,1,1,0)}(t,\beta) t^{-3}= (50,118,183,233,263,273,263,233,183,118,
50)_{t} \\
\hphantom{Q_{(2,1,1,0)}(t,\beta) t^{-3}=}{}
+ (214,491,738,908,992,992,908,738,491,214)_{t} \beta\\
\hphantom{Q_{(2,1,1,0)}(t,\beta) t^{-3}=}{}
+ (365,808,1167,1379,1448, 1379 , 1167,808,365)_{t} \beta^2 \\
\hphantom{Q_{(2,1,1,0)}(t,\beta) t^{-3}=}{}
+ (313,661,906,1020,1020,906,661,313)_{t} \beta^3\\
\hphantom{Q_{(2,1,1,0)}(t,\beta) t^{-3}=}{}
+
 (139,275,351,373,351,275,139)_{t} \beta^4\\
\hphantom{Q_{(2,1,1,0)}(t,\beta) t^{-3}=}{}
 +(29,52,60,60,52,29)_{t} \beta^5 + (2,3,3,3,2)_{t} \beta^6 , \\
 Q_{(2,1,1,0)}(1,\beta) = (1967,6686,8886,5800,1903,282,13) = (1,34,279,748,
688,204,13)_{\beta+1}.
\end{gather*}

It might be well to point out that since we know, see Theorem~\ref{theorem5.11}, that polynomials $Q_{M_{\boldsymbol{a}}}(1,\beta)$ in face are polynomials of~$\beta+1$
 with non-negative integer coef\/f\/icients, we can treat the polynomial
${\widetilde{Q}}_{M_{\boldsymbol{a}}}(\beta):= Q_{M_{\boldsymbol{a}}}(1,\beta-1)$ as a~$\beta$-analogue of the Kostant partition function in the dominant chamber. It seems an interesting {\it problem} to f\/ind an interpretation of polynomials
${\widetilde{Q}}_{M_{\boldsymbol{a}}}(\beta)$ in the framework of the representation theory of Lie algebras. For example,
\begin{gather*}
{\widetilde{Q}}_{(2,1,1,0)}(\beta) = (1,34,279,748,688,204,13)_{\beta},\\
{\widetilde{Q}}_{(2,1,1,0)}(\beta =1)= 1967 = K_{A_{5}}(3,3,4,4,-14).
\end{gather*}
\end{Example}

\begin{Exercises} \label{exercises5.9}\quad
\begin{enumerate}\itemsep=0pt
\item[(1)] Show that
\begin{gather*}
R_{n}(t,-1) = t^{2(n-2)} R_{n-1}\big({-}t^{-1}, 1\big).
\end{gather*}

\item[(2)] Show that the ratio
\begin{gather*}
\frac{R_n(0,\beta)}{(1+\beta)^{n-2}}
\end{gather*}
 is a polynomial in $(\beta+1)$~with non-negative coef\/f\/icients.

\item[(3)] Show that polynomial $R_n(t,1)$ has degree $e_n:= (n+1)(n-2)/2$, and
\begin{gather*}
{\rm Coef\/f}[t^{e_{n}}] R_n(t,1) =\prod_{k=1}^{n-1} {\rm Cat}_k .
\end{gather*}

\item[(4)] Show that
\begin{gather*}
{\widetilde{Q}}_{(n,2,3,0)}(\beta)= \left(1, 3 n+2, \binom{n+1}{2}+n, \binom{n+1}{3} +\binom{n}{2} \right)_{\beta},
\end{gather*}
 and
 \begin{gather*}
 K_{A_{4}} (n,3,4,-n-7) = \frac{(n+2)(n+3)(n+9)}{6}.
\end{gather*}
\end{enumerate}
\end{Exercises}

\begin{Problems}\label{problems5.2}\quad
\begin{enumerate}\itemsep=0pt
\item[$(1)$] Assume additionally to the conditions $(a')$ and $(b)$ above that
\begin{gather*}
x_{ij}^2= \beta x_{ij} + 1 \qquad \text{if} \quad 1 \le i < j \le n.
\end{gather*}
What one can say about a reduced form of the element $w_{0}$ in this case?

\item[(2)] According to a result by S.~Matsumoto and J.~Novak~{\rm \cite{MN}}, if
$\pi \in {\mathbb S}_n$ is a permutation of the cyclic type $\lambda \vdash n$,
then the total number of primitive factorizations $($see definition in~{\rm \cite{MN})}
of~$\pi$ into product of $n-\ell (\lambda)$ transpositions, denoted by
$\operatorname{Prim}_{n-\ell (\lambda)}(\lambda)$, is equal to the product of Catalan numbers:
\begin{gather*}
\operatorname{Prim}_{n-\ell (\lambda)}(\lambda) = \prod_{i=1}^{\ell(\lambda)} {\rm Cat}_{\lambda_{i}-1}.
\end{gather*}
Recall that the Catalan number ${\rm Cat}_n:= C_n={1 \over n} {2n \choose n}$.
Now take $\lambda=(2,3,\ldots,n+1)$. Then
\begin{gather*}
Q_n(1)= \prod_{a=1}^{n} {\rm Cat}_a=\operatorname{Prim}_{{n \choose 2}}(\lambda).
\end{gather*}
Does there exist ``a~natural'' bijection between the primitive factorizations
and monomials which appear in the polynomial $Q_n(x_{ij};\beta)$?

\item[$(3)$] Compute in the algebra ${\widehat{{\rm ACYB}}}_n(\alpha,\beta)$ the
specialization
\begin{gather*}
x_{ij} \longrightarrow 1, \quad j < n,\qquad
x_{ij} \longrightarrow t,\quad 1 \le i < n,
\end{gather*}
denoted by $P_{w{_n}}(t,\alpha,\beta)$,
 of the reduced polynomial $P_{s_{ij}}(\{x_{ij} \},\alpha, \beta)$
corresponding to the transposition
\begin{gather*}
s_{ij}:= \left( \prod_{k=i}^{j-2}
 x_{k,k+1} \right) x_{j-1,j} \left( \prod_{k=j-2}^{i} x_{k,k+1} \right) \in
{\widehat{{\rm ACYB}}}_n(\alpha,\beta).
\end{gather*}
For example,
\begin{gather*}
P_{s_{14}}(t,\alpha,\beta)= t^{5}+3(1+\beta) t^{4}+
((3,5,2)_{\beta}+3 \alpha) t^{3}+ (2(1+\beta)^2+\alpha(5+4 \beta)) t^{2}\\
\hphantom{P_{s_{14}}(t,\alpha,\beta)=}{}+
((1+\beta((1+3 \alpha)+2 \alpha^{2}) t +\alpha + \alpha^{2}.
\end{gather*}
\end{enumerate}
\end{Problems}

\subsection{Reduced polynomials of certain monomials}\label{section5.4}

In this subsection we compute the reduced polynomials corresponding to {\it
dominant} monomials of the form
\begin{gather*}
 x_{{\boldsymbol{m}}}:= x_{1,2}^{m_1} x_{23}^{m_2} \cdots x_{n-1,n}^{m_{n-1}}
\in \big({\widehat{{\rm ACYB}}}_n(\beta)\big)^{ab},
\end{gather*}
 where ${\boldsymbol{m}} =(m_1 \ge m_2 \ge \cdots \ge m_{n-1} \ge 0)$
 is a {\it partition}, and we apply the relations~$(a')$ and~$(b)$ in
the algebra $({\widehat{{\rm ACYB}}}_n(\beta))^{ab}$, see Def\/inition~\ref{definition5.1} and
Section~\ref{section5.3.1}, {\it successively}, starting from $x_{12}^{m_1}
 x_{23}$.

\begin{Proposition} \label{proposition5.9} The function
\begin{gather*}
 \Z_{\ge 0}^{n-1} \longrightarrow \Z_{\ge 0}^{n-1}, \qquad {\boldsymbol{m}}
\longrightarrow P_{{\boldsymbol{m}}}(t=1;\beta=1)
\end{gather*}
can be extended to a piece-wise polynomial function on the space
$\R_{\ge 0}^{n-1}$.
\end{Proposition}

We start with the study of powers of Coxeter elements. Namely, for powers
of {\it Coxeter} elements, one has\footnote{To simplify notation we set $P_{w}(\beta) := P_{w}(x_{ij}=1 ;\beta)$.}
\begin{gather*}
 P_{(x_{12} x_{23})^2}(\beta)=(6,6,1), \qquad P_{(x_{12} x_{23} x_{34})^2}(\beta)=
(71,142,91,20,1) = (1,16,37,16,1)_{\beta+1} , \\
 P_{(x_{12} x_{23} x_{34})^3}(\beta) = (1301,3903,4407,2309,555,51,1)= (1,45,315,579,315,45,1)_{\beta+1} , \\
 P_{(x_{12} x_{23} x_{34} x_{45})^2}(\beta)=(1266,3798,4289,2248,541,50,1)=
(1,44,306,564,306,44,1)_{\beta+1} , \\
 P_{(x_{12} x_{23} x_{34})^3}(\beta=1)=
12527 , \qquad P_{(x_{12} x_{23} x_{34})^4}(\beta=0)= 26599 , \\
 P_{(x_{12} x_{23} x_{34})^4}(\beta=1)= 539601 ,\qquad
 P_{(x_{12} x_{23} x_{34} x_{45})^2}(\beta=1)= 12193 , \\
 P_{(x_{12} x_{23} x_{34} x_{45})^3}(\beta=0)=50000 ,\qquad
 P_{(x_{12} x_{23} x_{34} x_{45})^3}(\beta=1)= 1090199 .
\end{gather*}

\begin{Lemma} \label{lemma5.3}
One has
\begin{gather*}
P_{x_{12}^n x_{23}^m}(\beta)= \sum_{k=0}^{\min(n,m)} {n+m-k
\choose m} {m \choose k} \beta^k = \sum_{k=0}^{\min(n,m)} {n \choose k}
{m \choose k} (1 +\beta)^k.
\end{gather*}
 Moreover,
\begin{itemize}\itemsep=0pt
\item polynomial $P_{(x_{12} x_{23} \cdots x_{n-1,n})^m}(\beta-1)$
is a symmetric polynomial in $\beta$ with
non-negative coefficients.

\item polynomial $P_{x_{12}^{n} x_{23}^{m}}(\beta)$ counts the number
of $(n,m)$-Delannoy paths according to the number of $NE$ steps\footnote{Recall that a $(n,m)$-Delannoy path is a lattice paths from $(0,0)$ to
$(n,m)$ with steps $E=(1,0)$,
$N=(0,1)$ and $NE=(1,1)$ only.
For the def\/inition and examples of the
Delannoy paths and numbers, see \cite[$A001850$, $A008288$]{SL} and
\url{http://mathworld.wolfram.com/DelannoyNumber.html}. }.
\end{itemize}
\end{Lemma}

\begin{Proposition} \label{proposition5.10} Let $n$ and $k$, $0 \le k \le n$, be integers. The number
\begin{gather*}
P_{(x_{12} x_{23})^n (x_{34})^k}(\beta=0)
\end{gather*}
is equal to the number of~$ n$ up, $n$ down permutations in the
symmetric group ${\mathbb{S}}_{2n+k+1}$, see~{\rm \cite[$A229892$]{SL}} and
Exercises~{\rm \ref{exercises5.3}(2)}.
\end{Proposition}

\begin{Conjecture}\label{conjecture5.10} Let $n$, $m$, $k$ be nonnegative integers. Then the number
\begin{gather*}
P_{x_{12}^n x_{23}^m x_{34}^k}(\beta=0)
\end{gather*}
is equal to the number of~$n$ up, $m$ down and~$k$ up permutations in the
symmetric group~$\mathbb{S}_{n+m+k+1}$.
\end{Conjecture}

For example,
\begin{itemize}\itemsep=0pt
\item Take $n=2$, $k=0$, the six permutations in $\mathbb{S}_5$ with
$2$ up, $2$ down are {12543}, {13542}, {14532}, {23541}, {24531}, {34521}.

\item Take $n=3$, $k=1$, the twenty permutations in $\mathbb{S}_7$ with
$3$ up, $3$ down are {1237654}, {1247653}, {1257643}, {1267543}, {1347652}, {1357642},
{1367542}, {1457632}, {1467532}, {1567432}, {2347651}, \linebreak {2357641},
 {2367541}, {2457631}, {2467531}, {2567431}, {3457621}, {3467521},
 {3567421}, {4567321}, see \cite[$A229892$]{SL}.

\item Take $n=3$, $m=2$, $k=1$, the number of $3$ up, $2$ down and~$1$ up
permutations in~$\mathbb{S}_7$ is equal to $50=P_{321}(0)$: {1237645},
{1237546}, \dots, {4567312}.

\item Take $n=1$, $m=3$, $k=2$, the number of~$1$ up, $3$ down and~$2$ up
permutations in $\mathbb{S}_7$ is equal to $55=P_{132}(0)$, as it can be easily checked.
\end{itemize}

On the other hand, $P_{x_{12}^4 x_{23}^3 x_{34}^2 x_{45}}(\beta=0) =7203
< 7910$, where $7910$ is the number of $4$ up, $3$~down, $2$ up and $1$ down
permutations in the symmetric group $\mathbb{S}_{11}$.

\begin{Conjecture}\label{conjecture5.11}
 Let $k_1,\ldots,k_{n-1}$ be a sequence of non-negative integer
numbers, consider monomial $M:=x_{12}^{k_{1}} x_{23}^{k_2}
\cdots x_{n-1,n}^{k_{n-1}}$. Then
reduced polynomial $P_{M}(\beta -1)$ is a unimodal
polynomial in $\beta$ with non-negative coefficients.
\end{Conjecture}
\begin{Example} \label{example5.9}
\begin{gather*}
 P_{3,2,1}(\beta)= (1,14,27,8)_{\beta +1}=P_{1,2,3}(\beta) ,\qquad
 P_{2,3,1}(\beta)= (1,15,30,9)_{\beta +1}=P_{1,3,2}(\beta) , \\
 P_{3,1,2}(\beta)= (1,11,18,4)_{\beta +1}=P_{2,1,3}(\beta) , \\
 P_{4,3,2,1}(\beta)= (1,74,837,2630,2708,885,68)_{\beta +1} , \qquad
 P_{4,3,2,1}(0)= 7203 = 3 \cdot 7^4 , \\ P_{5,4,3,2,1}(\beta)=
(1,394,19177,270210,1485163,3638790,
 4198361,2282942,\\
 \hphantom{P_{5,4,3,2,1}(\beta)=(}{} 553828,51945,1300)_{\beta +1} ,\\
 P_{5,4,3,2,1}(0)= 12502111 = 1019 \times 12269 .
\end{gather*}
\end{Example}

\begin{Exercises}\label{exercises5.10}\quad

(1) Show that if $ n \ge m$, then
\begin{gather*}
x_{ij}^{n} x_{jk}^{m} \Big{\arrowvert}_{x_{ij}=1=x_{jk}} = \sum_{a=0}^{n}
{m+a-1 \choose a} \left( \sum_{p=0}^{n-a} {m \choose p} \beta^{p} \right)
 x_{ik}^{m+a} .
 \end{gather*}

(2) Show that if $ n \ge m \ge k$, then
\begin{gather*}
P_{x_{12}^{n} x_{23}^{m} x_{34}^{k}}(\beta)=
P_{x_{12}^{n} x_{23}^{m}}(\beta)\\
\hphantom{P_{x_{12}^{n} x_{23}^{m} x_{34}^{k}}(\beta)=}{}+
 \sum_{ a \ge 1 \atop b,p \ge 0} {m \choose p} {k \choose a} {a-1 \choose b} {n+1 \choose p+a-b} {m+a-1-b \choose a} (\beta+1)^{p+a}.
 \end{gather*}
In particular, if $n \ge m \ge k$, then
\begin{gather*}
 P_{x_{12}^{n} x_{23}^{m} x_{34}^{k}}(0)= {m+n \choose n}+
\sum_{a \ge 1} {k \choose a} \left(\sum_{b=1}^{a}
{m+n+1 \choose m+b} {a-1 \choose b-1} {m+b-1 \choose a} \right) .
\end{gather*}
Note that the set of relations from the item $(1)$ allows to give an
explicit formula for the polynomial $P_{M}(\beta)$ for any {\it dominant}
sequence $M=(m_1 \ge m_2 \ge \cdots \ge m_{k}) \in (\Z_{ > 0})^{k}$.
Namely,
\begin{gather*}
P_{M}(\beta+1) =
 \sum_{\boldsymbol{a}} \prod_{j=2}^{k} {m_j+a_{j-1}-1 \choose a_{j-1} } \left(
\sum_{\boldsymbol{b}} \prod_{j=1}^{k-1} {m_{j+1} \choose b_{j} } \beta^{b_{j}} \right),
\end{gather*}
where the f\/irst sum runs over the following set ${\cal A}(M)$ of integer
sequences $ {\boldsymbol{a}} = (a_1,\ldots, a_{k-1}) $
\begin{gather*}
{\cal{A}}(M):= \{ 0 \le a_{j} \le m_{j}+a_{j-1}, \, j=1,\ldots, k-1 \},\qquad
a_{0}=0,
\end{gather*}
and the second sum runs over the set ${\cal B}(M)$ of all integer sequences
${\boldsymbol{b}} =(b_1,\ldots,b_{k-1})$
\begin{gather*}
{\cal B}(M) := \bigcup_{{\boldsymbol{a}} \in {\cal A}(M)} \{ 0 \le b_j \le
\min (m_{j+1},m_j-a_j+a_{j-1}) \}, \qquad j=1,\ldots,k-1.
\end{gather*}

(3) Show that
\begin{gather*}
 \# \big|{\cal A}\big(n,1^{k-1}\big) \big| = {n+1 \over k}{2k+n \choose k-1} = f^{(n+k,k)},
 \end{gather*}
where $f^{(n+k,k)}$ denotes the number of standard Young tableaux of shape
$(n+k,k)$. In particular, $ \#|{\cal A}(1^k)| = C_{k+1}$.

(4) Let $n \ge m \ge 1$ be integers and set $M=(n,m,1^{k})$.
Show that
\begin{gather*}
 P_{M}(x_{ij}=1 ; \beta=0)= \sum_{p=0}^{n} {m+p+1 \over k}~{m+p-1 \choose p}
{m+2k+p \choose k-1} := P_k(n,m).
\end{gather*}
In particular, $P_{1}(n,m)= {n+m \choose n} +m {n+m+1 \choose n}$,
\begin{gather*}
P_{k}(n,1)={n+1 \over k+1} {2k+2+n \choose k} ,\qquad
P_{k}(2,2)= \big(79k^2+341k+360\big){(2k+2) ! \over k! (k+5)! }.
\end{gather*}
Let us remark that
\begin{gather*}
P_k(n,1)= {\frac{n+1}{n+k+2}} {\binom{2(k+1)+n}{k+1}}=F_{k+1}^{(2)}(n)=
D(k,1,n,2),
\end{gather*}
where the $D(k,1,n,2)$ and $F_{k+1}^{(2)}(n)$ are def\/ined in Section~\ref{section5.2.4}.

(5) Let $T \in {\rm STY}((n+k,k))$ be a standard Young tableau of shape
$(n+k,k)$.
Denote by~$r(T)$ the number of integers $j \in [1,n+k]$ such that the integer
$j$ belongs to the second row of tableau~$T$, whereas the number $j+1$ belongs
to the f\/irst row of~$T$.

Show that
\begin{gather*}
 P_{x_{12}^{n} x_{23} \cdots x_{k+1,k+2}}(\beta -1) = \sum_{ T \in {\rm STY}((n+k,k))} \beta^{r(T)}.
\end{gather*}

(6) Let $M =(m_1,m_2,\ldots,m_{k-1}) \in \Z_{ > 0 }^{k-1}$ be a
composition. Denote by $\overleftarrow{M}$ the composition
$(m_{k-1},m_{k-2},\ldots,m_2,m_1)$, and set for short
\begin{gather*}
P_{M}(\beta):=
P_{\prod _{i=1}^{k-1} x_{i,i+1}^{m_i}}(x_{ij}=1; \beta).
\end{gather*}
Show that
$P_M(\beta)=P_{\overleftarrow{M}}(\beta)$.
Note that in general,
\begin{gather*}
P_{\prod\limits_{i=1}^{k-1}
x_{i,i+1}^{{m_i}}} (x_{ij};\beta) \not=
P_{\prod\limits_{i=1}^{k-1} x_{i,i+1}^{m_{k-i}}}(x_{ij}; \beta).
\end{gather*}

(7) Def\/ine polynomial $P_{M}(t,\beta)$ to be the following specialization
\begin{gather*}
 x_{ij} \longrightarrow 1, \quad i < j <n, \qquad x_{in} \longrightarrow t, \quad i=1,\ldots,n-1
 \end{gather*}
of a polynomial $P_{\prod\limits_{i=1}^{k-1} x_{i,i+1}^{m_i}}(x_{ij}; \beta)$.

Show that if $ n \ge m$, then
\begin{gather*}
P_{x_{12}^{n} x_{23}^m}(t,\beta)= \sum_{j=0}^{m} {m \choose j} \left(
\sum_{k=m-1}^{n+m-j-1} {k \choose m-1} t^{k-m+1} \right) \beta^j.
\end{gather*}
See Lemma~\ref{lemma5.2} for the case $t=1$.

(8) Def\/ine polynomials ${\widetilde{R}}_n(t)$ as follows
\begin{gather*}
 {\widetilde{R}}_n(t) :=P_{(x_{12} x_{23} x_{34})^n} \big({-}t^{-1}, \beta= -1\big) (-t)^{3n}.
 \end{gather*}
Show that polynomials ${\widetilde{R}}_n(t)$ have non-negative
coef\/f\/icients, and
\begin{gather*}
 {\widetilde{R}}_n(0) =\frac{(3n)!}{6(n!)^3} .
 \end{gather*}

$(9)$ Consider reduced polynomial $P_{n,2,2}(\beta)$ corresponding to
monomial $x_{12}^n (x_{23}x_{34})^2$ and
set ${\tilde{P}}_{n,2,2}(\beta):=P_{n,2,2}(\beta -1)$. Show
that
\begin{gather*}
 {\tilde{P}}_{n,2,2}(\beta) \in \N [\beta] \qquad \text{and}\qquad
{\tilde{P}}_{n,2,2}(1) =T(n+5,3),
\end{gather*}
where the numbers $T(n,k)$ are def\/ined in \cite[$A110952$, $A001701$]{SL}.
\end{Exercises}

\begin{Conjecture} \label{conjecture5.12}
Let $\lambda$ be a partition. The element
$s_{\lambda} (\theta_1^{(n)},\ldots,\theta_{m}^{(n)})$ of the algebra
$3T_n^{(0)}$ can be written in this algebra as a sum of
\begin{gather*}
 \left( \prod_{x \in \lambda} h(x) \right) \times
\dim {V_{\lambda'}}^{({\mathfrak {gl}}(n-m))} \times
\dim {V_{\lambda}}^{({\mathfrak {gl}}(m))}
\end{gather*}
monomials with all coefficients are equal to~$1$.

Here $s_{\lambda}(x_1,\ldots,x_m)$ denotes the Schur function corresponding
to the partition $\lambda$ and the set of variables $\{x_1,\ldots,x_m \}$;
for $x \in \lambda, $ $h(x)$ denotes the hook length corresponding to a~box~$x$; $V_{\lambda}^{({\mathfrak {gl}}(n))}$ denotes the
highest weight~$\lambda$ irreducible representation of the Lie algebra
${\mathfrak {gl}}(n)$.
\end{Conjecture}

\begin{Problems}\label{problems5.3}\quad
\begin{enumerate}\itemsep=0pt
\item[$(1)$] Define a bijection between monomials of the form
$\prod\limits_{a=1}^{s} x_{i_a,j_a}$ involved in the polynomial $P(x_{ij};\beta)$,
and dissections of a convex $(n+2)$-gon by $s$ diagonals, such that no two
diagonals intersect their interior.

\item[$(2)$] Describe permutations $w \in \mathbb{S}_n$ such that the
Grothendieck polynomial $\mathfrak{G}_{w}(t_1,\ldots,t_n)$ is equal to
the ``reduced polynomial'' for a some monomial in the associative
quasi-classical Yang--Baxter algebra $\widehat{{\rm ACYB}}_n(\beta)$.

\item[$(3)$] Study ``reduced polynomials'' corresponding to the monomials
\begin{itemize}\itemsep=0pt
\item transposition: $s_{1n}:= (x_{12}x_{23}
\cdots x_{n-2,n-1})^2 x_{n-1,n}$,
\item powers of the Coxeter element: $(x_{12}x_{23}\cdots
x_{n-1,n})^{k}$,
\end{itemize}
in the algebra $\widehat{{\rm ACYB}}_n(\alpha,\beta)^{ab}$.

\item[$(4)$] Construct a bijection between the set of $k$-dissections
of a convex $(n+k+1)$-gon and ``pipe dreams'' corresponding to the
Grothendieck polynomial
$\mathfrak{G}_{\pi_{k}^{(n)}}^{(\beta)}(x_1,\ldots,x_n)$.
As for a~de\-finition of ``pipe dreams'' for Grothendieck polynomials, see
{\rm \cite{KMY}} and {\rm \cite{FK1}}.
\end{enumerate}
 \end{Problems}

\begin{Comments} \label{comments5.9}
We don't know any ``good'' combinatorial interpretation of polynomials which
appear in Problem~\ref{problems5.3}(3) for general~$n$ and~$k$. For example,
\begin{gather*}
P_{s_{13}}(x_{ij}=1;\beta)=(3,2)_{\beta},\qquad
P_{s_{14}}(x_{ij}=1;\beta)=(26,42,19,2)_{\beta}, \\
P_{s_{15}}(x_{ij}=1;\beta)=(381,988,917,362,55,2)_{\beta},\qquad
P_{s_{15}}(x_{ij}=1;1)= 2705.
\end{gather*}
 On the other hand,
\begin{gather*}
P_{(x_{12}x_{23})^2 x_{34} (x_{45})^2}(x_{ij}=1;\beta)=(252,633,565,212,30,1),
\end{gather*}
 that is in deciding on dif\/ferent reduced decompositions of the
transposition $s_{1n}$. one obtains in general dif\/ferent reduced polynomials.

One can compare these formulas for polynomials $P_{s_{ab}}(x_{ij}=1;\beta)$
with those for the $\beta$-Grothendieck polynomials corresponding to
transpositions $(a,b)$, see Comments~\ref{comments5.5}.
\end{Comments}

\subsubsection{Reduced polynomials, Motzkin and Riordan numbers}\label{section5.4.1}

In this subsection we investigate reduced polynomials associated with Coxeter
element $C_n=u_{12} u_{23} \cdots u_{n-1,n}$ in commutative algebra
 ${\widehat{{\rm ACYB}}}_n(\alpha, \beta)$ in more detail. Recall that this algebra
is generated over the ring $\Z[z,\alpha, \beta]$ by the set of elements
$\{u_{i,j}, \, 1 \le i < j \le n \}$ subject to the following relations
\begin{gather*}
 u_{ij} u_{jk} = u_{ik} u_{ij}+u_{jk} u_{ik} + \beta u_{ik} + \alpha, \qquad i < j <k .
 \end{gather*}

Show that
\begin{gather*}
P_n(1,1,\beta=-1) = M_n,
\end{gather*}
where $M_n$ denotes the $n$-th {\it Motzkin} number that is the number of
{\it Motzkin $n$-paths}: paths from $(0,0)$ to $(n,0)$ in an $n \times n$
grid using only steps $U=(1,1)$, $(1,0)$ and $(1,-1)$. It is also the
number of Dyck $(n+1)$-paths with no steps $UUU$, see \cite[$A001006$]{SL} for
 a~wide variety of combinatorial interpretations, and vast literature
concerning the Motzkin numbers. For example,
\begin{gather*}
P_7(0,1,\beta = -1)= 36+37+24+18+5+6+0+1 = 127 =M_7.
\end{gather*}

Therefore we treat the polynomials
$P_n(t,\alpha,\beta=-1)$ as the $(t,\alpha)$-Motzkin numbers. For example,
\begin{gather*}
P_7(t,\alpha,\beta=-1)=t^7+6 \alpha t^5 + 5 \alpha t^4 +(0,4,14)_{\alpha} t^3 + (0,3,21)_{\alpha} t^2 + (0,2,21,14)_{\alpha} t\\
\qquad{} +(0,1,14,21)_{\alpha} = t^7+\alpha(1,2,3,4,5,6)_{t}+ \alpha^2(14,21,21,14)_{t} + \alpha^3(21,14)_{t}.
\end{gather*}
Therefore
\begin{gather*}
P_7(t,1,\beta=-1) = 1+21 \alpha+ 70 \alpha^2 + 35 \alpha^3, \qquad
P_7(1,1,\beta=-1)=127=M_7.
\end{gather*}

Show that
\begin{gather*}
P_n(0,1,\beta=-1)= A005043(n),
\end{gather*}
known as the {\it Riordan number}, or {\it Motzkin sum}~\cite{SL}. This
number, denoted by~$MS_n$, counts the number of Motzkin paths of length $n$
with no horizontal steps at level zero; it is also equal to the number of
Dyck paths of semilenght $n$ with no {\it peaks} at odd level, see \cite[$A005043$]{SL}
for a~bit more combinatorial interpretations, and literature
concerning the Motzkin sum or Riordan numbers. For example,
\begin{gather*}
P_7(t,1,-1) =
({\bf 36},37,24,18,5,6,0,1), \qquad 36 = MS_7.
\end{gather*}

Show that the Riordan number $MS_n$ is equal to the
number of underdiagonal paths from $(0,0)$ to the line $x=n-2$, using
 only steps $(1,0)$, $(0,1)$ and $NE= (2,1)$ and beginning with the step
 $NE=(2,1)$. Note that the number of such paths with no steps $NE$ is equal to
the Catalan number ${\rm Cat}_{n-1}$.

Let ${\cal{MS}} = \{ n \in \N \,|\, n= 2^{2k}(2 r+1)-1,\, k \ge 1, \, r \ge 0 \}$ be a subset of the set of all odd integers~\cite{DS}. Show that

$(a)$ $MS_n \equiv 1$ $({\rm mod}~2)$, if either $n \equiv 0$ $({\rm mod}~2)$ or $n \in
{\cal{MS}}_n$,

$(b)$ $MS_n \equiv 0$ $({\rm mod}~2)$, if $n$ is an odd integer and $ n \notin {\cal{MS}}$.

Show that
\begin{gather*}
 \frac{P_n(0,\alpha,\beta)}{\alpha}\Big|_{\alpha=0} =N_{n-1}(\beta+1),
\end{gather*}
where as before, $N_{n}(t)$ denotes the Narayana polynomial.

Let us set
\begin{gather*}
P_{n}(0,\alpha,\beta)= \sum_{k \ge 0} c_k(\beta+1) \alpha^k.
\end{gather*}
Show that polynomials $c_k(\beta+1)$, $k \ge 0$ are symmetric (unimodal?) polynomials of the variable $\beta+1$.

Show that~\cite{SL}
\begin{gather*}
P_n(1,1,0) = A052709(n+1).
\end{gather*}

Show that~\cite{SL}
\begin{gather*}
P_n(0,1,0) = A052705(n)
\end{gather*}
that is the number of underdiagonal paths from $(0,0)$ to the line $x=n-2$,
using only steps $R=(1,0)$, $V=(0,1)$ and $NE=(2,1)$.

For example,
\begin{gather*}
P_7(0,10) = {\bf 36}+106+120+64+15+1= 342 =A052705(7).
\end{gather*}

Show that~\cite{SL}
\begin{gather*}
{\frac{\partial}{\partial \alpha}} P_n(t,\alpha,\beta) \Big|_{{\alpha=0,\atop \beta=0,} \atop t=1} = A05775(n-1),
\end{gather*}
that is the number of paths in the half-plane $x \ge 0$ from $(0,0)$ to $(n-1,2)$ or $(n-1,-3)$, and consisting of steps $U=(1,1)$, $D=(1,-1)$ and $H=(1,0)$.
For example,
\begin{gather*}
{\rm l.h.s.} =106+130+99+48+5+6 =427 =A05775(6).
\end{gather*}

Let us set
\begin{gather*}
P_n(t,\alpha, \beta=1):= \sum_{k,l \ge 0} c_{k,l}^{(n)} t^{k} \alpha^{l}.
\end{gather*}
Show that
\begin{gather*}
(a)\quad \sum_{k=1}^{n} c_{k,n-k}^{(n)} t^k \alpha^{n-k} = (t+ \alpha)^{n-1},\\
(b) \quad c_{k,n-k-1}^{(n)}= (k+1) \binom{n-1}{k+2}, \qquad 0 \le k \le n-3,\\
(c) \quad c_{1,0}^{(n)} = c_{0,0}^{(n)}+ (-1)^{n-1} , \qquad n \ge 3.
\end{gather*}

\subsubsection{Reduced polynomials, dissections and Lagrange inversion formula}\label{section5.4.2}

 Let $ \{a_i,b_i,\beta_i,\alpha_i,\, 1 \le i \le n-1\} $ be a set of parameters, consider non commutative algebra generated over the ring $Z[\{a_i,b_i,\beta_i,\alpha_i \}_{1 \le i \le n-1}]$ by the set of generators $\{u_{ij},\, 1 \le i < j \le n \}$ subject to the set of relations
\begin{gather*}
u_{ij} u_{jk} =a_i u_{ik} u_{ij} + b_i u_{jk} u_{ik} + \beta u_{ik} +\alpha_i, \qquad 1 \le i < j < k \le n.
\end{gather*}
Consider reduced expression $R_n(\{u_{ij} \}_{1 \le i < j \le n})$ in the
above algebra which corresponds to the ``Coxeter element''
\begin{gather*}
C_n:= u_{12} u_{23} \cdots u_{n-1,n}.
\end{gather*}
Note that the reduced expression $R_n(\{u_{ij}\})$ is a linear combination of
noncommutative mono\-mials in the generators $\{u_{ij},\, 1 \le i < j \le n \}$
with coef\/f\/icients from the ring
\begin{gather*}
K_n:= \Z[\{a_i,b_i,\beta_i,\alpha_i\}_{1 \le i < n}].
\end{gather*}

Now to each monomial $U$ which appears in the reduced expression
$R_n(\{u_{ij}\})$ we associate a~{\it dissection} ${\cal D}:=
{\cal D}_{U}$ of a~convex $(n+1)$-gon as follows. First of
all let us label the vertices of a~convex $(n+1)$-gon selected, by the numbers
$n+1,n,\ldots,1$, written consequently and clockwise, starting from a f\/ixed
vertex, from here on named by $(n+1)$-vertex.

Next, let us take a monomial $ U= u_{i_{1},j_{1}} \cdots u_{i_{p}, j_{p}}$
which appears in the reduced expression $R_n(\{u_{ij} \})$ with coef\/f\/icient
$c(U) \in K_n$. We draw diagonals in a convex $(n+1)$-gon chosen which
connect vertices labeled correspondingly by numbers $i_{s}$ and $j_{s} +1$,
$s=1,\ldots, p$. It is clearly seen from the def\/ining relations in the algebra in question when being applied to the
Coxeter element above, that in fact, the diagonals we have drawn in a convex
$(n+1)$-gon selected, do not meet at interior points of our convex $(n+1)$-gon. Therefore, to each monomial $U$ which appears in the reduced polynomial
associated with the Coxeter element $C_n$ above, one can associate a~{\it dissecion} ${\cal D}:={\cal D}_{U}$ of a convex $(n+1)$-gon selected. Moreover, it
is not dif\/f\/icult to see (e.g., cf.~\cite{HN}) that there exists a natural
bijection $ U \Longleftrightarrow {\cal D}_{U}$ between monomials which
appear in the reduced expression $R_n(\{u_{ij}\})$ and the set of
dissections of a convex $(n+1)$-gon. As a corollary, to each dissection
${\cal D}:= {\cal D}_{U}$ of a~conves $(n+1)$-gon one can attache the element $c({\cal D}):= c(U) \in K_n$
which is equal to the coef\/f\/icient in front of monomial~$U$ in the reduced
expression corresponding to the Coxeter element~$C_n$.

To continue, let ${\boldsymbol{x}}=(x_1,\ldots,x_{n-1})$, $ {\boldsymbol{y}}=(y_1,\ldots,y_{n-1})$ and ${\boldsymbol{z}}=(z_1,\ldots,z_{n-1})$ be three sets of variables, and~${\cal D}$ be a~dissection of a convex $(n+1)$-gon. We associate with
dissection~${\cal D}$ a~monomial $m({\cal D}) \in K_n$ as follows
\begin{gather*}
m({\cal D }) := \prod_{k=1}^{n-1} x_{k}^{n(k)} y_{k}^{m(k)} z_{r(k)},
\end{gather*}
where $m(k):=m_{k}({\cal D})$ (resp.\ $r(k):=r_{k}({\cal D})$ and
$n(k):=n_{k}({\cal D})$) denotes the number of
(convex) $(m_{k}+2)$-gons constituent a dissection ${\cal D}$ taken (resp.\
the number of diagonals issue out of the vertex labeled by~$(n+1)$;
$n_{k}({\cal D}))$ stands for the number of (oriented) diagonals and edges which issue out of the vertex labeled by $k$, $k=1,\ldots,n$). Therefore we
associate with the reduced polynomial corresponding to the Coxeter element $u_{12}, \dots,u_{n-1,n}$ the following polynomial
\begin{gather*}
{\rm PL}_n({\boldsymbol{a}}, {\boldsymbol{b}}, {\boldsymbol{\beta}}, {\boldsymbol{\alpha}}, {\boldsymbol{x}}, {\boldsymbol{y}}, {\boldsymbol{z}})=
\sum_{\cal D} m({\cal D}) {c({\cal D})},
\end{gather*}
where the sum runs over all dissections ${\cal D}$ of a convex $(n+1)$-gon.

To begin with we set ${\boldsymbol{x}}={\bf 1}$ and consider the following specializations
\begin{gather*}
B_n({\boldsymbol{a}}, {\boldsymbol{y}}) ={\rm PL}_n({\boldsymbol{a}},{\boldsymbol{b}}={\bf 1},{\boldsymbol{\beta}}={\bf 1},{\boldsymbol{\alpha}}={\bf 0},{\boldsymbol{y}},{\boldsymbol{z}}={\bf 1}),\\
 P_n({\boldsymbol{z}},{\boldsymbol{a}},{\boldsymbol{b}},{\boldsymbol{\beta}}) = {\rm PL}_n({\boldsymbol{a}},{\boldsymbol{b}},{\boldsymbol{\beta}}, {\boldsymbol{\alpha}} ={\bf 0}, {\boldsymbol{y}}={\bf 1},{\boldsymbol{z}}),
\end{gather*}

Show that
\begin{gather*}
B_{n-1}({\boldsymbol{a}},{\boldsymbol{y}})) = {\rm Coef\/f}_{t^{n}}
\big( z-f(t y_1,\ldots,t y_n) \big)^{ [ -1 ]},
\end{gather*}
where $f(y_1,\ldots,y_n) = \sum\limits_{k =1}^{n-1} y_k u^{k+1}$, and for any
formal power series $g(u)$, ${\frac{d}{du}} g(u) \vert_{u=0} =1$, we denote
by $g(u)^{[-1]}$ the Lagrange Inverse formal power series associated with that
$g(u)$ that is a~unique formal power series such that $g(g^{[-1]}(u)) = u =
g^{[-1]}(g(u))$.

Now let us recall the statement of Lagrange's inversion theorem. Namely, let
\begin{gather*}
f(x) =x - \sum_{k \ge 1} y_k x^{k+1}
\end{gather*}
be a formal power series. Then the inverse power series
$f^{[-1]}(u)$ is given by the following formula
\begin{gather*}
 f^{[-1]}(y)= \sum_{n \ge 1} w_n u^n,
 \end{gather*}
where
\begin{gather*}
w_{n} :=w_n(p_1,\ldots,p_n) = {\frac{1}{n+1}} \sum_{p_1,\ldots, p_n \ge 0 \atop \sum j p_j=n} \binom{n+\sum p_j}{n,p_1, \ldots,p_n} y_1^{p_{1}} y_2^{p_{2}} \cdots y_n^{p_{n}},
\end{gather*}
where if $N=m_1+ \cdots + m_n$, then
\begin{gather*}
\binom{N}{m_1,\ldots,m_n} = \frac{N !}{m_1 ! m_2 ! \cdots m_n !}
\end{gather*}
denotes the multinomial coef\/f\/icient.

Therefore, the coef\/f\/icient
\begin{gather*}
b_n(p_1,\ldots,p_n):={\frac{1}{n+1}} \binom{n+\sum p_j}{n,p_1, \ldots,p_n}, \qquad \sum_{j} j p_j =n
\end{gather*}
is equal to the number of dissections of a convex $(n+2)$-gon which contain
exactly $p_j$ convex $(j+2)$-gons, see, e.g.,~\cite{EE}. Equivalently, the
number $b_n(p_1,\ldots,p_n)$ is equal to the number of cells of the
associahedron ${\cal K}^{n-1}$ which are isomorphic to the cartesian product
$({\cal{K}}^{0})^{p_1} \times \cdots \times ({\cal{K}}^{n-1})^{p_n}$~\cite{JL1, JL2}.
Based on a natural and well-known bijection between the set of dissections
of a convex $(n+2)$-gon and the set of plane trees with $(n+1)$ ends and such
that the all other vertices have degree at least~$2$, see, e.g.,~\cite{ST1},
one can readily seen that the number $w_n(p_1,\ldots,p_n)$ def\/ined above under constraint $\sum_{j} j p_j= n$, is equal to the number of plane trees with $n+1$ ends and having~$p_j$ vertices of degree~$j+1$.

\begin{Example} \label{example5.10}
For short we set $B_n= {\rm PL}_n({\boldsymbol{a}},{\boldsymbol{b}},{\boldsymbol{\beta}},{\boldsymbol{\alpha}}, {\boldsymbol{x}},{\boldsymbol{y}})$.

(1) Quadrangular:
\begin{gather*}
B_2= y_1^2(a_1 z_1+b_1 z_1 z_2) +y_2(\beta_1 z_1+ \alpha_1).
\end{gather*}

(2) Pentagon:
\begin{gather*}
B_3=y_1^3\big(a_1^2 z_1+a_1 b_1 z_1+a_2 b_1^2 z_1 z_2+
 a_1 b_1 z_1 z_3+b_1^2 b_2 z_1 z_2 z_3\big)\\
 \hphantom{B_3=}{}
 +y_1 y_2\big(2 a_1 \beta_1 z_1+b_1 \beta_1 z_1+b_1^2 \beta_2 z_1 z_2+b_1 \beta_1 z_1 z_3+a_1 \alpha_1 b_1 \alpha_1+\alpha_1 z_3\big) \\
 \hphantom{B_3=}{}
 +y_3\big(\beta_1 \alpha_1+ \beta_1^2 z_1+b_1^2 \alpha_2 z_1\big).
\end{gather*}

(3) Hexagon:
\begin{gather*}
B_4= {y_1^4}\big(\big(a_1^3+2 a_1^2 b_1+a_1 a_2 b_1^2+a_1 b_1^2 b_2\big)
z_1\\
\hphantom{B_4=}{}
+a_1^2 b_1 b_2 z_1 z_2+a_2 b_1^3 b_2 z_1 z_2+a_1 a_3 b_1^2 z_1 z_3+a_1^2 b_1 z_1 z_4+a_1 b_1^2 z_1 z_4 +a_3 b_1^3 b_2^2 z_1 z_2 z_3\\
\hphantom{B_4=}{}
+a_2 b_1^2 b_2 z_1 z_2 z_4
+a_1 b_1^2 b_3 z_1 z_3 z_4+b_1^3 b_2^2 b_3 z_1 z_2 z_3 z_4\big)
+
 {y_1^2 y_2}\big(a_1^2 \alpha_1+2 a_1 b_1 \alpha_1+a_2 b_1^2 \alpha_1\\
\hphantom{B_4=}{}
 +b_1^2 b_2 \alpha_1+(3 a_1^2 b \beta_1+4 a_1 b_1 \beta_1+a_2 b_1^2 \beta_1+b_1^2+b_2 \beta_1+a_1 b_1^2 \beta_2) z_1+a_2 b_1^2 \beta_2 z_1 z_2\\
 \hphantom{B_4=}{}
 +b_1^3 b_2 \beta_2 z_1 z_2+a_2 b_1^3 \beta_2 z_1 z_2+a_1 b_1^2 b_3 z_1 z_3+a_1 b_1 \beta_1 z_1 z_3+a_3 b_1^2 z_1 z_3+b_1^2 \beta_1 z_1 z_4\\
 \hphantom{B_4=}{}
 +a_1 b_1 \beta_1 z_1 z_4+ b_1^2 b_2 \beta_3 z_1 z_2 z_3+b_1^3 b_2 \beta_2 z_1 z_2 z_4+b_1^2 b_3 \beta_1 z_1 z_3 z_4\big)+
{y_1 y_3}\big(a_1 \beta_1 \alpha_1\\
 \hphantom{B_4=}{}
+2 b_1 \beta_1 \alpha_1+\big(2 a_1 \beta_1^2+2 b_1 \beta_1^2a_1 b_1^2 \alpha_3+a_2 b_1^2 \alpha_2+b_1^3 b_2 \alpha_2\big) z_1 +b_1^3 b_2 \alpha_3 z_1 z_2+b_1^3 \beta_2^2 z_1 z_2\\
 \hphantom{B_4=}{}
+b_1^3 \alpha_3 z_1 z_4+b_3 \alpha_1 z_3 z_3 z_4+a_3 \alpha_1 z_3+a_1 \alpha_1 z_4+b_1 \alpha_1 z_4+\beta_1 \alpha_1 z_4\big)+
{y_2^2}\big(a_1 \beta_1 \alpha_1\\
 \hphantom{B_4=}{}
+b_1^2 \beta_2 \alpha_2+\big(b_1^2\beta_1 \beta_2+a_1 \beta_1^2+a_1 \beta_1^2 \alpha_2\big) z_1 +\beta_3 \alpha_1 z_3+b_1 \beta_1 \beta_3 z_1 z_3\big)+
{y_4}(\alpha_1 \alpha_3+\beta_1^2 \alpha_1\\
 \hphantom{B_4=}{}
+b_1^2 \alpha_1 \alpha_2 \big(b_1^2 \beta_1 \alpha_2+b_1^3 \beta_2 \alpha_2+\beta_1^3+ b_1^2 \beta_1 \alpha_3\big) z_1\big).
\end{gather*}

\emph{Special cases.}
Generalized {\it Schr\"{o}der} or {\it Lagrange} polynomials:
\begin{gather*}
P_n({\boldsymbol{a}}, {\boldsymbol{b}}, {\boldsymbol{\beta}}, {\boldsymbol{y}},{\boldsymbol{z}})= B_n \big |_{{\boldsymbol{\alpha}}={\bf 0}}.
\end{gather*}
For example,
\begin{gather*}
P_4({\boldsymbol{a}},{\boldsymbol{b}},{\boldsymbol{y}})=
{y_1^4}\big(\big(a_1^3+2 a_1^2 b_1+a_1 a_2 b_1^2+a_1 b_1^2 b_2\big)
z_1+a_1^2 b_1 b_2 z_1 z_2+a_2 b_1^3 b_2 z_1 z_2+a_1 a_3 b_1^2 z_1 z_3\\
\hphantom{P_4({\boldsymbol{a}},{\boldsymbol{b}},{\boldsymbol{y}})= }{}
+a_1^2 b_1 z_1 z_4+a_1 b_1^2 z_1 z_4 +a_3 b_1^3 b_2^2 z_1 z_2 z_3+a_2 b_1^2 b_2 z_1 z_2 z_4+a_1 b_1^2 b_3 z_1 z_3 z_4\\
\hphantom{P_4({\boldsymbol{a}},{\boldsymbol{b}},{\boldsymbol{y}})= }{}
+b_1^3 b_2^2 b_3 z_1 z_2 z_3 z_4\big) +
 {y_1^2 y_2}\big(\big(3 a_1^2 b \beta_1+4 a_1 b_1 \beta_1+a_2 b_1^2 \beta_1+b_1^2+b_2 \beta_1\\
 \hphantom{P_4({\boldsymbol{a}},{\boldsymbol{b}},{\boldsymbol{y}})= }{}
 +a_1 b_1^2 \beta_2\big) z_1+a_2 b_1^2 \beta_2 z_1 z_2+b_1^3 b_2 \beta_2 z_1 z_2+a_2 b_1^3 \beta_2 z_1 z_2+a_1 b_1^2 b_3 z_1 z_3\\
\hphantom{P_4({\boldsymbol{a}},{\boldsymbol{b}},{\boldsymbol{y}})= }{}
 +a_1 b_1 \beta_1 z_1 z_3+a_3 b_1^2 z_1 z_3+b_1^2 \beta_1 z_1 z_4+a_1 b_1 \beta_1 z_1 z_4+ b_1^2 b_2 \beta_3 z_1 z_2 z_3\\
\hphantom{P_4({\boldsymbol{a}},{\boldsymbol{b}},{\boldsymbol{y}})= }{}
 +b_1^3 b_2 \beta_2 z_1 z_2 z_4+b_1^2 b_3 \beta_1 z_1 z_3 z_4\big) +
{y_1 y_3}\big(2 a_1 \beta_1^2+2 b_1 \beta_1^2 +b_1^3 \beta_2^2 z_1 z_2\\
\hphantom{P_4({\boldsymbol{a}},{\boldsymbol{b}},{\boldsymbol{y}})= }{}
+b_1 \beta_1^2 z_1 z_4\big)+
{y_2^2}\big(\big(b_1^2\beta_1 \beta_2+a_1 \beta_1^2\big) z_1 + b_1 \beta_1 \beta_3 z_1 z_3\big)+
{y_4} \beta_1^3 z_1.
\end{gather*}
After the specialization $a_i=b_i=\beta_i=z_i=1$, $i=1,2,3,4$, one will obtain
\begin{gather*}
P_4({\boldsymbol{a}}={\bf 1},{\boldsymbol{b}}= {\bf 1}, {\boldsymbol{\beta}} = {\bf 1}, {\boldsymbol{y}}, {\boldsymbol{z}} ={\bf 1}) = 14 y_1^4+21 y_1^2 y_2+6 y_1 y_3+3 y_2^2+y_4.
\end{gather*}

Generalized {\it Narayana} polynomials:
\begin{gather*}
P_n({\boldsymbol{a}}, {\boldsymbol{b}}, {\boldsymbol{y}},{\boldsymbol{z}})= B_n \big |_{{\boldsymbol{\alpha}}={\bf 0} \atop {\boldsymbol{\beta}}={\bf 0}},\\
P_n({\boldsymbol{a}}, {\boldsymbol{b}}, {\boldsymbol{y}},{\boldsymbol{z}})= {y_1^4}\big(\big(a_1^3+2 a_1^2 b_1+a_1 a_2 b_1^2+a_1 b_1^2 b_2\big)
z_1+a_1^2 b_1 b_2 z_1 z_2+a_2 b_1^3 b_2 z_1 z_2\\
\hphantom{P_n({\boldsymbol{a}}, {\boldsymbol{b}}, {\boldsymbol{y}},{\boldsymbol{z}})=}{}
+a_1 a_3 b_1^2 z_1 z_3+a_1^2 b_1 z_1 z_4+a_1 b_1^2 z_1 z_4 +a_3 b_1^3 b_2^2 z_1 z_2 z_3+a_2 b_1^2 b_2 z_1 z_2 z_4\\
\hphantom{P_n({\boldsymbol{a}}, {\boldsymbol{b}}, {\boldsymbol{y}},{\boldsymbol{z}})=}{}
+a_1 b_1^2 b_3 z_1 z_3 z_4+b_1^3 b_2^2 b_3 z_1 z_2 z_3 z_4\big).
\end{gather*}

Generalized {\it Motzkin--Schr\"{o}der} polynomials:
\begin{gather*}
{\rm MS}_n({\boldsymbol{a}}, {\boldsymbol{b}}, {\boldsymbol{y}},{\boldsymbol{z}})= B_n \big |_{{\boldsymbol{a}}={\bf 0}}.
\end{gather*}
For example,
\begin{gather*}
{\rm MS}_4 ({\boldsymbol{a}},{\boldsymbol{b}},{\boldsymbol{y}},{\boldsymbol{z}})= {y_1^1 y_2}
\big(a_1^2 \alpha_1 +2 a_1 b_1 \alpha_1+a_2 b_1^2 \alpha_1+b_1^2 b_2 \alpha_1\big) + {y_1 y_3} (a_1 \beta_1 \alpha_1+2 b_1 \beta_1 \alpha_1)\\
\hphantom{{\rm MS}_4 ({\boldsymbol{a}},{\boldsymbol{b}},{\boldsymbol{y}},{\boldsymbol{z}})=}{}
+ {y_2^2}\big(a_1 \beta_1 \alpha_1+b_1^2 \beta_2 \alpha_2\big) + {y_4}\big( \alpha_1 \alpha_3+b_1^2 \alpha_1 \alpha_2 +\beta_1^2 \alpha_1\big).
\end{gather*}

Generalized {\it Motzkin} polynomials:
\begin{gather*}
M_n({\boldsymbol{b}}, {\boldsymbol{y}},{\boldsymbol{z}})= B_n \big |_{{\boldsymbol{a}}={\bf 0} \atop {\boldsymbol{\beta}} = {\bf 0}}.
\end{gather*}
For example,
\begin{gather*}
M_4({\boldsymbol{b}},{\boldsymbol{y}},{\boldsymbol{z}})= {y_1^4} b_1^3 b_2^2 b_3 z_1 z_2 z_3 z_4
+ {y_1^2 y_2} b_1^2 b_2 \alpha_1+ {y_1 y_3}\big(b_1^3 b_2 \alpha_2+b_1\alpha_1 z_4+b_1^3 b_2 z_1 z_3\\
\hphantom{M_4({\boldsymbol{b}},{\boldsymbol{y}},{\boldsymbol{z}})=}{}
+b_1^3 \alpha_2 z_1 z_4+b_3 \alpha_1 z_3 z_4\big)
+{y_4}\big(\alpha_2 \alpha_3+b_1^2 \alpha_1 \alpha_2\big).
\end{gather*}

Generalized {\it Motzkin--Riordan} polynomials:
\begin{gather*}
{\rm MR}_n({\boldsymbol{a}},{\boldsymbol{b}},{\boldsymbol{\beta}},{\boldsymbol{\alpha}},{\boldsymbol{y}})= B_n \big |_{{\boldsymbol{z}}={\bf 0}}.
\end{gather*}

Generalized {\it Riordan} polynomials:
\begin{gather*}
{\rm RI}_n({\boldsymbol{b}},{\boldsymbol{\alpha}},{\boldsymbol{y}})= B_n \big |_{{\boldsymbol{z}}={\bf 0} {\top {\boldsymbol{a}}={\bf 0}}\atop{\boldsymbol{\beta}}={\bf 0}}.
\end{gather*}
For example,
\begin{gather*}
{\rm RI}_4({\boldsymbol{b}},{\boldsymbol{\alpha}},{\boldsymbol{y}})= {y_1^2 y_2} b_1^2 b_2 \alpha_1+{y_4} \big(\alpha_2 \alpha_3+b_1^2 \alpha_1 \alpha_2\big).
\end{gather*}
\end{Example}

Let us set $B_n(y_1,\ldots,y_n) = B_n({\boldsymbol{a}}={\bf 1},{\boldsymbol{b}}={\bf 1},
{\boldsymbol{\beta}} ={\bf 1},{\boldsymbol{y}}) $. Let $\beta$ be a~new parameter.
Show that
\begin{gather*}
B(1,\beta,\ldots,\beta^{n-1})= \mathfrak{G}_{1 \times w_{0}^{(n-1)}}^{(\beta)}(\underbrace{1,\ldots,1}_{n}),
\end{gather*}
where $\mathfrak{G}_{w}^{(\beta)}(X)$ denotes the $\beta$-Grothendieck polynomial corresponding to a permutation $w \in \mathbb{S}_n$. In particular,
\begin{gather*}
B_n(\underbrace{1,\ldots,1}_n) = {\rm Sch}_n,
\end{gather*}
where ${\rm Sch}_n$ denotes the $n$-th Schr\"{o}der number, that is the numbers of
paths from $(0, 0)$ to $(2n, 0)$, using only steps northeast $U=(1,1)$ or
 or $D=(1, - 1)$) or double $H= (2, 0)$, that never fall below the $x$-axis.

Assume that $n$ is devisible by an integer $d \ge 1$.
Show that if ${\boldsymbol{y}} = (y_j = \delta_{j+1,d})$, then
\begin{gather*}
B_n(0,\ldots,0,\underbrace{1}_{d-1},0,\ldots,0) = {\rm FC}_{n/d}^{(d+1)},
\end{gather*}
where ${\rm FC}_m^{p}$ denotes the Fuss--Catalan number, see, e.g., \cite{ST1}, and \cite[$A001764$]{SL} for a variety of combinatorial interpretations the Fuss--Catalan numbers ${\rm FC}_n^{(3)}$.

More generally, let $2 < d_1 < \cdots < d_k$ be a sequence of integers,
and set
\begin{gather*}
{\boldsymbol{y}}= ( \delta_{i+1,d_{j}},\, 1 \le j \le k).
\end{gather*}
Show that
the specialization $B_n({\boldsymbol{y}})$ counts the number of dissections of a convex $(n+2)$-gon on parts which are convex $(d+2)$-gons, where each~$d$
belongs to the set $\{d_1,\ldots,d_k \}$.
 We would like to point out that the polynomials
\begin{gather*}
{\rm FS}_n^{(d)}:= {\rm Coef\/f}_{y_{d}^n} \big(P_{nd}({\boldsymbol{a}},{\boldsymbol{b}}, {\boldsymbol{\beta}}, {\boldsymbol{y}}=(\delta_{i+1,d}), {\boldsymbol{z}}) \big).
\end{gather*}
can be treated as a multi-parameter analogue of the Fuss--Catalan numbers ${\rm FC}_n^{(d+1)}$.

{\it Colored dissections} \cite{SW}.
A~colored dissection of a convex polygon is a dissection where
each $(d+1)$-gon appearing in the dissection can be colored by one of~$b_d$
possible colors\footnote{We assume that if $b_{d}=0$, then the dissection in question doesn't contain parts which are $(d+1)$-gons.}, $d \ge 2$~\cite{SW}.
Show~\cite{SW} that if $b_2,\ldots,b_n$ be a sequence of non-negative integers, $B_n(b_2,\ldots,\ldots,b_n)$ is equal to the number of colored dissections of a convex $(n+2)$-gon.

Consider the specialization $y_i=i-1$, $i=1,\ldots,n$. Show that
\begin{gather*}
B_n(y):= {\rm SL}(0,1,\ldots,n-1)={\rm Fine}(n+1),
\end{gather*}
where ${\rm Fine}(m)$ denotes the $m$-th {\it Fine number}, that is the number of
ordered rooted trees with $m$ edges having root of even degree~\cite[$A000957$]{SL}.
Therefore, the Fine number $Fine(n+1)$ counts the number of dissections of a convex $(n+2)$-gon such that each $(d+3)$-gon appearing in the dissection can be colored by $d$ possible colors, $d \ge 1$.

Consider the specialization $y_{3k+1}=1$, $y_{3k+2}=0$, $y_{3k+3}=-1$, $k \ge 0$.
Show that
\begin{gather*}
B_n(y_1, \ldots, y_n)= M_n,
\end{gather*}
where $M_n$ denotes the $n$-th {\it Motzkin} number \cite[$A001006$]{SL}.

Recall that it is the number of ways to draw any number of nonintersecting chord joining~$n$ labeled points on a circle. The number $M_n$ is also equals to the number of {\it Motzkin} paths, that is paths from $(0,0)$ to $(n,n)$ in the $n \times n$ grid using only steps $U=(1,1)$, $H=(1,0)$ and $D=(1,-1)$, see \cite[$A001006$]{SL}
for references and a wide variety of combinatorial interpretations
of Motzkin's numbers.

 Consider the specialization $y_{3k+1}=0$, $y_{3k+2}= (-1)^{k}$,
$y_{3k+3} =(-1)^{k}$, $k \ge 0$. Show that
\begin{gather*}
B_n(y_1, \ldots, y_n)= {\rm MS}_n,
\end{gather*}
where ${\rm MS}_n$ denotes the {\it Motzkin sum} or {\it Riordan} number \cite[$A005043$]{SL}.

Recall that it is the number of Motzkin paths of length $n$ with no horizontal steps $H=(1,0)$ at level zero, see~\cite[$A005043$]{SL} for references and a wide variety of combinatorial interpretations of Riordan's numbers.

Consider the specialization $y_{2k+1}= (-1)^k$, $y_{2 k}= (-1)^{k+1}$,
 $k \ge 0$. Show that~\cite{SL}
\begin{gather*}
B_n(y_1,\ldots,y_n) = A052709(n),
\end{gather*}
that is the number of underdiagonal lattice paths from $(0,0)$ to $(n-1,n-1)$
 and such that each step is either $H=(1,0)$, $V=(0,1)$, or $D=(2,1)$.

Consider specialization $y_k= (-1)^{k} \frac{n !}{k !}$, $k \ge 1$.
Show that
\begin{gather*}
B_n(y_1, \ldots,y_n) = n^{n-2},
\end{gather*}
that is the number of {\it parking functions}, see, e.g., \cite{H, ST1}
and the literature quoted therein.

Consider the specialization $y_k = \frac{n !}{k !}$.
Show that~\cite{SL}
\begin{gather*}
 B_n(y_1,\ldots,y_n)= A052894(n),
 \end{gather*}
 where $A052894(n)$ denotes the number of {\it Schr\"{o}der trees}\footnote{Schr\"{o}der trees have been introduced in a~paper by W.Y.C.~Chen~\cite{Chen1990}. Namely, these are trees for which the set of subtrees at
any vertex is endowed with the structure of ordered partition. Recall that
{\it an ordered partition} of a~set in which the blocks are linearly ordered~\cite{Chen1990}.}.

\appendix

\section{Appendixes}

\subsection{Grothendieck polynomials}\label{appendixA.1}

\begin{Definition}\label{definition6.1}
 Let $\beta$ be a parameter. The Id-Coxeter algebra
${\rm IdC}_n(\beta)$ is an associative algebra
over the ring of polynomials $\Z[\beta]$ generated by elements
$\langle e_1,\ldots,e_{n-1} \rangle$ subject to the set of relations
\begin{itemize}\itemsep=0pt
\item$e_i e_j =e_j e_i$ if $| i-j | \ge 2$,
\item $e_i e_j e_i =e_j e_i e_j$ if $| i-j | =1$,
\item $e_i^{2} = \beta e_i$ $1 \le i \le {n-1}$.
\end{itemize}
\end{Definition}

It is well-known that the elements $\{ e_w,\, w \in \mathbb{S}_n \}$ form a
$\Z[\beta]$-linear {\it basis} of the algebra ${\rm IdC}_n(\beta)$.
Here for a permutation $w \in \mathbb{S}_n$ we denoted by~$e_w$
the product $e_{i_{1}} e_{i_{2}} \cdots e_{i_{\ell}} \in {\rm IdC}_n(\beta)$, where
$(i_{1},i_{2},\ldots, i_{\ell})$ is any {\it reduced} word for a~permutation~$w$, i.e., $w= s_{i_{1}} s_{i_{2}} \cdots s_{i_{\ell}}$ and $\ell = \ell(w)$
is the length of~$w$.

Let $x_1,x_2,\ldots,x_{n-1},x_n=y,x_{n+1}=z,\ldots$ be a set of mutually
commuting variables. We assume that~$x_i$ and~$e_j$ commute for all values of~$i$ and~$j$. Let us def\/ine
\begin{gather*}
h_i(x) =1+x e_i, \qquad A_{i}(x)= \prod_{a=n-1}^{i} h_{a}(x),\qquad
i=1,\ldots,n-1.
\end{gather*}

\begin{Lemma}\label{lemma6.1} One has
\begin{enumerate}\itemsep=0pt
\item[$(1)$] addition formula:
\begin{gather*}
h_i(x) h_i(y) =h_i(x \oplus y),
\end{gather*}
 where we set $(x \oplus y) := x+y+\beta xy$;
\item[$(2)$] Yang--Baxter relation:
\begin{gather*}
h_i(x) h_{i+1}(x \oplus y) h_i(y) = h_{i+1}(y) h_i(x \oplus y) h_{i+1}(x).
\end{gather*}
\end{enumerate}
\end{Lemma}
\begin{Corollary} \label{corollary6.1}\quad
\begin{enumerate}\itemsep=0pt
\item[$(1)$] $[ h_{i+1}(x)h_{i}(x),h_{i+1}(y) h_i(y)]=0$.
\item[$(2)$] $[A_{i}(x),A_{i}(y)]=0$, $i=1,2,\ldots,n-1$.
\end{enumerate}
\end{Corollary}

The second equality follows from the f\/irst one by induction using the addition
formula, whereas the f\/ist equality follows directly from the Yang--Baxter
relation.
\begin{Definition}[Grothendieck expression]\label{definition6.2}\quad
\begin{gather*}
 \mathfrak{G}_{n}(x_1,\ldots,x_{n-1}) := A_1(x_1) A_{2}(x_2) \cdots A_{n-1}(x_{n-1}).
 \end{gather*}
\end{Definition}

\begin{Theorem}[\cite{FK1}]\label{theorem6.1}
 The following identity
\begin{gather*}
 \mathfrak{G}_{n}(x_1,\ldots,x_{n-1}) = \sum_{w \in \mathbb{S}_n}
\mathfrak{G}_{w}^{(\beta)}(X_{n-1}) e_{w}
\end{gather*}
holds in the algebra ${\rm IdC}_n \otimes \Z[x_1,\ldots,x_{n-1}]$.
\end{Theorem}
\begin{Definition} \label{definition6.3}
We will call polynomial
$\mathfrak{G}_{w}^{(\beta)}(X_{n-1})$ as the {\it $\beta$-Grothendieck}
 polynomial corresponding to a permutation~$w$.
\end{Definition}

\begin{Corollary}\label{corollary6.2}\quad
\begin{enumerate}\itemsep=0pt
\item[$(1)$] If $\beta=-1$, the polynomials $\mathfrak{G}_{w}^{(-1)}(X_{n-1})$
coincide with the Grothendieck polynomials introduced by Lascoux and
M.-P.~Sch\"utzenberger~{\rm \cite{LS}}.

\item[$(2)$] The $\beta$-Grothendieck polynomial $\mathfrak{G}_{w}^{(\beta)}(X_{n-1})$ is divisible by $x_1^{w(1)-1}$.

\item[$(3)$] For any integer $k \in [1,n-1]$ the polynomial
$\mathfrak{G}_{w}^{(\beta-1)}(x_k=q,\,x_a=1,\, \forall \,a \not= k)$ is a~po\-ly\-nomial in the variables~$q$ and~$\beta$ with non-negative
integer coefficients.
\end{enumerate}
\end{Corollary}

\begin{proof}[Sketch of proof] It is enough to show that the specialized
Grothendieck expression $\mathfrak{G}_{n}(x_k=q$, $x_a =1, \,\forall\, a \not= k)$
can be written in the algebra $IdC_n(\beta-1) \otimes \Z[q,\beta]$ as a~linear combination of elements $\{e_w \}_{w \in \mathbb{S}_n} $ with
coef\/f\/icients which are polynomials in the variables~$q$ and $\beta$ with
 non-negative coef\/f\/icients. Observe that one can rewrite the
relation $e_k^2= (\beta-1) e_k$ in the following form
$e_k (e_k +1)=\beta e_k$. Now, all possible negative contributions to the
expression $\mathfrak{G}_{n}(x_k=q$, $x_a =1, \,\forall\, a \not= k)$ can appear
 only from products of a form $c_{a}(q):= (1+q e_k) (1+e_k)^{a}$. But using
the Addition formula one can see that $(1+q e_k)(1+e_k)= 1+(1 +q \beta)e_k$.
It follows by induction on~$a$ that~$c_{a}(q)$ is a polynomial in the
variables~$q$ and~$\beta$ with {\it non-negative} coef\/f\/icients.
\end{proof}

\begin{Definition} \label{definition6.4}\quad
\begin{itemize}\itemsep=0pt
\item The double $\beta$-Grothendieck expression
$\mathfrak{G}_n(X_n,Y_n)$ is def\/ined as follows
\begin{gather*}
\mathfrak{G}_n(X_n,Y_n) = \mathfrak{G}_n(X_n) \mathfrak{G}_n(-Y_n)^{-1}
\in {\rm IdC}_n(\beta) \otimes \Z[X_n,Y_n].
\end{gather*}
\item The double $\beta$-Grothendieck polynomials
$\{ \mathfrak{G}_{w}(X_n,Y_n) \}_{w \in \mathbb{S}_n}$ are def\/ined from the
decomposition
\begin{gather*}
 \mathfrak{G}_n(X_n,Y_n)= \sum_{w \in \mathbb{S}_n} \mathfrak{G}_w(X_n,Y_n) e_{w}
 \end{gather*}
of the double $\beta$-Grothendieck expression in the algebra ${\rm IdC}_n(\beta)$.
\end{itemize}
\end{Definition}

More details about $\beta$-Grothendieck and related polynomials can be found
in~\cite{K5,L}.

\subsection{Cohomology of partial f\/lag varieties}\label{appendixA.2}

Let $n=n_1+\cdots + n_k$, $n_i \in \Z_{\ge 1} \forall i$, be a composition of
$n$, $k \ge 2$. For each $j=1,\ldots,k$ def\/ine the numbers $N_j=n_1+\cdots+n_j,
N_0=0$, and $ M_j=n_j+\cdots+n_k$. Denote by
${\boldsymbol{X}}:= {\boldsymbol{X}}_{n_1,\ldots,n_k}=\{ x_{a}^{(i)}\, |\, i=1,\ldots,k,\,1 \le a \le n_i \}$
(resp.\ ${\boldsymbol{Y}}$, \dots) a set of variables of
the cardinality $n$. We set $\deg(x_a^{(i)})=a$, $i=1,\ldots,k$.
For each $i=1,\ldots,k$ def\/ine quasihomogeneous polynomial of degree~$n_i$ in
variables ${\boldsymbol{X}}^{(i)}=\big\{x_a^{(i)} \,|\, 1 \le a \le n_i \big\}$
\begin{gather*}
p_{n_i}\big({\boldsymbol{X}}^{(i)},t\big)=t^{n_{i}}+ \sum_{a=1}^{n_{i}} x_{a}^{(i)} t^{n_{i}-
a},
\end{gather*}
and put
\begin{gather*}
p_{n_1,\ldots,n_k}({\boldsymbol{X}},t)= \prod_{i=1}^{k}p_{n_i}({\boldsymbol{X}}^{(i)},t).
\end{gather*}
We summarize in the theorem below some well-known results about the
classical and quantum cohomology and $K$-theory rings of type~$A_{n-1}$
partial f\/lag varieties ${\cal F}l_{n_1,\ldots,n_k}$. Let $q_1,\ldots,q_{k-1}$, $\deg(q_i)=n_i+n_{i+1}$, $i=1,
\ldots,k-1$, be a set of ``quantum parameters''.

\begin{Theorem} \label{theorem6.2}
There are canonical isomorphisms
\begin{gather*}
 H^{*}({\cal F}l_{n_1,\ldots,n_k},\Z) \cong \Z[{\boldsymbol{X}}_{n_1,\ldots,n_k}] \big/
 \big\langle p_{n_1,\ldots,n_k}({\boldsymbol{X}},t)-t^{n} \big\rangle ,\\
 K^{\bullet}({\cal F}l_{n_1,\ldots,n_k},\Z) \cong \Z[{\boldsymbol{Y}}^{\pm 1}] \big/
\big\langle p_{n_1,\ldots,n_k}({\boldsymbol{Y}},t)-(1+t)^{n} \big\rangle ,\\
 H^{*}_{T}({\cal F}l_{n_1,\ldots,n_k},\Z) \cong \Z[{\boldsymbol{X}},{\boldsymbol{Y}}] \Big/
\left\langle \prod_{i=1}^{k} \prod_{a=1}^{n_i} (x_a^{(i)}+t)-p_{n_1,\ldots,n_k}({\boldsymbol{Y}},t) \right\rangle,\\
 QH^{*}({\cal F}l_{n_1,\ldots,n_k}) \cong
\Z[ {\boldsymbol{X}}_{n_1,\ldots,n_k},q_1,\ldots,q_{k-1}]\big/\big\langle \Delta_{n_1,\ldots,n_k}({\boldsymbol{X}},t)-t^n
\big\rangle\quad \text{\rm (cf.~\cite{ASS})}, \\
QH_{T}^{*}({\cal F}l_{n_1,\ldots,n_k}) \cong
\Z[{\boldsymbol{X}},{\boldsymbol{Y}},q_1,\ldots,q_{k-1}] \big/
\big\langle \Delta_{n_1,\ldots,n_k}({\boldsymbol{X}},t)- p_{n_1,\ldots,n_k}({\boldsymbol{Y}},t)
\big\rangle \quad \text{\rm (cf.~\cite{ASS})},
\end{gather*}
where\footnote{We prefer to use quantum parameters $\{q_i \,|\, 1 \le i \le
k-1 \}$ instead of the parameters $\{ (-1)^{n_i}q_i \,|\, 1 \le i \le k-1 \}$ have been used in~\cite{ASS}.}
\begin{gather*}
\Delta_{n_1,\ldots,n_k}({\boldsymbol{X}},t)=\\
\det
\begin{vmatrix}
p_{n_1}({\boldsymbol{X}}^{(1)},t) & q_1 & 0 & \cdots & \cdots & \cdots & 0 \\
-1 & p_{n_2}({\boldsymbol{X}}^{(2)},t) & q_2 & 0 & \cdots & \cdots & 0 \\
0 &-1 & p_{n_3}({\boldsymbol{X}}^{(3)},t) & q_3 & 0 & \cdots & 0 \\
\vdots & \ddots & \ddots & \ddots & \ddots & \ddots & \vdots \\
0 & \cdots & \cdots & 0 & -1 & p_{n_{k-1}}({\boldsymbol{X}}^{(k-1)},t) & q_{k-1} \\
0 & \cdots & \cdots & \cdots & 0 & -1 & p_{n_{k}}({\boldsymbol{X}}^{(k)},t)
\end{vmatrix}
.
\end{gather*}
\end{Theorem}

Here for any polynomial $P({\boldsymbol{x}},t)=\sum\limits_{j=0}^{r}b_j({\boldsymbol{x}})t^{r-j}$ in
variables ${\boldsymbol{x}}=
(x_1,x_2,\ldots)$, we
denote by $\langle P({\boldsymbol{x}},t) \rangle $ the ideal in the ring
$\Z[ {\boldsymbol{x}}]$ generated by the coef\/f\/icients $b_0({\boldsymbol{x}}),\ldots,
b_r({\boldsymbol{x}})$. A similar meaning have the symbols
\begin{gather*}
\left\langle \prod_{i=1}^{k} \prod_{a=1}^{n_i}(x_a^{(i)}+t)-p_{n_1,\ldots,n_k}({\boldsymbol{y}},t) \right\rangle,\qquad
\big\langle \Delta_{n_1,\ldots,n_k}({\boldsymbol{x}},t) -t^n \big\rangle
\end{gather*} and so on.

Note that $\dim({\cal F}_{n_1,\dots,n_k})=\sum\limits_{i <j} n_i n_j$ and the Hilbert
polynomial ${\rm Hilb}({\cal F}_{n_1,\dots,n_k},q)$ of the partial f\/lag variety
${\cal F}_{n_1,\ldots,n_k}$ is equal to the $q$-multinomial
 coef\/f\/icient
$ {n \brack n_1,\ldots,n_k}_{q}$, and also is equal to the
$q$-dimension of the weight $(n_1,\ldots,n_k)$ subspace of
 the $n$-th tensor power $(\C^{n})^{\otimes n}$ of the fundamental
representation of the Lie algebra ${\mathfrak{gl}}(n)$.

\begin{Comments} \label{comments6.1}
The cohomology and (small) quantum cohomology rings
$H^{*}({\cal F}_{n_1,\dots,n_k},\Z)$ and $QH^{*}({\cal F}_{n_1,\dots,n_k},\Z)$,
 of the partial f\/lag variety ${\cal F}_{n_1,\dots,n_k}$ admit yet another
representations we are going to present. To start with, let as before
$n=n_1+ \cdots +n_k$, $n_i \in \Z_{\ge 1}$, $\forall\, i$, be a composition.
Consider the set of variables
$\widehat{\boldsymbol{X}}=X_{n_1,\ldots,n_{k-1}}:= \{ x_{a}^{(i)} \,|\, 1 \le i \le n_a, \, a=1,\ldots,k-1 \},
$ and set as before $\deg x_{a}^{(i)}=a$. Note that the number of variables
$\widehat{\boldsymbol{X}}$ is equal to $n-n_k$. To
continue, let's def\/ine {\it elementary quasihomogeneous polynomials of degree}~$r$
\begin{gather*}
e_{r}\big(\widehat{\boldsymbol{X}}\big)=\sum_{I,A} x_{a_1}^{(i_1)} \cdots x_{a_s}^{(i_s)}, \qquad e_{0}\big(\widehat{\boldsymbol{X}}\big)=1,
\qquad
e_{-r}\big(\widehat{\boldsymbol{X}}\big)=0, \qquad r > 0,
\end{gather*}
where the sum runs over sequences of integers $I=(i_1,\ldots,i_s)$ and
$A=(a_1,\ldots,a_s)$ such that
\begin{itemize}\itemsep=0pt
\item $1 \le i_1 < \cdots < i_s \le k-1$,
\item $1 \le a_j \le n_{i_j}$, $j=1,\ldots, s$, and $r=a_1+\cdots+a_s$,
\end{itemize}
and {\it complete homogeneous polynomials} of degree $p$
\begin{gather*}
h_p\big(\widehat{\boldsymbol{X}}\big)= \det \big| e_{j-i+1}\big(\widehat{\boldsymbol{X}}\big)\big|_{1 \le i,j \le p}.
\end{gather*}
Finally, let's def\/ine the ideal $J_{n_1,\ldots,n_k}$ in the ring of polynomials
$\Z[X_{n_1,\ldots,n_{k-1}}]$ generated by polynomials
\begin{gather*}
h_{n_k+1}\big(\widehat{\boldsymbol{X}}\big), \ \ldots, \ h_n\big(\widehat{\boldsymbol{X}}\big).
\end{gather*}
 Note that the ideal $J_{n_1,\ldots,n_k}$ is generated by $n-n_k=
\#(X_{n_1, \ldots,n_{k-1}})$ elements.
\end{Comments}

\begin{Proposition}\label{proposition6.1}
 There exists an isomorphism of rings
\begin{gather*}
H^{*}({\cal F}_{n_1,\dots,n_k},\Z) \cong \Z[X_{n_1,\ldots,n_{k-1}}] / J_{n_1, \ldots, n_k}.
\end{gather*}
\end{Proposition}

In a similar way one can describe relations in the (small) quantum cohomology ring of
the partial f\/lag variety ${\cal F}_{n_1,\dots,n_k}$. To accomplish this let's
introduce {\it quantum quasihomogeneous elementary polynomials} of degree~$j$
$e_{j}^{({\boldsymbol{q}})}({\boldsymbol{X}}_{n_1,\ldots,n_r})$ through the decomposition
\begin{gather*}
 \Delta_{n_1,\ldots,n_r}({\boldsymbol{X}}_{n_1,\ldots,n_r}) = \sum_{j=0}^{N_r}
e_{j}^{({\boldsymbol{q}})}({\boldsymbol{X}}_{n_1,\ldots,n_r}) t^{N_r-j}, \!\qquad e_{0}^{({\boldsymbol{q}})}({\boldsymbol{x}})=1,\!\qquad
e_{-p}^{({\boldsymbol{q}})}({\boldsymbol{x}})=0, \!\qquad p > 0.
\end{gather*}
To exclude redundant variables $\big\{x^{(k)}_{a}, \, 1 \le a \le n_k \big\}$, let us def\/ine {\it quantum quasihomogeneous Schur polynomials}
$s_{\alpha}^{({\boldsymbol{q}})}({\boldsymbol{X}}_{n_1,\ldots,n_r})$ corresponding to a composition $\alpha =
(\alpha_1 \le \alpha_2 \le \cdots \le \alpha_p)$ as follows
\begin{gather*}
 s_{\alpha}^{({\boldsymbol{q}})}({\boldsymbol{X}}_{n_1,\ldots,n_r})=\det
 \big|e_{j-i+ \alpha_i}^{({\boldsymbol{q}})} ({\boldsymbol{X}}_{n_1,\ldots,n_r}) \big|_{1 \le i,j \le p}.
 \end{gather*}

\begin{Proposition}\label{proposition6.2}
The $($small$)$ quantum cohomology ring $QH^{*}({\cal F}_{n_1,\dots,n_k}, \Z)$
 is isomorphic to the quotient of the ring of polynomials $\Z[q_1,\ldots,q_{k-1}]~[ {\bf X}_{n_1, \ldots,n_{k-1}}]$ by the ideal $I_{n_1,\ldots,n_{k-1}}$ gene\-ra\-ted by the elements
\begin{gather*}
 g_r({\boldsymbol{X}}_{n_1,\ldots,n_{k-1}}):=s_{(1^{n_k}, r)}^{(q_1,\ldots,q_{k-1})}({\boldsymbol{X}}_{n_1,\ldots,n_{k-1}}) -
q_{k-1} e_{r-n_{k-1}}^{(q_1,\ldots,q_{k-2})}({\boldsymbol{X}}_{n_1,\ldots,n_{k-2}}),
\end{gather*}
where $n_{k}+1 \le r \le n$.
\end{Proposition}

It is easy to see that the Jacobi matrix
\begin{gather*}
\left( {\partial \over \partial x^{(i)}_{a}} g_{r}({\boldsymbol{X}}_{n_1,\ldots,n_{k-1}})
 \right)_{\{a=1,\ldots,k-1,\, 1 \le i \le n_{a} \atop n_{k}+1 \le r \le n \}}
\end{gather*}
corresponding to the set of polynomials $g_{r}({\boldsymbol{X}}_{n_1,\ldots,n_{k-1}})$,
$n_k \le r \le n$, has nonzero determinant, and the component of maximal degree $n_{\max}:=\sum\limits_{l < j} n_{i} n_{j}$
in the ring $QH^{*}({\cal F}_{n_1,\dots,n_k}, \Z)$ is a~$\Z[q_1,\ldots,q_{k-1}]$-module of rank one with generator
\begin{gather*}
\Lambda= \prod_{i=1}^{k-1} \prod_{a=1}^{n_a} \big(x^{(i)}_a \big)^{M_i}.
\end{gather*}
 Therefore, one can def\/ine a scalar product (the Grothendieck residue)
\begin{gather*}
\langle \bullet,\bullet \rangle \colon \ HQ^{*}({\cal F}_{n_1,\dots,n_k}, \Z) \times HQ^{*}({\cal F}_{n_1,\dots,n_k}, \Z)
\longrightarrow \Z[q_1, \ldots, q_{k-1}]
\end{gather*}
setting for elements $f$ and $g$ of degrees~$a$ and~$b$, $\langle f,h \rangle = 0$, if
$a+b \not= n_{\max}$, and $\langle f,h \rangle = \lambda(q)$, if~$a+b=n_{\max}$ and $f h=
\lambda(q) \Lambda$. It is well known that the Grothendieck pairing
$\langle \bullet,\bullet \rangle$ is nondegenerate (for any choice of parameters $q_1,\ldots,q_{k-1})$.

Finally we state ``a mirror presentation'' of the small quantum cohomology ring of
 partial f\/lag varieties. To start with, let $ n=n_1+\cdots + n_k$,
$k \in \Z_{ge 2}$ be a~composition of size $n$, and consider the set
\begin{gather*}
 \Sigma({\boldsymbol{n}})= \big\{(i,j) \in \Z \times \Z \,|\, 1 \le i \le N_a,\,
M_{a+1}+1 \le j \le M_{a}, \,a=1,\ldots, k-1 \},
\end{gather*}
where $N_a=n_1+\cdots+n_a$, $N_0=0,N_k=n$ $M_a=n_{a+1}+\cdots+n_k$, $M_0=n,M_{k}=0$.

With these data given, let us introduce the set of variables
\begin{gather*}
Z_{\boldsymbol{n}}= \{z_{i,j} \,|\, (i,j) \in \Sigma({\boldsymbol{n}}) \},
\end{gather*}
and def\/ine ``boundary conditions'' as follows
\begin{itemize}\itemsep=0pt
\item$z_{i,M_a+1}=0$, if $N_{a-1}+2 \le i \le N_a$, $a=1,\ldots,k-1$,

\item $z_{N_a+1,j}=\infty$, if $M_{a+1}+2 \le j \le M_a$, $a=1,\ldots,k-1$,

\item $z_{N_{a-1}+1,M_{a}+1}=q_{a}$, $a=1,\ldots,k$,
where $q_1,\ldots,q_k$ are ``quantum parameters.
\end{itemize}

Now we are ready, follow \cite{GiN}, to def\/ine {\it superpotential}
\begin{gather*}
W_{q,{\boldsymbol{n}}}=\sum_{(p,j) \in \Sigma({\boldsymbol{n}})} \left( {z_{i,j+1} \over z_{i,j}}+{z_{i,j} \over z_{i+1,j}} \right).
\end{gather*}

\begin{Conjecture}[cf.~\cite{GiN}]\label{conjecture6.1}
 There exists an isomorphism of rings
\begin{gather*}
 QH^{*}_{[2]}({\cal F}l_{n_1,\ldots,n_k},\Z) \cong \Z\big[q_1^{\pm 1},\ldots,q_{k}^{\pm 1}\big]\big[Z_{\boldsymbol{n}}^{\pm 1}\big] / J(W_{q,{\boldsymbol{n}}}),
 \end{gather*}
where $QH^{*}_{[2]}({\cal F}l_{n_1,\ldots,n_k},\Z)$ denotes the subring of the
ring~$QH^{*}({\cal F}l_{n_1,\ldots,n_k},\Z)$ generated by the elements
from $H^{2}({\cal F}l_{n_1,\ldots,n_k},\Z)$.

$J(W_{q,{\boldsymbol{n}}})$ stands for the ideal generated by the partial derivatives of the superpotential $W_{q,{\boldsymbol{n}}}$:
\begin{gather*}
J(W_{q,{\boldsymbol{n}}})= \left\langle {\partial W_{q} \over \partial z_{i,j}} \right\rangle, \qquad (i,j) \in \Sigma({\boldsymbol{n}}) .
\end{gather*}
\end{Conjecture}

 Note that variables $\{z_{i,j} \in \Sigma({\boldsymbol{n}}),\, i \not= N_{a}+1, \, a=0,\ldots,k-2\}$ are redundant, whereas the variables $\{z_{a,j}:=z_{N_{a}+1,j}^{-1}, \, j=1,\ldots,n_a,\, a=0,\ldots,k-2 \}$ satisfy the system of algebraic equations.

In the case of complete f\/lag variety ${\cal F}l_n$ corresponds to partition
${\boldsymbol{n}}=(1^n)$ and the superpotential $W_{q,1^n}$ is equal to
\begin{gather*}
 W_{q,1^n} = \sum_{1 \le i < j \le n-1} \left( {{z_{i,j+1} \over z_{i,j}}}+{{
z_{i,j} \over z_{i-1,j+1}}} \right),
\end{gather*}
where we set $z_{i,n}:=q_i$, $i=1,\dots,n$. The ideal $J(W_{q,1^n})$ is
generated by elements
\begin{gather*}
{{\partial W_{q,1^n} \over z_{i,j}}} ={{{1 \over z_{i,j-1}}}}+{{1 \over z_{i-1,j+1}}} -{{z_{i,j+1}+z_{i-1, j-1} \over z_{i,j}^2}}.
\end{gather*}

One can check that the ideal $J(W_{q,1^n})$ can be also generated by elements of
the form
\begin{gather*}
\sum_{j=0}^{i} A_j^{(i)}(q_1,\ldots,q_{n-i+1},z_{n-1},\ldots,z_{n-i+1}) z_{n-i}^{j-i-1}=1, \qquad A^{(i)}_{0}=q_1 \cdots q_{n-i+1},
\end{gather*}
where $z_i:=z_{1,i}^{-1}$, $i=1,\ldots n-1$. For example,
\begin{gather*}
z_{1}^n q_1\cdots q_{n}=1, \qquad q_1 q_2 z_{n-1}^{2}-q_2 z_{n-2} =1,\\
 q_1 q_2 q_3 z_{n-2}^{3}-2 q_1 q_2 q_3 z_{n-1} z_{n-2} z_{n-3}+q_2 q_3 z_{n-3}^2+q_3 z_{n-4}=1.
 \end{gather*}
Therefore the number of critical points of the superpotential~$W_q$ is equal to
$n!=\dim H^{*}({\cal F}l_n,\Z)$, as it should be. Note also that
$QH^{*}({\cal F}l_n,\Z) = QH^{*}_{[2]}({\cal F}l_n,\Z)$.

\subsection{Multiparamater 3-term relations algebras}\label{appendixA.3}

\subsubsection{Equivariant multiparameter 3-term relations algebras}\label{appendixA.3.1}

Let ${\boldsymbol{q}}= \{q_{ij} \}_{1 \le i \not= j \le n}$, $q_{ij}=q_{ji}$,
 be a collection of mutually commuting parameters and
${\boldsymbol{\beta}} = \{ \beta_{ij} \}_{1 \le i \not= j \le n}$, $\beta_{ij}=
 \beta_{ji}$ and ${\boldsymbol{\ell}}= \{ \ell_{ij} \}_{1 \le i \not= j \le n}$,
$ \ell_{ij}= \ell_{ji}$, be two sets of mutually commuting variables each.

\begin{Definition}\label{definition6.5}
 Denote by $3QT_n({\boldsymbol{\beta}}, {\boldsymbol{\ell}}, {\boldsymbol{q}})$ an associative algebra
generated over the ring
$\Z [{\boldsymbol{\beta}}, {\boldsymbol{\ell}}, {\boldsymbol{q}}]$ by the set of
generators $\{x_1,\ldots,x_n \}$ and that
$\{u_{ij}\}_{1 \le i \not= j \le n}$ subject to the set of relations
\begin{enumerate}\itemsep=0pt
\item[(1)] locality conditions: $[x_i,x_j]=0$, $[u_{ij},u_{kl}]=0$, $[x_k,u_{ij}]=0$ if~$i$, $j$, $k$, $l$
are pairwise distinct,
\item[(2)] generalized unitarity conditions: $u_{ij}+u_{ji}=\beta_{ij}$,
\item[(3)] Hecke type conditions: $u_{ij} u_{ji}= - q_{ij}$ if~$i \not= j$,
\item[(4)] twisted $3$-term relations: $u_{ij} u_{jk} =u_{jk} u_{ik} - u_{ik} u_{ji}, u_{jk} u_{ij}=u_{ik} u_{jk}-
u_{ji} u_{ik}$ if $i$, $j$, $k$ are distinct,
\item[(5)] crossing relations: $x_i u_{ji} = - u_{ij} x_j - \ell_{ij}$ if $i \not= j$.
\end{enumerate}
\end{Definition}

As before we def\/ine the (additive) Dunkl elements to be
\begin{gather}\label{equation6.1}
\theta_i= x_i+\sum_{j \not=i} u_{ij},\qquad i=1,\ldots,n.
\end{gather}
It should be pointed out that the Dunkl elements do not commute with
variables $\{x_{i}\}$, $\{\beta_{ij}\}$ and $\{\ell_{ij} \}$.

It is clearly seen from the def\/ining relations listed in Def\/inition~\ref{definition6.5} that
for any triple of distinct indices $(i,j,k)$ the elements $\{x_i,x_j,x_k,u_{ji},u_{ik},u_{jk}\}$ satisfy the twisted dynamical Yang--Baxter relations, and
thus the Dunkl elements $\{\theta_i\}_{1 \le i \le n}$ generate a commutative
subalgebra in the algebra $3QT_n({\boldsymbol{\beta}},{\boldsymbol{\ell}})$.

On the other hand, one can show that the set of def\/ining relations involve in
 the def\/inition of algebra $3QT_n({\boldsymbol{\beta}},{\boldsymbol{\ell}})$ implies the following
set of compatibility relations among the set of generators $\{ u_{ij} \}$ and the set of variables $\{ \beta_{ij}\}$ and $\{ \ell_{ij} \}$
\begin{gather*}
 \ell_{ij} u_{jk}+u_{ij} \ell_{jk}+\beta_{ij} u_{jk} x_i +u_{ij} \beta_{jk} x_j
 =
u_{jk} \ell_{ik}+\ell_{ik} u_{ij}+u_{jk} \beta_{ik} x_i+\beta_{ij} x_i u_{ij},
\end{gather*}
 if $i$, $j$, $k$ are distinct.

These relations are satisf\/ied, for example, if either $\beta_{ij}= \beta$, and $\ell_{ij} = h$, $\forall\, i,j$ for some parameters (i.e., a central elements)
$\beta$ and $h$, or~variables $\{\beta_{ij}\}$ and $\{ \ell_{ij} \}$ satisfy
the exchange relations with generators $\{ u_{ij} \}$, namely, the commutativity relations
\begin{gather*}
[\beta_{ij},u_{km}] = 0, \qquad [\ell_{ij},u_{km}] = 0 \qquad \text{if}\quad \{i,j \} \cap \{k, m \} = \varnothing
\end{gather*}
 and the exchange relations
\begin{gather*}
\beta_{ij} u_{jk} =u_{jk} \beta_{ik},\qquad \ell_{ij} u_{jk} =u_{jk} \ell_{ik} \qquad \text{if} \quad k \not= i,j.
\end{gather*}
It happens that in the f\/irst case, if $\beta=0$, then the (commutative)
algebra generated by additive Dunkl's elements and elementary symmetric
polynomials $\{e_k(X_n) \}_{1 \le k \le n}$ (resp.\ multiplicative Dunkl's
elements) is isomorphic to the equivariant quantum cohomology ring
(resp.\ to the equivariant quantum $K$-theory ring) of the type~$A_{n-1}$
complete f\/lag variety. In the second case a geometric interpretation of the
algebra generated by Dunkl's elements is missing.

Our main objective in this section is to to describe (part of) relations
among Dunkl's element using def\/ining relations involve in the Def\/inition~\ref{definition6.5}
of the algebra $3QT_n({\boldsymbol{\beta}}, {\boldsymbol{\ell}}, {\boldsymbol{q}})$, under the
following constraints
\begin{gather*}
 \ell_{ij} = h_{\max(i,j)}, \qquad h_2,\ldots,h_n \quad \text{are all central}.
 \end{gather*}
Note, that except the case $\beta_{j}=\beta$ and $h_{i}=h_{j}$, $\forall\, i, j$,
our assumption violates the crossing relations between the elements
$\beta_{ij}$, $\ell_{ij}$ and $u_{j,k}$, but nevertheless allows
to compute explicitly (part of) relations among the Dunkl's elements. We
expect that an abstract algebra generated over $\Q[{\boldsymbol{\beta}},{\boldsymbol{h}}]$ by a set
of mutually commuting elements $\theta_1, \ldots,\theta_n$ and elementary
symmetric polynomials $\{e_k(X_n)\}_{1 \le k \le n}$ subject to the set of
relations descending from those for Dunkl's elements which were mentioned
above, has some interesting combinatorial/geometric interpretations. Below we
state some results concerning relations among Dunkl elements in the algebra
$3QT_n({\boldsymbol{\beta}}, {\boldsymbol{\ell}}, {\boldsymbol{q}})$.

\begin{Theorem}[cf.\ Theorem~\ref{theorem3.3}, Section~\ref{section3}]\label{theorem6.3}
Let $k \ge 1$ be an integer. There exist polynomials
\begin{gather*}
R_k({\boldsymbol{q}},{\boldsymbol{h}},z_1,\ldots,z_n) \in \Z[\beta,{\boldsymbol{q}},\{h_j-h_i \}_{1 \le i < j \le n}][Z_n],\\
 T_k(\beta,{\boldsymbol{h}},z_1,\ldots,z_n) \in \Z[\beta,{\boldsymbol{h}}]
[Z_n]^{{\mathbb{S}}_n}
\end{gather*}
 such that
\begin{gather*}
R_k({\boldsymbol{q}},{\boldsymbol{h}},z_1,\ldots,z_n) = e_k^{({\boldsymbol{q}}+{\boldsymbol{h}})}(z_1,\ldots,z_n) + \text{monomials~of~total~degree}\\
\hphantom{R_k({\boldsymbol{q}},{\boldsymbol{h}},z_1,\ldots,z_n) =}{} \le k-2 \qquad \text{w.r.t. variables $\{z_{i} \}_{1 \le i \le n}$},\\
T_k(\beta,{\boldsymbol{h}},z_1,\ldots,z_n) =e_k(z_1,\ldots,z_n) +\sum_{ j <k}
c_{j,k} e_j(X_n), c_{j,k} \in \Z[\beta,{\boldsymbol{h}}],\\
R_k(\theta_1,\ldots,\theta_n) =T_k(x_1,\ldots,x_n),
\end{gather*}
where $e_k^{({\boldsymbol{q}}+{\boldsymbol{h}})}(z_1,\ldots,z_n)$ denotes the multiparameter quantum
elementary polynomial corresponding to the set of parameters $\{({\boldsymbol{q}}+{\boldsymbol{h}})\}=
\{q_{ij}+h_j \}_{1 \le i < j \le n}$.
\end{Theorem}

It is not dif\/f\/icult to see that the unitarity and crossing conditions imply the
following relations
\begin{gather*}
[x_i+x_j,u_{kl}]=0=[x_i x_j,u_{kl}], \qquad [x_i^2,u_{kl}]=0
\end{gather*}
are valid for all indices $i \not= j$, $k \not= l$. As a consequence of these
relations one can deduce that the all symmetric polynomials $e_k(X_n):=
e_k(x_1,\ldots,x_n)$, $k=1,\ldots,n$, belong to the {\it center} of the
algebra $3QT_n({\boldsymbol{q}},{\boldsymbol{h}})$, and therefore one has $[\theta_i,e_k(X_n)]=0$ for
all~$i$ and~$k$. Let us denote by $QH(\beta,{\boldsymbol{h}})$ a commutative
subalgebra in the algebra $3QT_n(\beta,{\boldsymbol{h}})$ generated by the elementary
symmetric polynomials $\{e_k(X_n) \}_{1 \le k \le n}$ and the Dunkl elements
$\{ \theta_i \}_{1 \le i \le n}$. It is an interesting {\it problem} to give
 a geometric/cohomological interpretation of the commutative algebra
$QH(\beta,{\boldsymbol{h}})$. We don't know any geometric interpretation of that
commutative algebra, except the special case~\cite{KM3}
\begin{gather}\label{equation6.2}
\beta =0, \qquad h_j=1, \quad \forall\, j, \qquad q_{ij}:= q_{i} \delta_{i+1,j}.
\end{gather}

\begin{Proposition}[\cite{KM3}]\label{proposition6.3}
Under assumptions \eqref{equation6.2}, the algebra $QH(0,{\bf 0})$ isomorphic to the
equivariant quantum cohomology $QH^{*}_{T}({\cal{F}}l_n)$ of the complete
f\/lag variety ${\cal{F}}l_n$.
\end{Proposition}

\begin{Examples}\label{examples6.1}
Let us list the relations among the Dunkl elements in the algebra
$3QT_n(\beta ,{\boldsymbol{h}})$ for $n=3,4$, and $\beta_j=\beta$, $\forall j$.
\begin{alignat*}{3}
& (1)\ && e_1(\theta_1,\ldots,\theta_n)= e_1(X_n) + {n \choose 2} \beta,&\\
& (2)\ && e_2^{({\boldsymbol{q}}+{\boldsymbol{h}})}(\theta_1,\ldots,\theta_n)= e_2(X_n)+ (n-1) \beta e_1(X_n) + \frac{ n (n-1)(n-2)(3 n-1)}{24} \beta^2, \quad n \ge 3,&\\
& (3) \ &&e_3^{({\boldsymbol{q}}+{\boldsymbol{h}})}(\theta_1,\theta_2,\theta_3)=e_3(X_3)+h_3 \beta, &\\
&&& e_3^{({\boldsymbol{q}} +{\boldsymbol{h}})}(\theta_1,\theta_2,\theta_3,\theta_4) = e_3(X_4)+ \beta e_2(X_4) +2 \beta^2 e_1(X_4) + 6 \beta^3+ \beta (h_3+3 h_4), &\\
& (4)\ && e_4^{({\boldsymbol{q}}+{\boldsymbol{h}})}(\theta_1,\theta_2,\theta_3,\theta_4) + \beta (h_4-h_3) \theta_4= e_4(X_4)+\beta h_4 e_1(X_4) +5 \beta^2 h_4.
\end{alignat*}
Note that $\frac{ n (n-1)(n-2)(3 n-1)}{24} =s(n-2,2)=e_2(1,2,\ldots,n-1)$ is
equal to the Stirling number of the f\/irst kind.
\end{Examples}

\begin{Conjecture}\label{conjecture6.2} The polynomial $R_k({\boldsymbol{q}},{\boldsymbol{h}},Z_n)$, see Theorem~{\rm \ref{theorem2.3}}, can be
written as a polynomial in the variables $\{h_{ij}:=h_j-h_i,\,1 \le i < j \le n,\, z_1,\ldots,z_n, \, \beta, \, q_{ij}, \, 1 \le i < j \le n \}$ with
nonnegative coefficients.
\end{Conjecture}
\begin{Exercises}[Pieri formula in the algebra $3T_n(0,h)$,~\cite{KM3}]\label{exercises6.1}
Assume that $\beta=0$ and $h_2= \cdots =h_n=h$, and denote by $\theta_i^{(n)}$, $i=1,\ldots,n$ the Dunkl elements~\eqref{equation6.1} in the algebra $3T_n(0,h)$.
Show that
\begin{gather*}
e_k\big(\theta_1^{(n)},\ldots, \theta_m^{(n)}\big) = \sum_{r \ge 0} (-h)^{r} N(m-k,2~r) \left\{ \sum_{S \subset [1,m] \atop I=\{i_a\},\, J=\{j_a\}} X_{S}
u_{i_{1},j_{1}} \cdots u_{i_{|I|},j_{|J|}} \right\},
\end{gather*}
where
\begin{gather*}
 N(a,2b)= (2 b-1) !! {a+2 b \choose 2 b},
 \end{gather*}
$X_{S}= \prod\limits_{ s \in S} x_{s}$, and the second summation runs over triples
of sets $\{ S,I,J \}$ such that $S \subset [1,m]$, $I \subset [1,m]
{\setminus} S $,
$|I|+|S|+2 r=k$, $|I|=|J|$, $ 1 \le i_a < m < j_a \le n$ and
$j_1 \le \cdots \le j_{|I|} $.
\end{Exercises}

\subsubsection[Algebra $3QT_n(\beta,h)$, generalized unitary case]{Algebra $\boldsymbol{3QT_n({\boldsymbol{\beta}},{\boldsymbol{h}})}$, generalized unitary
case}\label{appendixA.3.2}

Let ${\boldsymbol{\beta}}=(\beta_1,\ldots,\beta_{n-1})$, ${\boldsymbol{h}}=(h_2,\ldots,
h_n)$ and $\{q_{ij} \}_{1 \le i < j \le n}$ be collections of mutually
commuting parameters as in the previous section.
As before we def\/ine the Dunkl elements~$\theta_i$, $i=1,\ldots,n$, by the
formula~\eqref{equation6.1}. It is necessary to {\it stress} that the Dunkl
elements $\{\theta \}_{1 \le i \le n}$ {\it do not commute} in the
algebra $3QT_n(\boldsymbol{\beta},{\boldsymbol{h}})$~but satisfy a noncommutative
analogue of the relations displayed in Theorem~\ref{theorem6.3}. Namely, one needs to
replace the both elementary polynomials $e_k(Z_n)$ and the quantum
multiparameter elementary polynomials $e_k^{({\boldsymbol{q}})}(Z_n)$ by its
noncommutative versions. Recall that the noncommutative elementary polynomial
$\underline{e}_k(Z_n)$ is equal to
\begin{gather*}
\sum_{1 \le j_1 < j_2 < \cdots < j_k \le n} z_{j_{1}}~z_{j_{2}} \cdots z_{j_{k}}
\end{gather*}
and the noncommutative quantum multiparameters elementary polynomial
$\underline{e}_k^{({\boldsymbol{q}})}(Z_n)$ is equal to
\begin{gather*}
 \sum_{\ell} \sum_{1 \le i_1 < \cdots < j_{\ell} \le n \atop i_1 < j_1,\ldots, i_{\ell} < j_{\ell}} \underline{e}_{k-2\ell}(Z_{\overline{I \cup J}})
\prod_{a=1}^{\ell} u_{i_{a},j_{a}},
\end{gather*}
where $I=(i_1, \ldots, i_{\ell})$, $J=(j_1,\ldots,j_{\ell})$ should be distinct
elements of the set $\{1, \ldots, n \}$, and $Z_{\overline{I \cup J}}$
denotes set of variables $z_a$ for which the subscript $a$ is neither one of~$i_m$ nor one of the~$j_m$.

\begin{Example}\label{example6.1}
\begin{gather*}
\underline{e}_2^{({\boldsymbol{q}}+{\boldsymbol{h}})}(\theta_1,\ldots,\theta_n)=
e_{2}(X_n)+ \left(\sum_{j=1}^{n-1} \beta_j\right) e_1(X_n)+ \sum_{1 \le a < b \le n-1}
a b \beta_a \beta_b,\\
\underline{e}_3^{({\boldsymbol{q}}+{\boldsymbol{h}})}(\theta_1,\theta_2,\theta_3,\theta_4) + (\beta_3-\beta_1)(\theta_3 \theta_4+q_{34}+h_4+\beta_2(\theta_1+\theta_2))+(\beta_3-\beta_2)((\theta_1+\theta_2)\theta_4\\
\qquad{} +q_{14}+q_{24}+2 h_4+\beta_1 \theta_3)=
 e_3(X_4)+\beta_3e_2(X_4)+(\beta_1 \beta_3+\beta_2 \beta_3+\beta_3^2-\beta_1 \beta_2) e_1(X_4)\\
 \qquad{} + (3 \beta_3^2 -\beta_1 \beta_2)(\beta_1+2 \beta_2)+\beta_1(h_3+h_4)+2 \beta_2 h_4,\\
\underline{e}_4^{({\boldsymbol{q}}+{\boldsymbol{h}})}(\theta_1,\theta_2,\theta_3,\theta_4)+(\beta_2 h_4-\beta_1 h_3)\theta_4+ h_4 (\beta_2-\beta_1) \theta_3\\
\qquad{}= e_4(X_4)+\beta_2 h_4 e_1(X_4)+\beta_2 h_4 (2 \beta_2+3 \beta_3).
\end{gather*}
\end{Example}

\begin{Project}[noncommutative universal Schubert polynomials]\label{project6.1}
Let $w \in \mathbb{S}_n$ be a permutation and $\mathfrak{S}_w(Z_n)$ be the corresponding Schubert polynomial.

$(1)$ There exists a $($noncommutative$)$ polynomial $\mathfrak{Sh}_{w}(\{u_{ij} \}_{1 \le i < j \le n})$ with non-negative integer coefficients such that the
following identity
\begin{gather*}
\mathfrak{S}_w(\theta_1,\ldots,\theta_n) = \mathfrak{Sh}_w(\{u_{ij}\}_{1 \le i < j \le n})
\end{gather*}
holds in the algebra $3T_n^{(0)}$, where $\{ \theta_j \}_{1 \le j \le n}$ are
the Dunkl elements in the algebra $3T_n^{(0)}$.

$(2)$ There exist polynomials $R_w(\beta,{\boldsymbol{q}},{\boldsymbol{h}}, Z_n) \in \N [\beta,{\boldsymbol{q}},{h_j-h_i}_{1 \le i < j \le n}] [Z_n]$ and $T_w(\beta,{\boldsymbol{h}},Z_n) \in
\Z[\beta,{\boldsymbol{h}}] [Z_n]$ such that the following identity
\begin{gather*}
R_w(\beta,{\boldsymbol{q}},{\boldsymbol{h}},\theta_1,\ldots,\theta_n)=T_w(\beta,{\boldsymbol{h}},X_n) + \mathfrak{Sh}_w(\{u_{ij}\}_{1 \le i <j \le n})
\end{gather*}
holds in the algebra $3QT_n(\beta,{\boldsymbol{h}})$.

$3)$ Let $r \in \Z_{\ge 2}$ and $N=n_1+\cdots+ n_r$, $n_j \in \Z_{\ge 1}$,
$\forall\, j$, be a composition of $N$, and set $N_j=n_1+\cdots +n_j$,
$j \ge 1$, $N_0=0$.
Eliminate the Dunkl elements $\theta_{N_{r-1}+1}^{(N)}, \ldots, \theta_{N}^{(N)}$ from the set of relations among the Dunkl elements $\theta_{1}^{(N)},\ldots,\theta_{N}^{(N)}$ in the algebra $3QT_n(\beta,{\boldsymbol{h}})$, by the
use of the degree $1,\dots,n_r$ relations among the former. As a result
one obtains a set consisting of~$N_{r-1}$ relations among the~$N_{r-1}$
elements
\begin{gather*}
 \theta_{j.k_{j}}^{(N)}:= e^{({\boldsymbol{q}})}_{k_{j}}\big(\theta_{N_{j-1}+1}^{(N)},\ldots,
\theta_{N_{j}}^{(N)}\big), \qquad 1 \le k_j \le n_j, \qquad 1 \le j \le r-1.
\end{gather*}

Give a geometric interpretation of the commutative subalgebra
$QH_{n_{1},\ldots,n_{r}}(\beta,{\boldsymbol{h}}) \subset 3QT_n(\beta,{\boldsymbol{h}})$
generated by the set of elements $\theta_{j,k_{j}}^{(N)}$, $1 \le k_j \le n_j$,
$j=1,\ldots, r-1$.
\end{Project}

\subsection{Koszul dual of quadratic algebras and Betti numbers}\label{appendixA.4}

Let $k$ be a f\/ield of zero characteristic,
$F^{(n)}:=k\langle x_1,\ldots,x_n\rangle=\bigoplus_{j \ge 0} F_{j}^{(n)}$ be the free
associative algebra generated by $\{ x_i,\, 1 \le i \le n \}$. Let
$A=F^{(n)}/I$ be a quadratic algebra, i.e., the ideal of relations $I$ is
generated by the elements of degree~$2$, $I \subset F_{2}^{(n)}$.
Let $F^{(n)*}=\operatorname{Hom}(F_n,k)=\bigoplus_{j \ge 0}F_{j}^{(n)*}$ with a~multiplication induced by the rule $fg(ab)=f(a)g(b)$,
$f \in F_i^{(n)*}$, $g \in F_j^{(n)*}$, $ a \in F_i^{(n)}$, $b \in F_j^{(n)}$.
Let $I_{2}^{\bot}= \{f \in F_2^{(n)*}, f(I_2)=0 \}$, and denote by~$I^{\bot}$
 the two-sided ideal in~$F^{(n)*}$ generated by the set $I_{2}^{\bot}$.

\begin{Definition} \label{definition6.6} {\it The Koszul $($or quadratic$)$ dual $A^{!}$} of a quadratic
algebra $A$ is def\/ined to be $ A^{!}:=F^{(n)*}/I^{\bot}$.
\end{Definition}

The Koszul dual of a quadratic algebra $A$ is a quadratic algebra and
$(A^{!})^{!}=A$.

\begin{Examples} \label{examples6.2}\quad

(1) Let $A=F^{(n)}$ be the free associative algebra, then the quadratic dual
\begin{gather*}
A^{!}=k\langle y_1,\ldots,y_n\rangle /(y_iy_j, \, 1 \le i,j \le n).
\end{gather*}

(2) If $A=k[x_1,\ldots,x_n]$ is the ring of polynomials, then
\begin{gather*}
A^{!}=k[y_1,\ldots,y_n]/([y_i,y_j]_{-}, \, 1 \le i,j \le n),
\end{gather*}
 where we put by def\/inition $[y_i,y_j]_{-}=y_iy_j+y_jy_i$ if $i \not= j$,
and $[y_i,y_i]_{\_}=y_i^2$.

(3) Let $A=F^{(n)}/(f_1,\ldots,f_r)$, where
$f_i=\sum\limits_{1 \le j, k \le n} a_{ijk}x_j x_k$,
$i=1,\ldots, r$ are linear
independent elements of degree~$2$ in~$F^{(n)}$. Then the
quadratic dual of $A$ is equal to the quotient algebra
$A^{!}= k\langle y_1,\dots,y_n\rangle /J$, where the ideal $J=\langle g_1,\ldots,g_s\rangle $,
$s= n^2-r$, is generated by elements
$g_m=\sum\limits_{1 \le j,k \le n}b_{mjk} y_j y_k$. The coef\/f\/icients
$b_{mjk}$, $m=l,\ldots,s$, $1 \le j,k \le n$, can be def\/ined from the system of
linear equations $\sum\limits_{1 \le j,k \le n}a_{ijk} b_{mjk}=0$, $i=1,\ldots, r$,
$m=1,\ldots,s$ (cf.~\cite[Chapter~5]{MY-b}).
\end{Examples}

Let $A=\bigoplus_{j \ge 0}A_j$ be a graded f\/initely generated algebra over
f\/ield $k$.

\begin{Definition} \label{definition6.7} {\it The Hilbert series} of a graded algebra $A$ is def\/ined to be
 the generating function of dimensions of its homogeneous components:
 ${\rm Hilb}(A,t)=\sum\limits_{k \ge 0} \dim A_{k} t^k$.

{\it The Betti--Poincar\'{e} numbers} $B_{A}(n,m)$ of a graded algebra $A$
are def\/ined to be $B_A(i,j):= \dim \operatorname{Tor}_{i}^{A}(k,k)_{j}$.
{\it The Poincar\'{e}} series of algebra~$A$ is def\/ined to be the generating
function for the Betti numbers: $P_{A}(s,t):=\sum\limits_{i \ge 0,j \ge 0}B_{A}(i,j)s^it^j$.

Let $B$ is a $k$-module and $A$ is a $B$-module. The {\it Betti number}
$\beta_{ij}^{B(A)}$ of $A$ over $B$ is the rank of the free module $B[-j]$ the
$i$th module of a minimal resolution of $A$ over $B$ that is
$\beta_{ij}^{B}(A) = \dim_{k} {\rm Ext}_{i}^{B}(A,k)_{j}$. The graded Betti
series of~$A$ over~$B$ is the generating function
\begin{gather*}
{\rm Betti}_{B}(A,x,y):= \sum_{i \in \N,\, j \in \Z} \beta_{ij}^{B}(A) x^{i} y^{-j}
\in \Z[y,y^{-1}] [[x]].
\end{gather*}
\end{Definition}

\begin{Definition}\label{definition6.8} A quadratic algebra $A$ is called {\it Koszul} if\/f the Betti numbers
$B_{A}(i,j)$ are equal to zero unless $i=j$.
\end{Definition}

It is well-known that ${\rm Hilb}(A,t)P_{A}(-1,t)=1$, and
a quadratic algebra $A$ is {\it Koszul}, if and only if
$B_{A}(i,j)=0$ for all $i \not= j$. In this case
${\rm Hilb}(A,t)~{\rm Hilb}(A^{!},-t)=1$.

\begin{Example}\label{example6.2} Let $F_n^{(0)}$ be a quotient of the free associative algebra $F_n$
over f\/ield $k$~with the set of generators $\{x_1,\ldots,x_n \}$ by the
two-sided ideal generated by the set of elements $\{x_1^2,\ldots,x_n^2 \}$. Then the algebra $ F^{(0)}_f n$
 is {\it Koszul}, and ${\rm Hilb}\big(F_n^{(0)},t\big)={1+t \over 1-(n-1)t}$.
\end{Example}
We refer the reader to a nice written book by A.~Polishchuk and L.~Positselski~\cite{PP} to read more widely in the theory of quadratic
algebras, see also~\cite{MY}.

\subsection[On relations in the algebra $Z_n^{0}$]{On relations in the algebra $\boldsymbol{Z_n^{0}}$}\label{appendixA.5}

Let us def\/ine algebra $Z_n^{0}$ to be the subalgebra in $3T_{n}^{0}$ generated
by the elements $u_{i,n}$, $1 \le i \le n-1$. It is clear that
$Z_{n}^{0}$ is a ${\mathbb{S}}_{n-1}$-module,and well-known~\cite{FP} that if
one sets ${\rm Hilb}(Z_{k}^{0},t):=Z_{k}(t)$, then
\begin{gather*} {\rm Hilb}\big(3T_n^{(0)},t\big) = \prod_{k=2}^{n} Z_{k}(t).
\end{gather*}

There exists a natural action of algebra $3T_{n-1}^{0}$ on that $Z_{n}^{0}$.
 To def\/ine it, it's convenient to put $x_{i}:= u_{i,n}$,~$1\le i \le n-1$.
\begin{Definition}[cf.~\cite{K3} and Section~\ref{section2.3.4}]\label{definition6.9}
Def\/ine operators $\nabla_{i,j},
1 \le i < j \le n-1$, which act on $Z_{n}^{0}$, by the following rules
\begin{itemize}\itemsep=0pt
\item $\nabla_{i,j}(x_k)=0$ if $k \not= i,j$,

\item $\nabla_{i,j}(x_i)=x_{i}x_{j},\nabla_{i,j}(x_j)=-x_{j}x_{i}$,

\item twisted Leibniz rule:
\begin{gather*}
\nabla_{i,j}(x \cdot y)= \nabla_{i,j}(x) \cdot y + s_{i,j}(x) \cdot \nabla_{i,j}(y)
\end{gather*}
 for $x, y\in Z_{n}^{0}$ and all $1 \le i < j \le n-1$.
Here $s_{i,j} \in {\mathbb{S}}_{n-1}$
denotes the transposition that interchanges $i$ and $j$ and f\/ixes each
$ k \not= i,j$.
\end{itemize}
\end{Definition}

\begin{Proposition} \label{proposition6.4}
The operators $\nabla_{i,j}$, $1 \le i < j \le n-1$, satisfy all
defining relations of algebra~$3T_{n-1}^{0}$.
\end{Proposition}
In particular, the operators $\nabla_{i,j}$, satisfy the Coxeter and
Yang--Baxter relations:
\begin{itemize}\itemsep=0pt
\item Yang--Baxter relations:
\begin{gather*}
\nabla_{i,j}\nabla_{i,k}\nabla_{j,k}=
\nabla_{j,k}\nabla_{i,k}\nabla_{i,j},
\end{gather*}

\item Coxeter relations. Let $\nabla_j=\nabla_{j,j+1}$,
$ 1 \le j \le n-2$, then
\begin{gather*}
\nabla_j\nabla_{j+1}\nabla_j=\nabla_{j+1}\nabla_j\nabla_{j+1},\qquad
 [\nabla_i,\nabla_j]=0 \qquad \text{if}\quad |i-j| \ge 2.
 \end{gather*}
 \end{itemize}
Therefore, for each $w \in {\mathbb{S}}_{n-1}$ one can def\/ine the operator
$\nabla_{w}=\nabla_{a_{1}}\cdots\nabla_{a_{l}}$, where the sequence
$(a_1,\dots,a_l)$ is a reduce decomposition of the element~$w$.

Denote by ${\cal R}_n$ the {\it kernel} of the epimorphism $\iota \colon
Z_n \longrightarrow F_{n-1}$ given by $\iota(u_{k,n})=x_k$,
where $F_{n-1}:=\Q\langle x_1,\dots,x_{n-1}\rangle$ denotes the free associative algebra
generated by the elements $x_1,\dots$, $x_{n-1}$. There exists
 the decomposition
${\cal R}_n= \bigoplus_{k \ge 2} {\cal R}_{n,k}$, where ${\cal R}_{n,k}$
denotes the degree~$k$ part of~${\cal R}_n$. We denote by $r_{n,k}$ the
dimension of the space
${\cal R}_{n,k}/ \sum\limits_{j=1}^{n-1} (x_{j,n}{\cal R}_{n,k-1}+
{\cal R}_{n,k-1}x_{j,n})$, and put $r_n:=(r_{n,2},r_{n,3},\dots)$.

\begin{Example} \label{example6.3}
\begin{gather*}
r_3=(2,1), \qquad r_4=(3,3,2), \qquad r_5=(4,6,8,6,3),\\
r_6=(5,10,20,30,39,40,39,30,20,10,4).
\end{gather*}
\end{Example}

\begin{Remark} \label{remark6.1} The same formulas for the action of $\nabla_{i,j}$ on
$Z_n^{0}$ given in Def\/inition~\ref{definition6.9}, def\/ine an action of operators
$\nabla_{i,j}$ on the free algebra $F_{n-1}$. In this way we obtain a
representation of the algebra $3T_{n-1}$ on that $F_{n-1}$, cf.\ Section~\ref{section2.3.4}.
\end{Remark}

Let us denote by ${\widehat F}_n$ the quotient of the free associative
algebra $F_n= \langle x_1,\ldots,x_n\rangle $ by the two-sided ideal generated by the
elements $ \{x_i^2 x_j-x_j x_i^2, \, 1 \le i,j \le n \}$. It is not dif\/f\/icult
to see that the operators $\nabla_{i,j}$, $1 \le i < j \le n$, def\/ine a
representation of the algebra $3T_n^{0}$ on that ${\widehat F}_n$. Note that
\begin{gather*}
 {\widehat F}_n \backsimeq F_{n-1} \otimes \Z[y_1,y_2,\ldots,y_n],
 \end{gather*}
where $\deg (y_1)=1$, $\deg (y_j)=2$, $j=2,\ldots,n$. Therefore,
\begin{gather*}
{\rm Hilb}\big({\widehat F}_n,t\big)= {1 \over (1-t)(1-(n-1)t)(1-t^2)^{n-1}}.
\end{gather*}

\begin{Conjecture} \label{conjecture6.3}
The kernel ${\cal R}_n$ coincides with the two-sided ideal in the
free algebra $F_{n-1}$ generated by elements of the form
$\prod\limits_{k=1}^{s}\nabla_{i_k,j_k}(x_a^2)$ for some positive integers~$s$ and
$1 \le a \le n-1$.
\end{Conjecture}

In other words, the {\it all} relations in the algebra $Z_n^{0}$ are
consequence of the following relations $u\nabla_{w}(x_1^{2})=0$ for
some $u,w \in {\mathbb{S}}_{n-1}$.

\begin{Challenge} \label{challenge6.1}\quad
\begin{enumerate}\itemsep=0pt
\item[(1)] Compute the numbers $r_{n,k}$.

\item[(2)] Prove (or disprove) that there exists a positive integer $k_{\max}:=
k_{\max}^{(n)}$ such that $r_{n,k_{\max}} \not= 0$, but $r_{n,k}=0$ for all
integers $k > k_{\max}^{(n)}$.

\item[(3)] These examples suggest that there might be exist a certain symmetry
$r_{n,k}=r_{n,k_{\max}-k+2}$, if $3 \le k < k_{\max}$, between the numbers
$r_{n,k}$, and moreover, $r_{n,k_{\max}}=r_{n,2}-1$.
If so, how to explain these
properties of the numbers $r_{n,k}$?
\end{enumerate}
\end{Challenge}

 We {\it expect} that if $n \ge 4$, then $k_{\max}^{(n)}=2{n-2 \choose \lbrack (n-2)/2
\rbrack }$.

\begin{Example}[cyclic relations in the algebra $Z_n^{0}$] \label{example6.4} The following
relation
\begin{gather*}
\prod_{j=1}^{n-1} \nabla_{n-j,n-j+1}\big(x_1^2\big)=
\sum_{i=1}^{n}
x_i \left( \prod_{a=i+1}^{n}x_a \prod_{a=1}^{i-1}x_a \right) x_i
\end{gather*}
holds in the free algebra $F_n$. Therefore in the algebra $Z_n^{0}$ one has the
following cyclic relation of the degree $n$ and length $n-1$:
\begin{gather*}
\sum_{i=1}^{n-1}
x_i \left( \prod_{a=i+1}^{n-1}x_a \prod_{a=1}^{i-1}x_a \right) x_i = 0.
\end{gather*}
\end{Example}

 If $n \ge 5$, then by applying to monomials of the form
$\prod\limits_{j=2}^{n-1} \nabla_{n-j,n-j+1}(x_1^2)$ the action of either operators
$\nabla_{a,n-1}$, $2 \le a \le n-3$, or those
 $\nabla_{a,b}$, $1 \le a \le b-2 \le n-4$, new, more complicated relations in the algebra $Z_n^{0}$,
 i.e., non-cyclic relations, can {\it appear}.
 These are relations of the
length~$2n$ and degree $n+1$ in the algebra $Z_n^{0}$. Conjecturally all relations in the algebra $Z_n^{0}$ can be obtained by this method.

\begin{Proposition}\label{proposition6.5}
\begin{gather*}
r_{n,k}=(k-2)! {n-1 \choose k-1}, \qquad 2 \le k \le 5,\\
r_{n,6}= 4!{n-1 \choose 5}+3{n-1 \choose 4},\qquad
r_{n,7}=5!{n-1 \choose 6}+40{n-1 \choose 5}, \\
r_{n,8}=6!{n-1 \choose 7}+430{n-1 \choose 6}+ 39{n-1 \choose 5}.
\end{gather*}
\end{Proposition}

\subsubsection[Hilbert series ${\rm Hilb}\big(3T_n^{0},t\big)$ and
${\rm Hilb}\big(\big(3T_n^{0}\big)^{!},t\big)$: Examples]{Hilbert series $\boldsymbol{{\rm Hilb}\big(3T_n^{0},t\big)}$ and
$\boldsymbol{{\rm Hilb}\big(\big(3T_n^{0}\big)^{!},t\big)}$: Examples\footnote{All computations in this section were performed by using the
computer system {\tt Bergman}, except computations of
${\rm Hilb}(3T_{6}^{0},t)$ in degrees from twelfth till f\/ifteenth.
The last computations were made by J.~Backelin, S.~Lundqvist and J.-E.~Roos
from Stockholm University, using the computer algebra system {\tt aalg}
mainly developed by S.~Lundqvist.}}

\begin{Examples}\label{examples6.3}
\begin{gather*}
{\rm Hilb}\big(3T_{3}^{0},t\big)=[2]^2[3] ,\qquad {\rm Hilb}\big(3T_{4}^{0},t\big)=[2]^2[3]^2[4]^2,\qquad
 {\rm Hilb}\big(3T_{5}^{0},t\big)=[4]^4[5]^2[6]^4 , \\
 {\rm Hilb}\big(3T_{6}^{0},t\big) =(1,15,125,765,3831,16605,64432,228855,755777,2347365,6916867,\\
\hphantom{{\rm Hilb}\big(3T_{6}^{0},t\big) =(}{}
19468980 , 52632322,137268120,346652740,850296030,\dots)\\
\hphantom{{\rm Hilb}\big(3T_{6}^{0},t\big)}{}
={\rm Hilb}\big(3T_{5}^{0},t\big)(1,5,20,70,220, 640,1751,4560,11386,27425,64015,\\
\hphantom{{\rm Hilb}\big(3T_{5}^{0},t\big) =(}{}
145330,321843 , 696960,1478887,3080190,\dots),\\
 {\rm Hilb}\big(3T_{7}^{0},t\big)={\rm Hilb}\big(3T_{6}^{0},t\big)(1,6,30,135,560,2190,8181,29472,103032,\\
 \hphantom{{\rm Hilb}\big(3T_{7}^{0},t\big)=(}{}
351192 , 1170377,\dots),\\
 {\rm Hilb}\big(3T_{8}^{0},t\big)={\rm Hilb}\big(3T_{7}^{0},t\big)(1,7,42,231,1190,5845,27671, 127239,571299,2514463 ,
\\
 {\rm Hilb}\big(\big(3T_{3}^{0}\big)^{!},t\big)(1-t)=(1,2,2,1) , \qquad
 {\rm Hilb}\big(\big(3T_{4}^{0}\big)^{!},t\big)(1-t)^2=(1,4,6,2,-5,-4,-1) ,
\\
 {\rm Hilb}\big(\big(3T_{5}^{0}\big)^{!},t\big)(1-t)^2=(1,8,26,40,19,-18,-22,-8,-1) ,\\
 {\rm Hilb}\big(\big(3T_{6}^{0}\big)^{!},t\big)(1-t)^3=
(1,12,58,134,109,-112,-245,-73,68,50,12,1) , \\
 {\rm Hilb}\big(\big(3T_{7}^{0}\big)^{!},t\big)(1-t)^3=
(1,\!18,136,545,1169,1022,\!-624,\!-1838,\!-837,312,374,123,\!18,1).
\end{gather*}
We {\it expect} that ${\rm Hilb}((3T_{n}^{0})^{!},t)$ is a rational function
with the only pole at $t=1$ of order $[n/2]$, and the polynomial
${\rm Hilb}((3T_{n}^{0})^{!},t)(1-t)^{[n/2]}$ has degree equals to $[5n/2]-4$, if
$n \ge 2$.
\end{Examples}

\subsection{Summation and Duality transformation
formulas \cite{NK}}\label{appendixA.6}

{\bf Summation formula.} Let $a_1+\cdots+a_m=b$. Then
\begin{gather*}
 \sum_{i=1}^{m} [a_i] \left( \prod_{j \not= i}{{[x_i-x_j+a_j] \over [x_{i}-x_{j}]}} \right)
{{[x_i+y-b] \over [x_i+y]}} =
[b] \prod_{1 \le i \le m} {{[y+x_{i}-a_{i}] \over [y+x_i]}}.
\end{gather*}
{\bf Duality transformation, case $N=1$.} Let $a_1+\cdots+a_m=b_1+\cdots+b_n$. Then
\begin{gather*}
\sum_{i=1}^{m} [a_i] \prod_{j \not= i}{{[x_i-x_j+a_j] \over [x_{i}-x_{j}]}}
\prod_{1 \le k \le n}{{[x_i+y_k-b_k] \over [x_i+y_k]}}\\
\qquad{} =
 \sum_{k=1}^{n} [b_k] \prod_{l \not= k} {{[y_k-y_l+b_l] \over [y_k-y_l]}}
\prod_{1 \le i \le m}{{[y_k+x_i-a_i] \over [y_k+x_i]}}.
\end{gather*}

\subsection*{Acknowledgments}

I would like to express my deepest thanks to Professor Toshiaki Maeno for
many years fruitful collaboration.
I'm also grateful to Professors
Yu.~Bazlov, I.~Burban, B.~Feigin, S.~Fomin, A.~Isaev, M.~Ishikawa, M.~Noumi,
B.~Shapiro and Dr.~Evgeny Smirnov for fruitful discussions on dif\/ferent
stages of writing~\cite{K}.

My special thanks are to
Professor Anders Buch for sending me the programs for
computation of the $\beta$-Grothendieck and double $\beta$-Grothendieck
polynomials. Many results and examples in the present paper have been checked
by using these programs, and
Professor Ole Warnaar (University of Queenslad) for
 a warm hospitality and a kind interest and fruitful discussions of some
results from~\cite{K} concerning hypergeometric functions.

These notes represent an update version of Section~\ref{section5} of my notes~\cite{K}, which have been designed as an extended version of~\cite{Kir}, and
are based on my talks given at\footnote{To save place I~will mention only the Universities and Institutions
which I visited and gave talks/lectures, starting from the year 2010. I~want
to thank the all Universities and Institutions which I~visited, for warm hospitality and f\/inancial support.}
 \begin{itemize}\itemsep=0pt
\item The Simons Center for Geometry and Physics, Stony Brook
University, USA, January 2010;
\item Department of Mathematical Sciences at the Indiana University~--
Purdue University Indianapolis (IUPUI), USA, {\it Departmental Colloquium},
January 2010;

\item The Research School of Physics and Engineering, Australian National
University (ANU), Canberra, ACT 0200, Australia, April 2010;

\item The Institut de Math\'{e}matiques de Bourgogne, CNRS U.M.R.\
5584, Universit\'{e} de Bourgogne, France, October 2010;

\item The School of Mathematics and Statistics University of Sydney,
NSW 2006, Australia, November 2010;

\item The Institute of Advanced Studies at NTU, Singapore, {\it 5th Asia-Pacif\/ic Workshop on Quantum Information Science in conjunction
with the Festschrift in honor of Vladimir Korepin}, May 2011;

\item The Center for Quantum Geometry of Moduli Spaces, Faculty of
Science, Aarhus University, Denmark, August 2011;

\item The Higher School of Economy (HES), and The Moscow State
University, Russia, November 2011;

\item The Research Institute for Mathematical Sciences (RIMS), the
Conference {\it Combinatorial representation theory}, Japan, October 2011;

\item The Korean Institute for Advanced Study (KIAS), Seoul, South Korea,
May/June, 2012, August 2014;

\item The Kavli Institute for the Physics and Mathematics of the
Universe (IPMU), Tokyo, August 2013, August 2015;

\item The University of Queensland, Brisbane, Australia,
October--November 2013;

\item The University of Warwick, the University of Nottingham and the University of York, Clay Mathematics Institute, Oxford, United Kingdom, May/June 2015.
\end{itemize}

I would like to thank Professors Leon Takhtajan and Oleg Viro
 (Stony Brook), J{\o}rgen E. Andersen (CGM, Aarhus University), Bumsig Kim
(KIAS, Seoul), Vladimir Matveev (Universit\'e de Bourgogne),
Vitaly Tarasov (IUPUI, USA), Vladimir Bazhanov (ANU), Alexander Molev
(University of Sydney), Sergey Lando (HES, Moscow), Sergey Oblezin
(Nottingham, UK), Maxim Nazarov (York, UK), Kyoji Saito (IPMU,
Tokyo), Kazuhiro Hikami (Kyushu University), Reiho Sakamoto (Tokyo
University of Science), Junichi Shiraishi (University of Tokyo) for
invitations and hospitality during my visits of the Universities and the
Institutes listed above.

Part of results stated in Section~\ref{section2},~{\bf II} has been obtained during my
visit of the University of Sydney, Australia. I would like to thank Professors
A.~Molev and A.~Isaev for the keen interest and useful comments on my paper.

\LastPageEnding

\end{document}